\documentclass[reqno,english,11pt]{amsart}

\usepackage{amsmath,amsfonts,amssymb,graphicx,amsthm,url}
\usepackage[noadjust]{cite}
\usepackage{stmaryrd}
\usepackage{mathrsfs,booktabs,tabularx}
\usepackage{xifthen,xcolor,tikz,setspace}
\usetikzlibrary{decorations.pathmorphing,patterns,shapes,calc,decorations}
\usetikzlibrary{decorations.pathreplacing}
\usepackage{mathtools,upgreek,bbm}
\usepackage[final]{showkeys} 

\definecolor{refkey}{gray}{.75}
\definecolor{labelkey}{gray}{.5}
\usepackage[algo2e,boxed,vlined,algoruled]{algorithm2e}
\usepackage{comment}
\usepackage[shortlabels]{enumitem}

\usepackage[letterpaper]{geometry}
\usepackage[colorinlistoftodos]{todonotes}
\presetkeys{todonotes}{inline, color=green}{}
\usepackage{accents}
\usepackage[algo2e,boxed,vlined,algoruled]{algorithm2e}

\usepackage[small]{caption}
\usepackage[colorlinks=true]{hyperref}
\colorlet{DarkGreen}{green!50!black}
\colorlet{DarkGray}{gray!60!black}

\numberwithin{equation}{section}

\renewcommand{\restriction}{\mathord{\upharpoonright}}
\renewcommand{\epsilon}{\varepsilon}

\newcommand{\one}{\mathbbm{1}}
\newcommand\norm[1]{\lVert#1\rVert}

\usepackage{color}
\definecolor{refkey}{gray}{.5}
\definecolor{labelkey}{gray}{.5}
\definecolor{light}{gray}{.9}

\usepackage{soul}

\usepackage[capitalise]{cleveref}

\newtheorem{theorem}{Theorem}[section]
\newtheorem*{theorem*}{Theorem}
\newtheorem{lemma}[theorem]{Lemma}

\newtheorem{claim}[theorem]{Claim}
\crefname{claim}{Claim}{Claims}
\newtheorem{proposition}[theorem]{Proposition}

\crefname{step}{Step}{Steps}

\newtheorem{observation}[theorem]{Observation}
\newtheorem{fact}[theorem]{Fact}
\newtheorem{corollary}[theorem]{Corollary}

\theoremstyle{definition}{
	
	\newtheorem{definition}[theorem]{Definition}
	
	\newtheorem*{definition*}{Definition}

	\newtheorem{remark}[theorem]{Remark}
	\newtheorem*{remark*}{Remark}
	
}

\newcommand{\E}{\mathbb E}

\renewcommand{\P}{\mathbb P}

\renewcommand{\H}{\mathbb H}
\newcommand{\R}{\mathbb R}
\newcommand{\Z}{\mathbb Z}

\newcommand{\sfh}{\mathsf h}
\newcommand{\sfv}{\mathsf v}
\newcommand{\sfw}{\mathsf w}
\newcommand{\sfx}{\mathsf x}
\newcommand{\sfy}{\mathsf y}
\newcommand{\sfz}{\mathsf z}

\newcommand{\GFF}{\textsc{gff}\xspace}
\newcommand{\ZGFF}{\ensuremath{\Z\textsc{gff}}\xspace}
\newcommand{\SOS}{\textsc{sos}\xspace}
\newcommand{\Dim}{\textsc{d} }
\newcommand{\hatpi}{{\widehat\pi}}
\newcommand{\hatq}{{\widehat q}}
\newcommand{\tildeq}{{\widetilde {\mathsf q}}}
\newcommand{\hatZ}{{\widehat Z}}
\newcommand{\tildeZ}{{\widetilde {\mathsf Z}}}
\newcommand{\partialvtx}{\partial_{\mathtt v}}

\newcommand{\cA}{\ensuremath{\mathcal A}}
\newcommand{\cB}{\ensuremath{\mathcal B}}
\newcommand{\cC}{\ensuremath{\mathcal C}}
\newcommand{\cD}{\ensuremath{\mathcal D}}
\newcommand{\cE}{\ensuremath{\mathcal E}}
\newcommand{\cF}{\ensuremath{\mathcal F}}
\newcommand{\cG}{\ensuremath{\mathcal G}}
\newcommand{\cH}{\ensuremath{\mathcal H}}
\newcommand{\cI}{\ensuremath{\mathcal I}}

\newcommand{\cP}{\ensuremath{\mathcal P}}

\newcommand{\cR}{\ensuremath{\mathcal R}}
\newcommand{\cS}{\ensuremath{\mathcal S}}

\newcommand{\cW}{\ensuremath{\mathcal W}}

\newcommand{\cY}{\ensuremath{\mathcal Y}}
\newcommand{\cZ}{\ensuremath{\mathcal Z}}
\newcommand{\llb }{\llbracket}
\newcommand{\rrb }{\rrbracket}

\newcommand{\fu}{\mathfrak{u}}
\newcommand{\fs}{\mathfrak{s}}
\newcommand{\fF}{\mathfrak{F}}

\newcommand{\fI}{\mathfrak{I}}
\newcommand{\fL}{\mathfrak{L}}

\newcommand{\fR}{\mathfrak{R}}
\newcommand{\fT}{\mathfrak{T}}
\newcommand{\fP}{\mathfrak{P}}

\newcommand{\sfA}{{\ensuremath{\mathsf A}}}

\newcommand{\sfC}{{\ensuremath{\mathsf C}}}

\newcommand{\sfL}{{\ensuremath{\mathsf L}}}

\newcommand{\sfR}{{\ensuremath{\mathsf R}}}
\newcommand{\sfS}{{\ensuremath{\mathsf S}}}
\newcommand{\sfU}{{\ensuremath{\mathsf U}}}
\newcommand{\sfW}{{\ensuremath{\mathsf W}}}
\newcommand{\sfu}{{\ensuremath{\mathsf u}}}

\newcommand{\sB}{{\ensuremath{\mathscr B}}}
\newcommand{\sC}{{\ensuremath{\mathscr C}}}

\newcommand{\sE}{{\ensuremath{\mathscr E}}}

\newcommand{\sG}{{\ensuremath{\mathscr G}}}

\newcommand{\sL}{{\ensuremath{\mathscr L}}}

\newcommand{\sN}{{\ensuremath{\mathscr N}}}

\newcommand{\fD}{\mathfrak{D}}
\newcommand{\fU}{\mathfrak{U}}

\newcommand{\f}{{\ensuremath{\mathbf f}}}

\newcommand{\bE}{{\ensuremath{\mathbf E}}}

\newcommand{\bd}{{\ensuremath{\mathbf d}}}

\newcommand{\bY}{{\ensuremath{\mathbf Y}}}

\newcommand{\bP}{{\ensuremath{\mathbf P}}}

\renewcommand{\epsilon}{\varepsilon}
\DeclareMathOperator{\var}{Var}
\DeclareMathOperator{\dist}{dist}
\DeclareMathOperator{\diam}{diam}

\DeclareMathOperator{\sgn}{sgn}

\newcommand{\Bdd}{{\mathsf{Bdd}}}

\newcommand{\Low}{{\mathsf{Low}}}
\newcommand{\Cpts}{{\mathsf{Cpts}}}

\renewcommand{\d}{\mathrm{d}}

\newcommand{\n}{{\vec{\mathsf{n}}}}
\renewcommand{\o}{{\mathsf{o}^*}}

\newcommand{\tv}{{\textsc{tv}}}

\newcommand{\btl}{\blacktriangleleft}
\newcommand{\btr}{\blacktriangleright}

\makeatletter
\crefname{step}{Step}{Steps}
\crefname{case}{Case}{Cases}
\newcommand{\superimpose}[2]{%
	{\ooalign{$#1\@firstoftwo#2$\cr\hfil$#1\@secondoftwo#2$\hfil\cr}}}

\newcommand{\sbullet}{%
	\hbox{\fontfamily{lmr}\fontsize{.4\dimexpr(\f@size pt)}{0}\selectfont\textbullet}}

\makeatother

\title{The limiting law of the Discrete Gaussian level lines}

\author{Joseph Chen}
\address{J.\ Chen\hfill\break
	Courant Institute\\ New York University\\
	251 Mercer Street\\ New York, NY 10012, USA.}
\email{jlc871@courant.nyu.edu}

\author{Eyal Lubetzky}
\address{E.\ Lubetzky\hfill\break
	Courant Institute\\ New York University\\
	251 Mercer Street\\ New York, NY 10012, USA.}
\email{eyal@courant.nyu.edu}

\begin{document}
	
	\begin{abstract}
		Consider the $(2+1)$\Dim Discrete Gaussian (\ZGFF, integer-valued Gaussian free field) model in an $L\times L$ box above a hard floor. Bricmont, El-Mellouki and Fr\"ohlich~(1986) established that, at low enough temperature, this random surface exhibits entropic repulsion: the floor propels the average height to be poly-logarithmic in $L$. The second author, Martinelli and Sly (2016) showed that, for all but exceptional values of $L$, the surface has a plateau whose height concentrates on an explicit integer $H(L)$, and fills nearly the full square. It was conjectured there that the boundary of this plateau---the top level line of the surface---should have random  fluctuations of $L^{1/3+o(1)}$.
		
		We confirm this conjecture of [LMS16] and further recover the limiting law of the top level line: there exists an explicit sequence $N=L^{1-o(1)}$ such that the distance of the top level line from $I$, the interval of length $N^{2/3}$ centered along the side boundary, converges, after rescaling it by $N^{1/3}$ and the width of the interval by $N^{2/3}$, to a Ferrari--Spohn diffusion. In particular, the level-line fluctuations at, say, the center of $I$, have a limit law involving the Airy function rescaled by $N^{1/3}$. This gives the first example of one of the $(2+1)$\Dim $|\nabla \phi|^p$ models (approximating 3\Dim Ising and crystal formation) where the conjectured limit law of its level lines is confirmed (\ZGFF is the case $p=2$).
		
		More generally, we find the joint limit law of any finite number of top level lines: rescaling their distances from the side boundary, each by its $(N_n^{2/3},N_n^{1/3})$, yields a product of Ferrari--Spohn laws. These new results extend to the full universality class of $|\nabla\phi|^p$ models for any fixed $p>1$.
	\end{abstract}
	
	\maketitle
	\vspace{-0.3in}
	
	\section{Introduction}
	The $(2+1)$\Dim Discrete Gaussian model, also known as the integer-valued Gaussian free field and denoted here \ZGFF, is a random surface model extensively studied in the context of the roughening transition in crystals (see the work of Chui and Weeks~\cite{ChuiWeeks76} in 1976, and the related models in~\cite{BCF51} dating back to the 1950's). It is dual to the Villain XY model~\cite{Villain75}, and as such undergoes a Kosterlitz--Thouless phase transition (see~\cite{Kosterlitz2016})---one of two models (along with Solid-On-Solid) where this transition was established by Fr\"ohlich and Spencer in their celebrated works~\cite{FrohlichSpencer81a,FrohlichSpencer81b}.
	
	For $\beta>0$ (the inverse-temperature, which in our context will be taken fixed and large enough), the $(2+1)$\Dim \ZGFF model with a floor (or a hard wall) at height $0$ is a probability distribution over functions that assign nonnegative integer heights to the sites of the square grid $\Lambda=\llb 1,L\rrb^2$. Writing $x\sim y$ for the nearest-neighbor relation in $\Z^2$, the probability of $\phi:\Lambda\to\Z_+$ is given by
	\begin{equation}\label{eq:pi-def}
		\pi^0_\Lambda(\phi) \propto \exp\Big(-\beta \sum_{x\sim y}|\phi_x-\phi_y|^2\Big)\,,
	\end{equation}
	with zero boundary conditions ($\phi_x=0$ for all $x\notin\Lambda$). Define $\hatpi^0_\Lambda$ as the analogue of $\pi^0_\Lambda$ in the no-floor setting (i.e., when $\phi$ accepts values in $\Z$ as opposed to $\Z_+$), and define the infinite-volume measure~$\hatpi_\infty$ as the weak limit of $\hatpi^0_\Lambda$ as $L\to\infty$. 
	The aforementioned \ZGFF phase transition, occurring in $\phi\sim \hatpi_\infty$ at
	a critical $\beta_{\textsc{r}}$ (empirically, $\beta_{\textsc{r}}\approx0.665$),  can be demonstrated at $\phi_o$ 
	as follows: for $\beta\leq\beta_{\textsc{r}}$ (delocalized regime) the surface is rough, in that $\lim_{L\to\infty} \var(\phi_o)=\infty$, while for $\beta>\beta_{\textsc{r}}$ (localized regime) the surface is rigid, in that $\var(\phi_o)=O(1)$ (rigidity, moreover with exponential tails of $|\phi_o|$, was shown in~\cite{BrandenbergerWayne82} at large~$\beta$; roughness at small~$\beta$ was established in~\cite{FrohlichSpencer81a,FrohlichSpencer81b};  monotonicity in $\beta$ was shown in \cite{AHPS21} and continuity of the phase transition was established in~\cite{Lammers22}). 
	
	We focus on $\phi\sim\pi^0_\Lambda$ in the regime of $\beta$ large enough, where most $x\in\Lambda$ would have~$\phi_x=0$ under the no-floor measure $\hatpi^0_\Lambda$, yet the floor in $\pi^0_\Lambda$ induces a nontrivial surface (see \cref{fig:zgff-illustration}). 
	This entropic repulsion effect was identified in a pioneering work of Bricmont, El-Mellouki and Fr\"ohlich~\cite{BEF86}: despite the penalizing boundary conditions, the surface is propelled to height at least $c \sqrt{\log L}$, where it gains entropy from extra (permitted once at this height) downward fluctuations.
	
	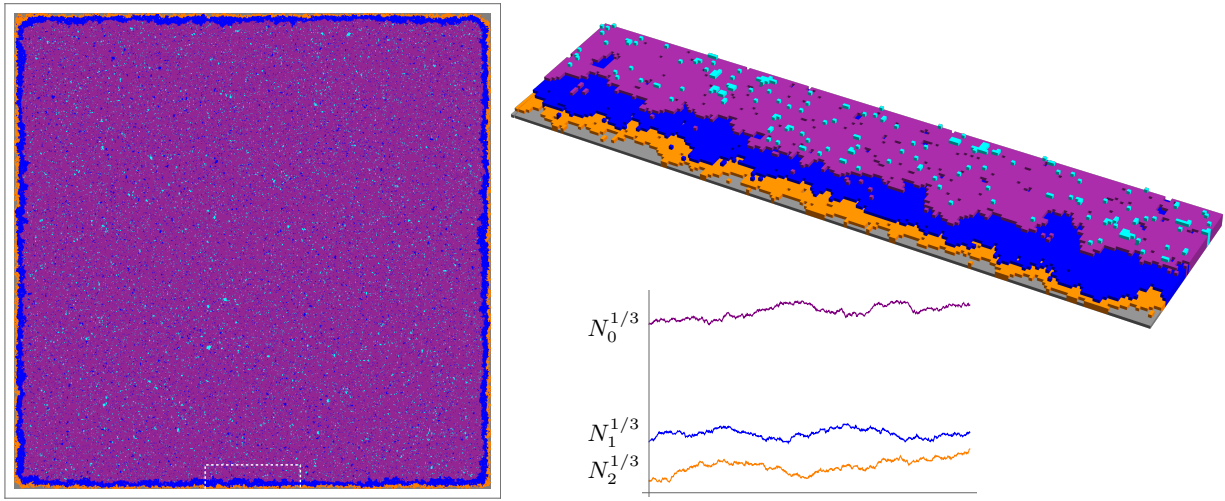
\begin{figure}
		\vspace{-0.15in}
		\begin{tikzpicture}
			\node at (0,-0.2) {\includegraphics[width=0.36\textwidth]{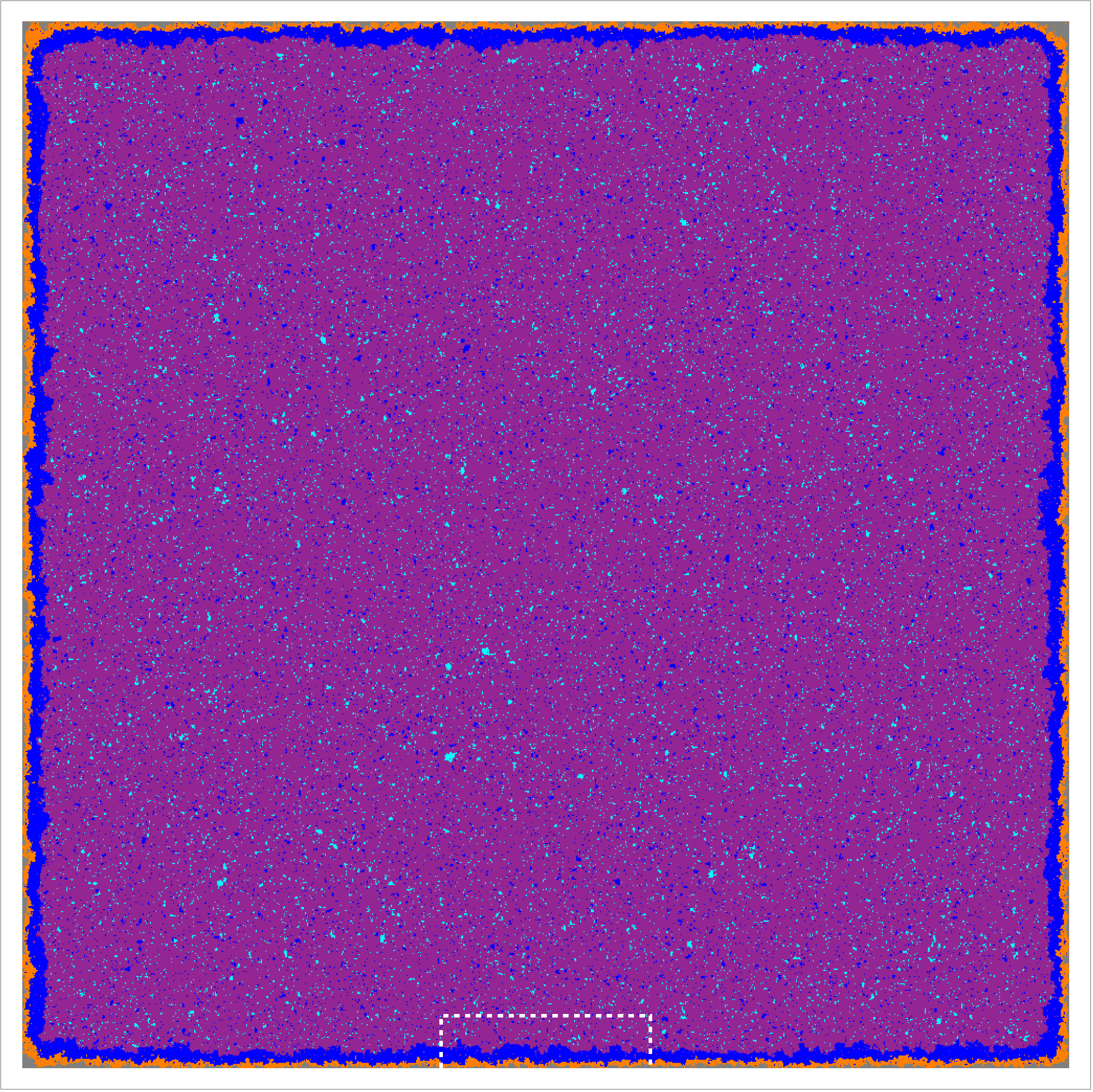}};
			
			\node at (8.1,1){\includegraphics[width=0.5\textwidth]{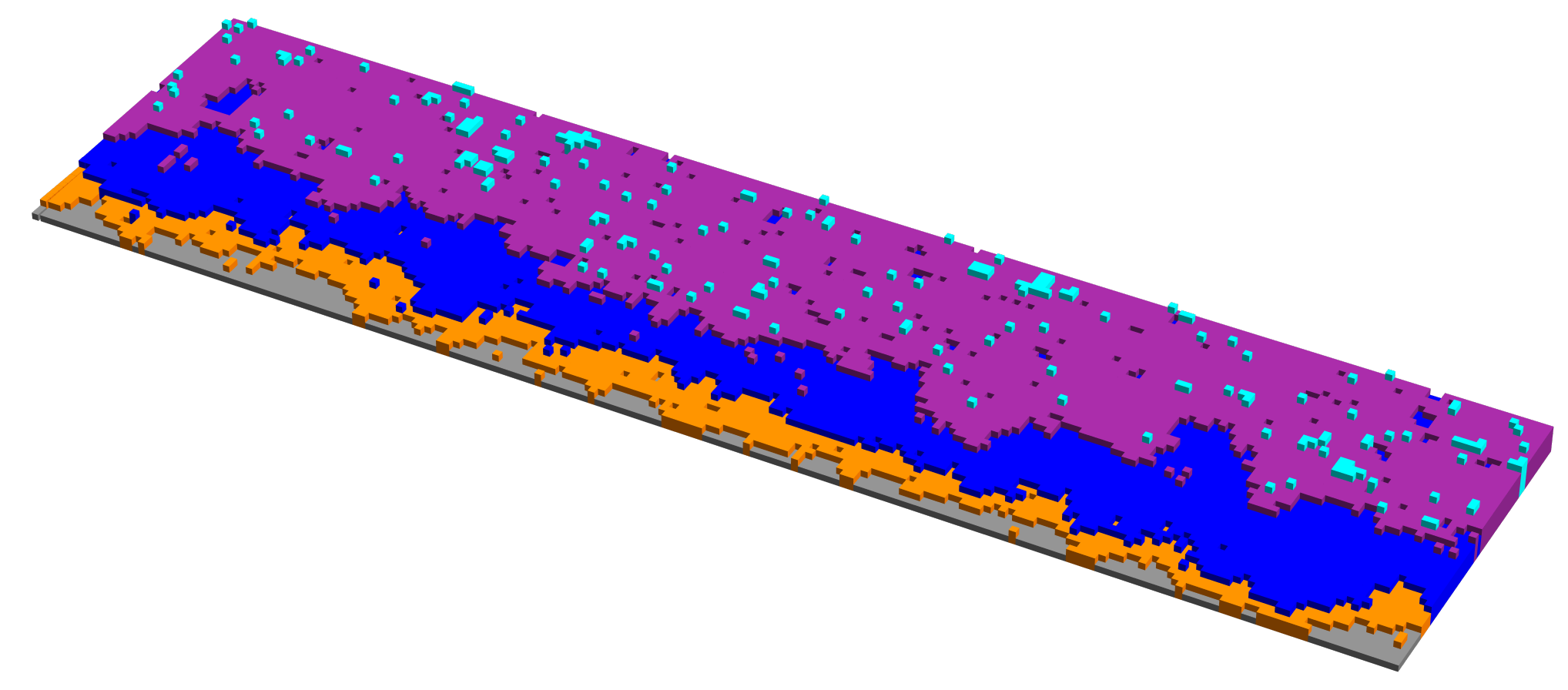}};
			\node at (7.4,-1.9) {\includegraphics[width=0.27\textwidth]{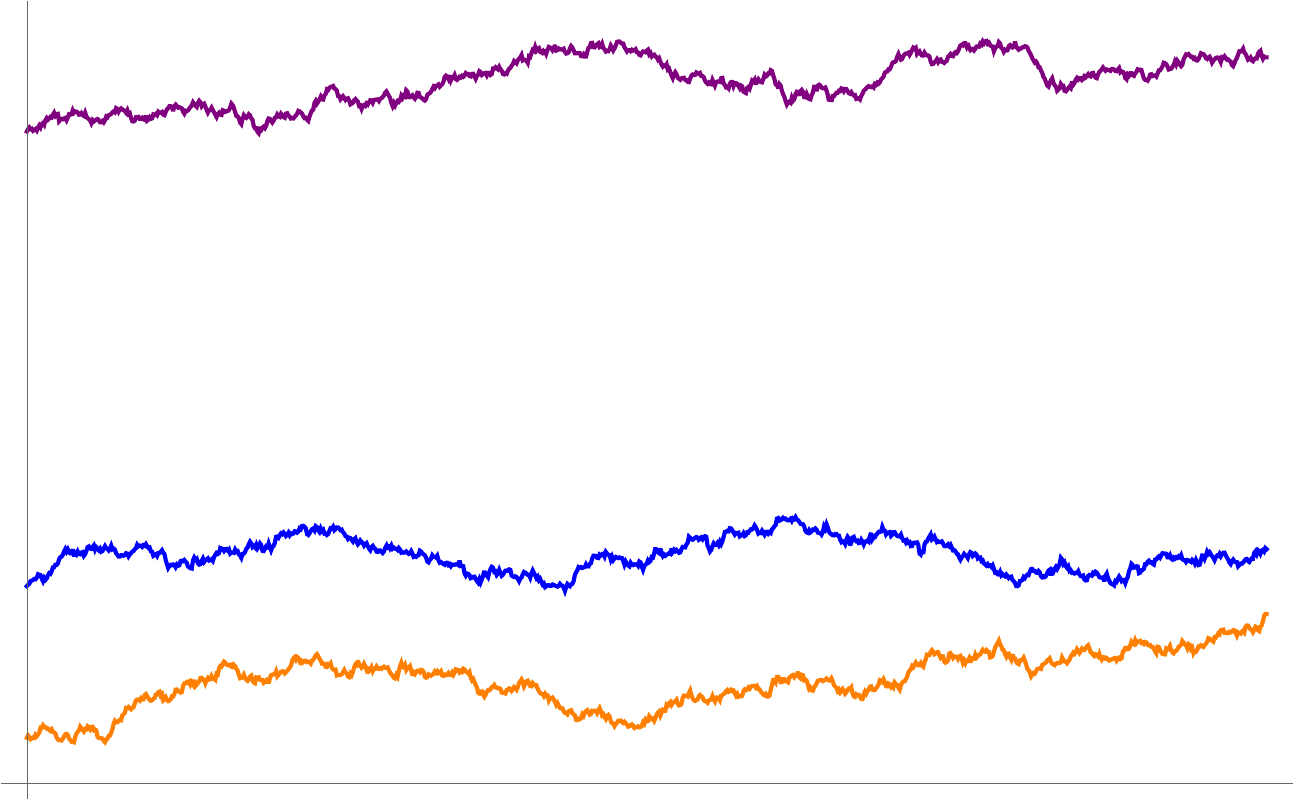}};
			\node [font=\footnotesize] at (4.8,-1.) {$N_1^{1/3}$};
			\node [font=\footnotesize] at (4.8,-2.4) {$N_2^{1/3}$};
			\node [font=\footnotesize] at (4.8,-2.9) {$N_3^{1/3}$};
			
		\end{tikzpicture}
		\vspace{-0.16in}
		\caption{Illustration of the low temperature \ZGFF on $\llb 1,L\rrb^2$, where we look at the law of the top level lines in a rectangle (magnified on right) along the center of the bottom side. On bottom right: independent Ferrari--Spohn diffusions, the limiting law of these level lines.}
		\vspace{-0.2in}
        \label{fig:zgff-illustration}
	\end{figure}
	The study of entropic repulsion in the paper above considered two closely-related models: the \ZGFF and Solid-On-Solid (\SOS), where the term $|\phi_x-\phi_y|^2$ in \cref{eq:pi-def} is replaced by $|\phi_x-\phi_y|$. Caputo et al.~\cite{CLMST12,CLMST14,CLMST16} made significant progress in characterizing the shape of the \SOS surface above a floor at low temperature. They showed that the macroscopic level lines form a sequence of nested loops, with the height of the top one concentrating on one of the two values $\lfloor\frac1{4\beta}\log L\rfloor-1,\lfloor\frac1{4\beta}\log L\rfloor$.
	These works further identified the deterministic scaling limits of the top level lines in $[0,1]^2$: these are the boundaries of the shapes formed by dilating the associated Wulff shape to explicit radii, then taking the union of all their translations in $[0,1]^2$. So, the limits are curved near the corners of the box (where the dilated Wulff shape is visible), and overlap the boundary in intervals surrounding the center-sides (the ``flat portions'' of the limit), where the random fluctuations are of interest.

	The top level line in \SOS was shown in~\cite{CLMST16} to have random fluctuations of at most $L^{1/3+o(1)}$ from the flat portions of its limit. Key to that was showing that the law of the top level line in the relevant region resembles that of a random walk, conditioned to be nonnegative, and penalized exponentially in the area below it. The latter is known~\cite{ISV15} to have $L^{1/3}$ fluctuations, and moreover its scaling limit is known to have the law of a Ferrari--Spohn diffusion~\cite{FerrariSpohn05}, defined below (in fact, the continuous analogue---a Brownian excursion tilted by its area---is equivalent via a Girsanov transformation to Brownian motion staying above a parabolic barrier, which was studied in~\cite{FerrariSpohn05}). 
	It was recently established~\cite{CKL24} that the order of the \SOS top level line fluctuations is at least~$L^{1/3}$. This is expected to be the correct order, but bounding the fluctuations from above is challenging---in \SOS (and also in the \ZGFF), there is a diverging number of level lines pushing against each other. Showing that this interaction does not change the fluctuation order at the top level remains an open problem in \SOS.  
    (If the boundary conditions in \SOS are manipulated to force just a single level line, then it would have a Ferrari--Spohn limit, as was shown in~\cite{CKL24} via adapting the recent proof~\cite{IOSV21} of this limit for the 2\Dim Ising model with critical prewetting). The conjectured \SOS limit should, instead of Ferrari--Spohn, take after the line-ensemble of non-crossing random walks with geometric area tilts (see, e.g.,~\cite{BCG25,CaputoGanguly23,CIW19a,CIW19b,DLZ23,HKS25,Serio23} studying it in the discrete and continuous setting). 
    See \cref{sec:open-prob-related} for more on the open problem of these laws  in \SOS.   
	
	For \ZGFF, the second author, Martinelli and Sly~\cite{LMS16} showed that, as in the case of \SOS, the surface heights concentrate on two values $H,H+1$, and for ``most'' values of $L$ the surface forms a plateau of area $(1-\epsilon_\beta)L^2$ at height $H$ (for $\epsilon$  arbitrarily small as $\beta$ increases). In more detail, let 
	\begin{equation}
		\label{eq:H-def}
		H(L) := \max\Big\{h \,:\; \hatpi_\infty(\phi_o= h) \geq \frac{5\beta}{L} \Big\}\,.
	\end{equation}
	Further define the $h$ level line(s) of $\phi\sim\pi^0_\Lambda$, for each $h\geq 0$, as the collection of loops obtained from edges dual to nearest-neighbors $x\sim y$ such that $\phi_x < h$ and $\phi_y \geq h$ (for concreteness, in case $4$ dual-edges share an endpoint, one ``splits'' them into two loops along the \textsc{northeast} diagonal). Call a loop $\gamma$ \texttt{large} if its length is at least $\log L$ (one could even use a threshold of $(C/\beta)\log L$).
	\begin{theorem*}[\cite{LMS16}\footnote{The cutoff for \texttt{large} loops in \cite{LMS16} (called \emph{macroscopic} loops there) was $\log^2 L$, though the proofs hold also for a threshold of $\log L$ provided $\beta$ is large enough; see the footnote below the definition of macroscopic loops in \cite[\S1.2]{LMS16}. }]
		Fix $\beta$ large, and consider the \ZGFF $\phi\sim\pi^0_\Lambda$ per \cref{eq:pi-def} and $H(L)$ per \cref{eq:H-def}. Then with high probability, 
		there is a unique \texttt{large} $h$ level-line loop for each $h=0,\ldots,H$, and there are no \texttt{large} $h$ level-line loops for any $h\geq H+2$. This sequence of loops is nested; the loops for $h\leq H-1$ have area $(1-o(1))L^2$, and the loop for $h=H$ has area at least $(1-\epsilon_\beta)L^2$.
	\end{theorem*}
	As noted in \cite[Remark~1.5]{LMS16}, for all but an exceptional set\footnote{Precisely, for all side-lengths $L$ outside a set $\sB\subset \Z$ of zero logarithmic density ($\sum_{k\in\sB\cap\llb1,n\rrb}\frac1k = o(\log n)$).} of values of $L$, the area of the $H$ level-line loop is in fact $(1-o(1))L^2$, and there are no \texttt{large} $(H+1)$ level-line loops (making the one at height $H$ the top level line). An open problem left in that work (see \cite[\S1.5]{LMS16}) was to prove that this top level line should have $L^{1/3+o(1)}$ fluctuations away from the corners of $\Lambda$. 
	
	Our main result identifies the correct scale of the fluctuations of the top level line from an interval set at the center of the side boundary---which as conjectured, is indeed  $L^{1/3+o(1)}$, yet with a necessary $L^{o(1)}$ correction---and moreover recovers its Ferrari--Spohn limit law, defined as follows.
	Let $\mathsf{Ai}(x)$ be the Airy function (of the first kind), i.e., the solution to $y''(x) = x y$ with the initial condition $y=0$ at $x=\infty$. For $\sigma>0$, let $\varphi_{\sigma}(x) := \mathsf{Ai}((2/\sigma^2)^{1/3}x -\omega_1)$, where $\omega_1$ is the ``first'' zero of $\mathsf{Ai}$, that is, $\omega_1=\min\{x>0\,:\;\mathsf{Ai}(-x)=0\}$.
	The stationary Ferrari--Spohn diffusion we consider (see \cite{FerrariSpohn05,ISV15}) 
	is the diffusion $\mathsf{FS}_\sigma$ on $(0,\infty)$ with Dirichlet boundary condition at $0$ and generator
	\begin{align}\label{def:fs-generator}
		\mathsf{L}_{\sigma} = \frac{\sigma^2}2 \frac{\d^2}{\d x^2} +  \sigma^2\frac{\varphi_{\sigma}'(x)}{\varphi_{\sigma}(x)}\frac{\d}{\d x}\,.
	\end{align}
	This diffusion is ergodic and reversible with respect to the probability density $\frac{(2/\sigma^2)^{1/3}}{\mathrm{Ai'(-\omega_1)^2}}  \varphi_{\sigma}(x)^2\one_{\{x>0\}}$, 
	where the rescaled Airy function $\varphi_{\sigma}$ given above is also the first eigenfunction of the operator $\mathsf{L}_\sigma$.
	Our main result addresses side-lengths $L$ excluding an explicit set $\sB$ of integers of logarithmic density zero (see \cref{rem:Lh-B-concrete}), about which the plateau transitions from one height to the next.
	\begin{theorem}\label{thm:1}
		Fix $\beta>0$ large enough and consider $\phi\sim \pi^0_\Lambda$, the $(2+1)$\Dim \ZGFF model on $\Lambda=\llb 1,L\rrb^2$ for $L\in\Z\setminus\sB$ (with $\sB$ as in \cref{rem:Lh-B-concrete}) with a floor and zero boundary conditions. Set $H(L)$ per \cref{eq:H-def}, and
		consider the bottom boundary interval $I = \llb \frac{L}2-N^{2/3}, \frac{L}2 + N^{2/3}\rrb$ for  
		\begin{equation}\label{eq:N-def} N := \frac1{ \hatpi_\infty(\phi_o = H)}\quad(=L^{1-o(1)})\,.\end{equation}
		Let $\fL_1$ be the top \texttt{large} level line of $\phi$, and define its vertical distance from~$I$ via
		\[ \rho(x) = \min\{y\geq 0 \,:\; (\tfrac{L}2+x,y)\in\fL_1\}\quad\mbox{for}\quad -N^{2/3}\leq x \leq N^{2/3}\,.\] 
		Then $ Y_1(t) := N^{-1/3}\rho(t N^{2/3})$ converges weakly to the stationary law $\mathsf{FS}_\sigma$ on $[-1,1]$ as $L\to\infty$
        for a fixed $\sigma>0$.
        The analogous $\bar{Y}_1(t)$ w.r.t.\ $\bar{ \rho}(x)= \max\{y \leq \frac L2\,:\; (\frac{L}2+ x,y)\in\fL_1\}$ has $\|\bar{Y}_1-Y_1\|_\infty\xrightarrow[]{\textsc{p}}0$.
		
		More generally, for fixed $m$, let $\fL_1\subset\ldots\subset\fL_{m}$ be the top $m$ \texttt{large} level lines, and
		\begin{equation}\label{eq:Nn-def}
			N_n := \frac{1}{\hatpi_\infty(\phi_o = H+1-n)}\quad(=L^{1-o(1)})\qquad(n=1,\ldots,m)\,.
		\end{equation}
		Let $I_n = \llb \frac{L}2- N_n^{2/3},\frac{L}2+N_n^{2/3}\rrb$, and let $\rho_n(x) = \min\{y\geq 0 \,:\; (\frac{L}2+x,y)\in\fL_n\}$ denote the vertical distance of~$\fL_n$ from $I_n$. The joint law of $Y_n(t) := N_n^{-1/3}\rho_n(t N_n^{2/3})$ ($n=1,\ldots,m$) converges weakly to that of
		$m$ i.i.d.\ stationary Ferrari--Spohn diffusions $\mathsf{FS}_{\sigma}$ on~$[-1,1]$ for an explicit fixed~$\sigma>0$. The analogous $\bar{Y}_n(t)$ w.r.t.\ $\bar{\rho}_n(x)=\max\{y\leq \frac{L}2:(\frac{L}2+x,y)\in\fL_n\}$ have $\max_{n\leq m}\|\bar{Y}_n-Y_n\|_\infty\xrightarrow[]{\textsc{p}}0$.
	\end{theorem}

	\begin{remark}\label{rem:Nn-scales}
		The sequence $N=N_1$ from~\cref{eq:N-def} satisfies
		$  e^{-c\sqrt{\beta \log L / \log\log L}} \leq  N / L \leq 1 / (5\beta)$,
		and each of these two bounds gives the behavior of the fluctuations of $\fL_1$ for infinitely many values of $L$. Namely, the fluctuations of $\fL_1$ are of order $L^{1/3}$ for infinitely many values of $L$, yet they are $o(L^{-1/3})$ for infinitely many other values of $L$. As for $N_2\geq N_3\geq\ldots\geq N_m$ from \cref{eq:Nn-def}, while these scales are all $L^{1-o(1)}$ for any fixed $n$, they are of lower order: $N_n = N_{n-1} e^{-\Theta(\sqrt{\beta\log L/\log\log L})}$ (in particular, the fluctuations of $\fL_n$ for $n\geq 2$ have order $N_n^{1/3} = L^{1/3-o(1)}=o(L^{1/3})$).
	\end{remark}

	\begin{remark}\label{rem:Lh-B-concrete}
		As one increases the side-length $L$ of the box, once  the top level line concentrates on height $h$ (at some $L_h$), it will remain so until $L$ further increases by a factor of $\exp(c\beta \frac{h}{\log h})$. Precisely, the transition marking the onset of level $h+1$ occurs at $L\approx (4+\epsilon_\beta)\beta/\hatpi_\infty(\phi_o=h+1)$.
		As such, one can define the set $\sB$ of exceptional values of the side-lengths $L$ (see \cref{fig:exceptional_set})  via
		\begin{equation}\label{eq:exceptional-set-def} \sB = \bigcup_{h\geq 1} \llb \tfrac34 L_h, L_{h}\rrb \qquad\mbox{where}\qquad
		L_h = \lceil 5\beta /\hatpi_\infty(\phi_o=h)\rceil \qquad(h=1,2,\ldots)\,.\end{equation} 
		With this choice, $\sum_{k\in\sB\cap\llb1,n\rrb}\frac1k = O(\sqrt{\log n \log\log n}) = o(\log n)$ (i.e., zero logarithmic density), and the \ZGFF surface will concentrate on height $H(L)=h$ for $L\in\llb L_h,\ldots,\frac34 L_{h+1}\rrb$ ($h=1,2,\ldots$).
	\end{remark}

        \begin{remark}\label{rem:FS-on-R}
        Taking the intervals $I_n=\llb \frac{L}2- K N_n^{2/3},\frac{L}2+K N_n^{2/3}\rrb$ for any fixed $K>0$ (in lieu of $K=1$), one has convergence to i.i.d. Ferrari--Spohn diffusions $\mathsf{FS}_{\sigma}$ on~$[-K,K]$.
    \end{remark}

	\begin{figure}
		\vspace{-0.15in}
		\begin{tikzpicture}
			\node at (0,0) {\includegraphics[width=0.5\textwidth]{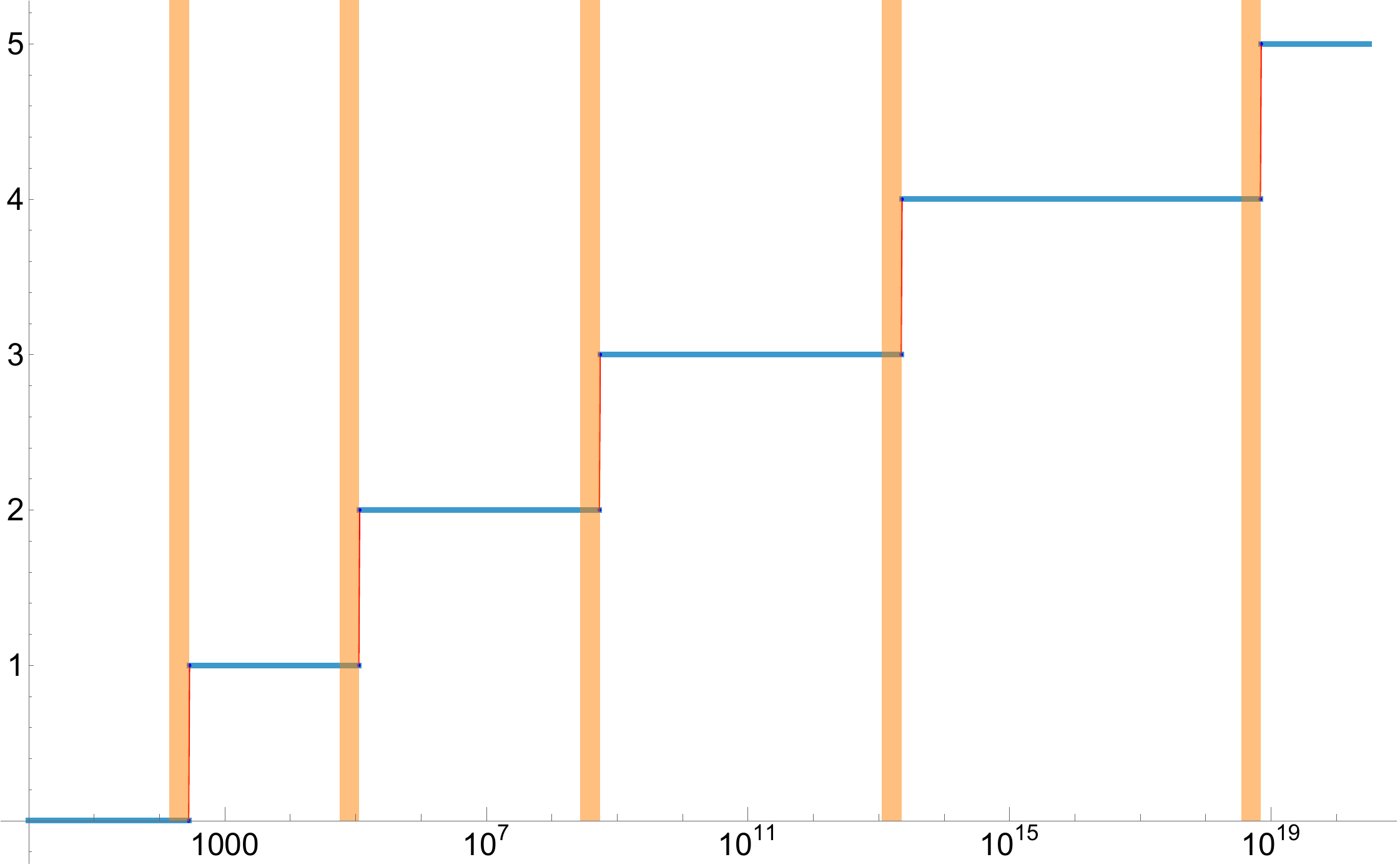}};
			\node[font=\tiny] at (-3.95,2.8) {$h$};
			\node[font=\tiny] at (4.3,-2.25) {$L$};
			
		\end{tikzpicture}
		\vspace{-0.15in}
		\caption{The exceptional set $\sB$ of values of $L$ as per \cref{rem:Lh-B-concrete}, highlighted in orange, which delimits the intervals $\llb L_h, \frac34 L_{h+1}\rrb$ where the top level line concentrates on height $h$. As depicted in this log-plot, the set $\sB$ has zero logarithmic density.}
		\vspace{-0.1in}
        \label{fig:exceptional_set}
	\end{figure}
	
	The new results extend to the family of  $|\nabla\phi|^p$ models, defined as follows. In lieu of \cref{eq:pi-def}, the probability of $\phi:\Lambda\to\Z_+$ (with $0$ boundary conditions and inverse-temperature $\beta>0$) is given by
	\begin{equation}\label{eq:pi-p-def}
		\pi^{(p),0}_\Lambda(\phi) \propto \exp\Big(-\beta \sum_{x\sim y}|\phi_x-\phi_y|^p\Big)\,,
	\end{equation}
	so that $p=2$ is the \ZGFF and $p=1$ is the \SOS model. It was shown in~\cite{LMS16} that, for all $p>1$, and all ``typical'' values of $L$, the surface is typically a plateau (with microscopic fluctuations) at a single deterministic height $H$ (that scales differently with $L$ for different values of $p$), as is the case for \SOS and \ZGFF. As for the limit law of the top level lines, the pictures in \SOS and \ZGFF are different (a Ferrari--Spohn limit in the latter, vs.\ a conjectured variant thereof in the former). 
    
    In what follows, put $\pi^{(p)}_\Lambda$ instead of $\pi^{(p),0}_\Lambda$ for brevity ($0$ boundary conditions by default), and let
    \begin{equation}
        \label{eq:gradphi-def-H-Nn}
        		H^{(p)}(L) := \max\Big\{h \,:\; \hatpi^{(p)}_\infty(\phi_o= h) \geq \frac{5\beta}{L} \Big\}\quad,\quad
                N^{(p)}_n := \frac{1}{\hatpi^{(p)}_\infty(\phi_o = H^{(p)}+1-n)}\quad(=L^{1-o(1)})\,,
    \end{equation}
    the analogues of $H$ and $N_n$ defined in \cref{eq:H-def,eq:Nn-def} for the $\ZGFF$.

\begin{figure}
    \centering
    \vspace{-0.1in}
    \begin{tikzpicture}
			\node at (0,0) {\includegraphics[width=0.45\textwidth]{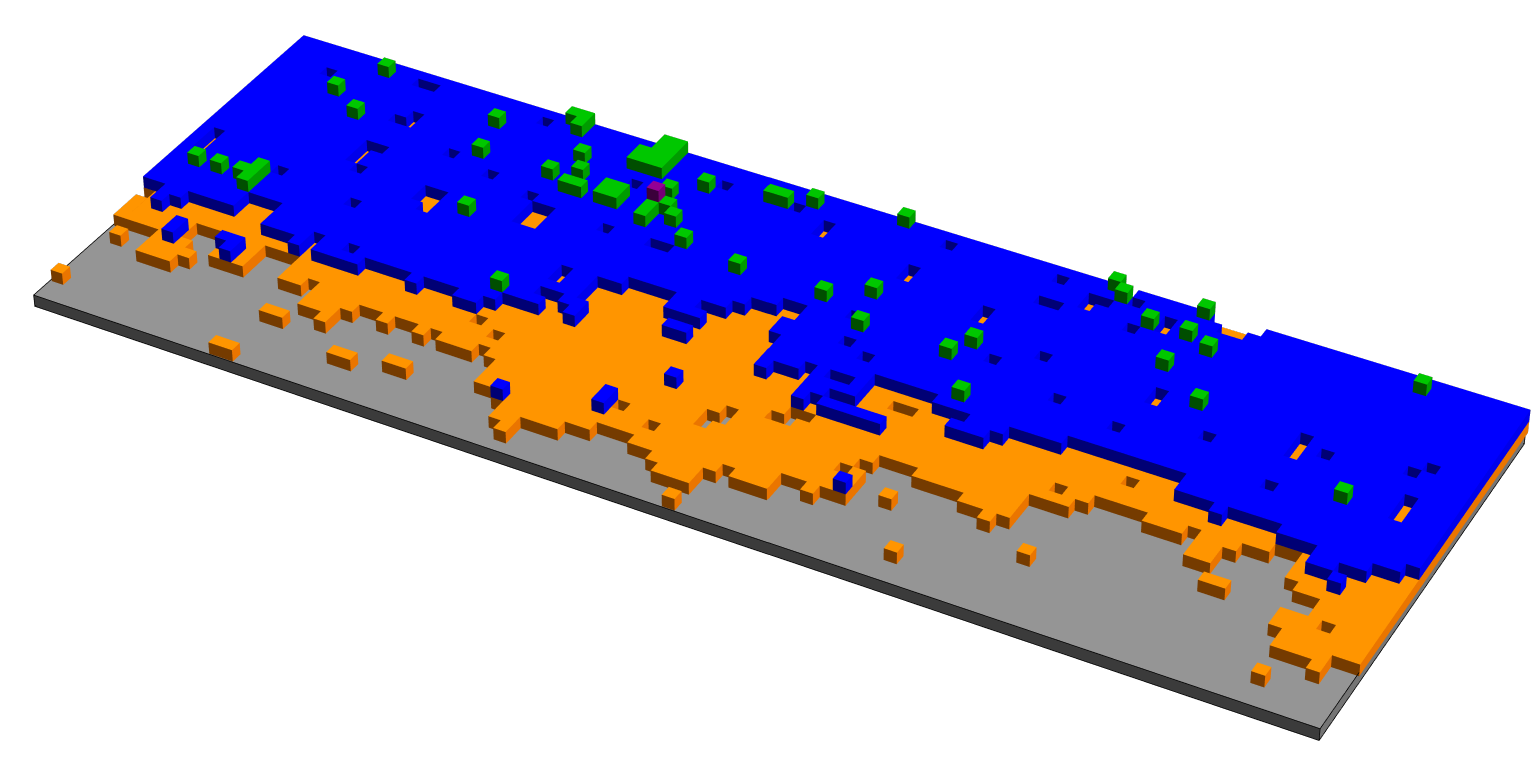}};

            \node [font=\footnotesize] at (-1,-1.2) {$p=1$ (\SOS)};
            \node [font=\scriptsize] at (-1,-1.6) {$N_{n+1}\asymp N_n$};

            \draw[thick,->, color=white] (1.15,-1.17) -- (1.68,-.39);
            \draw[thick,->, color=white] (1.25,-1.2) -- (1.64,-.62);

            \node at (7.5,-0.3) {\includegraphics[width=0.45\textwidth]{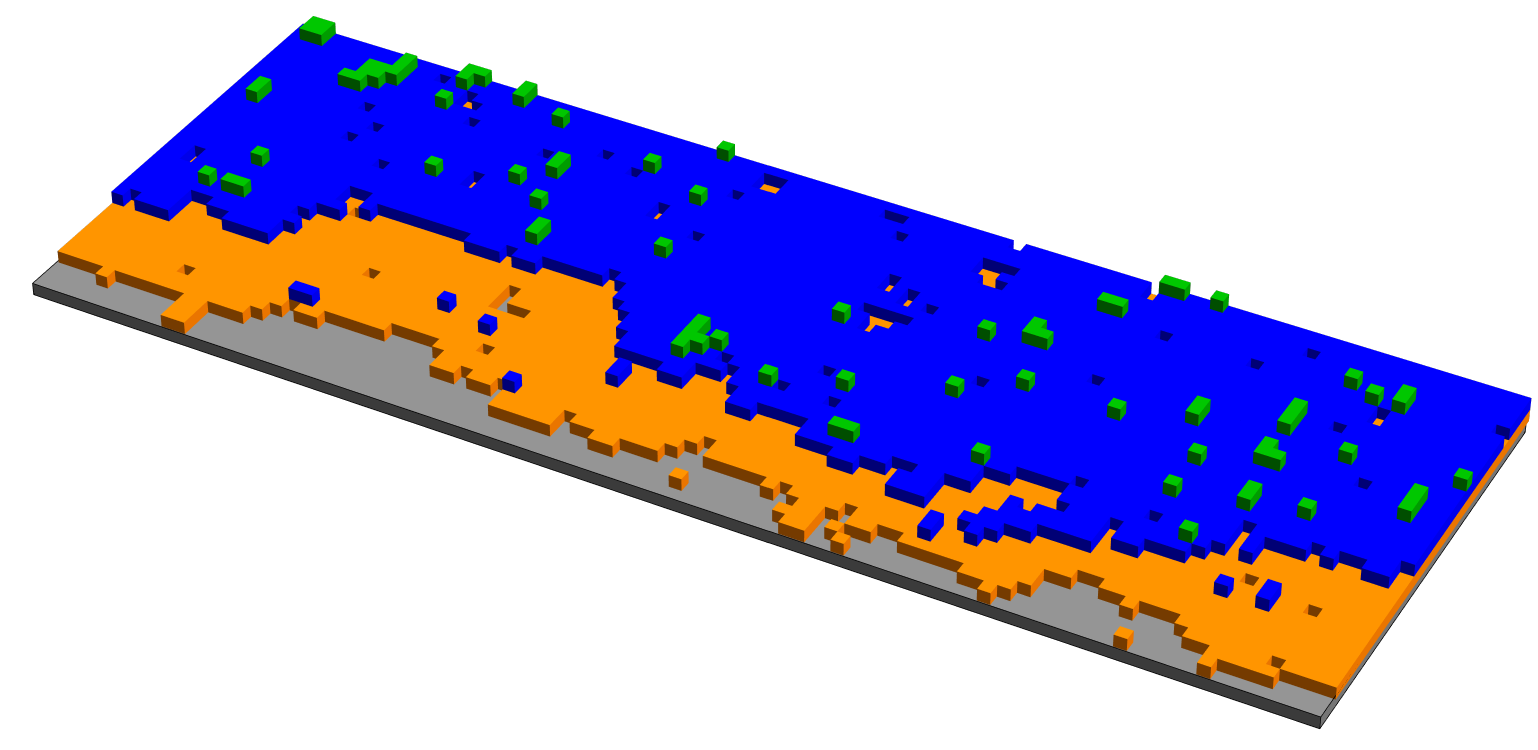}};

            \node [font=\footnotesize] at (6.15,-1.2) {$p=1.5$};
            \node [font=\scriptsize] at (6.0,-1.6) {$N_{n+1}\approx N_{n}e^{-c (\log L)^{1/3}}$};

            \draw[thick,->, color=white] (9.32,-1.65) -- (9.67,-1.15);
            \draw[thick,->, color=white] (9.42,-1.68) -- (9.5,-1.56);

            \node at (0,-3.6) {\includegraphics[width=0.45\textwidth]{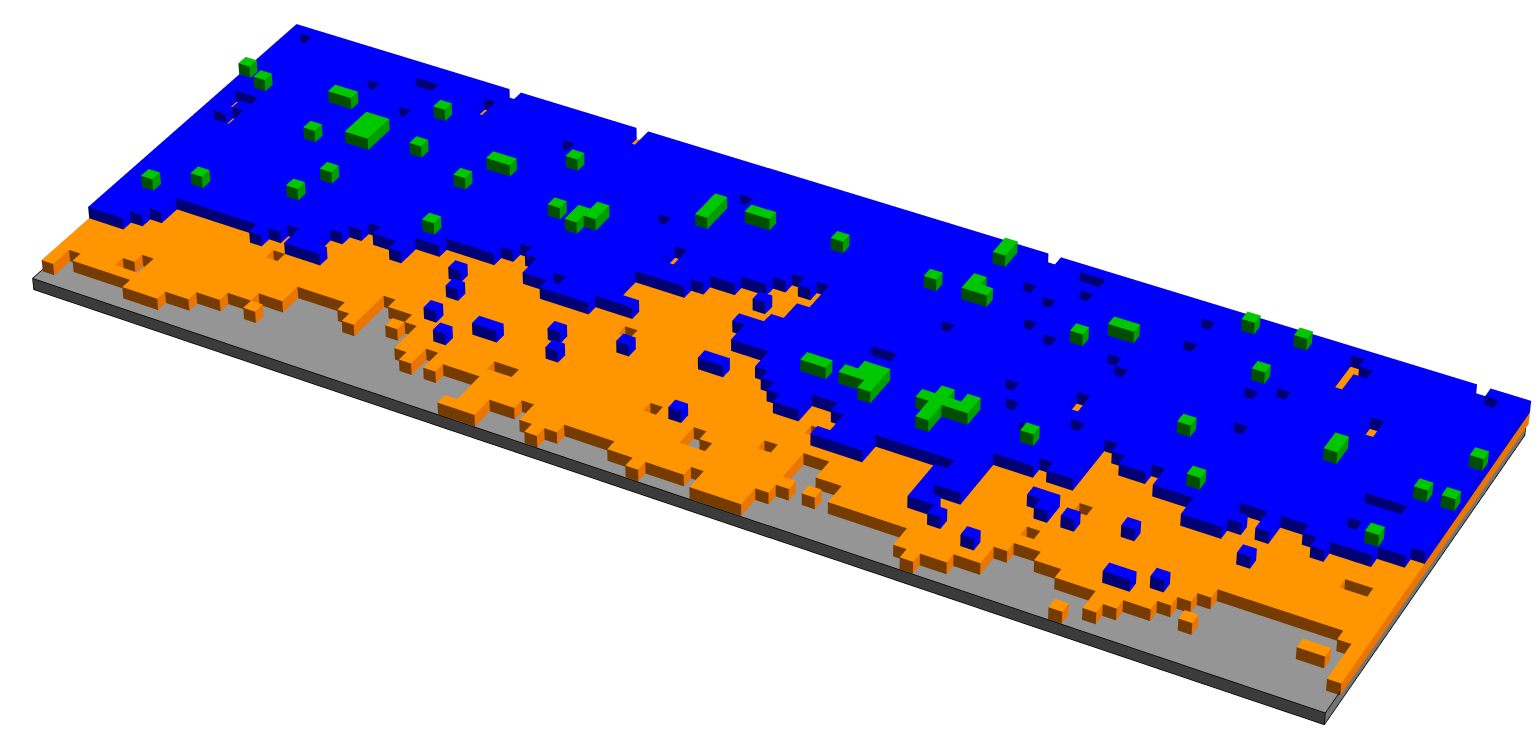}};
            \node [font=\footnotesize] at (-1,-4.8) {$p=2$ (\ZGFF)};
            \node [font=\scriptsize] at (-1,-5.2) {$N_{n+1}\approx N_{n} e^{-c \sqrt{\log L/\log\log L}}$};

            \draw[thick,->, color=white] (1.15,-4.73) -- (1.63,-4.);
            \draw[thick,->, color=white] (1.25,-4.76) -- (1.35,-4.6);

            \node at (7.5,-3.9){\includegraphics[width=0.45\textwidth]{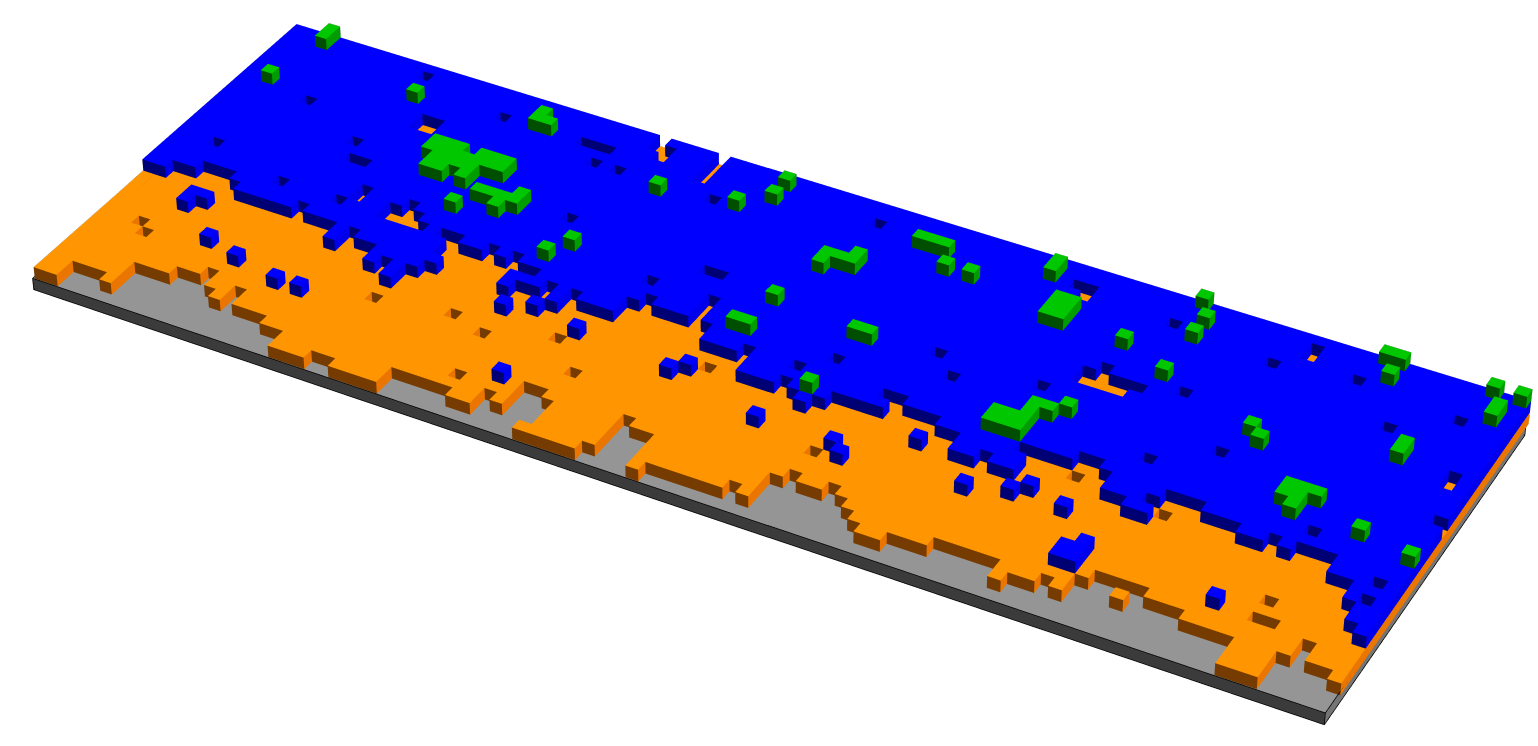}};
            \node [font=\footnotesize] at (6.15,-4.8) {$p=3$};
            \node [font=\scriptsize] at (6.0,-5.2) {$N_{n+1}\approx N_{n} e^{-c \sqrt{\log L}}$};

            \draw[thick,->, color=white] (9.32,-5.25) -- (9.72,-4.65);
            \draw[thick,->, color=white] (9.42,-5.28) -- (9.5,-5.16);

	\end{tikzpicture}
    \vspace{-0.2in}
    \caption{Comparison of the scales $N_n^{(p)}\asymp 1/\hatpi_\infty^{(p)}(\phi_o=L)$
    in the $|\nabla\phi|^p$-model for different values of $p$. As the scaling for the level line $\fL_n$ is $((N_n^{(p)})^{2/3},(N_n^{(p)})^{1/3})$, the $|\nabla\phi|^p$ models enjoy a scale separation between the level lines for $p>1$, unlike the \SOS model ($p=1$).}
    \label{fig:grad-phi-ll}
    \vspace{-0.1in}
\end{figure}
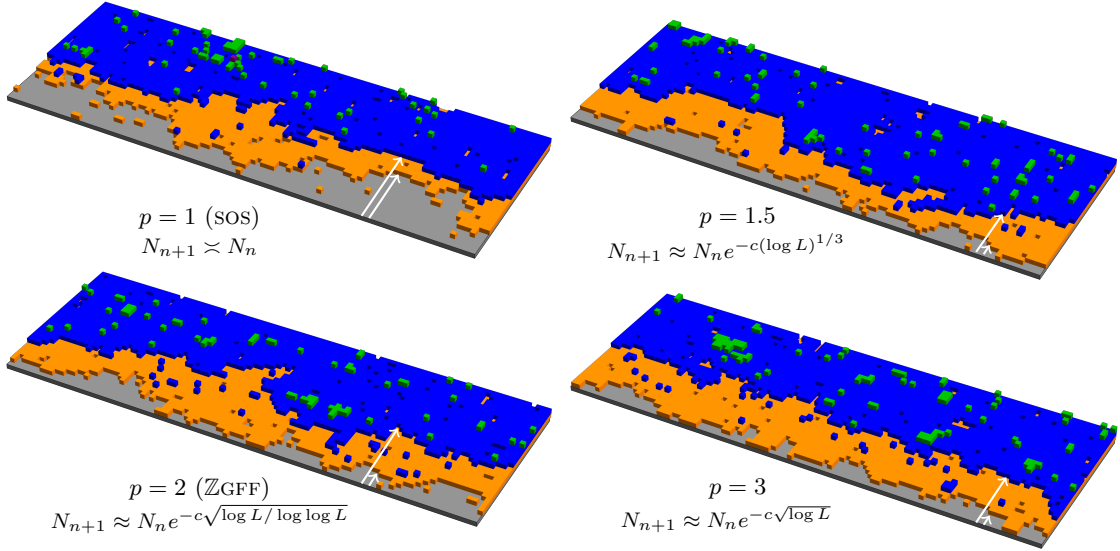

    The next theorem shows that for all $p>1$, the level lines have the same limit law as for \ZGFF.
	
	\begin{theorem}
		\label{thm:grad-phi-p}
		Fix $p>1$, and consider the setting of \cref{thm:1} with the $|\nabla\phi|^p$ model $\pi_\Lambda^{(p)}$
		replacing $\pi_\Lambda^0$ and $H^{(p)},N_n^{(p)}$ replacing $H,N_n$.
        Then the conclusion of \cref{thm:1} holds: the law of $(Y_n)_{n=1}^{m}$, the rescaled vertical distances of the top $m$ level lines from the bottom boundary intervals, converges weakly to the law of $m$ i.i.d.\ stationary Ferrari--Spohn diffusions $\mathsf{FS}_\sigma$ on~$[-1,1]$ for $\sigma>0$.
	\end{theorem}

    \begin{remark}\label{rem:gradphi-Nn-scales}
    Analogously to \cref{rem:Nn-scales}, for $2<p<\infty$, the sequence $N_1^{(p)}$ from \cref{eq:gradphi-def-H-Nn} satisfies $e^{-c\sqrt{\beta \log L}}\leq N_1^{(p)}/L \leq 1/(5\beta)$ and $N_n^{(p)} = N_{n-1}^{(p)} e^{-\Theta(\sqrt{\beta\log L})}$. For $1<p<2$, one has that $N_1^{(p)}$ from \cref{eq:gradphi-def-H-Nn} satisfies that $e^{-c \beta^{1/p}(\log L)^{(p-1)/p}}\leq N_1^{(p)}/L \leq 1/(5\beta)$ and $N_n^{(p)} = N_{n-1}^{(p)} e^{-\Theta(\beta^{1/p}(\log L)^{(p-1)/p})}$.
    In both cases, as in \cref{rem:Nn-scales}, the upper and lower bounds on $N^{(p)}_1/L$ give the correct behavior of the fluctuations along infinitely many values of $L$: the fluctuations of $\fL_1$ have order $L^{1/3}$ in the former case and $L^{1/3-o(1)}=o(L^{1/3})$ in the latter. Lower level lines, $\fL_2,\ldots,\fL_m$ for any fixed $m$, all have fluctuations $L^{1/3-o(1)}=o(L^{1/3})$. 
    \end{remark}

    \begin{remark}
        \label{rem:gradphi-bad-set}
        Analogously to \cref{rem:Lh-B-concrete}, one defines the exceptional set $\sB$ for the $|\nabla\phi|^p$ model in \cref{thm:grad-phi-p} exactly as in \cref{eq:exceptional-set-def} but with  $\hatpi^{(p)}_\infty$ replacing $\hatpi_\infty$ there.
        The fact that $\sB$ has zero logarithmic density extends to all $1<p<\infty$, owing to the fact that the large deviation rate of $\hatpi_\infty^{(p)}$ is super-linear: $\lim_{h\to\infty}-\frac1h \log\hatpi_\infty^{(p)}(\phi_o=h)=\infty$ (see \cref{rem:bad-set-calc} for more on this).
    \end{remark}

\medskip
We end this subsection with a discussion on the exceptional values $\sB$ and how these are handled in a follow-up work \cite{ChenLubetzky26+}. Recall that the intervals $\llb \frac34 L_h,L_h\rrb$ that make the set $\sB$ as per \cref{eq:exceptional-set-def} capture the emergence of a new macroscopic layer in the surface at height $h$, and that \cref{thm:1} took $L\notin \sB$, so the top level is guaranteed to be $H$ w.h.p.\ and one needs not address level~$H+1$.
In order to address these exceptional intervals, and in particular side-lengths where the probability that level $H+1$ exists is uniformly bounded away from both $0$ and $1$, additional ideas are needed. 
We establish this in the companion paper \cite{ChenLubetzky26+}, where we also derive the global limit shape of the top level lines (near the corners of the box, whereas here we looked at the local law at its flat parts), and the critical point and window around it for the onset of level $h$.

\begin{theorem}[\cite{ChenLubetzky26+}]  \label{thm:1-companion}
		Fix $\beta>0$ large enough and consider $\phi\sim \pi^0_\Lambda$, the $(2+1)$\Dim \ZGFF model on $\Lambda=\llb 1,L\rrb^2$ 
        with a floor and zero boundary conditions. Set $H(L)$ and $N_n$ per \cref{eq:H-def,eq:Nn-def}.
For fixed $m$, let $\fL_n$ ($n=0,\ldots,m$) be the \texttt{large} $(H+1-n)$ level lines, where possibly $\fL_0$ does not exist. 
Let $I_n = \llb \frac{L}2- N_n^{2/3},\frac{L}2+N_n^{2/3}\rrb$, let $\rho_n(x) := \min\{y\geq 0 \,:\; (\frac{L}2+x,y)\in\fL_n\}$ be the vertical distance of~$\fL_n$ from $I_n$, and 
$Y_n(t) := N_n^{-1/3}\rho_n(t N_n^{2/3})$.
There exists a constant $\sigma>0$  such that:
\begin{enumerate}[1.,leftmargin=2em]
\item The law of $(Y_n(t))_{n=1}^m$ under $\pi_\Lambda^0$ converges weakly to the law of 
$m$ i.i.d.\ copies of the stationary Ferrari--Spohn diffusion $\mathsf{FS}_{\sigma}$ on~$[-1,1]$.
\item For $L$ with $\pi_\Lambda^0(\fL_0\mbox{ exists})>L^{-10}$, the law of $(Y_n(t))_{n=0}^m$ under $\pi_\Lambda^0(\cdot\mid\fL_0\mbox{ exists})$ converges weakly to that of $m+1$ i.i.d.\ copies of 
$\mathsf{FS}_{\sigma}$.
\item 
     For $L$ with $\pi_\Lambda^0(\fL_0\mbox{ does not exist}) > L^{-10}$, the law of $(Y_n(t))_{n=1}^m$ under $\pi_\Lambda^0(\cdot\mid \fL_0\mbox{ does not exist})$ converges weakly to that of $m$ i.i.d.\ copies of $\mathsf{FS}_{\sigma}$.
\end{enumerate}
The same statements hold for $\bar{Y}_n(t)$ corresponding to $\bar{\rho}_n(x)=\max\{y\leq \frac{L}2:(\frac{L}2+x,y)\in\fL_n\}$.
	\end{theorem}

    \subsection{Proof ideas}\label{subsec:proof-ideas}
        Our starting point is the observation that the level lines in the $\ZGFF$, and more generally, in the $|\nabla\phi|^p$ model for any $p>1$, ought to be ``separated'' from one another. Indeed, $\fL_n$, the $n$-th level line from the top,  is expected to behave like a random walk tilted by a term of $\exp[\lambda_{L,n} \cA(\fL_n)]$, where $\lambda_{L,n} = \hatpi^{(p)}_\infty(\phi_o = H^{(p)}+1-n)$ and $\cA(\fL_n)$ is the area of its interior. Consequently, as such area-tilted random walks are known to have a Ferrari--Spohn limit law after rescaling their width and height by $(\lambda^{2/3},\lambda^{1/3})$ respectively, one expects the distance of $\fL_n$ from (say) the bottom boundary to be of order $\hatpi^{(p)}_\infty(\phi_o = H^{(p)}+1-n)^{-1/3}$, i.e., $(N_n^{(p)})^{-1/3}$ for $N_n^{(p)}$ from \cref{eq:gradphi-def-H-Nn}. 
        In the case of \SOS ($p=1$), the fact that $\hatpi^{(1)}_\infty(\phi_o=h)\asymp e^{-4\beta h}$ would have every level line $\fL_n$, for $n\geq 1$ fixed, be at the same scale: all would be found at $\Theta(L^{1/3})$ from the bottom boundary, as seen in the top left of \cref{fig:grad-phi-ll} (see \cref{sec:open-prob-related} for more on this).

For all $p>1$, however, one has $\hatpi^{(p)}_\infty(\phi_o=h+1) / \hatpi^{(p)}_\infty(\phi_o=h)=o(1)$, whence the above heuristic implies that the scales $N_n^{(p)}$ satisfy $N^{(p)}_{n+1}= o(N^{(p)}_n)$ for all $n$; i.e., each level line is at a microscopic scale compared to the one above it (see \cref{fig:grad-phi-ll}). Let us focus on the \ZGFF model for simplicity. There, the above ratio of $\hatpi_\infty(\cdot)$ terms is $e^{-\Theta(h/\log h)}$. For $h \asymp H(L)$, this is $\exp[-(\log L)^{1/2-o(1)}]$, e.g., we expect the second level line~$\fL_2$ to be at an $\exp[-(\log L)^{1/2-o(1)}]$ fraction of the ``height'' (distance from the bottom boundary) of the top level line~$\fL_1$, and similarly for the height of $\fL_{n+1}$ vs.\ $\fL_n$. Intuitively, this should mean there is no interaction between the level lines; if these are effectively independent area-tilted RWs,  their limit would be a product of Ferrari--Spohn diffusions.

Showing this separation of the level lines rigorously is nontrivial since the typical height of $\fL_{n+1}$ ought to be $N_n^{1/3} e^{-(\log L)^{1/2-o(1)}}$, so one needs to first control $\fL_n$ and show that it is concentrated about its height of $N_n^{1/3}$ to a finer degree. However, even for the \SOS level lines ($p=1$), where much more was known on $\{\fL_n\}$, the best-known estimates on the height of $\fL_{n}$ had an $L^{o(1)}$ error...
Moreover, we emphasize that, even after one manages to show that the level lines are well separated, the interactions between them could still have a huge impact on their limiting laws. A notoriously challenging obstacle in these and related models is the \emph{pinning} problem, whereby a level line might opt to stick to its  predecessor/successor (in this case, random; in other settings, it can be the boundary of the box) due to said interactions. Handling these delicate interactions (and associated pinning issues) is where most of our efforts in the proof are concentrated.

Our crucial idea here is to separate the analysis into two cases: one showing that $\fL_n$ stochastically dominates a Ferrari--Spohn diffusion, and another showing it is stochastically dominated by one with the same parameters. If one were to try to show directly that $\fL_n$ behaves as an area-tilted random walk, current techniques break down due to the combined interactions (and pinning hazards) with random level lines below and above. For instance, the depinning result of \cite{IST15} is quite delicate, even failing when the exponential decay of interactions has rate $\leq \beta/2$ (for us it is $\beta$, a non-issue). Specifically, it is only valid when the boundary of the domain is flat, and it appears that extending it to ``nearly flat'' domains is highly nontrivial, if at all possible using their method. In our setting, even the top level line $\fL_1$ faces pinning issues from below (to $\fL_{2}$) and from above (to microscopic holes along the boundary of the mesoscopic rectangle of interest, where we wish to derive a Ferrari--Spohn law). 
Luckily, the stochastic domination allows us to employ FKG adjustments that only go in the right direction for the respective side of the bound (upper/lower) currently studied. This does away with the pinning issues on one of the two sides of $\fL_n$ (the interactions on the remaining side are handled differently in the upper and lower bounds, as we explain below).

\subsubsection*{Upper bound on the limit law of $\fL_n$} The upper bound is proved in three steps. 
\begin{enumerate}[(i)]
    \item \label{it:sketch-growth}
We show that each $\fL_n$ does not exceed height $N_n^{1/3}(\log L)^C$, a sufficiently fine degree of accuracy to obtain the aforementioned separation. 
This is achieved via a ``growth gadget,'' that is, a result showing that conditional on the fact that  the level-line loop $\fL_n$ exceeds a certain area, it is likely to exceed it farther.

    \item\label{it:sketch-drop} Then, disregarding lower level lines via the separation established above, as well as upper level lines via monotonicity, we show that $\fL_n$ continues to drop down from $N_n^{1/3}(\log L)^C$ to its equilibrium height $N_n^{1/3}$ using finer random walk estimates. 

\item \label{it:sketch-area-tilt} Lastly, we show that the law of the level line intersected in a small box is approximately that of an area-tilted RW, when its initial height is $O(N_n^{1/3})$.
\end{enumerate}

(As a byproduct of \cref{it:sketch-growth}, we refine the corresponding estimate for \SOS; see \cref{rem:sos-log}.)

In each case, one aims to estimate the law of a level line $\fL_n$ in a mesoscopic box $V$ with boundary conditions $H+1-n,H-n$, and show it is approximately governed by $\exp[-\beta|\fL_n|+\lambda_{L,n} \cA(\fL_n)]$ where $\lambda_{L,n}=\hatpi_\infty(\phi_o=H+1-n)$. 
E.g., the prequel~\cite{CLMST16} on \SOS featured a growth gadget analogous to \cref{it:sketch-growth} where the region $V$ had area $L^{4/3+o(1)}$ (then applied to boxes of size $L^{2/3+o(1)}\times L^{2/3+o(1)}$). The paper~\cite{CKL24} on \SOS had an analysis analogous to \cref{it:sketch-area-tilt} in a box of size $(C L^{2/3})\times (C L^{2/3})$.

Unfortunately, for the \ZGFF, we can afford to obtain such estimates at regions of area $L (\log L)^c$.
The culprit is the two point probability $\hatpi_\infty(\phi_x \geq h, \phi_y \geq h)$ for a pair of nearest-neighbors $x \sim y$. In \SOS, this probability is $\approx e^{-6\beta h}$, so when $\hatpi_\infty(\phi_x \geq h) \approx e^{-4\beta h}$ is $1/L$, the two point probability is $L^{-3/2}$. This translates to a domain restriction of $|V| \ll L^{3/2}$, a computation made rigorous in \cite{CLMST16}. In the \ZGFF model, the two point probability turns out to be $L^{-1 - o(1)}$, since the large deviation event $\phi_x \geq h$ is dominated by $\phi$ which climbs to $h$ in the shape of a harmonic pinnacle~\cite{LMS16}, whereby its neighbors are already at height $h-O(h/\log h)$. The consequence is that we can only work on domains of size $|V| \leq L^{1 + o(1)}$, which forces us to take boxes of say $L^{2/3}(\log L)^c\times L^{1/3}(\log L)^c$. 

This is a serious restriction, due to the pinning issues. When the height of $V$ is only $L^{1/3 + o(1)}$, even though the probability of a random walk reaching height $L^{1/3 + o(1)}$ on an $L^{2/3}$ interval is only $e^{-L^{o(1)}}$, the interactions between $\fL_n$ and $\partial V$ could potentially tilt the measure by a factor of $e^{\epsilon_\beta L^{2/3}}$. As we mentioned, the case of a flat boundary was handled in \cite{IST15}; however, our $V$ is a perturbation of a rectangle with wiggly boundaries, for which that proof method breaks down.

In the direction of showing an upper bound on $\fL_n$, we circumvent these pinning issues by using FKG to forget the floor constraint on a subset of $V$, picking up an area tilt only on a set $F \subset V$ instead. This effectively removes the size restriction from $V$, allowing us to use $L^{2/3+o(1)}\times L^{2/3+o(1)}$ rectangles, where the top boundary is too far to induce pinning issues. We handle interactions with the bottom boundary of $V$ by conditioning $\fL_n$ to stay sufficiently far away from it (again by FKG). 

As evident from this strategy, we moreover have that with high probability, if we condition on $\fL_{n+1},\fL_{n+2},\ldots$, then the conditional limit law of $\fL_n$ is stochastically dominated by a Ferrari--Spohn diffusion. One can view this as an induction  revealing level lines from exterior to interior, ultimately showing that the joint limit law of $\fL_1,\ldots,\fL_m$ is dominated by independent Ferrari--Spohn laws.

\bigskip
\subsubsection*{Lower bound on the limit law of $\fL_n$}
For the lower bound on the top level line $\fL_1$, we first remove the lower level lines by FKG (intuitively, they only push $\fL_1$ upwards). Then, we study the restriction of $\fL_1$ to a domain $V$ which is a perturbation of an $N_1^{2/3} \times N_1^{1/3}(\log L)$ rectangle, with a flat boundary at the bottom and wiggly boundary on the other three sides. All the monotonicity workarounds from the upper bound now no longer apply, being in the wrong direction. In particular, the top boundary of $V$ cannot be placed far away anymore. Instead, we condition $\fL_1$ to stay away from the top boundary, and handle the (now flat) bottom boundary by extending the depinning results of~\cite{IST15} to our setting. Once we establish the depinning, we can reduce the problem to the area-tilted 2\Dim random walk setting of prior works~\cite{CKL24,IOSV21} which show convergence to Ferrari--Spohn.
This proves a lower bound for $\fL_1$, which in particular shows that its typical height is $O(N_1^{1/3})$. 

Turning to the next level line, $\fL_2$, we look at its restriction to a domain $V$ which is a perturbation of a $N_2^{2/3} \times N_2^{1/3}\log L$ rectangle. Since $N_2^{1/3}\log L = o(N_1^{1/3})$, the behavior of $\fL_1$ can be isolated from the analysis of $\fL_2$, and as before the lower level lines $\fL_3, \fL_4, \ldots$ can be removed by FKG. This then allows us to show the desired lower bound for $\fL_2$, and we can proceed inductively for $\fL_n$.

As in the upper bound on $\fL_n$ (but with a reverse order of which level lines are revealed, where here the induction reveals them from interior to exterior), one obtains that with high probability, conditional on
$\fL_1,\ldots,\fL_{n-1}$, the limit law of $\fL_n$ stochastically dominates a Ferrari--Spohn diffusion. Consequently, the joint limit law of $\fL_1,\ldots,\fL_{m}$ dominates independent Ferrari--Spohn laws.

\begin{remark}
    Note that the upper bound had to reveal the level lines from exterior to interior (revealing $\fL_{n-1}$ would introduce a non-flat interacting boundary at distance $L^{1/3+o(1)}$ above $\fL_n$)  whereas the lower bound had to reveal the level lines from interior to exterior (revealing $\fL_{n+1}$ would introduce a non-flat interacting boundary at distance $L^{1/3+o(1)}$ below $\fL_n$).
\end{remark}

    \subsection{Proof outline and organization}\label{subsec:proof-outline}~The paper is organized as follows.

    In \cref{sec:law-of-disagreement-poly}, we establish a polymer representation for the law of a level line $\fL$ using classical cluster expansion techniques. This requires some refinements of large deviation results of \cite{LMS16}. It turns out that the natural object to study here is not $\fL$, but the entire component of bonds $\gamma$ dual to height disagreement edges containing $\fL$, which we call a disagreement polymer. Thus,  the polymer model of interest no longer falls under the family of ``Ising polymers'' for which much has been proven (e.g., the existence of a surface tension~\cite{DKS92} and the depinning from a flat boundary~\cite{IST15}). Instead, the geometry of $\gamma$ is allowed to be more complex than that of a contour, and its law has an additional term capturing the area tilt of finite components surrounded by $\gamma$.

    In \cref{sec:geom-disagreement-polymers}, we prove several foundational results for the aforementioned class of disagreement polymers, which are already known for ``Ising polymers.'' We begin by defining animals, cone-points, and irreducible components in order to establish a product structure to the polymer law, via Ornstein--Zernike analysis. Next, we extend results of \cite[Sec.~4]{DKS92} to our model, proving the existence and key properties of the surface tension, as well as elementary depinning results. We then define the Wulff shape corresponding to the surface tension, and continue the Ornstein--Zernike theory to show the proper renormalization needed to turn the product structure into a random walk.

   In \cref{sec:initial-UB-JEMS}, 
   we show one side (growth of a droplet) of the macroscopic scaling limit of the $(H+1-n)$ level line $\fL_n$, which should be given by a translation of Wulff shapes. (Hereafter, we prove results on $\fL_n$, but remember that the object we work with is $\gamma$.) The behavior in the $\ZGFF$ is markedly different from that in the \SOS model: here the Wulff shape is only visible on the $O(L)$ scale for select values of $L$, and only for the top level line $\fL_1$ (see \cref{rem:wulff-shape-invisible}). A direct implication of this scaling limit is that $\fL_n$ has a maximum deviation of $N_n^{1/3}(\log L)^C$ from the boundary of~$\Lambda$, away from the corners. Following the strategy employed for the \SOS model in \cite{CLMST16}, we look at the region $\cD$ known to be contained inside $\fL_n$ and iteratively grow this region. The key is to show that in a rectangle $V$ with $H+1-n, H-n$ boundary conditions, $\fL_n$ drops to a lower height. Then we can place such rectangles all around the boundary of the contained region $\cD$ to expand $\cD$ in every direction. A similar dropping lemma is needed in \cref{sec:UB}, but the difference is that here we need to handle boundary conditions at an angle. The lattice effect makes this a nontrivial change (being the reason there is even a Wulff shape to begin with), leading to a different analysis.

   In \cref{sec:UB}, once we show that the height of $\fL_n$ above the bottom side of $\partial \Lambda$ is at most $N_n^{1/3}(\log L)^C$, we prove that inside an $N_n^{2/3}(\log L)^C \times 2N_n^{2/3}(\log L)^C$ rectangle, the curve of $\fL_n$ is stochastically dominated by an area tilted random walk in the middle $N_n^{2/3}$ interval. As mentioned above, this requires another dropping lemma, this time showing the stronger result that $\fL_n$ drops from height $N_n^{1/3}(\log L)^C$ to $O(N_n^{1/3})$. (In turn, we must assume that the endpoints of $\fL_n$ are at an angle of $\approx 0$.) Then, combining monotonicity considerations with previous results on area tilted random walks, we prove the sought upper bound on the limit law in \cref{thm:1}.

 In \cref{sec:LB}, we prove that inside a $TN_n^{2/3}\times N_n^{1/3}(\log L)^C$ rectangle, the curve $\fL_n$ stochastically dominates an area tilted random walk in the middle $N_n^{2/3}$ interval. Since this random walk has the same increment law as the one studied in \cref{sec:UB}, it converges to the same Ferrari--Spohn diffusion, proving the lower bound on the limit law in  \cref{thm:1}.

In \cref{sec:gradphi}, we prove \cref{thm:grad-phi-p}, extending our results to the $|\nabla\phi|^p$ model for  $1 < p < \infty$. 
To that end, we prove a lower bound on the ratio $\hatpi^{(p)}_\infty(\phi_o = h)/\hatpi^{(p)}_\infty(\phi_o = h-1)$, which was not addressed as part of the large deviation results from \cite{LMS16} when $p \neq 2$. This was not needed in that work, which only concerned the height histogram of those models, but is imperative here to obtain an asymptotically sharp bound on the law of the disagreement polymer.

    \subsection{Open problems and related works}\label{sec:open-prob-related}
    In this work, we showed that the scaling limit of the $m$ top level lines, $\fL_1,\ldots,\fL_{m}$, in the \ZGFF above a floor at large enough $\beta$, is $m$ independent Ferrari--Spohn diffusions. Moreover, this holds for the $|\nabla\phi|^p$ model for any fixed $p > 1$, which captures the complete universality class when varying $p$, as the behavior for $p = 1$ ($\SOS$) is expected to be different. It remains an open problem to establish the limit law for the countably many level lines $\fL_1,\fL_2,\ldots$ in \SOS, conjectured to be a Brownian line ensemble with a geometric area tilt; see 
\cite{DLZ23,CaputoGanguly23,CIW19a,CIW19b,BCG25,HKS25,Serio23} for a plethora of works studying this object.
    
    For $p > 1$, several notable problems remain open. \cref{thm:1,thm:grad-phi-p} address the law of the level lines away from the corners. It is expected that at the corners, the level lines should have fluctuations of $L^{1/2-o(1)}$ (as opposed to $L^{1/3-o(1)}$), with an associated scaling limit of independent Brownian motions. (This is open for all $p\geq 1$, i.e., including \SOS.) In addition, our results exclude an exceptional set $\sB$ of values of $L$, addressed in the companion paper~\cite{ChenLubetzky26+} (cf.~\cref{thm:1-companion} above). The reader is referred to \cite{ChenLubetzky26+} for additional questions on the critical window marking the onset of a new level line.
    Specifically, the question of establishing the width of the critical window is open.

Lastly, we mention recent developments on the rough phase of the $|\nabla\phi|^p$ model (with no floor). At high enough temperature with uniform boundary conditions, convergence of the $\ZGFF$ to the Gaussian Free Field (\GFF) was shown in~\cite{BPR22a,BPR22b}; this is expected, but remains open for all other $p\geq 1$. When the boundary conditions are tilted, e.g., taking the value $\lfloor \theta_1 x + \theta_2 y\rfloor$ at every boundary point $(x,y)$ for some slope $\theta=(\theta_1,\theta_2)\neq 0$, 
then, as long as at least one of the $\theta_i$'s is noninteger, the surface is known~\cite{Sheffield05} to delocalize at some unknown rate (``qualitative delocalization'') for all $\beta$ and~$p\geq 1$. At high enough temperature, this was extended also to integer slopes in~\cite{LammersOtt24} for \SOS. More recently, logarithmic delocalization at high enough temperature was established for all $0 < p \leq 2$ in~\cite{OttSchweiger25}. It is further believed that the rough phase induced by the slope should have a \GFF scaling limit, mirroring the high temperature regime with flat boundary conditions.
However, this remains a formidable open problem at large $\beta$, where even logarithmic fluctuations are not yet established. One recent result in this direction is for \SOS tilted by an added long range potential~\cite{LaslierLubetzky25}, which indeed has a \GFF scaling limit (for any slope  with $\theta_1,\theta_2>0$ including when both are integers; this is in contrast to  $|\nabla\phi|^p$ for all $p>1$, including  \ZGFF, which are rigid at large $\beta$ when $\theta_1, \theta_2 \in \Z$).

	\section{Law of disagreement polymers}\label{sec:law-of-disagreement-poly}
	
	The goal of this section is to formulate the cluster expansion framework for the level lines of the \ZGFF configuration $\phi$. In what follows, we will typically refer to dual-edges (in $(\Z^2)^*$) as bonds, to help distinguish them from the edges of $\Z^2$. We say $\sfu,\sfv\in \Z^2$ are $*$-adjacent if their $L^\infty(\R^2)$ distance is $1$ (i.e., they are either adjacent in $\Z^2$, or their bounding cells share a single corner). For any point $\sfu \in \R^2$, we will denote the $x$ and $y$ coordinates by $\sfu_1, \sfu_2$ respectively. Further let $\partial U$ be the boundary bonds of $U$, i.e., the set of bonds dual to $uv$ with $u\in U$ and $v\notin U$, and let 
    $\partialvtx U$ be the external vertex boundary of $U$, i.e., every vertex $v\notin U$ adjacent to some $u\in U$.
	
	\begin{definition}[Disagreement polymer]\label{def:gamma}
		Let $\phi$ be a
		$\Z$-valued height function on the vertices of a connected domain $V\subset \Z^2$. Associate to each bond $e\in (\Z^2)^*$, dual to some edge $(x,y)\in\Z^2$, the gradient $(\nabla \phi)_e := \phi_x - \phi_y$, where $x$ is taken to be the \textsc{north} vertex if $(x,y)$ is vertical and the \textsc{west} vertex if $(x,y)$ is horizontal. We say the bond $e$ is a \emph{disagreement bond} of $\phi$ if $(\nabla\phi)_e\neq 0$ (including when $\phi_x$ or $\phi_y$ are specified by the boundary conditions). A \emph{disagreement polymer}~$\gamma$ is a (maximal) connected component of the disagreement bonds of $\phi$, and we let $\cP_\phi$ denote the disagreement polymers in $\phi$. For such a polymer, let $D_i$ be the connected components of $V \setminus \gamma$, noting that by the maximality of $\gamma$, within each $D_i$, all the vertices that are $*$-adjacent to $V\setminus D_i$ must have the same height $h_i$ in~$\phi$. Call the triple $(\gamma, \{D_i\}, \{h_i\})$ a \emph{labeled disagreement polymer}. For brevity,  we denote labeled disagreement polymers by $\gamma$ (omitting the $\{h_i\}$ from the notation).
	\end{definition}

    \begin{figure}
    \vspace{-0.15in}
    \centering
    \includegraphics[width=0.5\textwidth]{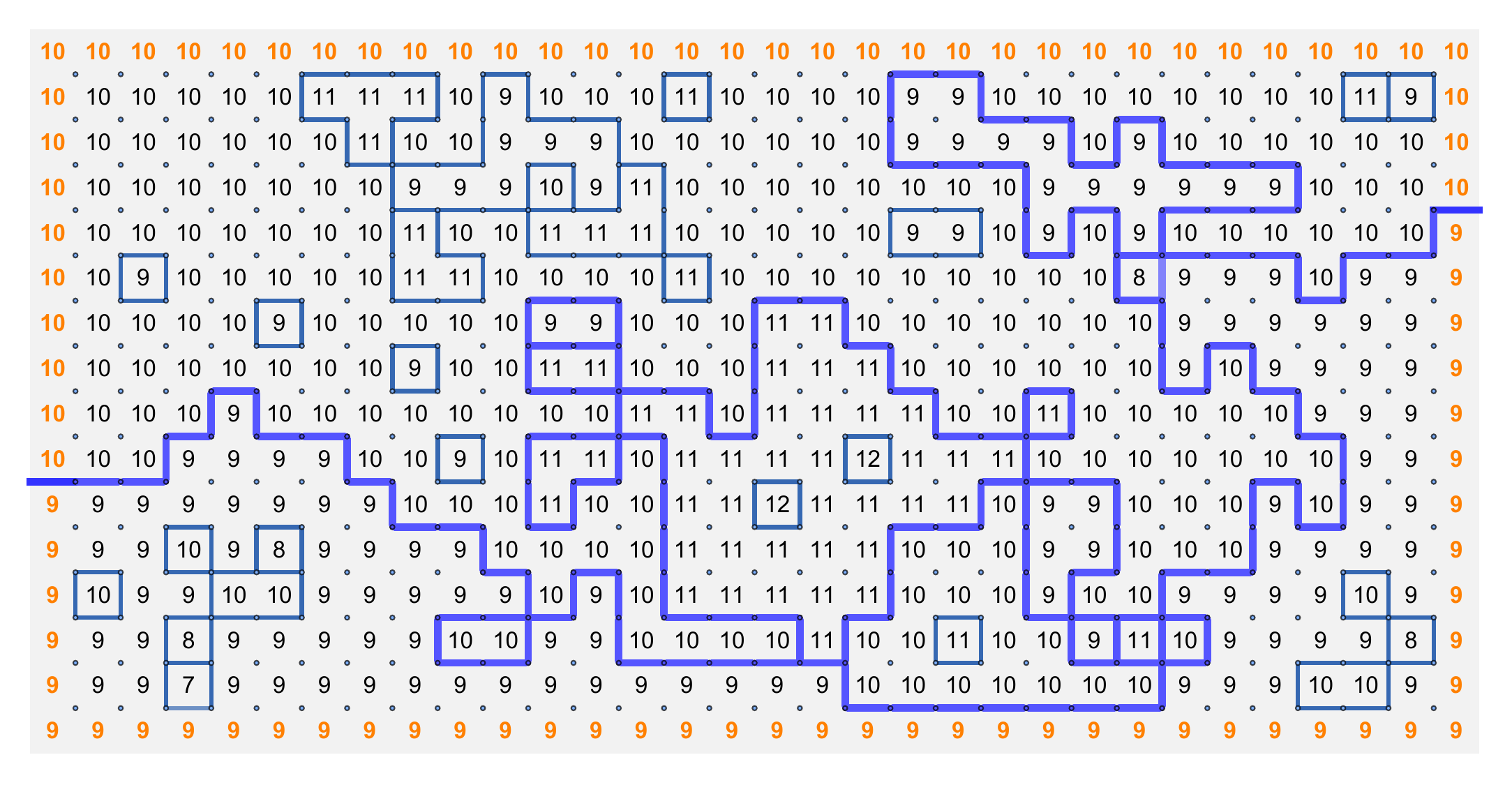}
    \hspace{-0.1in}
    \includegraphics[width=0.5\textwidth]{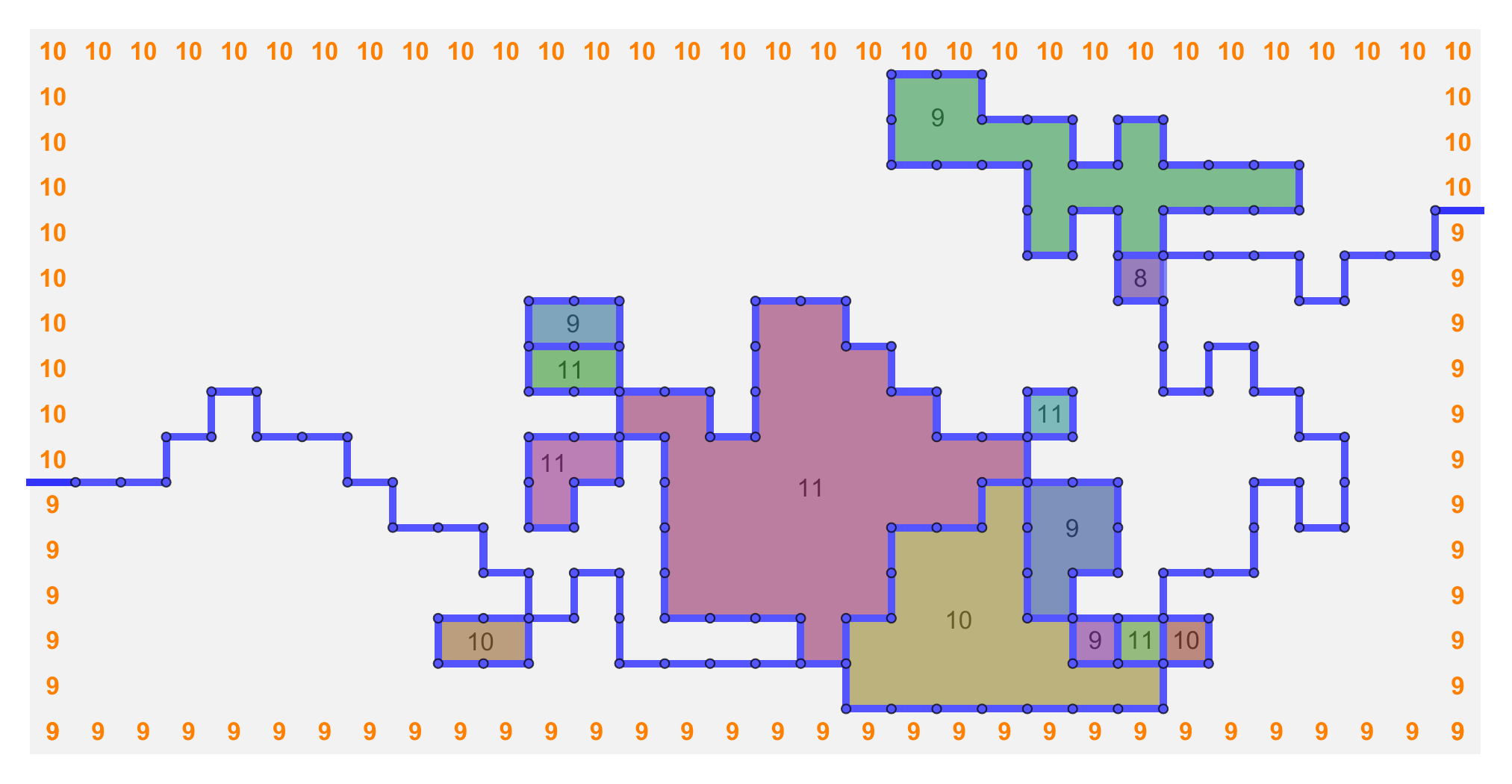}

    \vspace{-0.15in}
    \caption{Left: A height function $\phi$ on a rectangle with $h, h-1$ boundary conditions at $h = 10$, with all disagreement bonds in blue. The disagreement polymer $\gamma$ is highlighted. Right: The corresponding labeled disagreement polymer $(\gamma, \{D_i\}, \{h_i\})$. Each region $D_i$ is assigned a single height $h_i$ which can be read off from $\phi$ on the left (its interior boundary).}
    \label{fig:disagree-poly}
    \vspace{-0.1in}
    \end{figure}
    
	NB.\ In the graph on $(\Z^2)^*$ whose edges are $\gamma$, no vertex can have degree $1$ (degrees $0,2,3,4$ are possible), hence the restriction that all $v\in D_i$ that are $*$-adjacent to $V\setminus D_i$ have the same height. See the example in \cref{fig:disagree-poly} illustrating this.
	\begin{definition}[Energy, length and decorations of a disagreement polymer]\label{def:disagree-polymer-len-energy-decor}
		Let $(\gamma,\{D_i\},\{h_i\})$ be a labeled disagreement polymer. Its \emph{length} $\sN(\gamma)$ and \emph{energy} 
		$\sE_\beta(\gamma)$ are defined as
		\[\sN(\gamma)=\sum_{e\in\gamma}|(\nabla \phi)_e|\quad,\quad 
		\sE_\beta(\gamma) := \beta\sum_{e \in \gamma} |(\nabla \phi)_e|^2\,.\]
		To describe the law of the disagreement polymer, we will use a family of decoration functions $\Phi$ on subsets $\sfW\subset \Z^2$---each referred to as a cluster if it is also connected---
        satisfying the following properties.
		\begin{enumerate}[(i)]
			\item \label{it:phi(W)-conn} If $\sfW$ is not connected, then $\Phi(\sfW) = 0$.
			
			\item \label{it:phi(W)-symmetric}
			The function $\Phi$ is translation invariant, as well as invariant under a rotation by $\pi/2$ or reflection with respect to one of the axes.
			
			\item There exists a constant $C > 0$ such that for every $\sfW$, we have 
			$
			|\Phi(\sfW)| \leq \exp(-(\beta - C)\bd(\sfW))
			$, 
			where $\bd(\sfW)$ is the size of the smallest connected set of bonds in $(\Z^2)^*$ containing all the boundary bonds of $\sfW$.\label{it:phi(W)-decay-bound}
		\end{enumerate}
		
	\end{definition}

	The main result of this section is the following, where here and in what follows, we use the notation $D_i^\circ = D_i^\circ(\gamma)$ to denote $D_i \setminus \Delta_\gamma$ where 
	\[ \Delta_\gamma := \{ u\in V\,:\; \dist(u,\gamma) \leq 1/\sqrt{2}\}\,,\]
	as the value of $\phi$ on $D_i\cap \Delta_\gamma$ (and elsewhere on $\partialvtx D_i^\circ$) are specified to be the corresponding $h_i$.
	
	\begin{proposition}[Cluster expansion with a floor]\label{prop:CE-law-with-floor} Fix $n\geq 1$. Let $V\subset \Z^2$ be a connected domain, and consider the \ZGFF model $\pi_{V;F}^{\xi}$ with a floor at~0 imposed only on a subset $F\subset V$, and boundary conditions $\xi$ that are $H+1-n$ on a $*$-connected path in $\partialvtx V$ and $H-n$ elsewhere so that they induce a unique disagreement polymer $(\gamma,\{D_i\},\{h_i\})$ in $V\cup \partialvtx V$ that contains boundary disagreements\footnote{We say the disagreement polymer contains boundary disagreements if it contains a bond dual to $\sfu\sfv$ for $\sfu,\sfv\notin V$.}.  Then for $\beta\geq \beta_0$, the law of this unique disagreement polymer $\gamma$ is given by 
		\begin{align}\label{eq:CE-with-area}
			\pi^{\xi}_{V;F}(\gamma) =\frac1{Z} \exp\bigg(-\sE_\beta(\gamma) + \sum_{\substack{\sfW\subset V \\ \sfW \cap \Delta_\gamma \neq \emptyset}} \Phi(\sfW)
			\bigg)\prod_{i\geq 0}\hatpi_{D_i^\circ}^{h_i}\big(\phi_x \geq 0,\, \forall x \in D_i^{\circ}\cap F\big)\,,
		\end{align}
		for a normalizer $Z= Z(V,F,\xi,\beta,n)$ and a decoration function $\Phi(\sfW)$ as per \cref{def:disagree-polymer-len-energy-decor}.
		
		Moreover, if we further have $|F| \leq L e^{\kappa\sqrt{\log L}}$ and $|\partial F|\leq L^{1-\delta}$ for fixed $\kappa>0$ and $0<\delta<\frac13$, denote by $D_0$ and $D_1$ the regions of $\gamma$ containing the boundary vertices of $V$ at heights $H+1-n$ and $H-n$, respectively, and let
		\begin{align*} E:=\left\{|\gamma|\leq L^{1-\delta}\,,\;
			\;|F| - |D_0 \cap F| - |D_1 \cap F| \leq L^{1-\delta}\right\}\,,\end{align*}
		then the following holds
		for all $\beta\geq \beta_0$ (not depending on $\kappa,\delta$). The probability distribution given~by 
		\[ p_{V;F}^{\xi}(\gamma) := 
		\frac1{\overline Z} \exp\bigg(-\sE_\beta(\gamma) + \frac{|D_0\cap F|}{N_n}+ \sum_{\substack{\sfW\subset V \\ \sfW \cap \Delta_\gamma \neq \emptyset}} \Phi(\sfW)  \bigg)\prod_{i \geq 2}\hatpi_{D_i^\circ}^{h_i}(\phi_x \geq 0,\, \forall x \in D_i^{\circ}\cap F)
		\]
        for $\gamma\in E$, where $\overline Z = \overline{Z}(V,F,\xi,\beta,n)$ is a normalizer, satisfies		\begin{align}\label{eq:CE-with-area-cond}
			\pi^\xi_{V;F}(\gamma \mid E) =(1+o(1))p_{V;F}^{\xi}(\gamma)\,.
		\end{align}
	\end{proposition}
	
	We prove this result in \cref{sec:Cluster-Expansion}, after we establish some properties of the measure $\hatpi_\infty$.
	
	\subsection{Large deviations without a floor}
	The following was established in \cite[Eq.~(3.1)--(3.4)]{LMS16}) for the infinite volume \ZGFF measure $\hatpi_\infty$:  for every fixed large enough $\beta$ and every $h\geq 2$,
	\begin{align}
		&e^{-c_0\beta h/\log h} \leq \frac{\hatpi_\infty(\phi_o = h)}{\hatpi_\infty(\phi_o = h-1)} \leq e^{-c_1\beta h/\log h}\,,\label{eq:LD-ratio-inf}\\
		&\hatpi_\infty(\phi_o = h) = \exp\bigg[-2\pi\beta\frac{h^2}{\log h} + O\Big(\frac{h^2}{\log^2h}\Big)\bigg]\,,\label{eq:LD-inf}\\
		&\hatpi_\infty(\phi_z = h \mid \phi_o = h) \leq e^{-c_2 \beta h /\log h}\,,\label{eq:LD-conditional-inf}
	\end{align}
	where $c_0,c_1,c_2>0$ are absolute constants.
	\begin{remark}
		A stronger upper bound of $\exp(-c_2 h^2/\log^2 h)$ was stated in \cite[Eq.~(3.3)]{LMS16}, even though the bound in \cref{eq:LD-conditional-inf} was the one later proved in \cite[\S3.3]{LMS16}. The weaker bound still suffices for the proof of the main theorem there: one needs only to modify the hypothesis in \cite[Prop.~4.4]{LMS16} (where this result is used) from $h\geq \log\log L$ to $h\geq (\log\log L)^2$, which would reproduce the same estimates needed in that proof from the weaker \cref{eq:LD-conditional-inf}. The proof applied \cite[Prop.~4.4]{LMS16} in a box $\Lambda_\ell$ with $h=(\log \ell)^{1/2+o(1)}$, and hence the stronger hypothesis $h \geq (\log\log\ell)^2$ would be valid.
		NB. In~\cite{LMS16}, the aim was to show $\hatpi_V^0(\min_{x\in V}\phi_x \geq -h)\leq \exp[-(1+o(1))I(h)]$ for the correct $I(h)$, while our analysis of the level lines requires a much more precise bound of the form $(1+o(1))\exp[-I(h)]$.
	\end{remark}
	For our proofs, it will be crucial to have the stronger upper bound of $\exp(-c h^2 / \log^2 h)$ on the probability appearing in the left hand of \cref{eq:LD-conditional-inf}; we adapt the argument of \cite{LMS16} to obtain it, as well as estimates in an arbitrary region~$V$ containing a ball of a certain radius around the origin $o$.
	\begin{theorem}[adapting {\cite[Thm.~3.1]{LMS16}}]\label{thm:LD-DG}
		There exist constants $\beta_0>0$ and $c>c'>0$ so that the following holds for every $\beta\geq \beta_0$ and integer $h\geq 2$. Let $V \subset \Z^2$ be a region containing $\cB_r(o)$, the ball of radius $r=\lceil 2c h/\log h\rceil $ centered at the origin $o$, as well as $\cB_{r+1}(z)$ for a site $z\in V$. Then
		\begin{align}
			\exp\Big(-c\beta \frac{h}{\log h}\Big) \leq &\;\frac{\hatpi_V^0(\phi_o = h)}{\hatpi_V^0(\phi_o = h-1)} \leq 
			\exp\Big(-c'\beta \frac{h}{\log h}\Big)
			\,,\label{eq:LD-ratio}\\
			\exp\Big(-2\pi\beta\frac{h^2}{\log h} - c \beta \frac{h^2}{\log^2h}\Big)
			\leq  &\quad\;\;\hatpi_V^0(\phi_o = h)\!\!\!\!\quad\;\leq \exp\Big(-2\pi\beta\frac{h^2}{\log h} + c \beta\frac{h^2}{\log^2h}\Big)\,,\label{eq:LD}\\
			&\!\!\!\!\!\!\!\!\!
			\hatpi_V^0(\phi_z = h \mid \phi_o = h) \leq \exp\Big(-c \beta \frac{h^2}{\log^2h}\Big)\,.\label{eq:LD-conditional}
		\end{align}
	\end{theorem}
	\begin{proof}
		The arguments of \cref{eq:LD-ratio-inf,eq:LD-inf} extend more or less verbatim to the setting given here of a more general domain $V$ in \cref{eq:LD-ratio,eq:LD}, provided that $V\supset \cB_r(o)$ for the given $r$, and we begin by explaining this point. Let $R = \lfloor h / \log h\rfloor$. 

    The lower bound given in \cite{LMS16} on \cref{eq:LD-ratio-inf} was stated for $\hatpi_\infty(\phi_o=h)/\hatpi_\infty(\phi_o=h-1)$, carried out on $V=\cB_L(o)$ for $L\gg R$, but in fact holds for $\hatpi_V$ for any domain $V\supset \cB_R(o)$\footnote{The proof of this bound in \cite{LMS16} is concluded immediately after \cite[Claim~3.5]{LMS16}, using nothing only that $V\supset \cB_R(o)$ until that point; it does appeal to  \cite[Claim~3.6]{LMS16}, but the latter is already phrased for a general domain $V$.}.

    \begin{remark}
    To see why the requirement $V\supset \cB_R(o)$ is the only one needed for the lower bound on \cref{eq:LD-ratio}, we briefly summarize that argument. The proof first reduced the lower bound to showing that all $x$ neighbors of $o$ would have $\phi_x\geq h-c R$, for a large enough $c$, as expected given that the (harmonic) optimum $\phi^*$ for the $\R$-valued Dirichlet problem has $h-\phi^*_x \asymp h/\log h$ at those sites: this step is valid in any domain~$V$.
	Next, factoring out the contribution of $\phi^*$, the problem was reduced, sequentially, to 
	\begin{enumerate}[(i)]\item the configuration $\sigma := \phi-\phi^*$ on $V\setminus\{o\}$ (already here the assumption $\cB_R(o)\subset V$ is used, as one defines $\phi^*$ as the $\R$-valued solution on $\cB_R(o)$, and $0$ outside); \item then to its integer part $\bar\sigma := \lfloor \sigma\rfloor$ in $V\setminus\{o\}$; \item and finally to $\bar\sigma$ in a $\ZGFF$ model on $V\setminus \{o\}$ where the interaction strength in the interior $\cB_R(o)$ is modified to $1/2$ vs.\ $1$ in its exterior.
	\end{enumerate} (This second series of reductions  incurs a cost that is absorbed by further increasing $c$.) The dominant term in the lower bound is $\exp(-c \beta R)$, yielding the sought bound that appears here, for the more general $V$, in the left-hand of \cref{eq:LD-ratio}.
    \end{remark}
    
    The upper bound of \cite{LMS16} on \cref{eq:LD-ratio-inf} holds for every $V\supset \cB_R(o)$ as follows: arguing very similarly to the proof there\footnote{See the argument that appears immediately after \cite[Lem.~3.10]{LMS16}.}, if one chooses $r=\delta R$ for a small enough $\delta>0$ (fixed independently of $\beta$), then
    \[ \frac{\hatpi_V(\phi_o=h)}{\hatpi_V(\phi_o=h-1)} \leq e^{-\frac34\beta r} + e^{\epsilon_\beta r}\frac{\hatpi_{\cB_r(o)}(\phi_o=h)}{\hatpi_{V}(\phi_o=h-1)} 
    \leq e^{-\frac34\beta \delta R} + e^{-c_0 \beta \log(\frac1\delta) R^2}\frac{\hatpi_{V}(\phi_o=h)}{\hatpi_{V}(\phi_o=h-1)}\,, 
    \]
    using that $\hatpi_{\cB_r(o)}(\phi_o=h) \leq e^{-c_0\beta \log(\frac1\delta)R^2}\hatpi_{\cB_R(o)}(\phi_o=h)$ and that $\hatpi_{\cB_R(o)}(\phi_o=h)\leq e^{\epsilon_\beta R}\hatpi_V(\phi_o=h)$ using \cite[Lem.~3.6]{LMS16} for the first inequality and \cite[Cor.~3.9]{LMS16} for the second one\footnote{Here one appeals only to the first part of \cite[Cor.~3.9]{LMS16}, which is valid for any $V\supset \cB_R(o)$.}.
    Rearranging the above equation yields the sought upper bound on $\hatpi_V(\phi_o=h)/\hatpi_V(\phi_o=h-1)$ in \cref{eq:LD-ratio}.

    It is for \cref{eq:LD-inf} where one needs to extend the radius of the ball contained in $V$ from $R$ to $c R$; more precisely, one has the following more general form of \cite[Cor.~3.9]{LMS16}, via the exact same proof\footnote{The upper bound in that result used the fact  $r\geq 2c_0 R$ to replace $\hatpi_V(\phi_o=h-1)$ in \cite[Eq.~(3.12)]{LMS16} by $\hatpi_V(\phi_o=h)$; the lower bound holds for all $r$.}:
    \begin{claim}[extension of {\cite[Cor.~3.9]{LMS16}} via the same proof] \label{clm:hatpi-V-vs-Br}
    Let $c_0>0$ be an absolute constant satisfying, for all $h\geq 1$ and $V\supset \cB_R(o)$, that $\hatpi_V(\phi_o=h)/\hatpi_V(\phi_o=h-1) \geq \exp(-c_0 \beta h/\log h)$. Setting $r = \lceil 2 c_0 h/\log h\rceil$, 
    for every $V\supset \cB_{r}(o)$ one has
        \[ e^{-\epsilon_\beta r} \leq \frac{\hatpi_V(\phi_o=h)}{\hatpi_{\cB_r(o)}(\phi_o=h)} \leq e^{\epsilon_\beta r}\,.\] 
    \end{claim}
    The required bound in \cref{eq:LD} now follows from \cite[Lem.~3.10]{LMS16}, which showed that
    \[\left|\log \hatpi_{\cB_r(o)}(\phi_o=h) - 2\hatpi \beta \frac{h^2}{\log r}\right| \leq c \beta \frac{h^2}{\log^2 r} + c\beta r^2\,,
    \]
    where $c>0$ is some absolute constant\footnote{As seen in the short proof of \cite[Lem.~3.10]{LMS16}, the constant $c'$ associated with the error term $e^{c' r^2}$ in its statement is of the form $c' = c \beta$ for some absolute constant $c>0$.}.

        It remains to establish \cref{eq:LD-conditional}, via a small modification of the argument of \cite{LMS16}. Letting
		\[ X = \max_{x\sim z}\phi_x\quad,\quad Y := \min_{x\sim z}\phi_x\,,\] 
		the proof of \cref{eq:LD-conditional-inf} in~\cite{LMS16} considered the events $\{X \leq h\}$ and $\{Y \geq h-\delta\sqrt{h/\log h}\}$ to derive the sought estimate. Here we will instead consider $\{X \leq h+\delta_1 h/\log h\}$ and $\{Y \geq h - \delta_2 h/\log h\}$. Specifically, let $c'>0$ be the constant in the upper bound form \cref{eq:LD-ratio}, and let
		\[ E_1 := \{ X \leq h + \Delta\}\quad ,\quad E_2 := \{ Y \geq h-5\Delta\}\quad\mbox{for}\qquad \Delta := \Big\lfloor (c'/200)\frac{h}{\log h} \Big\rfloor\,.\]
		By iterating the upper bound in \cref{eq:LD-ratio}, using here that $\cB_{r}(x)\subset V$ for each $x\sim z$, one has that
		\begin{align}
			\hatpi_V^0(E_1^c \mid \phi_o = h) &\leq  4 \max_{x\sim z}\frac{\hatpi_V^0(\phi_x > h + \Delta)}{\hatpi_V^0(\phi_o = h)} 
            \leq 4 \max_{x\sim z} \frac{\hatpi_V^0(\phi_x = h)}{\hatpi_V^0(\phi_o=h)} \exp\Big(-c'\beta \Delta \frac{h}{\log h}\Big)
            \nonumber\\ &\leq \exp\bigg(-(c'-o(1)) \beta \frac{ h \Delta}{\log h}\bigg)\,,\label{eq:E1c-upper}
		\end{align}
        with the second line following from  \cref{clm:hatpi-V-vs-Br} to show that $\hatpi_V(\phi_x=h)/\hatpi_V(\phi_o=h) \leq \exp(\epsilon_\beta r)$, which is negligible as $r \asymp \frac h{\log h} = o( \Delta \frac{h}{\log h})$.
        It thus suffices to show that, for some $c>0$,
		\begin{align}\label{eq:LD-conditional-reduced-to-E1}
			\hatpi_V^0(\phi_z = h \mid \phi_o = h,\, E_1) &\leq  \exp\Big(-c \beta \frac{h^2}{\log^2 h}\Big)\,.
		\end{align}
		Next, on the events $E_1\cap E_2^c$, we can reveal $\phi$ at the neighbors $\{x_i\}_{i=1}^4$ of $z$ and then increase those values by monotonicity, to find that
		\begin{align*} \hatpi_V^0(\phi_z = h \mid \phi_o = h,\, E_1,\, E_2^c) &\leq
			\hatpi_V^0\left(\phi_z \geq h \mid \phi_{x_1} = h-5\Delta\,,\,\phi_{x_i}=h+\Delta\;\mbox{ for $2\leq i\leq 4$}\right) \\ &\leq (1+\epsilon_\beta)e^{-\beta(3 \Delta^2 + (5\Delta)^2)}=(1+\epsilon_\beta)e^{-28\beta \Delta^2}\,.\end{align*}
		Thus, bounding $\hatpi_V^0(\phi_z = h \mid \phi_o = h,\, E_1) \leq 
		\hatpi_V^0(\phi_z = h \mid \phi_o = h,\, E_1,\,E_2^c) + \hatpi_V^0(E_2 \mid \phi_o = h,\, E_1)
		$,
		in order to show \cref{eq:LD-conditional-reduced-to-E1}, it is enough to show that for some $c>0$,
		\begin{align}\label{eq:LD-reduced-to-E2-given-E1}
			\hatpi_V^0(E_2 \mid \phi_o = h,\, E_1) &\leq  \exp\Big(-c \beta \frac{h^2}{\log^2 h}\Big)\,.
		\end{align}
		Again by monotonicity,
		\begin{align} \hatpi_V^0(\phi_z \geq h+\Delta \mid \phi_o = h,\, E_1,\, E_2) &\geq 
			\hatpi_V^0\left(\phi_z \geq h+\Delta \mid \phi_{x_i} = h-5\Delta \;\mbox{ for $1\leq i\leq 4$}\right)\nonumber\\
			&\geq (1-\epsilon_\beta)e^{-4\beta  (6\Delta)^2}=
			(1-\epsilon_\beta)e^{-144\beta\Delta^2}\,.
			\label{eq:eta_geq-h+Delta-lower-bound}\end{align}
		On the other hand,
		\begin{align*} \hatpi_V^0(\phi_z \geq h+\Delta&\mid \phi_o =h,\, E_1\,, E_2) \leq \frac{\hatpi_V^0(\phi_z \geq h+\Delta\mid \phi_o = h\,,E_1)}{\hatpi_V^0(E_2\mid \phi_o=h\,,E_1)} \\ &\leq \frac{\hatpi_V^0(\phi_z\geq h+\Delta\mid \phi_o=h) }{\hatpi_V^0(E_2\mid \phi_o=h,\,E_1)\hatpi_V^0(E_1\mid\phi_o=h)} \leq \frac{
				\exp(-(c'-o(1))\beta\frac{h\Delta}{\log h})}
			{\hatpi_V^0(E_2\mid \phi_o=h,\,E_1)}\,,
		\end{align*}
		where the last transition used \cref{eq:E1c-upper} to show that $\hatpi_V^0(E_1\mid\phi_o=h) = 1-o(1)$, whereas the numerator was bounded from above first by $\hatpi_V^0(\phi_z \geq h+\Delta) / \hatpi_V^0(\phi_o=h)$ and then by iterating the upper bound in \cref{eq:LD-ratio} and thereafter using \cref{clm:hatpi-V-vs-Br}, as was done in \cref{eq:E1c-upper}.
		Combining this with \cref{eq:eta_geq-h+Delta-lower-bound}, we find that  \[ \hatpi_V^0(E_2 \mid\phi_o=h\,,\,E_1) \leq (1+\epsilon_\beta)\exp\bigg(\beta\Big(144\Delta-(c'-o(1))\frac{h}{\log h} \Big) \Delta\bigg) \leq \exp\Big(-(c'/4)\beta \frac{h\Delta}{\log h}\Big)\]
		by the choice of $\Delta$. This establishes \cref{eq:LD-reduced-to-E2-given-E1} and completes the proof.
	\end{proof}
	The following lemma shows that the large deviation estimate in \cref{eq:LD} is also an upper bound for all points in $V$, in particular at sites $x$ near the boundary of $V$ (whereas \cref{thm:LD-DG} was only applicable for sites $x$ at distance at least $c h/\log h$ from $\partial V$).
	\begin{lemma}\label{lem:UB-LD-any-point}
		There exists $c_0>0$ such that, for $\beta$ large and every $V \subset \Z^2$, $h \geq 2$, and $x \in V$,
		\begin{equation}
			\hatpi_V^0(\phi_x = h) \leq \exp\bigg(-2\pi\beta\frac{h^2}{\log h} + c_0 \frac{h^2}{\log^2 h}\bigg)\,.
		\end{equation}
	\end{lemma}
	\begin{proof}
		Let $r = \lceil 2c h/\log h\rceil$ for the constant $c>0$ from \cref{thm:LD-DG}, and let $V' = V \cup \cB_r(x)$. Then
		\[  \hatpi_{V'}^0(\phi_x \geq h) \geq \hatpi_{V'}^0(\phi_x \geq h,\, \phi\restriction_{\partial V \setminus \partial V'} \geq 0) \geq \bigg(\inf_{\substack{\xi\,:\;\xi\restriction_{\partial V'}=0, \\ \xi\restriction_{\partial V\setminus \partial V'}\geq 0}}
		\hatpi_V^\xi(\phi_x\geq h)    
		\bigg)
		\hatpi_{V'}^0(
		\phi\restriction_{\partial V \setminus \partial V'}\geq 0) \,. \]
		By FKG, this is minimized at $\xi\equiv 0$; namely,
		\begin{align}\label{eq:force-zero-boundary}
			\hatpi_{V'}^0(\phi_x \geq h) \geq \hatpi_V^0(\phi_x \geq h)\hatpi_{V'}^0(\phi\restriction_{\partial V \setminus \partial V'} \geq 0)
			\geq \hatpi_V^0(\phi_x \geq h)\prod_{x \in \partial V \setminus \partial V'}\hatpi_{V'}^0(\phi_x \geq 0)\,,
		\end{align}
		with the last inequality again following from FKG.
		A standard Peierls argument (see, e.g., \cite{BrandenbergerWayne82}) shows that for any $h \geq 0$ and any $V$, \[\hatpi_V^0(\phi_x > h) \leq \epsilon_\beta\hatpi_V^0(\phi_x = h)\,,
		\]
		where $\epsilon_\beta \to 0$ as $\beta \to \infty$ (and similarly, $\hat\pi_V^0(\phi_x<-h) \leq \epsilon_\beta\hatpi_V^0(\phi_x=-h)$ by symmetry).  In particular, conditional on $\phi\equiv0$ on a subset of $\partial V\setminus\partial V'$, the probability that $\phi_x=0$ for an additional vertex $x$ along that boundary is at least $\exp(-\epsilon_\beta)$, and so     
		\begin{align}\label{eq:Peierls-rigidity}\prod_{x \in \partial V \setminus \partial V'}\hatpi_{V'}^0(\phi_x \geq 0) \geq e^{-\epsilon_\beta|\partial V \setminus \partial V'|} \geq e^{-\epsilon_\beta h^2 / \log^2 h} \,,\end{align}
		where the last inequality holds (for a different $\epsilon_\beta$) using that,     since $\partial V \setminus \partial V' \subset \cB_r(x)$, we can infer that $|\partial V \setminus \partial V'| \leq \pi r^2 \leq  O( h^2/\log^2 h)$. Combining \cref{eq:force-zero-boundary,eq:Peierls-rigidity}, we find that
		\begin{equation}
			\hatpi_V^0(\phi_x = h) \leq e^{\epsilon_\beta h^2/\log^2h}\hatpi_{V'}^0(\phi_x = h),\,
		\end{equation}
		whence the proof concludes via \cref{thm:LD-DG}.
	\end{proof}
	
	We can now bound the probability of the floor event. Our bound will be asymptotically tight, whereas the analogous estimate in \cite[Prop.~4.4]{LMS16} (which considered a similar event: rather than $B=\bigcap_{x\in V}\{\phi_x\geq -h\}$ addressed here, it pertained $B$ intersected with another event forbidding \texttt{large} nonzero paths) had a $(1+o(1))$ term, not as a prefactor, but within the exponent in that probability. It is crucial that we have this more precise estimate for our cluster expansion expression in \cref{sec:Cluster-Expansion}, and it comes at a cost of area and boundary constraints on the domain. As we describe below (see \cref{rem:bbq-worse-setup}), these constraints are more stringent in \ZGFF compared to the case of \SOS, where the corresponding estimate was applicable to an $L^{2/3+\epsilon}\times L^{2/3+\epsilon}$ box (whereas we can only afford to address a $L^{2/3}(\log L)^c\times L^{1/3}(\log L)^c$ box in the setting of \ZGFF).
    (NB.\ the probability estimated below resembles  $\hatpi_V^0(\phi_x \geq 0,\forall x\in V)$---i.e., replacing $h$ by $0$ in the lemma; the asymptotic of $\log\hatpi_V^0(\phi_x\geq 0,\forall x\in V)$ was established in \cite{CMT17} for the $|\nabla\phi|^p$ models (and previously in \cite{CMT15} for the \SOS); here we are interested in the asymptotic probability for $h=H+1-n$.)
    
    \begin{lemma}\label{lem:area-estimate}
		Fix $0<\delta<\frac13$, $\kappa > 0$, and $n\geq 1$. Let $V \subset \Z^2$ be a connected region, let $F\subset V$ be a subset satisfying $|F| \leq Le^{\kappa\sqrt{\log L}}$ and $|\partial F| = O(L^{1-\delta})$, and set $h = H+1-n$. Then
		\[
		\hatpi_V^0(\phi_x \geq -h,\, \forall x \in F) = (1+o(1))\exp\Big(-\hatpi_\infty(\phi_o < -h)|F| \Big)\,.
		\]
	\end{lemma}
	\begin{proof}
		We begin with the lower bound. By FKG, 
		\begin{align}\nonumber
			\hatpi_{V}^0(\phi_x \geq -h,\, \forall x \in F) &\geq \prod_{x \in F} \hatpi_{V}^0(\phi_x \geq -h) = \prod_{x \in F}\left(1 - \hatpi_{V}^0(\phi_x < -h)\right) \\ \label{eq:area-lower-bound}
			&\geq\exp\bigg(-\sum_{x \in F}\hatpi_{V}^0(\phi_x < -h) -\sum_{x \in F} \hatpi_{V}^0(\phi_x < -h)^2\bigg)\,,
		\end{align}
		using $1-s\geq \exp(-s-s^2)$ for $0\leq s \leq \frac12$, applied to $\hatpi_V^0(\phi_x<-h)=o(1)$. Moreover, by \cref{eq:LD-ratio-inf,eq:LD-inf} and the definition in \cref{eq:H-def} of $H$, we have
		$ \hatpi_{\infty}(\phi_o < -h) \leq L^{-1+o(1)} $
		(moving from $\hatpi_{\infty}(\phi_o = -(h+1))$ for $h=H$ to $h=H+1-n$ incurs a multiplicative cost of $\exp\big(n c_0 \beta \frac{H}{\log H}\big) = L^{o(1)}$), and hence by \cref{lem:UB-LD-any-point}
		we have that for all $x$,
		\begin{equation}\label{eq:x-close-to-boundary}
			\hatpi_{V}^0(\phi_o < -h) \leq L^{-1+o(1)}\,,
		\end{equation}
		as moving from the upper bound $\hatpi_{\infty}(\phi_o < -h)$ to $\hatpi_{V}^0(\phi_o = -h)$ incurs an additional multiplicative cost of $\exp\big( c \frac{H^2}{\log^2 H}\big) = L^{o(1)}$. 
		This implies that 
		\[ \sum_{x \in F} \hatpi_{V}^0(\phi_x < -h)^2 \leq L^{-2+o(1)}|F| \leq L^{-1+o(1)}=o(1)\]
		using our hypothesis on $|F|$ (with room to spare: the stringent requirement on $|F|$ would come from the upper bound). \Cref{eq:x-close-to-boundary} further  implies that 
	\begin{equation}\label{eq:sum-dist-log2-from-F}\sum_{\substack{x\in F\\ \dist(x,\partial F)\leq \log L}} \hatpi_V^0(\phi_x<-h) \leq L^{-1+o(1)}|\partial F| \leq L^{-\delta+o(1)}=o(1)
        \end{equation}
		by our assumption on $|\partial F|$. 
		A straightforward consequence of the Peierls argument of \cite{BrandenbergerWayne82} is the following decay of correlation property of $\hatpi$: for every $x\in F$ and $R$ such that $\dist(x,\partial V)>2R$,
		\begin{equation}\label{eq:decay-of-corr}
			\left\|\hatpi_{V}^0(\phi\restriction_{\cB_R(x)}\in\cdot)-\hatpi_\infty(\phi\restriction_{\cB_R(o)}\in\cdot)\right\|_{\tv} \leq e^{-c\beta R}\,.
		\end{equation}
		(To see this, note that one can couple $\phi\sim\hatpi_V^0$ and $\phi'\sim \hatpi_\infty^0$ so that they agree on $\cB_R(x)$ and $\cB_R(o)$, respectively,  provided there is no $*$-connected path of sites $P$ connecting $\partial B_{R/2}(x)$ to $\partial V$, along which every site $y$ has $\phi_y\neq 0$ or $\phi'_y\neq 0$; in particular, either $\phi$ or $\phi'$ have more than $|P|/2$ nonzero sites along $P$. This occurs with probability at most $\exp(-c \beta R)$ by Peierls applied to both~$\phi,\phi'$.)
		Thus, if $x$ is such that $\dist(x,\partial V)\geq \log L\}$ (which holds if $\dist(x,\partial F)\geq \log L$ as $F\subset V$) then 
		\[\left|\hatpi_{V}^0(\phi_x < -h) - \hatpi_\infty(\phi_o < -h)\right| \leq e^{-c\beta \log L}\,,\]
		and hence our bound on $|F|$ again yields
		\begin{align*} \sum_{\substack{x\in F\\ \dist(x,\partial F)\geq \log L}} \hatpi_V^0(\phi_x<-h) &\leq |F| (\hatpi_\infty(\phi_o<-h) + e^{-c\beta\log L}) \leq |F|\hatpi_\infty(\phi_o<-h) +o(1)\,.  \end{align*}
		Overall, we conclude a lower bound of 
		\begin{equation}
			\hatpi_V^0(\phi_x \geq -h,\, \forall x \in F) \geq (1-o(1))\exp\big(-\hatpi_\infty(\phi_o < -h)|F| \big)\,.
		\end{equation}
		
		We proceed to the upper bound, which is where the stronger assumptions on $|F|$ will be needed. The proof will follow a  grid partitioning argument using Bonferroni's inequalities, as was done in \cite[Prop.~7.7]{CLMST14} and~\cite[Prop.~A.1]{CLMST16} for the \SOS model, in \cite[Prop.~6.2]{GheissariLubetzky23} for the 3\Dim Ising model, and in \cite[Prop~4.4]{LMS16} for the \ZGFF (where it was applied more crudely, as the focus there was the height histogram of the \ZGFF surface as opposed to the fluctuations of its level lines).
		
		Let \[ \mathfrak{u} := \delta/3\quad,\quad \mathfrak{s}:=\delta/4\,,\]
		and partition $\Z^2$ into squares $P_i$ of side-length $L^{\fu} + L^{\fs}$, and let $Q_i\subset P_i$ be the concentric squares of side-length $L^{\fu}$. We will refer to $P_i\setminus Q_i$ as the \emph{shell} of the square $Q_i$. Observe that if
		\begin{align*}
			V_1 &= \bigcup \big\{ P_i \,:\; P_i\cap F^c \neq \emptyset \big\}\,,\\ 
			V_2 &= \bigcup \big\{ (P_i \setminus Q_i) \cap F \,:\; P_i\cap V_1=\emptyset\big\}\,,\\
			V_3 &= F\setminus (V_1\cup V_2)\,,
		\end{align*}
		then $|V_1| \leq |\partial F| L^{2\fu} = O(L^{1-\delta/3})$ by our assumption on $\partial F$ and choice of $\fu$, while $|V_2|\leq L^{\fs-\fu}|F|$ (each shell $P_i\setminus Q_i$, fully contained in $F$, has $|P_i|L^{\fs-\fu}$ sites), which is at most $L^{1-\delta/12+o(1)}$ by the assumption on $F$ (and choice of $\fu,\fs$). Combined, we have
		\[ \hatpi_\infty(\phi_o<-h)(|V_1|+|V_2|) \leq L^{-\delta/12+o(1)} = o(1)\,,\]
		so it suffices to show that, for the increased event that restricts its attention only to $V_3$ we have 
		\begin{equation}
			\label{eq:V3} 
			\hatpi_V^0\bigg(\bigcap_{x\in V_3}\{\phi_x\geq -h\}\bigg) \leq (1+o(1))\exp\Big( -\hatpi_\infty(\phi_o<-h)|V_3|\Big)\,.
		\end{equation}
		We now aim to decrease the event under consideration via intersecting it with 
		\[ \fD_\fs = \big\{ \mbox{$\phi$ has no disagreement polymers $\gamma$ with $\diam(\gamma) \geq L^{\fs}/4$}\big\}\,.\]
		Indeed, by the standard Peierls argument,
		\[ \hatpi_V^0(\fD_\fs^c) \leq |V|\exp\big(-(\beta-c)L^{\fs}/4\big) = o\Big(\exp\Big( -\hatpi_\infty(\phi_o<-h)|V_3|\Big)\Big)
		\]
		since $\hatpi_\infty(\phi_o<-h)|V_3| \leq L^{o(1)}$, and thus \cref{eq:V3} will follow once we show that
		\begin{align}\label{eq:V3-cap-B} 
			\hatpi_V^0\bigg( \fD_\fs \cap \bigcap_{x\in V_3}\{\phi_x\geq -h\}\bigg)  \leq (1+o(1))\exp\Big( -\hatpi_\infty(\phi_o<-h)|V_3|\Big)\,.\end{align}
		Denoting by $\{i_j\}_{j\geq 1}$ the indices of the $Q_i$'s that contribute to $V_3$, we reveal $\phi\restriction_{P_{i_j}}$ in step $j$, and denote the associated filtration by $(\cF_j)_{j\geq 0}$. For brevity, write 
        \[ Q'_{i_j}:= F\cap Q_{i_j}\,.\]
        We will argue that
		\begin{equation}\label{eq:grill-good-points}
			\hatpi_V^0\bigg(\fD_\fs\cap \bigcap_{x \in Q'_{i_j}}
			\{\phi_x \geq -h\} \;\Big|\; \cF_{j-1}\bigg) \leq \exp\Big(-\Big(1 -e^{-c\beta \frac{h^2}{\log^2 h}}\Big)\hatpi_\infty(\phi_o < -h)|Q'_{i_j}| \Big)\,,
		\end{equation}
		from which \cref{eq:V3-cap-B} will readily follow as $\sum_j |Q'_{i_j}|=|V_3|\leq |F| \leq Le^{\kappa\sqrt{\log L}}$ by assumption, so
		\begin{align} e^{-c\beta\frac{h^2}{\log^2 h}}\,\hatpi_\infty(\phi_o<-h)\sum_{j}|Q'_{i_j}| &\leq 
			e^{-c\beta\frac{h^2}{\log^2 h}+c_0  \beta n \frac{h}{\log h}}\,\hatpi_\infty(\phi_o<-H) |V_3| \nonumber\\  &\leq e^{-(c\beta-o(1))\frac{h^2}{\log^2 h}}\,,\label{eq:grill-absorbing-cond}\end{align}
		using the lower bound of \cref{eq:LD-ratio-inf} (or the one in \cref{eq:LD-ratio}) in the first inequality (to move from $\hatpi_\infty(\phi_o<-h)$ to $\hatpi_\infty(\phi_o<-H)$ at a cost of $e^{c_0\beta n\frac{h}{\log h}}$) and \cref{eq:H-def} and the hypothesis on $|F|$ in the second inequality, as  $h^2/\log^2 h \asymp \log L / \log\log L$, thus $e^{\kappa \sqrt{\log L}} = o (\exp( ch^2/\log^2 h))$ for any fixed $\kappa,c>0$.
		
		To establish \cref{eq:grill-good-points}, denote by $\sfC$ the outermost $*$-connected circuit of zeros in $P_{i_j}$, noting that for every $\phi\in \fD_\fs$, the distance of $\sfC$ from $\partial P_{i_j}$ cannot exceed $L^{\fs}/4$. 
		Letting $\sfU$ denote the interior of~$\sfC$, we see that $\cB_r(x)\subset \sfU$ for every $x\in Q_{i_j}$ with $r \asymp L^{\fs} \gg \frac{h}{\log h}$, our assumption in \cref{thm:LD-DG}. 
		We now apply Bonferroni's inequalities to infer that
		\[
		\hatpi_{\sfU}^0\bigg(\bigcap_{x \in Q'_{i_j}}
		\{\phi_x \geq -h\}\bigg) \leq 1-\sum_{x\in Q'_{i_j}} \hatpi_\sfU^0(\phi_x < -h) + \frac12 \sum_{\substack{x,y\in Q'_{i_j}\\ x\neq y}} \hatpi_\sfU^0(\phi_x < -h\,,\, \phi_y<-h)\,.
		\]
		By the decay of correlations bound in \cref{eq:decay-of-corr}, we can replace $\hatpi_{\sfU}^0$ with $\hatpi_\infty$ in each of the sums at an additive error cost of 
		\[ |Q'_{i_j}|^2\exp(-c\beta L^{\fs})
		= \chi_1\, |Q'_{i_j}|\hatpi_\infty(\phi_o<-h)  \quad\mbox{for}\quad \chi_1 < \exp\left(-(c\beta-o(1))L^{\fs}\right) < L^{-100}\,.\]
		Next, we again use the decorrelation estimate in \cref{eq:decay-of-corr} to infer that
		\begin{align*}\sum_{\substack{x,y\in Q'_{i_j}\\\dist(x,y)\geq \log L}} \hatpi_\infty(\phi_x<-h,\, \phi_y<-h) &\leq \Big(\sum_{x\in Q'_{i_j}} \hatpi_\infty(\phi_x<-h)\Big)^2 + |Q'_{i_j}|^2 e^{-c\beta \log L} \\
			&\leq (\chi_2+\chi_3) |Q'_{i_j}|\hatpi_\infty(\phi_o<-h)\,, \end{align*}
		where $\chi_2,\chi_3$ correspond to the two terms on the right-hand of the first line, and satisfy
		\[ \chi_2 < L^{-1+2\fu+o(1)}<L^{-1+\delta}\quad,\quad\chi_3 < \exp\left(-(c\beta-o(1))\log L\right) < L^{-1}\,.\]
		Finally, and this will be the dominant term in our error, we address pairs $x,y\in Q'_{i_j}$ at distance at most $\log L$ as follows.
        Noting that $\hatpi_\infty(\phi_x>h,\,\phi_y>h) \leq (1+\epsilon_\beta)\hatpi_\infty(\phi_x=h+1,\,\phi_y=h+1)$ by a standard Peierls argument (enumerating the two inner most $(h+2)$ level-line loops that surround each of $x,y$ (possibly this is the same loop), and subtracting $1$ in each of their interiors if nonempty), it suffices to bound the sum of said probability. This is achieved by         
        \cref{eq:LD-conditional} as follows:
		\begin{align*}\sum_{\substack{x,y\in Q'_{i_j}\\ 0<\dist(x,y)\leq \log L}} \!\!\!\!\!\hatpi_\infty(\phi_x=-h-1,\, \phi_y=-h-1) &= 
			\sum_{x\in Q'_{i_j}} 
			\hatpi_\infty(\phi_x=-h-1) \\ &\qquad\times \!\!
            \sum_{\substack{y\in Q'_{i_j} \\ 0<\dist(x,y)\leq \log L}} \!\!\!\!\!\hatpi_\infty(\phi_y=-h-1\mid \phi_x=-h-1) \\
			&\leq \chi_4\, \hatpi_\infty(\phi_o<-h) |Q'_{i_j}|\,,
		\end{align*}
		where
		\[ \chi_4 \leq (\log L)^2 \exp\Big(-c\beta \frac{h^2}{\log^2 h}\Big) \,.\]
		Summing these, we see that for $\chi = \sum_{i=1}^4 \chi_i \leq \exp(-(c\beta-o(1))\frac{h^2}{\log^2 h})$ (dominated by $\chi_4$), we have
		\[
		\hatpi_{\sfU}^0\bigg(\bigcap_{x \in Q'_{i_j}}
		\{\phi_x \geq -h\}\bigg) \leq 1- ( 1- \chi)\hatpi_\infty^0(\phi_o<-h)|Q'_{i_j}| \leq \exp\Big(-(1-\chi)\hatpi_\infty^0(\phi_o<-h)|Q'_{i_j}|\Big)\,,
		\]
		thus establishing \cref{eq:grill-good-points} and completing the proof.
	\end{proof}
	
	\begin{remark}\label{rem:bbq:cond-bound}
		The probability $\hatpi_\infty(\phi_y=h\mid \phi_x=h)$, which we controlled through \eqref{eq:LD-conditional}, governed the error-term in the proof of \cref{lem:area-estimate}.
		Had we instead used \cref{eq:LD-conditional-inf}---which features an exponent of $c\beta \frac{h}{\log h}$ as opposed to $c \beta \frac{h^2}{\log^2 h}$ from \cref{eq:LD-conditional}---it would have competed with the exponent $c_0\beta n \frac{h}{\log h}$ from \cref{eq:LD-ratio-inf} which appears in our estimate for $\hatpi_\infty(\phi_o>-h)$ already when considering $h=H+1-n$ for $n=2$. 
		We will need to consider $n=2$, the second-from-top level line, already to address the law of the top level line, as we need to establish that there is ample spacing between them (and of course, in \cref{thm:1} we proceed to further obtain the joint law of the $m
		$ top level lines).
	\end{remark}
	\begin{remark}
		\label{rem:bbq-worse-setup} 
		The analysis of $\dist(x,y)\leq\log L$ in the final part of the upper bound generated the term  $\chi_4 |F| \hatpi_\infty(\phi_o<-h) \leq p |F|/L$ for  $p=\exp(-c\frac{\log L}{\log\log L}) = L^{-o(1)}$ 
		derived from \cref{eq:LD-conditional}. Thus, we can only obtain an $o(1)$ additive error in the exponent if $|F|=L^{1+o(1)}$, and moreover we can handle $Le^{c\sqrt{\log L}}$ but not any arbitrary $L^{1+o(1)}$. The analogous setting in \SOS had $p\leq L^{-1/2}$ (due to the different nature of the large deviation problem), resulting in a much larger applicable domain of area $L^{4/3+o(1)}$ in \cite[Prop.~A.1]{CLMST16} (see the hypothesis on the domain $\Lambda$ in Eq.~(A.2) there).
	\end{remark}
	
	\subsection{Cluster expansion and proof of \cref{prop:CE-law-with-floor}}\label{sec:Cluster-Expansion}
	Cluster expansion for disagreement polymers under without a floor (i.e., under $\hatpi_V^\xi$) will follow from classical results, as formulated in the next proposition. We will thereafter use it in conjunction with \cref{lem:area-estimate} to derive \cref{prop:CE-law-with-floor}.
	
	\begin{proposition}[Cluster expansion]\label{prop:CE-law}
		Let $V\subset \Z^2$ be a connected domain, and consider the \ZGFF model $\hatpi_V^\xi$ with boundary conditions $\xi$ that are $1$ on a $*$-connected path in $\partialvtx V$ and $0$ elsewhere so that they induce a unique disagreement polymer $(\gamma,\{D_i\},\{h_i\})$ in $V\cup \partialvtx V$ that contains boundary disagreements. Then for $\beta\geq \beta_0$, the law of this unique disagreement polymer is given by 
		\begin{equation}\label{eq:CE-without-area}
			\hatpi^\xi_V(\gamma) =\frac1{Z} \exp\bigg(-\sE_\beta(\gamma) + \sum_{\substack{\sfW \subset V \\ 
					\sfW \cap \Delta_\gamma \neq \emptyset}} \Phi(\sfW)\bigg)
		\end{equation}
		for a normalizer $Z = Z(\xi,V,\beta)$ and a decoration function $\Phi(\sfW)$ as per \cref{def:disagree-polymer-len-energy-decor}.    
	\end{proposition}
	\begin{proof}
		The proof will be a standard application of the framework of the Koteck\'{y}--Preiss~\cite{KoteckyPreiss86} cluster expansion.
		Observe that, by \cref{def:gamma,def:disagree-polymer-len-energy-decor}, for any connected domain $V$ we have
		\[ \hatpi^0_V(\phi) = \frac1{\widehat \cZ^0_V} \exp\Big(-\sum_{\gamma \in \cP_\phi} \sE_\beta(\gamma)\Big)\,,\]
		where the sum goes over the disagreement polymers arising from $\phi$, each of the form $(\gamma,\{D_i\},\{h_i\})$.
		Equivalently, letting $\cP_V$ be the set of all disagreement polymers that can arise in $\phi\sim\hatpi_V^0$ (or, equivalently, arising from any $\phi\sim\hatpi_V^k$ for boundary conditions all-$k$), we put \[ \widehat \cZ^0_{\cP_V} = \sum_{\substack{\{\gamma_j\}\subset \cP_V\\ \gamma_j\text{ pairwise disjoint}}} \prod_j e^{-\sE_\beta(\gamma_j)}\,,\] 
		so that the partition function $\widehat \cZ^0_V$ from above is synonymous with~$\widehat \cZ^0_{\cP_V}$.    
		Since $\sE_\beta(\gamma)\geq \beta \sN(\gamma)$,
		\[ \sum_{\gamma\,:\;\gamma\cap\gamma_0\neq \emptyset} e^{(\beta-C_0)\sN(\gamma)}e^{-\sE_\beta(\gamma)}
		\leq |\gamma_0| \sum_{\gamma\ni e_0}e^{-C_0 \sN(\gamma)} \leq \sN(\gamma_0) \]
		for a large enough $C_0$,     
		satisfying the criterion of the main theorem of \cite{KoteckyPreiss86} for $a(\gamma)=\sN(\gamma)$ and $\dist(\gamma)=(\beta-C_0-1)\sN(\gamma)$ in their notation. By that theorem, one has
		\[ \log \widehat \cZ^0_{\cP_V} = \sum_{\fP\subset\cP_V} \Phi_0(\fP) \quad\mbox{where}\quad \Phi_0(\fP) = \sum_{\fP'\subset \fP} \ (-1)^{|\fP|-|\fP'|} \log \widehat \cZ_{\fP'}^0\,,\]
		and the function $\Phi_0$ satisfies $\Phi_0(\{\gamma_j\})=0$ if the graph formed by the edges of $\{\gamma_j\}$ is not connected, and  
		$ \sum_{\fP: \fP\cap\gamma_0\neq \emptyset} |\Phi_0(\fP)| e^{(\beta-C_0-1)\sN(\fP)}\leq \sN(\gamma_0)$, where $\sN(\fP):=\sum_j \sN(\gamma_j)$ for $\fP=\{\gamma_j\}$.
		It follows that $|\Phi_0(\fP)|\leq \exp\big(-(\beta-C_1)\sN(\fP)\big)$. One can then, as done in \cite[Sec.~3.9]{DKS92} (see also~\cite[{\S}A.2]{CLMST14}), define for any subset of vertices $\sfW\subset \Z^2$,
		\[\Phi(\sfW) = \sum\Big\{ 
		\Phi_0(\fP)\,:\;
		\fP=\big\{(\gamma_j,\{D_{j,i}\},\{h_{j,i}\})\big\}\mbox{ such that }\bigcup D_{j,i}=\sfW\}
		\Big\}\,,\] 
		whence (not just under $0$ boundary conditions, but for any all-$k$ boundary conditions)
		\begin{equation}\label{eq:log-Z-0} \log \widehat \cZ_{V}^0 = \sum_{\sfW\subset V} \Phi(\sfW)\,,\end{equation}
		and the aforementioned bound on $\Phi_0(\fP)$ implies that (after summing it over all possible $\{h_{j,i}\}$ to get a decay of $\exp(-(\beta-C)\sum |\gamma_j|)$ and noting that the boundary bonds $\partial \sfW$ are a subset of $\bigcup_j \gamma_j$),
		\[ |\Phi(\sfW)| \leq \exp\left(-(\beta-C)\bd(\sfW)\right)\,,\]
		recalling that $\bd(\sfW)$ is the size of the smallest connected set of bonds containing all boundary bonds of $\sfW$.
		The fact that $\Phi$ is invariant under translation ($\Phi(\sfW)=\Phi(\sfW+x)$) as well as under rotation by $\pi/2$ or reflection with respect to one of the axes is apparent from its combinatorial construction (as the graphs formed by the $\gamma_j$'s are isomorphic to the corresponding translations/rotations/reflections).
		
		Finally, let $\phi\sim\hatpi_V^\xi$ with $\xi$ consisting of an interval of $0$ and an interval of $1$ as per \cref{prop:CE-law}. We will move from \cref{eq:log-Z-0} to \cref{eq:CE-without-area} 
		via observing that, by definition, if $(\gamma,\{D_i\},\{h_i\})\in\cP_\phi$, then 
		\begin{align*} \hatpi_V^\xi(\gamma) = \frac1{\widehat\cZ^\xi_V} \sum_{\phi:\cP_\phi \ni \gamma} e^{-\sum_{\gamma' \in \cP_\phi} \sE_\beta(\gamma')} 
     =   \frac{\widehat \cZ_V^0}{\widehat \cZ_V^\xi}e^{-\sE_\beta(\gamma)} \frac{\prod_i \widehat \cZ_{D_i^\circ}^{h_i} }{\widehat \cZ_V^0}\,. 
		\end{align*}
		We combine this with the fact that, by \cref{eq:log-Z-0},
		\begin{align}\label{eq:hat-Z-ratio}
			\frac{\prod_i \widehat \cZ_{D_i^\circ}^{h_i}}{\widehat \cZ_V^0}  &= \exp\Big( - \sum_{\sfW\subset V} \Phi(\sfW) + 
			\sum_i \sum_{\sfW\subset D_i^\circ} \Phi(\sfW) 
			\Big) = \exp\Big(\sum_{\substack{\sfW\subset V\\\sfW \cap \Delta_\gamma\neq \emptyset}} \Phi(\sfW) \Big)\,,
		\end{align}
		to conclude the proof.
	\end{proof}
	
	\begin{proof}[Proof of \cref{prop:CE-law-with-floor}]
		We begin with the standard identity (see the proof of \cref{prop:CE-law}) 
		\begin{align*}\pi_{V;F}^\xi(\gamma) &= \frac1{\cZ_{V;F}^\xi} e^{-\sE_\beta(\gamma)}\prod_i \cZ_{D_i^\circ;F}^{h_i} \\ &= \frac{\widehat \cZ_V^0}{\cZ_{V;F}^\xi} e^{-\sE_\beta(\gamma)}\frac1{\widehat \cZ_V^0 } \prod_i\widehat \cZ_{D_i^\circ;F}^{h_i}\,\hatpi_{D_i^\circ}^{h_i}(\phi_x \geq 0,\, \forall x \in D_i^\circ\cap F)\,.
		\end{align*}
		Absorbing $\widehat \cZ_V^0 / \cZ_{V;F}^\xi$ into the partition function, then applying 
		\cref{eq:hat-Z-ratio} to rewrite $(\prod_i\widehat \cZ_{D_i^\circ}^{h_i})/\widehat \cZ_V^0$ in terms of the decoration function $\Phi$,
		establishes \cref{eq:CE-with-area}.
		
		To derive \cref{eq:CE-with-area-cond}, we let $\gamma\in E$, and wish to apply \cref{lem:area-estimate} to   $\hatpi_{D_i^\circ}^{h_i}(\phi_x\geq 0,\,\forall x\in D_i^\circ\cap F)$.
Our hypothesis on $F$ gives $|D_i^\circ\cap F| \leq |F| \leq Le^{\kappa\sqrt{\log L}}$. Next, $|\partial(D_i^\circ \cap F)| \leq |\partial (D_i\cap F)| + O(|\Delta_\gamma|)$, which is at most $|\partial (D_i\cap F)|+O(|\gamma|)$. We further claim that
\[ \partial(D_i \cap F) \subset \gamma \cup \partial F\,; \]
indeed, every bond $b$ in the left-hand is dual to some edge $uv$ for $u\in D_i \cap F$ and $v \notin D_i\cap F$, which, if $v\in F^c$, is counted in $\partial F$, and otherwise $v\in D_i^c$ and must have $b\in \gamma$ (as all other bonds of $\partial D_i$ have $v\in V^c \subset F^c$). Altogether, we conclude that
\[ |\partial (D_i^\circ \cap F)| \leq |\partial F| + O(|\gamma|) = O(L^{1-\delta})\]
by the definition of $E$, satisfying the hypothesis of \cref{lem:area-estimate}. Applying that lemma we obtain that, for $i=0,1$, 
		\begin{align*}
			\hatpi_{D_i^\circ}^{H+1-n-i}\big(\phi_x \geq 0,\, \forall x \in D_i^\circ\cap F\big) &=
			\hatpi_{D_i^\circ}^{0}\big(\phi_x \geq -(H+1-n-i),\, \forall x \in D_i^\circ\cap F\big) 
			\\ &=(1+ o(1))\exp\big(-\hatpi_\infty(\phi_o < -(H+1-n-i))|D_i \cap F| \big)\,,
		\end{align*}
using here that 
\[ \hatpi_\infty(\phi_o<-(H+1-n))|(D_i \setminus D_i^\circ)\cap F| \leq \hatpi_\infty(\phi_o<-(H+1-n))O(|\gamma|) = L^{-\delta+o(1)}=o(1)\,.\]  
		Recalling that $|F| \leq |D_0\cap F|+|D_1\cap F|+L^{1-\delta}$, we see that
		\[
		\hatpi_\infty(\phi_o < -(H-n)\big)|D_1 \cap F| =  \hatpi_\infty(\phi_o < -(H-n)\big)\big(|F|-|D_0\cap F|+O(L^{1-\delta})\big)\,,
		\]
		whereby we again have $\hatpi_\infty(\phi_o<-(H-n))L^{1-\delta}=L^{-\delta+o(1)}=o(1)$, and
		\begin{align*}
			\Big(\hatpi_\infty\big(\phi_o < -(H-n)\big)- \hatpi_\infty\big(\phi_o < -(H+1-n)\big)\Big)|D_0\cap F| &= \hatpi_\infty\big(\phi_o = -(H+1-n)\big) |D_0\cap F| \\&= \frac{|D_0\cap F|}{N_n}
		\end{align*}
		by the definition in \cref{eq:Nn-def}.
		Thus,  absorbing $\exp[-\hatpi_\infty(\phi_o<-(H-n))|F|]$ into the partition function (being independent of $\gamma$), we find that
		\begin{align*}
			\pi^\xi_{V;F}(\gamma) \propto (1+o(1))\exp\bigg(-\sE_\beta(\gamma) + \frac{|D_0\cap F|}{N_n} + \sum_{\substack{\sfW\subset V\\ \sfW \cap \Delta_\gamma \neq \emptyset}}\Phi(\sfW)\bigg)\prod_{i\geq 2}\hatpi_{D_i^\circ}^{h_i}(\phi_x \geq 0,\, \forall x \in D_i^\circ\cap F)\,,
		\end{align*}
		as required.
	\end{proof}
	
	\section{Geometry of the disagreement polymers}\label{sec:geom-disagreement-polymers}
	Our eventual goal is to show that $\gamma$ behaves like an area tilted random walk. Hence, we would like to show that upon removing the main area tilt $\frac{|D_0|}{N_n}$ from \cref{eq:CE-with-area-cond}, $\gamma$ falls under the Ornstein--Zernike setup. We will prove this in a more general polymer model setting. Throughout this section, we can take any fixed $n \in \Z_+$ and $L$ large.
	
	To begin, we need to define the set of legal polymers from $A$ to $B$. Fix any two vertices $A, B \in (\Z^2)^*$ and a simply connected domain $V$ such that $A, B$ are on $\partial V$. As $\partial V$ is the boundary in $\R^2$ of a simply connected domain, consider any conformal map that sends $\partial V$ to the unit circle centered at the origin, mapping $A$ to the point $(0, -1)$ and $B$ to the point $(0, 1)$. Let $\xi$ be the boundary condition which is $h$ along the interval of $\partial V$ that maps to the arc of the unit circle in the upper half plane, and $h-1$ along the rest of $\partial V$, where $h$ can be chosen as desired (the choice of $h$ is irrelevant in the following, all that matters here is that the boundary heights differ by 1). For every height function on $V$ with boundary condition $\xi$, there is a unique labeled disagreement polymer $\gamma$ which contains the boundary disagreements. Define $\cP_V(A, B)$ as the set of all such possible disagreement polymers in this setting. We can then extend this definition to domains $V$ with infinite volume by defining $\cP_V(A, B) = \bigcup_{V' \subset V} \cP_{V'}(A, B)$, where the union is over all simply connected $V' \subset V$ which have finite volume. 
	
	Next we will define the polymer weights. Fix $L$ and define the weight of $\gamma$ interacting with a domain $U$ with boundary conditions at height $H+1-n$ and $H-n$ as 
	\begin{equation}\label{eq:def-polymer-weights}
		\hatq^n_U(\gamma)= \exp\big(-\sE_\beta(\gamma) + \sum_{\sfW \cap \Delta_\gamma \neq \emptyset} \Phi(\sfW)\one_{\{\sfW \subset U\}}\big)\!\prod_{i \geq 2}\!\hatpi_{D_i^\circ}^{h_i}(\phi_x \geq 0,\, \forall x \in D_i^\circ)\,,
	\end{equation}
    for $\Phi(\sfW)$ as from \cref{prop:CE-law}. For brevity, define the shorthand notation \[\Phi_U(\sfW) = \Phi(\sfW)\one_{\{\sfW \subset U\}}\,.\]
	\begin{remark}\label{rem:conditions-on-Phi-U}
		It is clear that if $\Phi(\sfW)$ satisfies the properties of \cref{def:disagree-polymer-len-energy-decor}, then $\Phi_U(\sfW)$ satisfies the decay condition \cref{it:phi(W)-decay-bound}. As usual, the polymer results of this paper hold more generally for any choice of $\Phi_U$ satisfying the decay condition such that $\Phi_U(\sfW) = \Phi(\sfW)$ for any $\sfW \subset U$, with the exception of \cref{prop:compare-tau} which uses the explicit form of $\Phi_U$ to compare different surface tensions.
	\end{remark}
	Note that the weight $\hatq^n_U(\gamma)$ makes sense even when $\gamma$ is the disagreement polymer for a domain $V \neq U$. Moreover, the definition of $\hatq^n_U(\gamma)$ does not depend on this reference domain $V$, as $\{D_i\}_{i \geq 2}$ are the finite areas encapsulated by $\gamma$. Hence, it makes sense to consider polymer partition functions of the form
	\[\hatZ_{V, U}^n(A, B) := \sum_{\gamma \in \cP_{V}(A, B)} \hatq^n_U(\gamma)\,.\]
	If $\cE$ is an event about $\gamma$, then we define
	\[\hatZ_{V, U}^n(A, B \mid \cE) := \sum_{\gamma \in \cP_{V}(A, B) \cap \cE} \hatq^n_U(\gamma)\,.\]
	When the points $A, B$ are clear from context, we will drop them from the notation so we instead have $\hatZ^n_{V, U}$ and $\hatZ^n_{V, U}(\cE)$.

    Closely related will be the following polymer weight appearing in \cref{prop:CE-law},
    \begin{equation}\label{eq:def-polymer-weights-tilde}
		\tildeq_U(\gamma)= \exp\big(-\sE_\beta(\gamma) + \sum_{\sfW \cap \Delta_\gamma \neq \emptyset} \Phi(\sfW)\one_{\{\sfW \subset U\}}\big)\,,
	\end{equation}
    with analogous notation for the partition function $\tildeZ_{V, U}$. Note that $\tildeq$ has no reference to $n$ or $L$.
	
	\subsection{Non-negative decoration functions and product structure}\label{sec:product-structure}~

	Following standard treatment, we now adjust $\sE_\beta(\gamma)$ to make the decoration functions $\Phi_U(\sfW)$ non-negative (see, e.g., \cite[Sec.~3.1]{IST15} or \cite[Sec.~2.4]{CKL24}). This will allow us to eventually couple the polymer with a random walk. The following steps can be done for both $\tildeq$ and $\hatq^n$, but we focus on $\hatq^n$ as the random walk coupling will only be needed for the latter polymer model. 
    
    Let $\nabla_\gamma = \bigcup_{b = (\mathsf{y}, \mathsf{y} + e_i) \in \gamma} \{b, b+e_i, b-e_i\}$, viewed as a multiset ($|\nabla_\gamma|=3|\gamma|$). Then we can consider the decoration functions given by 
	\begin{equation}\label{eq:def-Phi'}\Phi'_U(\sfW; \gamma) := |\sfW \cap \nabla_\gamma|e^{-(\beta-C)\bd(\sfW)} + \Phi_U(\sfW)\,,\end{equation}
    for the constant $C$ from \cref{it:phi(W)-decay-bound} from \cref{def:disagree-polymer-len-energy-decor}. Since $\Phi(\sfW)$ satisfies the decay bound \cref{it:phi(W)-decay-bound} in \cref{def:disagree-polymer-len-energy-decor}, we have that $\Phi'_U(\sfW;\gamma) \geq 0$. Moreover, $\Phi'_U(\sfW;\gamma)$ also satisfies the decay bound, i.e. there exists a constant $C' > 0$ such that for any $U, \sfW, \gamma$, we have
    \begin{equation}\label{eq:decay-bound-W-gamma}
        |\Phi'_U(\sfW;\gamma)| \leq \exp(-(\beta - C')\bd(\sfW))\,.
    \end{equation}
%
    Let $c(\beta) = \sum_{\sfW \cap b \neq \emptyset} e^{-(\beta - C)\bd(\sfW)}$, noting that $c(\beta) \to 0$ as $\beta \to \infty$. Then, we obtain that $\sum_{\sfW \cap \Delta_\gamma \neq \emptyset} \Phi_U(\sfW) = -3c(\beta)|\gamma| + \sum_{\sfW \cap \nabla_\gamma \neq \emptyset} \Phi'_U(\sfW;\gamma)$. (Note that, unlike $\Phi_U(\cdot)$, the function  $\Phi'_U(\cdot;\gamma)$ no longer vanishes on $W\not\subset V$.) Hence, we get
	\begin{align*}\hatq^n_U(\gamma)= \exp\big(-\sE_\beta(\gamma) - 3c(\beta)|\gamma| + \sum_{\sfW \cap \nabla_\gamma \neq \emptyset} \Phi'_U(\sfW;\gamma)\big)\!\prod_{i \geq 2}\!\hatpi_{D_i^\circ}^{h_i}(\phi_x \geq 0,\, \forall x \in D_i^\circ)\,.
	\end{align*}
	For simplicity of notation, we now define 
	\begin{equation}\label{eq:define-E*_beta}\sE_\beta^*(\gamma) = \sE_\beta(\gamma) + 3c(\beta)|\gamma| - \sum_{i\geq 2}\log \hatpi_{D_i^\circ}^{h_i}(\phi_x \geq 0,\, \forall x \in D_i^\circ)\,,\end{equation}
	so that
	\begin{equation}\label{eq:weight-of-q-E*}\hatq^n_U(\gamma) = \exp\big(-\sE^*_\beta(\gamma)+\sum_{\sfW \cap \nabla_\gamma \neq \emptyset} \Phi'_U(\sfW;\gamma)\big)\,.
	\end{equation}
	In many parts of the paper, we will only care about the total interaction. For simplicity, we define as shorthand the notation 
	\begin{equation}\label{eq:define-fI}\fI_U(\gamma) = \sum_{\sfW \cap \nabla_\gamma \neq \emptyset} \Phi'_U(\sfW;\gamma)\,,
	\end{equation}
	so that we can equivalently write $\hatq^n_U(\gamma) = \exp\big(-\sE^*_\beta(\gamma)+ \fI_U(\gamma)\big)$. Recall that the boundary condition $n$ is needed to determine either the heights $h_i$, if viewing $\gamma$ as labeled disagreement bonds, or the energy $\sE_\beta(\gamma)$, if viewing $\gamma$ as a triple $(\gamma, \{D_i\}, \{h_i\})$, so that the term $\sE^*_\beta$ implicitly depends on $n$. 
	It will also be convenient to define
	\[\beta' = \beta + 3c(\beta)\,.\]
	Now define $\Psi_U(\sfW, \gamma)=(\exp(\Phi'_U((\sfW;\gamma))-1)\one_{\{\sfW \cap \nabla_\gamma \neq \emptyset\}}$. We have now ensured that $\Psi_U(\sfW, \gamma) \geq 0$, so that the weights can eventually be interpreted as probabilities. For this, observe that we can write 
	\begin{equation*}
		\exp\bigg(\sum_{\sfW \cap \nabla_\gamma \neq \emptyset} \Phi'_U(\sfW;\gamma)\bigg) = \prod_{\sfW \cap \nabla_\gamma \neq \emptyset}\bigg((e^{\Phi'_U(\sfW;\gamma)} - 1)+1\bigg) = \sum_{\underline\sfW = \{\sfW_i\}} \prod_i \Psi_U(\sfW_i;\gamma)\,,
	\end{equation*}
	where the final sum is over all possible finite collections of components $\underline{\sfW}$.
	\begin{definition}\label{def:animal}
		Define an animal as a pair $\Gamma = [\gamma, \underline\sfW]$, where $\gamma$ is a disagreement polymer of a height function in some domain $V$, and $\underline{\sfW}$ is a finite collection of connected components of vertices. We can assign each animal with the weight
		\[\hatq^n_U(\Gamma) = \exp(-\sE^*_\beta(\gamma))\prod_{\sfW \in \underline\sfW} \Psi_U(\sfW;\gamma)\,.\]
	\end{definition}
	
	Observe that we can again sum over all possible finite collections $\underline{\sfW}$ to write 
	\begin{equation}\label{eq:gamma-sum-Gamma}
		\hatq^n_U(\gamma) = \sum_{\Gamma = [\gamma, \underline\sfW]}\hatq^n_U(\Gamma)\,.
	\end{equation}
	Hence, we will apply the previous notation of $\cP_{V}(A, B)$ and $\hatZ^n_{V, U}(A, B)$ to animals. More precisely, we say $\Gamma \in \cP_{V}(A, B)$ to mean that $\Gamma = [\gamma, \underline{\sfW}]$ for some $\gamma \in \cP_{V}(A, B)$ and any finite collection of components $\underline \sfW$, and $\hatZ^n_{V, U}(A, B)$ is also a partition function for animals through \cref{eq:gamma-sum-Gamma}.
	
	\begin{definition}
		We say a point $m = (m_1, m_2)$ is a cut-point of $\Gamma$ if $m \in \Gamma$ and the intersection of $\Gamma$ and the vertical line $x = m_1$ is the point $m$. That is, $\gamma$ only crosses the line once, and so in particular no components of $\underline{\sfW}$ or finite regions $D_i$ cross the line. 
	\end{definition}
	
	Suppose $\Gamma = \Gamma_1 \circ \Gamma_2$ is the decomposition of $\Gamma$ before and after a cut point $m$. More precisely, if $\Gamma = [\gamma, \underline \sfW]$, then we can write $\Gamma_1 = [\gamma_1, \underline\sfW_1]$ and $\Gamma_2 = [\gamma_2, \underline\sfW_2]$ where $\gamma = \gamma_1 \circ \gamma_2$ for $\gamma_1 \in \cP_V(A,m)$ and $\gamma_2 \in \cP_V(m, B)$, and $\underline\sfW = \underline\sfW_1 \cup \underline\sfW_2$ where all components of $\underline\sfW_1$ lie to the left of $m$ and all components of $\underline\sfW_2$ lie to the right of $m$. By the definition of the law $\hatq^n_U$, we immediately have 
	\begin{equation}\label{eq:product-structure}\hatq^n_U(\Gamma) = \hatq^n_U(\Gamma_1)\hatq^n_U(\Gamma_2)\,.
	\end{equation}
	
	\begin{remark}\label{rem:energy-conditions}
		The only properties of $\sE^*_\beta(\gamma)$ that will be used is that
		\begin{equation}\label{eq:energy-property}
			\sE^*_\beta(\gamma) \geq (\beta + 3c(\beta))\sN(\gamma) =: \beta'\sN(\gamma)\,,
		\end{equation}
		along with the product structure of the energy established above in \cref{eq:product-structure}. (Recall from \cref{def:disagree-polymer-len-energy-decor} that $\sN(\gamma)$ is the number of bonds in $\gamma$, counting the absolute value of the gradient along the bonds.) In particular, one could instead sum over $|\nabla\phi|^p$ instead of $|\nabla\phi|^2$ in the definition of $\sE_\beta(\gamma)$. Or, one can begin with the weight $\tildeq$ instead of $\hatq^n$, removing the $D_i^\circ$ terms from \cref{eq:define-E*_beta}. 
        
        Moreover, in bounding the number of disagreement polymers $\gamma$ that satisfy some criterion, we will only use the fact that each $\gamma$ is a connected set of dual bonds, where each bond is labeled with an integer whose absolute value represents the multiplicity of that bond. In particular, we will not use any of the additional geometric restrictions that come with the fact that $\gamma$ is the disagreement polymer of a height function $\phi$; this is only needed to ensure that the regions $D_i$ and heights $h_i$ are well-defined. However, to keep notation consistent, we will still denote the integer label on a bond $e$ by $(\nabla \phi)_e$.
	\end{remark}
	
	\subsection{Cone-points and irreducible components}\label{subsec:cone-points}~
	We now define cone-points as a subset of the cut-points. See \cref{fig:disagree-poly-diamond} for an illustration of the diamonds formed by the cone-points.
	\begin{definition}[Cone-points]
		The $\delta$-forward cone from $\sfu$ is the set $\cY^\btl_\delta(\sfu) := \sfu + \{(x, y) \in \Z^2: |y| \leq \delta x\}$. Similarly, the backward cone from $\sfu$ is the set $\cY^\btr_\delta(\sfu) = \sfu - \{(x, y) \in \Z^2: |y| \leq \delta x\}$. When $\delta$ is omitted from the notation, we assume $\delta = 1$. When $\sfu$ is omitted from the notation, we assume $\sfu = (0, 0)$. For a disagreement polymer $\gamma$, $\sfu$ is a cone-point of $\gamma$ if $\gamma \subset \cY^\btr(\sfu) \cap \cY^\btl(\sfu)$. Similarly, for an animal $\Gamma = [\gamma, \underline\sfW]$, $\sfu$ is a cone-point for $\Gamma$ if $\Gamma \subset \cY^\btr(\sfu) \cap \cY^\btl(\sfu)$.
	\end{definition}
	
	\begin{definition}[Irreducible components]\label{def:irreducible-comp}
		An animal $\Gamma \in \cP_V(\sfu, \sfv)$ is called left-irreducible if it has no cone-points, and $\Gamma \subset \cY^\btr(\sfv)$. Similarly, $\Gamma$ is right-irreducible if it has no cone-points and $\Gamma \subset \cY^\btl(\sfu)$. We say $\Gamma$ is irreducible if it is both right and left irreducible. The set of all left-irreducible, right-irreducible, and irreducible animals with starting point at $\o$ is denoted $\sfA_\sfL, \sfA_\sfR, \sfA$ respectively. (So each of these sets consist of $\Gamma$ in $\bigcup_B \cP_V(\o, B)$.)
	\end{definition}

    \begin{figure}
    \centering
    \includegraphics[width=0.75\textwidth]{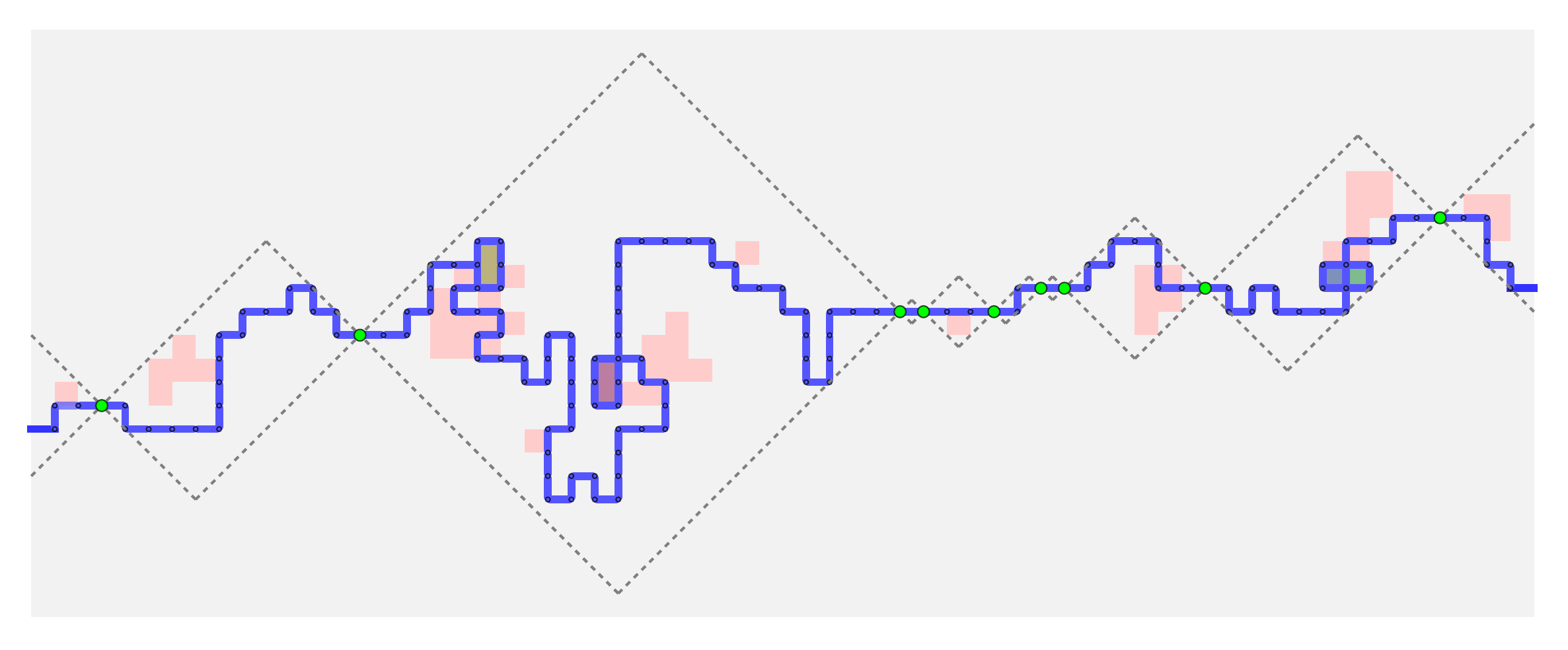}
    \caption{An animal $\Gamma = [\gamma, \underline{\sfW}]$, with cone-points in green. The forward and backward cones emanating from the cone-points form diamonds which encapsulate $\Gamma$. The disagreement polymer $\gamma$ is in blue, and the components $\underline\sfW$ are in pink.}
    \label{fig:disagree-poly-diamond}
    \end{figure}
    
	A few remarks are in order. First note that the only difference between $A$ being a cone-point for $\gamma$ vs. $\Gamma$ is the additional criteria on the clusters in $\underline\sfW$. Note also that by definition, every cone-point is a cut-point. Suppose now that $\Gamma$ has at least two cone-points. Then, we can decompose $\Gamma$ according to all of its cone-points, so that
	\[\Gamma = \Gamma^{(\sfL)}\circ\Gamma^{(1)} \circ \ldots \circ \Gamma^{(k)} \circ \Gamma^{(\sfR)}\]
	where $\Gamma^{(\sfL)} \in \sfA_\sfL$, each $\Gamma^{(i)} \in \sfA$, and $\Gamma^{(\sfR)} \in \sfA_\sfR$. 
	Repeatedly applying \cref{eq:product-structure}, we obtain the product structure of the weights:
	\begin{equation}\label{eq:cone-pt-decomposition}\hatq^n_U(\Gamma) = \hatq^n_U(\Gamma^{(\sfL)})\hatq^n_U(\Gamma^{(\sfR)})\prod_{i=1}^k\hatq^n_U(\Gamma^{(i)})\,.
	\end{equation}	
	We will now show that a typical animal $\Gamma$ has lots of cone-points. First, suppose $S$ is a simple (non self-intersecting) path of dual bonds, and is given the weight $e^{-\beta|S|}$. Define also the partition function
	\begin{equation*}
		Z_V^{\mathsf{SW}}(A, B) = \sum_{\substack{S: A \mapsto B\\ S \in V}} e^{-\beta|S|}\,,
	\end{equation*}
	where the sum is over all simple paths from $A$ to $B$ which stay inside $V$. We will focus our attention on domains for which a typical simple path $S$ drawn from the measure given by the above weights will have linear length and number of cone-points.
	
	\begin{definition}
		Fix $\epsilon, \delta \in (0, 1)$, and $A, B$ with $B \in \cY_\delta^\btl(A)\setminus \{A\}$. We call $V$ an $\epsilon$-\emph{nice} domain with respect to $A, B$ if there exists $\beta_0, v_0, \delta_0, c > 0$ such that uniformly over $\beta \geq \beta_0$ and $r \geq 1 + \epsilon$, we have
		\begin{align}
			Z_V^{\mathsf{SW}}(A, B \mid |S| \geq r\norm{A - B}_1) \leq ce^{-v_0\beta r\norm{A-B}_1}Z_V^{\mathsf{SW}}(A, B)\,,\label{eq:SW-length}\\
			Z_V^{\mathsf{SW}}(A, B \mid |\mathsf{Cpts}(S)| < 2\delta_0\norm{A-B}_1) \leq ce^{-v_0\beta\norm{A-B}_1}Z_V^{\mathsf{SW}}(A, B)\,.\label{eq:SW-cone-points}
		\end{align}
	\end{definition}
    Throughout this paper, we will refer back to the following cigar shape. 
    \begin{definition}[Cigar shape]\label{def:cigar-shape}
        Fix $A, B$ with angle denoted $\theta_{A, B}$ and $B_1 - A_1 = M_{A, B}$. Define the curves $\cC^\pm_{A, B}$ by
		\begin{equation*}
			\cC^\pm(t) = \tan(\theta_{A, B})(t-A_1) + \left(\frac{(t - A_1)(M_{A, B} - t + A_1)}{M_{A, B}}\right)^{1/2}(\log L)^2\,.
		\end{equation*}
		Define the cigar shape $\sC = \sC(\overline{AB})$ as the region in between the curves $\cC^+_{A, B}$ and $\cC^-_{A, B}$.
    \end{definition}
	\begin{remark*}The domains that we will consider in this paper include $\Z^2$, the upper half plane $\H_+$, and domains containing $\sC(\overline{AB})$ or $\sC(\overline{AB}) \cap \H_+$. In each setting, \cref{eq:SW-length} is easily satisfied by a Peierls argument mapping $S$ to a minimal length path from $A$ to $B$. The existence of a linear number of cone-points was established on the domain $\Z^2$ in \cite[Section 2.7]{CIV03}, but the proof method is robust and has since been applied to reach the same conclusion for other domains, such as $\H_+$ in \cite[Lemma 3]{IST15}. By the product structure of \cref{eq:product-structure}, this implies that for the domains $\Z^2$ and $\H_+$, the simple paths behave like random walks. As $\sC(\overline{AB})$ is defined to contain a typical random walk from $A$ to $B$, it is then easy to extend the results to domains containing $\sC(\overline{AB})$ or $\sC(\overline{AB}) \cap \H_+$ as well.
    \end{remark*}
	
	To go from simple paths to disagreement polymers, it will be useful to define the upper and lower envelopes of a disagreement polymer.
	\begin{definition}
		Suppose $\gamma \in \cP_V(A, B)$. The upper envelope $\mathsf{UE}(\gamma)$ is the highest (by lexicographical ordering) simple path from $A$ to $B$ that is a subset of the bonds of $\gamma$ (so in particular, the labels of the bonds are ignored). We can analogously define the lower envelope $\mathsf{LE}(\gamma)$.
	\end{definition} 
	
	The next lemma states that for nice domains, a typical animal $\Gamma$ has linear length and number of cone-points.
	\begin{lemma}\label{lem:cpts-length}
		Fix $\epsilon, \delta \in (0, 1)$, and $A, B$ with $B \in \cY_\delta^\btl(A)\setminus \{A\}$. Let $V$ be an $\epsilon$-nice domain, and $U$ any domain. There exists $\beta_0, v_0, \delta_0, c > 0$ such that uniformly over $n$, $\beta \geq \beta_0$, and $r \geq 1 + 2\epsilon$, we have
		\begin{align}
			\hatZ^n_{V, U}(A, B \mid |\gamma| \geq r\norm{A - B}_1) \leq ce^{-v_0\beta r\norm{A - B}_1}\hatZ^n_{V, U}(A, B)\,,\label{eq:linear-length}\\
			\hatZ^n_{V, U}(A, B \mid |\mathsf{Cpts}(\Gamma)| < 2\delta_0\norm{A - B}_1) \leq ce^{-v_0\beta\norm{A - B}_1}\hatZ^n_{V, U}(A, B)\,.\label{eq:cone-points}
		\end{align}
	\end{lemma}
	\begin{proof}
		The proof of \cref{eq:linear-length} is standard, we simply compare $\gamma$ to $\gamma_0$, a shortest path from $0$ to $\mathsf{y}$ where all the edge labels are 1. Note that since $V$ is $\epsilon$-nice, the length of $\gamma_0$ is at most $(1 + \epsilon)\norm{A-B}$. We have by \cref{eq:decay-bound-W-gamma,eq:weight-of-q-E*} that
		\begin{equation}
			\hatq^n_U(\gamma) \leq \hatq^n_U(\gamma_0)\exp(-\sE^*_\beta(\gamma) + \beta'|\gamma_0|+e^{-(\beta'-C)}(|\gamma| + |\gamma_0|))\,.
		\end{equation}
		Now, the number of rooted connected sets of bonds of size $k$ is at most $C^k$. For each such set of bonds $\gamma$, we then obtain the following upper bound by allowing each $(\nabla \phi)_e$ to take on any value in $\Z \setminus \{0\}$, and applying \cref{eq:energy-property}:
		\begin{equation}
			\sum_{\substack{\gamma \in \cP_V(A, B),\\
					|\gamma| = k}}\exp(-\sE_\beta^*(\gamma)) \leq C^k\prod_{e \in \gamma}\sum_{j = 1}^\infty 2\exp(-\beta' j) \leq C^ke^{-\beta' k}\,.
		\end{equation}
		Hence, we have
		\begin{align*}
			\hatZ^n_{V, U}(A, B \mid |\gamma| \geq r\norm{A - B}_1) &\leq \sum_{k \geq r\norm{A - B}_1}  \sum_{\substack{\gamma \in \cP_V(A, B),\\
					|\gamma| = k}} \hatq^n_U(\gamma)\\
			&\leq \hatq^n_U(\gamma_0)\sum_{k \geq r\norm{A - B}_1}\exp(-\beta' k + \beta'(1+\epsilon)\norm{A - B}_1+ Ck)\\
			&\leq \hatq^n_U(\gamma_0)C\exp(-(\beta' -C)r\norm{A - B}_1 + \beta'(1+\epsilon)\norm{A - B}_1)\\
			&\leq Ce^{-v_0\beta' r\norm{A - B}_1}\hatZ^n_{V, U}(A, B)\,
		\end{align*}
		as long as $v_0 < 1 - \frac{1+\epsilon}r - \frac C{\beta'}$.
		
		To show \cref{eq:cone-points}, we will first show the same inequality holds when considering cone-points of $\gamma$ instead of $\Gamma$. First note that by \cref{eq:decay-bound-W-gamma} we have
		\begin{equation}\label{eq:compare-no-decoration}
			\left|\log\frac{\exp\big(-\sE_\beta^*(\gamma) + \fI_U(\gamma))}{\exp\big(-\sE_\beta^*(\gamma))}\right| \leq e^{-(\beta-C)}|\gamma|\,.
		\end{equation}
		As $e^{-(\beta - C)}|\gamma|$ is smaller than $v_0\beta$ for sufficiently large $\beta$, it suffices to prove the bound replacing $\hatq^n_U(\gamma)$ with $\overline q^n(\gamma) := e^{-\sE_\beta^*(\gamma)}$. That is, if we define $\overline Z^n_{V, U}(A, B) = \sum_{\gamma \in  \cP_V(A, B)} \overline q^n(\gamma)$, we will show that
		\begin{equation}\label{eq:cone-points-gamma}
			\overline Z^n_{V, U}(A, B \mid |\mathsf{Cpts}(\gamma)| < 2\delta_0\norm{A - B}_1) \leq ce^{-v_0\beta\norm{A - B}_1}\overline Z^n_{V, U}(A, B)\,.
		\end{equation}
		
		Now, for every $\gamma \in \cP_V(A, B)$, let $S = \mathsf{UE}(\gamma)$. Observe that since each $S$ is a simple path, the energy $\sE^*_\beta(\gamma)$ is simply $\beta$ times the length of the path, $\beta|\gamma|$. (In particular, there are no domains $D_i$ enclosed by $\gamma$ that contribute an area term to the energy, and so there is no dependence on $n$ anymore). Hence, we have 
		\begin{equation}\label{eq:q1-weight-SW}
			\sum_{S: S = \mathsf{UE}(\gamma), \gamma \in \cP_V(A, B)} \overline q^n(S) = Z_V^{\mathsf{SW}}(A, B)\,.
		\end{equation}
		 Moreover, for each $\gamma$ such that $\mathsf{UE}(\gamma) = S$, we can view the set of bonds in $\gamma$ as a collection $\{S, (B_e)_{e \in S}\}$ where $B_e$ are connected components of dual bonds indexed by the first bond $e \in S$ that the component is connected to. For every $S$, denote by $\cP_V(S)$ the set of labeled disagreement polymers $\gamma$ such that $\mathsf{UE}(\gamma) = S$. Using \cref{eq:energy-property}, we can upper bound the sum over $\gamma \in \cP_V(S)$ by enumerating over collections $\{S, (B_e)_{e \in S}\}$ and then enumerating over adding integer labels on $\gamma$:
		\begin{align*}
			\sum_{\gamma \in \cP_V(S)} \overline q^n(\gamma) \leq \prod_{e \in S}(1+\sum_{k \geq 2} e^{-\beta' (k-1)} + \sum_{k \leq -1}e^{-\beta'(|k|-1)})&\big(1 + \sum_{l=1}^\infty C^l(\sum_{j=1}^\infty 2e^{-\beta' j})^l\big)\overline q^n(S)\\
			&\leq (1+Ce^{-\beta})^{|S|}\overline q^n(S)\\
			&\leq \exp(Ce^{-\beta}|S|)\overline q^n(S)\,.
		\end{align*}
		The sums over $k$ capture the possibilities that at each bond $e \in \gamma$, the gradient $(\nabla\phi)_e$ is either $\geq 2$ or $\leq -1$ instead of 1. The sum over $l$ includes the possible sizes of the connected component $B_e$, and there are $C^l$ possible components of size $l$ for some constant $C$. For each such component $B_e$, each bond of $B_e$ can also have a height disagreement of $j$ along that edge for $|j| \geq 1$. The above with \cref{eq:q1-weight-SW} implies that if $\cE$ is measurable with respect to $\mathsf{UE}(\gamma)$, then 
		\begin{equation}\label{eq:GD1-omega-measurable}
			\overline Z^n_{V, U}(A, B \mid \cE, |\mathsf{UE}(\gamma)| \leq 1.1\norm{A - B}_1) \leq \exp(Ce^{-\beta}\norm{A - B}_1)Z_V^{\mathsf{SW}}(A, B \mid \cE, |\mathsf{UE}(\gamma)| \leq 1.1\norm{A - B}_1)\,.
		\end{equation}
		Combining \cref{eq:SW-cone-points,eq:GD1-omega-measurable} thus implies that
		\begin{align*}
			&\overline Z^n_{V, U}(A, B \mid \mathsf{Cpts}(\mathsf{UE}(\gamma)) < 2\delta_0\norm{A - B}_1,|\mathsf{UE}(\gamma)| \leq 1.1\norm{A - B}_1) \\
			&\leq e^{-c\beta\norm{A - B}_1}Z_V^{\mathsf{SW}}(A, B)\leq e^{-c\beta\norm{A - B}_1}\overline Z^n_{V, U}(A, B)\,,
		\end{align*}
		where the last inequality follows as $Z_V^{\mathsf{SW}}(A, B)$ is just a restriction of  $\overline Z^n_{V, U}(A, B)$ to simple paths.  
		
		Moreover, since $|\mathsf{UE}(\gamma)| \leq |\gamma|$, we can also bound the event that $\mathsf{UE}(\gamma)$ is too long. We have from \cref{eq:linear-length} and \cref{eq:compare-no-decoration} that 
		\begin{equation}
			\overline Z^n_{V, U}(A, B  \mid |\mathsf{UE}(\gamma)| > 1.1\norm{A - B}_1) \leq e^{-c\beta\norm{A - B}_1}\overline Z^n_{V, U}(A, B)\,.
		\end{equation}
		
		Hence, to prove \cref{eq:cone-points-gamma}, it suffices to show
		\begin{align}\label{eq:cpts-omega-to-gamma}
			\overline Z^n_{V, U}(A, B \mid \mathsf{Cpts}(\gamma) < \tfrac32\delta_0\norm{A - B}_1, \mathsf{Cpts}(\mathsf{UE}(\gamma)) \geq 2\delta_0\norm{A - B}_1,& |\mathsf{UE}(\gamma)| \leq 1.1\norm{A - B}_1) \nonumber\\
            &\leq e^{-c\beta\norm{A - B}_1}\overline Z^n_{V, U}(A, B)\,.
		\end{align}
		However, observe that in the computation above \cref{eq:GD1-omega-measurable}, each component $B_e$ adds on a multiplicative factor of $O(e^{-\beta |B_e|})$ to the weight of $\overline q^n(S)$ (more precisely, this factor is in the interval $[e^{-\beta |C_e|},\,  (\sum_{j=1}^\infty 2e^{-\beta j})^{|C_e|}]$, depending on the admissible assignments of height differences along the edges of $B_e$). In \cite[Propositions 2.12]{CKL24}, it was proved that for an animal $\Gamma = [\gamma, \underline{\sfW}]$, many cone-points of $\gamma$ are also cone-points of the whole animal $\Gamma$, precisely in the sense of \cref{eq:cpts-omega-to-gamma}. The proof was a combinatorial computation using the fact that each cluster $\sfW$ of the animal contributes a factor of $e^{-(\beta - C) \bd(\sfW)}$ to the weight, and hence the exact same computation proves \cref{eq:cpts-omega-to-gamma} with components $B_e$ taking the role of clusters. This finally establishes \cref{eq:cone-points-gamma}. From here we can show that many cone-points of $\gamma$ imply many cone-points of $\Gamma$ to conclude \cref{eq:cone-points} again using the computation of \cite[Propositions 2.12]{CKL24} (this time actually using it to enumerate over clusters in the animal), as it did not matter there that $\gamma$ had the simpler geometry of a simple path.
	\end{proof}
	
	\subsection{Surface tension existence and properties}
	Fix $n$ (which determines the heights $h_i$ in the weight $\hatq^n_U(\gamma)$), and fix a unit vector $\n$ with angle $\theta$. Recall the definitions of $\cP_{\Z^2}$ and $\hatZ^n_{\Z^2, \Z^2}, \tildeZ_{\Z^2,\Z^2}$ from the beginning of \cref{sec:geom-disagreement-polymers}. Let $N$ be such that the point $N\n$ lies on the lattice. Define the surface tensions as \begin{align}\label{eq:surface-tension}
		&\tau_{\beta, n}(\theta):=\tau_{\beta, n}(\n) := -\lim_{N \to \infty} \frac1{\norm{N\n}_1} \log \hatZ^n_{\Z^2, \Z^2}(\o, N\n)\\
        &\tau_{\beta}(\theta):=\tau_{\beta}(\n) := -\lim_{N \to \infty} \frac1{\norm{N\n}_1} \log \tildeZ_{\Z^2, \Z^2}(\o, N\n)\label{eq:surface-tension-tilde}
	\end{align}
	where the limit is over $N$ such that $N\n$ lies on the lattice. We will use both the notation of $\n$ and $\theta$. As usual, if the above limit exists, we can extend the surface tension to a function over all of $\R^2$ by homogeneity. We will show the following proposition about the surface tension.
	\begin{proposition}\label{prop:surface-tension-properties}
		The above limits for $\tau = \tau_{\beta, n}$ or $\tau_\beta$ exist and satisfy the following properties:
		\begin{enumerate}[(i)]
			\item The convergence of the limit is uniform over all unit vectors $\n$.\label{it:tau-uniform-convergence}
			
			\item $\tau$ is analytic as a function from $\R^2$ to $\R^2$\label{it:tau-analytic}
			
			\item $\tau$ is strictly convex, i.e., for any two vectors $\sfu, \sfv$ pointing in different directions, 
			\[\tau(\sfu) + \tau(\sfv) > \tau(\sfu + \sfv)\,.\]\label{it:tau-strict-convexity}
			
			\item $\tau$ is symmetric under rotations by $\pi/4$, reflections across the $x,y$ axes and diagonals $y = \pm x$.\label{it:tau-symmetric}
		\end{enumerate}
	\end{proposition}
	The statement of \cref{it:tau-symmetric} follows by construction: it is clear that each listed symmetry defines a bijection from $\cP_{\Z^2}(0, N\n)$ to $\cP_{\Z^2}(0, N\n')$, where $\n'$ is the image of $\n$ under the symmetry. Considering the weights in \cref{eq:def-polymer-weights,eq:def-polymer-weights-tilde}, the energy $\sE_\beta(\gamma)$ is preserved under the symmetry since $\sN(\gamma)$ remains unchanged, the encapsulated domains $D_i$ are just rotated/reflected so that their corresponding area terms are unchanged, and finally the functions $\Phi(\sfW)$ are also preserved under symmetries as stated in \cref{def:disagree-polymer-len-energy-decor}.
	
	The statements of \cref{it:tau-uniform-convergence,it:tau-analytic,it:tau-strict-convexity} (and many more properties) regarding the surface tension of a polymer model have been well studied in \cite[Ch.~4]{DKS92} in the case where the shape of $\gamma$ is a contour and $\sE_\beta^*(\gamma) = \beta'|\gamma|$. (See  \cite[Sec.~4.12]{DKS92} for analyticity and uniform convergence, and \cite[Sec.~4.21]{DKS92} for strict convexity.) We provide a brief review of the beginning proof ideas there in order to show what adjustments need to be made in our setting.\footnote{The careful reader will notice slight differences due to conventional inconsistencies, such as a choice of $\beta$ vs. $2\beta$ in the law of $\gamma$ --- these are inconsequential and hence ignored.} For brevity we will write the proof in terms of $\hatZ^n$ and $\tau_{\beta, n}$, but the proof holds verbatim for $\tildeZ$ and $\tau_\beta$ (and even more generally, as per \cref{rem:energy-conditions}).
	
	\begin{remark}\label{rem:DKS-idea} The key point is that the polymer model in \cite{DKS92} is shown to be a perturbation of minimal length paths between the start and end points. Even though our model has more complicated features --- $\gamma$ is allowed to be any connected component of bonds, and there is also an additional penalty in the area term for finite regions enclosed by $\gamma$ in $\hatq^n$ --- any portion of $\gamma$ that is not locally behaving like a minimal length path is sufficiently penalized due to \cref{eq:energy-property}, so that our setting is also a perturbation of minimal length paths.
	\end{remark}
	
	We begin by recalling \cite[Secs.~4.1--4.4]{DKS92}, which set up the following definitions.\footnote{Dictionary between the notation of this paper and \cite{DKS92}, modulo that this paper considers a more general set of polymers:
		$\cP_{\Z^2}(\o, t^{\n}_N)=\cI_{N, \n}$, 
		$\cP_{\Z^2}(N)=\cI_N$,
		the complex parameter $z$ is denoted $H$, 
		$\hatq^n_z(\Gamma_i) = \Psi_H(\hat\xi_i)$,
		$\hatZ^n_{\Z^2, \Z^2}(N, z) = \Xi(N, H)$.
	} Define the point-to-line polymers
	\[\cP_{\Z^2}(N) := \bigcup_{\n} \cP_{\Z^2}(\o, t^{\n}_N)\,.\]
	Let $h(\gamma)$ be the difference in height between the start and end of $\gamma$. That is, if $\gamma \in \cP_V(A, B)$, then $h(\gamma) := B_2 - A_2$, the difference between the $y$-coordinates of $B$ and $A$. Define also the partition function with a complex parameter $z$ as
	\[\hatZ^n_{\Z^2, \Z^2}(N, z) = \sum_{\gamma \in \cP_{\Z^2}(N)} \exp\big(-\sE^*_\beta(\gamma) + \tfrac12\beta' h(\gamma)z + \sum_{\sfW \cap \nabla_\gamma \neq \emptyset} \Phi'_{\Z^2}(\sfW; \gamma)\big)\]
	(Ultimately, we only care about $z = 0$, but including this parameter is needed for studying the characteristic polynomial of a random variable in later sections of \cite{DKS92}.) Recalling the definitions of cut-points and animals $\Gamma = [\gamma, \underline\sfW]$ from \cref{sec:product-structure} as well as the notation $\Psi(\sfW; \gamma)=(\exp(\Phi'((\sfW;\gamma))-1)\one_{\{\sfW \cap \nabla_\gamma \neq \emptyset\}}$, we can define the animal weights $\hatq^n_z(\Gamma)$ for $\Gamma = [\gamma, \underline \sfW]$
	\begin{equation*}
		\hatq^n_z(\Gamma) := \exp(-\sE^*_\beta(\gamma) + \tfrac12\beta'h(\gamma)z)\prod_{\sfW \in \underline\sfW}\Psi(\sfW)\,.
	\end{equation*}
	As before, we can then obtain a product structure for the weights, 
	\begin{equation*}
		\hatZ^n_{\Z^2, \Z^2}(N, z) = \sum_{\Gamma_0, \ldots, \Gamma_l} \prod_{i =1}^l \hatq^n_z(\Gamma)\,,
	\end{equation*}
	where $\{\Gamma_i\}_{i = 1}^l$ is the decomposition by cut points of some animal whose polymer $\gamma$ is in $\cP_{\Z^2}(N)$. We assume that each resulting $\Gamma_i$ cannot be further decomposed by cut points, and we call such a $\Gamma_i$ as wild (not to be confused with the previously defined irreducible components, which refer specifically to cone-points).
	
	Next, in \cite[Sec.~4.5]{DKS92}, the authors consider only $\Gamma$ capturing the $\beta \to \infty$ behavior. That is, they consider restricting to the set of $\Gamma = [\gamma, \underline\sfW]$ where $\gamma$ has minimal length in the sense that there are no two horizontal bonds with the same $x$-coordinate, and $\underline\sfW = \emptyset$. This is denoted by $\cI^\infty_{N, \n}$, a notation that we will leave unchanged to emphasize that even though we start with a more complicated set of polymers $\gamma$ than in \cite{DKS92}, once we impose the above restrictions, we are reduced to the same set of $\Gamma$ considered there (and with the same weight as well). The case of no fixed endpoint is also considered, with $\cI^\infty_N := \bigcup_\n \cI_{N_,\n}^\infty$. The partition function summing over just animals in $\cI_N^\infty$ is then directly computed (\cite[Eq. 4.5.7-8]{DKS92}) as 
	\[\sum_{\Gamma \in \cI^\infty_N}\hatq^n_z(\Gamma) = \bigg(\frac{\sinh(\beta')}{\cosh(\beta') - \cosh(z\beta')}\bigg)^N =: (Q_z)^N\,.\]
	
	Now for each wild animal $\Gamma_i$ in the decomposition of some $\Gamma$, define $J(\Gamma_i)$ as the projection onto the $x$-coordinate of the starting and ending cut-points of $\Gamma_i$. In \cite[Sections 4.6--4.7]{DKS92}, the focus is on the sum over all possible wild animals projecting onto a fixed interval $I$, 
	\[\sum_{\Gamma_i: J(\Gamma_i) = I} \hatq^n_z(\Gamma_i)\,.\]
	To be precise, the sum should be interpreted as restricting to $\Gamma_i$ rooted at the origin, since $\hatq^n_z(\Gamma_i)$ is invariant under vertical shifts of $\Gamma_i$. The key bound in \cite[Lemma 4.7]{DKS92} shows that this decays exponentially with $(\beta-C)|I|$ relative to the total weight of (not necessarily wild) animals with $J(\Gamma) = I$ and $\Gamma \in \cI_N^\infty$. This is a mathematical statement of the perturbation described in \cref{rem:DKS-idea}. The proof there relies on the fact that the polymers are contours, and thus will need to be modified to fit our setting. We do this in \cref{lem:DKS-polymer-weights}.
	
	In \cite[Sec.~4.8]{DKS92}, the key bound is used with the cluster expansion machinery of \cite{KoteckyPreiss86} to obtain a formula for $\log \hatZ^n_{\Z^2, \Z^2}(N, z)$ as a sum of some abstract functions $\Phi(I)$ over all intervals $I \subset [0, N]$. From here onwards, the authors work primarily with the functions $\Phi$ and the properties of $\Phi$ proven in \cite{KoteckyPreiss86}, which otherwise loses the information of the original polymer model. The first exception is in \cite[Sec.~4.14]{DKS92} where the polymer weights are reintroduced to add extra parameters to the weights. However, the same cluster expansion machinery is immediately used, justified also by \cref{lem:DKS-polymer-weights} (see the equation below \cite[Eq. 4.14.18]{DKS92}), after which the authors work with the resulting $\Phi(I)$ functions again. The second exception is in the proof of the sharp triangle inequality \cite[Sec.~4.21]{DKS92}. It is straightforward to check that the proof there requires only the results proved previously in the chapter, and is not impacted by our different polymer setting. Hence, to obtain the results of \cite[Chapter 4]{DKS92} in our setting, it suffices to reprove the key bound given in \cite[Lemma 4.7]{DKS92}, which we now state more explicitly.
	
	\begin{lemma}\label{lem:DKS-polymer-weights}
		Fix $\delta > 0$. Let $z$, $z_r$ for $r = 1, \ldots, N$ be complex numbers such that $|\Re z| < 2 - \delta/\beta'$. Let $s$ be a real number in $[0, \delta/3]$. For any interval $I \subset [0, N]$ with integer endpoints, define
		\begin{equation*}
			\hat X_{N, s}(I) := (\prod_{r \in I} Q_{z_r})^{-1} \sum_{\Gamma: J(\Gamma) = I} \hatq^n_z(\Gamma)\exp(s|I|)\,.
		\end{equation*}
		Then, there exists some $\beta_0$ and constant $c$ such that for all $N$, $\beta \geq \beta_0$, $z$, $z_r$, and all such intervals $I$, we have
		\begin{equation*}
			|\hat X_{N, s}(I)| \leq \exp(-(\beta - \beta_0)(|I| - 1))\,.
		\end{equation*}
	\end{lemma}
	\begin{proof}
		Let $\Gamma = (\gamma, \sfW_1, \ldots, \sfW_i)$, and let $J(\Gamma) = I = [m', m'']$. We can write $\sN(\gamma) = \sN_\sfh(\gamma) + \sN_\sfv(\gamma)$ where $\sN_\sfh(\gamma), \sN_\sfv(\gamma)$ denote the contribution to $\sN(\gamma)$ from horizontal and vertical bonds of $\gamma$, respectively. Moreover, we define $\sN_\sfh'(\gamma) := \sN_\sfh(\gamma) - (|I| - 1)$  and $\sN_\sfv'(\gamma) := \sN_\sfv(\gamma) - h(\gamma)$. We begin with the observation that for every integer $m' < m < m''$, either $\gamma$ intersects the vertical line $x = m$ more than once, or there is some cluster $\sfW$ that intersects the line $x = m$. This implies that
		\begin{equation}\label{eq:no-cuts-excess}
			|J(\gamma)| - 1 \leq \sN_\sfh'(\gamma) + \sum_{\sfW \in \Gamma}\bd(\sfW)\,.
		\end{equation}
		(Note that this is off from \cite[Eq. 4.7.6]{DKS92} by a factor of $1/2$, because if $\gamma$ is required to be a contour, then the only way $\gamma$ can intersect a vertical line more than once is to intersect it three times. The only impact is that the decay bound in the lemma is off by a factor of 2, which is of no consequence.) By the assumption in \cref{eq:energy-property}, for an upper bound we can replace $\sE^*_\beta(\gamma)$ with $\beta'\sN(\gamma)$ in $\hatq^n_z(\Gamma)$. With this simplification we can now follow the calculations exactly as up to \cite[Eq. 4.7.18]{DKS92} to obtain 
		\begin{equation}\label{eq:bound-hatX}
			|\hat X_{N, \gamma}(I)| \leq e^{-(\beta - \beta_1)(|I| - 1)}\sum_{\gamma: J(\gamma) = I} \exp(-\frac\delta2|h(\gamma)| - \beta'(\sN_\sfv(\gamma) - |h(\gamma)|)- \beta_2\sN_\sfh(\gamma))\,,
		\end{equation}
		where $\beta_1$, $\beta_2$ are just constants that can be taken arbitrarily large depending on $\delta$, but not dependent on $\beta'$. (Note that above \cite[Eq. 4.7.13]{DKS92}, the additional bound $|h(\gamma)| \leq \sN_\sfv(\gamma)$ is used to upper bound $-\frac\delta2|h(\gamma)| - \beta'(\sN_\sfv(\gamma) - |h(\gamma)|) \leq -\frac\delta2\sN_\sfv(\gamma)$, which is in the final expression obtained in \cite[Eq. 4.7.18]{DKS92}, the analog of \cref{eq:bound-hatX} above. We will need the sharper bound as we have written.)
		
		It now remains to show that the sum 
		\begin{equation*}
			\Xi := \sum_{\gamma: J(\gamma) = I} \exp(-\frac\delta2|h(\gamma)| - \beta'(\sN_\sfv(\gamma) - |h(\gamma)|)- \beta_2 \sN_\sfh(\gamma))
		\end{equation*}
		is bounded by a constant independent of $\beta'$. First observe that in order for $\gamma$ to climb height $h(\gamma)$, we need at least $|h(\gamma)|$ vertical bonds $e$ such that $\sgn((\nabla \phi)_e) = \sgn(h(\gamma))$. Hence, we partition the bonds of $\gamma$ into the three sets
		\begin{align*}
			&B_1 = \{e\in \gamma: \sgn((\nabla \phi)_e) = \sgn(h(\gamma)),\, e \text{ is horizontal}\}\\
			&B_2 = \{e\in \gamma: \sgn((\nabla \phi)_e) \neq \sgn(h(\gamma)),\, e \text{ is horizontal}\}\\
			&B_3 = \{e \in \gamma: e \text{ is vertical}\}\,,
		\end{align*} 
		and obtain the upper bound
		\begin{align}\label{eq:B1-B2-vert-bonds}
			\Xi \leq \sum_{\gamma: J(\gamma) = I}\prod_{e \in B_1} e^{-\frac\delta2-\beta'(|(\nabla \phi)_e|-1)}\prod_{e \in B_2}e^{-\beta'|(\nabla \phi)_e|}\prod_{e \in B_3}e^{-\beta_2|(\nabla \phi)_e|}\,.
		\end{align}
		
		We next enumerate over the bonds by partitioning $\gamma$ into fragments. For every $\gamma$ such that $J(\gamma) = I$, we consider slicing $\gamma$ at all vertical lines $x = x_i$ for integer $x_i$, so that we are cutting all horizontal bonds in half. We end up with a collection of fragments $F$, each consisting of a (possibly empty) vertical segment $\gamma$, and at least one (but possibly more) horizontal half-bond attached to it. Each bond (or half-bond) also inherits the label $(\nabla \phi)_e$ from $\gamma$.
		
		For a fragment $F$, we can denote its height by $h(F)$, and the number of vertical and horizontal bonds by $\sN_\sfv(F)$ and $\sN_\sfh(F)$ respectively (accounting for the labels as in \cref{def:disagree-polymer-len-energy-decor}) where now $\sN_\sfh(F)$ may be a half integer. We additionally define a marked fragment as a pair $(F, e)$ of a fragment and one of its horizontal half-bonds. The marked half-bonds will tell us later how to reconstruct $\gamma$ given these fragments. Finally, if $R$ is any connected (in $\R^2$) set of vertical bonds and horizontal half bonds, we call a horizontal half-bond of $R$ open if it is not part of a whole horizontal bond in $R$. We will consider an ordering of open half-bonds of $R$ by associating to each open half-bond its point at which another half-bond could join to it to create a whole bond, and then applying a lexicographic ordering to such points.
		
		We now construct an injection from $\gamma$ to a list of marked fragments $\fF = \fF(\gamma)$, to be indexed as $\fF[i]$. Starting at $i = 1$, set $\fF[1]$ to be $(F, \emptyset)$ where $F$ is the fragment containing the unique half-bond of $\gamma$ intersecting the left-most vertical line $x = m'$ (we allow $e = \emptyset$ for this starting fragment only). At each step $i$, define $R_i$ to be the subset of $\gamma$ consisting of the fragments which have been added to $\fF$ by the end of step $i$, so that $R_1$ is the fragment $\fF[1]$ (rooted at the origin, where $\gamma$ is assumed to be rooted). Now let $i = 2$. Let $e$ be the minimal open horizontal half-bond of $R_{i-1}$. Let $F$ be the fragment in $\gamma \setminus R_{i-1}$ that contains the half-bond $e'$ such that $e'$ and $e$ combine to make a whole bond in $\gamma$. Set $\fF[i] = (F, e')$. Then define $R_i$ as instructed above, set $i = i+1$, and repeat until we have reached $R_N = \gamma$, where $N$ is the number of fragments in $\gamma$.
		
		Now we show that this map is injective. Suppose we start with $\fF$ and we know that $\fF = \fF(\gamma)$ for some $\gamma$, and that $\fF$ was built through a sequence of $R_i$ as described above (but of course we do not know what $\gamma$ or the $R_i$ are). We know the number of fragments in $\gamma$ is $N := |\fF|$. We start by placing the fragment of $\fF[1]$ so that its leftmost horizontal bond starts at the origin, and call this $T_1$. By construction, $T_1 = R_1$. Now set $i = 2$. We know $T_{i - 1} = R_{i-1}$. Let $e$ be the minimal open half-bond of $T_{i-1}$. Calling $\fF[i] = (F, e')$, let $T_i$ be the result of adding $F$ to $T_{i-1}$ so that $e'$ and $e$ join together to make a whole vertical bond. By construction, $R_i$ also results from attaching $(F, e')$ to the minimal open half-bond of $R_{i-1}$ via $e'$. Since $R_{i-1} = T_{i-1}$, then $R_i = T_i$. Now set $i = i+1$ and repeat until we have exhausted the whole list $\fF$, so we have reached $T_N$. Since $T_N = R_N = \gamma$, we are done as we have constructed $\gamma$ only from knowing $\fF$. 
		
		Thus, we can upper bound the sum over $\gamma$ by a sum over marked fragments. Consider $F$ with height $h(F) = k$. We next bound the contribution to \cref{eq:B1-B2-vert-bonds} of the bonds in such a fragment. We begin with the horizontal half-bonds. There are $2k+2$ choices of where to put the marked half-bond. Accounting for the possible label of this half bond, its total contribution to the weight is at most 
		\[(2k+2)\sum_{j = 1}^\infty 2e^{-j\beta_2/2}\,.\]
		Similarly, each of the other $2k+1$ locations can either have a labeled half-bond or not, for a total contribution of \[(1+\sum_{j = 1}^\infty 2e^{-j\beta_2/2})^{2k+1}\,.\]
		
		For the $k$ vertical bonds, we split up into the two cases of $B_1$ and $B_2$. Assume without loss of generality that $\sgn(h(\gamma)) = 1$, as the $-1$ case is analogous and the result is the same. For $e \in B_1$, we get a contribution of $\sum_{j \geq 1}e^{-\frac\delta2 -\beta'(j-1)}$. For $e \in B_2$, we get a contribution of $\sum_{j \leq -1}e^{-\beta'|j|}$. Hence, we get a total contribution of 
		\[(\sum_{j \geq 1}e^{-\frac\delta2 - \beta'(j-1)} + \sum_{j \leq -1}e^{-\beta'|j|})^k\,.\]
		Altogether, we have that 
		\begin{align}\label{eq:fragment-weight}\sum_{k = 0}^\infty &\sum_{F: h(F) = k} \prod_{e \in B_1\cap F} e^{-\frac\delta2-\beta'(|(\nabla \phi)_e|-1)}\prod_{e \in B_2\cap F}e^{-\beta'|(\nabla \phi)_e|}\prod_{e \in B_3 \cap F}e^{-\beta_2|(\nabla \phi)_e|} \\
			&\leq \sum_{k = 0}^\infty \bigg((2k+2)\sum_{j = 1}^\infty 2e^{-j\beta_2/2}\bigg)\bigg(1+\sum_{j = 1}^\infty 2e^{-j\beta_2/2}\bigg)^{2k+1}\bigg(\sum_{j \geq 1}e^{-\frac\delta2 - \beta'(j-1)} + \sum_{j \leq -1}e^{-\beta'|j|}\bigg)^k\nonumber \\
			&\leq \sum_{k = 0}^\infty Ce^{-\beta_2/2}ke^{cke^{-\beta_2/2}}e^{-k\delta/3}\nonumber
		\end{align}
		for sufficiently large $\beta'$ compared to $\delta$.
		We can then choose $\beta_2 = \beta_2(\delta)$ and a constant $C = C(\delta)$ so that $ke^{-k\delta/3 + cke^{-\beta_2/2}} < Ce^{-k\delta/4}$ for all $k \geq 0$. Then the above sum is upper bounded by
		\begin{equation}\label{eq:frag-weight-UB}\sum_{k = 0}^\infty Ce^{-\beta_2/2}ke^{cke^{-\beta_2/2}}e^{-k\delta/3} \leq C'e^{-\beta_2/2}\frac1{1 - e^{-\delta/4}} \leq C''e^{-\beta_2}\,.
		\end{equation}
		Finally, enumerating over $\gamma$ as a collection of marked fragments (and ignoring any other compatibility conditions) and noting that there must be at least $|I|$ fragments, we have by \cref{eq:B1-B2-vert-bonds,eq:fragment-weight,eq:frag-weight-UB} that
		\begin{align*} 
			\Xi &\leq \sum_{N \geq |I|} \sum_{\fF: |\fF| = N} \prod_{i = 1}^N \prod_{e \in B_1 \cap \fF[i]} e^{-\frac\delta2-\beta'(|(\nabla \phi)_e|-1)}\prod_{e \in B_2\cap \fF[i]}e^{-\beta'|(\nabla \phi)_e|}\prod_{e \in B_3\cap \fF[i]}e^{-\beta_2|(\nabla \phi)_e|}\\
			&\leq \sum_{N \geq |I|} \bigg(\sum_F \prod_{e \in B_1\cap F} e^{-\frac\delta2-\beta'(|(\nabla \phi)_e|-1)}\prod_{e \in B_2\cap F}e^{-\beta'|(\nabla \phi)_e|}\prod_{e \in B_3 \cap F}e^{-\beta_2|(\nabla \phi)_e|}\bigg)^N\\
			&\leq \sum_{N \geq |I|} (C''e^{-\beta_2})^N\\
			&\leq e^{-(\beta_2-C)|I|}\,,
		\end{align*}
		which is a $\beta$ independent constant.
	\end{proof}
	
	As a result of the above discussion, we can now conclude the remaining properties of \cref{prop:surface-tension-properties}, as well as the following results from \cite[Chapter 4]{DKS92}:
	\begin{proposition}\label{prop:DKS-propositions}
		The following statements hold for any finite $n$ and sufficiently large $\beta$. Let $A=(A_1,A_2), B=(B_1,B_2)$ be any two points, and assume that $A$ is to the left of $B$. Let $\theta_{A, B}$, $M_{A, B}$, $\ell_{A, B}$ denote the angle, horizontal length, and length of the line segment $\overline{AB}$ respectively, and assume that $|\theta_{A, B}|$ is bounded away from $\pi/2$. Let $\cS = \cS(A, B)$ denote the infinite vertical strip with $A, B$ on its boundary, and let $W \supset \cS$.
		\begin{enumerate}[(i)]
			\item\label{it:DKS-large-deviations} \cite[Eq. 4.15.4]{DKS92} Let $\overline{h}_t(\gamma)$ be the maximum height of $\gamma$ above $\overline{AB}$ at the line $x = A_1 + t$. There exist constants $C, c> 0$ such that for any $A, B$ as above,
			\begin{equation*}
				\hatZ^n_{W, \Z^2}(A, B \mid \overline{h}_t(\gamma) \geq j) \leq C\sqrt{M_{A, B}}e^{-c(j \wedge j^2/t)}\hatZ^n_{W, \Z^2}(A, B)\,.
			\end{equation*}
			Let $\ell_t(\gamma)$ be the number of bonds in $\gamma$ to the left of $x = A_1 + t$. For some constant $K(\theta_{A, B})$, we have
			\begin{equation*}
				\hatZ^n_{W, \Z^2}(A, B \mid \ell_t(\gamma) \geq j + K(\theta_{A, B})t) \leq C\sqrt{M_{A, B}}e^{-c(j \wedge j^2/t)}\hatZ^n_{W, \Z^2}(A, B)\,.
			\end{equation*}
			
			\item \cite[\S4.12.3]{DKS92} There exists $C > 0$ such that for any $A, B$ as above,
			\begin{align*}
				&\left|\log \hatZ^n_{\Z^2, \Z^2}(A, B) + \tau_{\beta, n}\ell_{A, B} + \tfrac12\log M_{A, B}\right| \leq C\,, \mbox{ and }\\
				&\left|\log \hatZ^n_{\cS, \Z^2}(A, B) + \tau_{\beta, n}\ell_{A, B} + \tfrac12\log M_{A, B}\right| \leq C\,.
			\end{align*} 
			\label{it:DKS-convergence-rate}
		\end{enumerate}
	\end{proposition}
	\begin{remark*}
		Note that \cref{it:DKS-large-deviations} above was originally written only for $W = \Z^2$. However, an immediate consequence of \cref{it:DKS-convergence-rate} is that for any $W \supset \cS$,
		\begin{equation}\label{eq:change-domain-S-Z2}\left|\log \hatZ^n_{\Z^2, \Z^2}(A, B) - \log \hatZ^n_{W, \Z^2}(A, B)\right| \leq C\,,
        \end{equation}
		and this implies the  generalization.
	\end{remark*}
    
    Now assume additionally that $1 \leq \ell_{A, B} \ll L$. We next prove two useful lemmas that allow us to switch between domains and modify interactions, as long as we contain the cigar shape capturing typical random walk fluctuations (recall \cref{def:cigar-shape}). 
    \begin{lemma}\label{lem:cigar-likely-in-Z2}
        Let $V \supset \sC(\overline{AB})$. Then, 
        \[|\log\hatZ^n_{V, \Z^2}(A, B) - \log\hatZ^n_{\Z^2, \Z^2}(A, B)| \leq C\,.\]
        Consequently, both the large deviation bounds in \cref{it:DKS-large-deviations} and the surface tension result in \cref{it:DKS-convergence-rate} of \cref{prop:DKS-propositions} hold for $\hatZ^n_{V, \Z^2}(A, B)$.
    \end{lemma}
    \begin{proof}
        It suffices to consider $V = \sC(\overline{AB})$. It is clear that $\hatZ^n_{\sC(\overline{AB}), \Z^2}(A, B) \leq \hatZ^n_{\Z^2, \Z^2}(A, B)$. To show the other side, we take a union bound over \cref{it:DKS-large-deviations} of \cref{prop:DKS-propositions} to obtain
        \[\hatZ^n_{\cS, \Z^2}(A, B \mid \gamma \not\subset \sC(\overline{AB})) \leq M_{A, B}e^{-c(\log L)^2}\hatZ^n_{\cS, \Z^2}(A, B) = o(1)\hatZ^n_{\cS, \Z^2}(A, B)\,.\]
        This implies that 
        \[\hatZ^n_{\sC(\overline{AB}), \Z^2}(A, B) \geq (1 - o(1))\hatZ^n_{\cS, \Z^2}\,,\]
        and then we can apply \cref{eq:change-domain-S-Z2} to conclude.
    \end{proof}

    \begin{lemma}\label{lem:strip-to-Z2-interactions}
        Let $V \supset \sC(\overline{AB})$. Let $W$ contain the infinite vertical strip with sides $x = A_1 + (\log L)^2$ and $x = B_1 - (\log L)^2$. Then, we have 
        \[|\log \hatZ^n_{V, W}(A, B) - \log \hatZ^n_{V, \Z^2}(A, B)| \leq C(\log L)^2\,.\]
    \end{lemma}
    \begin{proof}
        Let $\fT_{A, B}$ be the set of points to the left of $x = A_1 + 2(\log L)^2$ and to the right of $x = B_1 - 2(\log L)^2$. Since $W$ contains the strip, we have
        \[|\log \hatq^n_{W}(\gamma) - \log \hatq^n_{\Z^2}(\gamma)| \leq Ce^{-\beta}|\gamma \cap \fT_{A, B}| + e^{-c(\log L)^2}|\gamma \cap \fT_{A, B}^c|\,.\]
        Let $\widehat\bP$,  $\widehat\bE$ be the probability and expectation for the polymer model with weight $\hatq^n_{\Z^2}$ and partition function $\hatZ^n_{V, \Z^2}$. Then we have
        \[\frac{\hatZ^n_{V, W}(A, B)}{\hatZ^n_{V, \Z^2}(A, B)} \leq \widehat\bE[\exp(Ce^{-\beta}|\gamma \cap \fT_{A, B}| + e^{-c(\log L)^2}|\gamma \cap \fT_{A, B}^c|)]\,.\]
        An easy Peierls argument shows that 
        \[\widehat\bP(|\gamma| \geq 2\ell_{A, B} + j) \leq e^{-(\beta - C)j}\,,\]
        while \cref{lem:cigar-likely-in-Z2} shows that $|\gamma \cap \fT_{A, B}|$ has an exponential tail past $O((\log L)^2)$ that beats the gain of $\exp(Ce^{-\beta}|\gamma \cap \fT_{A, B}|)$ for sufficiently large $\beta$, so that we have
        \[\frac{\hatZ^n_{V, W}(A, B)}{\hatZ^n_{V, \Z^2}(A, B)} \leq e^{C(\log L)^2}\,.\]
        For the lower bound, we have
        \begin{align*}\frac{\hatZ^n_{V, W}(A, B)}{\hatZ^n_{V, \Z^2}(A, B)} &\geq \widehat\bE[\exp(-Ce^{-\beta}|\gamma \cap \fT_{A, B}| - e^{-c(\log L)^2}|\gamma \cap \fT_{A, B}^c|)] \\
        &\geq e^{-C(\log L)^2}\widehat\bP(|\gamma| \leq 3\ell_{A, B}, |\gamma \cap \fT_{A, B}| \leq 2(\log L)^2)\\
        &\geq (1-o(1))e^{-C(\log L)^2}\,.\qedhere
        \end{align*}
    \end{proof}

    We end with a payoff of this subsection - one of the main benefits of comparing two partition functions is the ability to show an event has low probability in one polymer model by bounding its probability in a simpler polymer model. This is stated more generally as follows.
    \begin{lemma}\label{lem:rare-events-changing-Z}
		Let $V_1, V_2, U_1, U_2$ be any domains such that $V_1 \subset V_2$ and $U_1 \subset U_2$. Suppose we have the bounds on the partition functions
		\begin{align}
			&\left|\log \hatZ^n_{V_1, U_1}(A, B) - \log \hatZ^n_{V_1, U_2}(A, B)\right| \leq f_1(A,B)\,, \mbox{ and }\label{eq:cor-bound-change-interaction}\\
			&\left|\log \hatZ^n_{V_1, U_2}(A, B) - \log \hatZ^n_{V_2, U_2}(A, B)\right| \leq f_2(A, B)\,,\label{eq:cor-bound-change-domain}
		\end{align}
		where we allow the more general form of $\Phi_{U}$ as in \cref{rem:conditions-on-Phi-U}.
		Then, for any subset $\cE\subset \cP_{V_1}(A, B)$ such that  
		\[\hatZ^n_{V_2, U_2}(A, B \mid \cE) \leq p\hatZ^n_{V_2, U_2}(A, B)\,,\]
		we also have
		\[\hatZ^n_{V_1, U_1}(A, B \mid \cE) \leq e^{2f_1(A, B) + \tfrac12f_2(A, B)}\sqrt{p}\hatZ^n_{V_1, U_1}(A, B)\,.\]
	\end{lemma}
	\begin{proof}
		Let $\bP^n_{V_2, U_2}$ be the probability measure given by
		\[\bP^n_{V_2, U_2}(\gamma) = \frac{\hatq_{U_2}(\gamma)}{\hatZ^n_{V_2, U_2}(A, B)}\,,\]
		and let $\bE^n_{V_2, U_2}$ be expectation under $\bP^n_{V_2, U_2}$. Let $\bP^n_{V_1, U_2}, \bP^n_{V_1, U_1}$ be defined similarly. Define $\Delta\fI(\gamma) = \fI_{U_1}(\gamma) - \fI_{U_2}(\gamma)$. Then, we have
		\[\bE^n_{V_1, U_2}[\exp(\Delta\fI(\gamma))] = \frac{\hatZ^n_{V_1, U_1}(A, B)}{\hatZ^n_{V_1, U_2}(A, B)} \quad\mbox{ and }\quad \bP^n_{V_1, U_1}(\cE) = \frac{\bE^n_{V_1, U_2}[\one_\cE \exp(\Delta\fI(\gamma))]}{\bE^n_{V_1, U_2}[\exp(\Delta\fI(\gamma))]}\,.\]
		The assumption \cref{eq:cor-bound-change-interaction} implies that $\bE^n_{V_1, U_2}[\exp(\Delta\fI(\gamma))] \geq e^{-f_1(A, B)}$. Applying the bound with respect to $\overline \Phi_{U_1} := 2\Phi_{U_1} - 2\Phi_{U_2}$ (justified since $U_1 \subset U_2$), we also have $\bE^n_{V_1, U_2}[\exp(2\Delta\fI(\gamma))] \leq e^{f_1(A, B)}$.
		By Cauchy-Schwarz for the first inequality and \cref{eq:cor-bound-change-domain} for the second, this implies that
		\[\bP^n_{V_1, U_1}(\cE) \leq e^{2f_1(A, B)}\sqrt{\bP^n_{V_1, U_2}(\cE)} \leq e^{2f_1(A, B) + \tfrac12f_2(A, B)}\sqrt{\bP^n_{V_2, U_2}(\cE)}\,,\]
		concluding the proof.
	\end{proof}
	
	\subsection{Wulff Shape}~
	In this subsection, we recall a few facts about the Wulff shape, which will allow for a more fine-tuned analysis of the polymers via \cref{prop:OZ-normalization}. We move from working with $\tau_{\beta, n}$ to $\tau_\beta$ for the following reason. The polymer model with weight $\hatq^n$ arises naturally from the $\ZGFF$ above a floor, as seen in \cref{prop:CE-law-with-floor}. The previous \cref{prop:DKS-propositions} allows us to estimate $\hatZ^n_{\Z^2, \Z^2}$ by $\tau_{\beta, n}$. However, this surface tension depends on $L$ and $n$ while $\tau_\beta$ does not, making the latter more suitable for analysis once we show that $\tau_{\beta, n}$ is sufficiently comparable to $\tau_{\beta}$. We will postpone such a comparison to \cref{prop:compare-tau}.
    
    For any function $\tau$, we can define a corresponding Wulff shape
	\begin{equation}\label{eq:def-Wulff-shape}
		\cW  = \cW(\tau) := \bigcap_{\sfy \in \R^2} \{\sfh \in \R^2: \sfh \cdot \sfy \leq \tau(\sfy)\}\,.
	\end{equation}
	Now for $\tau$ being the surface tension $\tau_{\beta}$ defined in \cref{eq:surface-tension-tilde}, by definition we have
	\begin{equation*}
		\log \tildeZ_{\Z^2, \Z^2}(\o, \sfy) = -\tau_{\beta}(\sfy)(1+o_{\norm{\sfy}_1}(1))\,.
	\end{equation*}
	Thus, the sum $\sum_{\sfy \in \Z^2} e^{\sfh \cdot \sfy}\tildeZ_{\Z^2, \Z^2}(\o, \sfy)$ converges if and only if $\sfh\cdot \sfy < \tau_{\beta}(\sfy)$ for all $\sfy \in \Z^2$ such that $\norm{\sfy}_1$ is sufficiently large. But then since $\tau_{\beta}$ was defined by homogeneity and is also a continuous function, this is equivalent to requiring that $\sfh\cdot \sfy < \tau_{\beta}(\sfy)$ for all $\sfy \in \R^2$. Hence, an equivalent definition for the Wulff shape is
	\[\cW := \overline{\bigg\{\sfh \in \R^2: \sum_{\sfy \in \Z^2} e^{\sfh\cdot \sfy}\tildeZ_{\Z^2, \Z^2}(\o, \sfy) < \infty\bigg\}}\,.\]
	
	
	Now, for any $\sfh \in \R^2$, we can define the weight 
	\[W^\sfh(\Gamma) = e^{\sfh \cdot X(\Gamma)}\tildeq_{\Z^2}(\Gamma)\]
	where $X(\Gamma)$ is the difference between the ending and starting point of $\Gamma$ (i.e., if $\Gamma \in \cP_{\Z^2}(A, B)$, then $X(\Gamma) := B - A$). A special value of $\sfh$ will be $\sfh_\sfy = \nabla \tau_{\beta}(\sfy)$, so that by the homogeneity of~$\tau_{\beta}$, we have  $\sfh_\sfy \cdot \sfy = \tau_{\beta}(\sfy)$. Note that $\sfh_\sfy$ is only dependent on the angle of $\sfy$ and not the length. We next state an important result which allows us to view the weights of irreducible components as probabilities, when normalized by $e^{\sfh_\sfy \cdot X(\Gamma)}$. Recall the definitions of the sets $\sfA, \sfA_\sfL, \sfA_\sfR$ from \cref{def:irreducible-comp}.
	\begin{proposition}\label{prop:OZ-normalization}
		For any $\delta \in (0, 1)$, there exists $\beta_0$ such that for all $\beta > \beta_0$, any $\sfy \in \cY_\delta(\o) \setminus \{\o\}$, we have
		\[\sum_{\Gamma \in \sfA}\P^{\sfh_\sfy}(\Gamma) := \sum_{\Gamma \in \sfA}e^{\sfh_\sfy\cdot X(\Gamma)}\tildeq_{\Z^2}(\Gamma) = 1\,.\]
		Moreover, the expectation is collinear to $\sfy$, i.e. there exists some constant $\alpha=\alpha(\beta,\sfy)>0$ such that
		\[\E^{\sfh_\sfy}[X(\Gamma)] = \alpha\sfy\,.\]
		Lastly, there exists a constant $\nu_g > 0$ such that for $k \geq 1$,
		\[\sum_{\Gamma \in \sfA_\sfL \cup \sfA_\sfR} \P^{\sfh_\sfy}(\Gamma)1_{\{|\Gamma| > k\}} \leq Ce^{-\nu_g \beta k}\,.\]
		This last exponential decay property holds more generally, replacing $\sfh_\sfy$ with any $\sfh \in \cW$, or replacing $\P^{\sfh_\sfy}(\Gamma)$ with $e^{\sfh_\sfy\cdot X(\Gamma)}\tildeq_{U}(\Gamma)$ for any domain $U$.
	\end{proposition}
    Versions of this proposition have appeared before in slightly different settings \cite{CKL24,IV08,IST15}. The closest to our setting is the treatment in \cite{CKL24}, where the definition of cones is exactly the same. The only differences are in the weights $\tildeq_{\Z^2}(\Gamma)$ and the set of admissible $\Gamma$, both of which are irrelevant to the proof of \cite[Prop. 2.14, Eq. 5.18]{CKL24}, which only uses the properties of the surface tension in \cref{eq:surface-tension} and the existence of cone-points in \cref{lem:cpts-length}. Hence, the result also holds in our disagreement polymer setting.

    We can now conclude by defining the variance $\sigma^2$ which appears in \cref{thm:1}.
        \begin{definition}[Variance of the effective 2\Dim RW]\label{def:sigma}
            For the particular choice of $\sfy = (1, 0)$, define~$\sigma^2$ as the variance of the $y$-coordinate $X(\Gamma)_2$ under the measure $\P^{\sfh_\sfy}$ from \cref{prop:OZ-normalization}. 
        \end{definition}
	
	\section{Initial upper bound on the displacement of a level line}\label{sec:initial-UB-JEMS}
	The goal of this section is to prove \cref{thm:level-line-contains-Wulff}, stating that w.h.p., the $(H+1-n)$ level line $\fL_n$ contains a translation of Wulff shapes, which in turn provides an initial upper bound on the displacement of the level lines from the sides of $\Lambda$. We will use the following shorthand notation: 
	\begin{definition}
		For a subset $A \subseteq \Z^2$, let $\cE_n(A)$ be the event that $\fL_n$ contains $A$.
	\end{definition}
	
	We begin with some notation surrounding the Wulff shape. Recalling the definition of the Wulff shape $\cW(\tau)$ from \cref{eq:def-Wulff-shape}, let $\cW_1(\tau)$ be $\cW(\tau)$ rescaled to have unit area. Define \[\mathsf{w}_1(\tau) = \int_{\partial \cW_1(\tau)} \tau (\theta_s)ds\,,\]
	where $\theta_s$ is the direction of the normal with respect to $\partial \cW_1(\tau)$ at $s$. Our target shape for the level lines will be a translation of Wulff shapes.
	\begin{definition}
		Recall $\tau_\beta$ from \cref{eq:surface-tension-tilde}. Let $\sL(\ell, r)$ be the set obtained by taking the union of all translates of $\ell\cW_1(\tau_{\beta})$ inside the unit square, and then dilating the shape by a factor of $(1 + r)$. 
	\end{definition}
	Define the parameters
	\begin{align*}
		&\ell_n^*  = \mathsf{w}_1(\tau_{\beta}) N_n/2L\,,\\
		&\kappa_{n, b} := N_n^{1/3}(\log L)^b/L\,.
	\end{align*}

    \begin{remark}\label{rem:size-of-ell0}
        Observe that since $N_1 \leq L/5\beta$, then $\ell_1^* \leq \mathsf{w}_1(\tau_{\beta})/{10\beta}$. Moreover, $\mathsf{w}_1(\tau_{\beta}) \approx 4\beta$ (up to a multiplicative $(1\pm\epsilon)$). Hence, $\ell_n^* \leq \ell_1^* \leq 1/2$.
    \end{remark}
	We next state the main theorem of this section.
	
	\begin{theorem}\label{thm:level-line-contains-Wulff} 
		Fix any $L$ and $n \geq 1$, and consider the \ZGFF on an $L \times L$ box.\footnote{Here, we do not require that $L\notin \sB$.} Recall that $H = H(L) := \max\{ h : \hatpi_\infty(\phi_o\geq h) \geq 5\beta / L\}$ and $N_n = N_n(L) := 1/ \hatpi_\infty(\phi_o \geq H+1-n)$. Then, w.h.p., we have
		\[\cE_n(L\sL(\ell_n^*(1+(\log L)^{-5/3}), -(H+1-n)\kappa_{n,15}))\,.\]
		In particular, if $x$ is a point on the bottom boundary of $\Lambda$ at distance at least, say, $L/10$ from the corners, then 
		\begin{enumerate}[(a)]
			\item \label{it:growth-vertical-dist} the vertical distance of $\fL_{n}$, the $(H+1-n)$ level line, from $x$ is at most $N_n^{1/3} (\log L)^{16}$;\label{it:H-n-level-line-dist}
			\item by the upper bound of \cref{eq:LD-ratio}, the analogous distance of $\fL_{n+1}$ is at most \[N_{n+1}^{1/3}(\log L)^{16} \leq N_n^{1/3} \exp\Big(-c \sqrt{\beta\tfrac{\log L}{\log\log L}}\Big)\,.\]\label{it:H-n-1-level-line-dist}
		\end{enumerate}
	\end{theorem}
    \begin{remark}\label{rem:wulff-shape-invisible}
		The shape $L\sL(\ell_n^*(1+(\log L)^{-5/3}), -(H+1-n)\kappa_{n,15}))$ is flat away from the corners, but around the corners there will be an arc of the Wulff shape. When $N_1 = O(L)$, then the size of this arc is $O(L)$ and the Wulff shape is visible (note that the shape and hence the $H$ level line will still occupy a $1 - \epsilon_\beta$ fraction of the sites however). However, for most values of $L$ (i.e., with the exception of a zero logarithmic density set, similar to \cref{rem:Lh-B-concrete}), $\ell_1^* = o(1)$ so that $L\sL(\ell_1^*(1+(\log L)^{-5/3}), -H\kappa_{1,15}))$ is flat all the way up to distance $o(L)$ away from the corners, and there is no visible Wulff shape at the $O(L)$ scale. When we consider $n \geq 2$, we always have $\ell_n^* = o(1)$, so the latter picture is always the case. This is in contrast with the \SOS picture where the Wulff shape is always visible, even for the $(H+1-n)$ level line for finite $n$.
	\end{remark}

     \begin{remark}\label{rem:sos-log}
 While our focus is the \ZGFF, and later (in \cref{sec:gradphi}) we also extend \cref{thm:level-line-contains-Wulff}  to the $|\nabla\phi|^p$ model for all $p>1$, we note that \cref{it:growth-vertical-dist} of the theorem applies also to the case $p=1$: Using \cite[Lemma~5.9]{CLMST16} in lieu of  \cref{prop:CE-law-with-floor}, the proof of \cref{thm:level-line-contains-Wulff} extends to \SOS to show that, for $H=\lfloor \frac1{4\beta}\log L\rfloor $ and any fixed $n\geq 1$, the vertical distance of its $(H+1-n)$ level line is at most $L^{1/3}(\log L)^{16}$ with high probability (refining the $L^{\epsilon}$ in \cite[Thm.~6.2]{CLMST16} into a $\mathrm{polylog}(L)$).
    \end{remark}

    \begin{remark}In this section, we study $\fL_n$ in a mesoscopic rectangle of side-length $O(N_n^{2/3}(\log L)^c)$. We could instead take any side-length $N_n^{2/3}f(L)$ where $(\log L)^C \leq f(L) \leq e^{C\sqrt{\log L}}$, with the upper bound coming from the area restriction to apply \cref{prop:CE-law-with-floor} for this rectangle. (The lower bound will be needed for the application of \cref{thm:growth-gadget} below.) It turns out that a smaller $f(L)$ provides a stronger result on the distance of $\fL_n$ from the bottom of $\partial \Lambda$, as this is limited by how close we can place our rectangle to the bottom without exiting $\Lambda$ (see, e.g., the definition of $\ell_x$ in the proof of \cref{prop:wulff-to-cL}, or \cref{fig:growth-gadget}). Thus, in this paper we choose the smallest possible $f(L)$. However, a larger $f(L)$ provides a better result on the limit shape near the corners: if we choose $f(L) = e^{C\sqrt{\log L}}$, then our proof of \cref{thm:level-line-contains-Wulff} will instead give that, w.h.p., we have 
    \[\cE_n(L\sL(\ell_n^*(1+e^{-\frac{C}3\sqrt{\log L}}), -(H+1-n)N_n^{1/3}e^{3C\sqrt{\log L}}/L))\,.\]
    That is, the term $1+(\log L)^{-5/3}$ arising from \cref{eq:min-start-Wulff} can be lowered to $1+e^{-C'\sqrt{\log L}}$, and this is the dominant error term at the corners of the limit shape, studied in our companion paper \cite{ChenLubetzky26+}.
    \end{remark}
	
	\subsection{Growth gadget}
	In this section, we will prove that in a region around a line segment at the correct scale, the level line will drop below a certain point, motivated by the treatment in \cite[Section 5.2]{CLMST16}. This will eventually be used (see \cref{lem:grow-droplet}) to show that the level line drops far enough below where a Wulff shape would be, which will allow us to iteratively grow the region which we know the level line contains.

    We start with a lemma stating that for a rectangle in the middle of $\Lambda$, we can enforce the desired $H+1-n$, $H-n$ boundary conditions at the cost of moving from the rectangle to a wiggly domain approximating the rectangle. The proof will be postponed to \cref{sec:random-boundary}.
    \begin{lemma}\label{lem:boundary-construction}
        Fix $n \geq 1$, and let $R$ be an $\ell_1 \times \ell_2$ rectangle, where $\ell_1, \ell_2 \gg (\log L)^6$. Let $\Lambda_{9L/10} \subset V \subset \Lambda$, and $0 \leq k \leq H-n$. Suppose that, w.h.p.\ under $\pi_{V}^k$, $R$ is contained in the interior of $\fL_{n-1}$ and that there are two points $A, B$ on the sides of $\partial R$ with the top arc of $\partial R$ from $A$ to $B$ in the interior of $\fL_n$, such that if $d$ is the distance from $\{A, B\}$ to the top or bottom of $R$, then $\sqrt{\ell_1}(\log L)^2 \leq d$. 
        Then there exists a $\pi^k_V$-measurable distribution on connected regions $Q\subset R$ with marked boundary conditions $\xi$ satisfying  \cref{it:Q-simp-conn,it:dist-Q-R-bdy,it:good-boundary-points,it:Q-bc-arc}
        below, such that the following holds. If $\cA_1$ is the area of the interior of $\fL_n$ intersected with~$Q$, and $\cA_2$ is the area above the $(H+1-n)$ level line in $Q$ under $\pi^\xi_Q$, then $\cA_2 \subset \cA_1$ w.h.p.\ under $\pi^k_V$.        
        \begin{enumerate}
            \item \label{it:Q-simp-conn}$Q$ is simply connected,

            \item\label{it:dist-Q-R-bdy} $\dist(\partial Q, \partial R) \leq \log L$,

            \item\label{it:good-boundary-points} There exists $A', B' \in \partial Q$ such that $\sC(\overline{A'B'}) \subset Q$, $\cY^{\btl}(A')$ does not intersect the left side of $Q$, $\cY^{\btr}(B')$ does not intersect the right side of $Q$, and $\max\{d(A, A'), \dist(B, B')\} \leq 2(\log L)^5$.

            \item \label{it:Q-bc-arc} The boundary conditions $\xi$ assigns height $H+1-n$ on the top arc from $A'$ to $B'$, and $H-n$ on the bottom arc.
        \end{enumerate}
    \end{lemma}

    The main theorem of this subsection is the following, showing that the $H+1-n$ level line drops below a certain height. Fix $a > 4$ and $b\geq3a$. Motivated by the above lemma, consider a $N_n^{2/3}(\log L)^a \times 2N_n^{2/3}(\log L)^a$ rectangle $R$ with $A, B$ distance at $CN_n^{1/3}(\log L)^b$ above the bottom of $R$ for some $C > 1$. (The requirements on $\ell_1, \ell_2, d$ are easily satisfied.) Let $Q = Q_{a, b}$ satisfy the above properties in \cref{lem:boundary-construction}. Assume for simplicity that the slope of $\overline{AB}$ is $\theta = \theta_{A, B} \in [0, \pi/4]$ and the midpoint of $\overline{AB}$ is the origin. For future reference, we define here also $F$ as the intersection of $Q$ with the parallelogram that shares sides with $R$, has top and bottom sides parallel to $\overline{AB}$, with height $N_n^{1/3}(\log L)^b$ and centered at the origin. 
	\begin{theorem}\label{thm:growth-gadget}
    The following holds uniformly over all possible $Q$ as above. Let $\fL_n$ be the (unique) $(H+1-n)$ level line induced by $\xi$. Then, with $\pi^\xi_Q$-probability $1 - e^{-c\log L}$, $\fL_n$ lies below the point $X= (0, Y + \sigma(\log L)^a + \log L)$, where
		\begin{align}\label{eq:def-Y-sigma}
			&Y = -\frac{N_n^{1/3}(\log L)^{2a}}{8(\tau_\beta(\theta) + \tau_\beta''(\theta))\cos(\theta)^3}\,,\\
			&\sigma^2 = \frac{N_n^{2/3}(\log L)^a}{4(\tau_\beta(\theta) + \tau_\beta''(\theta))\cos(\theta)^3}\,.\nonumber
		\end{align}
	\end{theorem}
    
	\begin{remark}\label{rem:choice-of-a-b}
		In \cref{subsec:level-line-contain-wulff}, we will apply the above for the choice of $a = 5, b = 15$. As the error probability is smaller than, say $e^{-c(\log L)^2}$, any union bound over a polynomial number of applications of \cref{thm:growth-gadget} will still have an error probability of the same form.
	\end{remark}
	
	In contrast to the polymer model studied in \cref{sec:geom-disagreement-polymers}, we now need to reintroduce the primary area term. Let $V$ be a finite simply connected set, and consider $\gamma \in \cP_{V}(A, B)$ for any $A, B \in V$. Let $F \subset V$, and assume for simplicity that $\overline{AB} \subset F$. Define $\cA_F(\gamma)$ as the signed area of the region above $\mathsf{UE}(\gamma)$ with respect to the line segment $\overline{AB}$, after restricting to $F$. That is, $\cA_F(\gamma)$ is equal to $|D_0 \cap F|$ minus the area above the line segment $\overline{AB}$ in $F$. (Note that although $|D_0|$ depends on the reference domain $V$, the normalized area $\cA_F(\gamma)$ does not.) Define the polymer weight with boundary conditions $H+1-n$ and $H-n$, area tilt $\mu$ on $F$, and domain of interaction $U$, by
	\begin{align}\label{eq:def-qUmu-weights}
		q^n_{U; F, \mu}(\gamma) &= \exp\big(-\sE^*_\beta(\gamma)+ \frac{\mu}{N_n}\cA_F(\gamma) + \sum_{\sfW \cap \nabla_\gamma \neq \emptyset} \Phi'_U(\sfW;\gamma)\big)\\
		&=: \exp\big(-\sE^*_\beta(\gamma)+ \frac{\mu}{N_n}\cA_F(\gamma) + \fI_U(\gamma)\big)\,,\nonumber
	\end{align}
	where $\sE^*_\beta(\gamma)$ and $\Phi'_U(\sfW;\gamma)$ are as defined in \cref{eq:define-E*_beta,eq:def-Phi'}. For convenience, recall that
	\[\sE_\beta^*(\gamma) = \sE_\beta(\gamma) + 3c(\beta)|\gamma| - \sum_{i\geq 2}\log \hatpi_{D_i^\circ}^{h_i}(\phi_x \geq 0,\, \forall x \in D_i^\circ)\,.\]
	As usual, we can define the partition function
	\[Z^n_{V, U; F, \mu}(A, B) = \sum_{\gamma \in \cP_V(A, B)}q^n_{U; F, \mu}(\gamma)\,,\]
	with the notation $Z^n_{V, U; F, \mu}(A, B \mid \cE)$ denoting the restriction of the sum to $\gamma \in \cE$, for any event $\cE$.
	Note that when $\mu = 0$, we recover the definitions of $\hatq^n_U$ and $\hatZ^n_{V, U}$ from \cref{sec:geom-disagreement-polymers}. Of course, $q^n_{U; F, \mu}(\gamma)$ is just a renormalization of the cluster expansion weight $p^\xi_{V; F}(\gamma)$, a fact we will frequently apply without reference.
    
    We next prove a few preliminary lemmas. In order to prove \cref{thm:growth-gadget}, we will in several instances want to show that at a specific point, $\gamma$ has the geometry of a simple path.
    
    \begin{definition} For a bond $b \in \gamma$, call a pocket $D_b = D_b(\gamma)$ the connected component of all finite regions in $\Z^2 \setminus \gamma$ containing a region that has $b$ as part of its boundary. If there is no such finite region adjacent to $b$, then $D_b = \emptyset$. 
	\end{definition}

	\begin{lemma}\label{lem:cut-gamma-in-half}
		Fix any $V, U, F$, and any points $A, B \in V$. Let $b(\gamma)$ be the first horizontal bond where the upper envelope $\mathsf{UE}(\gamma)$ hits the vertical line $x = m$ for a fixed $m \in [A_1, B_1]$. Then there exists a constant $C > 0$ such that for any $k \geq 0$, we have 
		\[Z^n_{V, U; F, \mu}(A, B \mid |\partial D_{b(\gamma)}| \geq k) \leq e^{-(\beta - C)k/2}Z^n_{V, U; F, \mu}(A, B)\,.\]
		In particular, we have
		\[Z^n_{V, U; F, \mu}(A, B \mid D_{b(\gamma)} = \emptyset) \geq (1 - \epsilon_\beta)Z^n_{V, U; F, \mu}(A, B)\,.\]
        The same holds replacing $Z^n_{V, U;F, \mu}$ with $\tildeZ_{V, U}$ or $\hatZ^n_{V, U}$.
	\end{lemma}
	\begin{proof}
		We will prove the case of $Z^n_{V, U;F, \mu}$, as the proof holds verbatim for the other cases. By construction, $D_b$ is the union of some $D_i$ for $i \geq 2$ such that one arc of $\partial D_b$ is a subset of $\mathsf{UE}(\gamma)$ and the remaining arc is a subset of $\mathsf{LE}(\gamma)$. Let $t_1, t_2$ be the first and last points on $\mathsf{UE}(\gamma)$ that are intersected by $\partial D_b$. 
		Let $\Delta(D_{b(\gamma)})$ be the length of the shortest path from $t_1$ to $t_2$ that does not exit $D_{b(\gamma)}$. In particular, we have $|\partial D_{b(\gamma)}| \geq 2\Delta(D_{b(\gamma)})$, and so 
		\[|\partial D_{b(\gamma)}| - \Delta(D_{b(\gamma)}) \geq \tfrac12|\partial D_{b(\gamma)}|\,.\]
		Let $\f$ be the map which replaces the portion of $\gamma$ between $t_1$ and $t_2$ by a path attaining length $\Delta(D_{b(\gamma)})$ (which can be chosen arbitrarily --- say by the minimal lexicographic one that does not exit $D_{b(\gamma)}$). Note that this map can only lower $\mathsf{UE}(\gamma)$ so that $\cA_F(\gamma) \leq \cA_F(\f(\gamma))$. Moreover, the map only removes some regions $D_i$ for $i \geq 2$ and does not change any others. Thus, by the definition of $\sE^*_\beta(\gamma)$ and the decay property of $\Phi_U$, we have for some $C > 0$ that
		\begin{equation}\label{eq:delete-Db-energy}q^n_{U; F, \mu}(\gamma) \leq q^n_{U; F, \mu}(\f(\gamma))e^{-(\beta - C)\tfrac12|\partial D_{b(\gamma)}|}\,.
		\end{equation}		
		Moreover, for a given $\omega$ in the image of $\f$, the number of preimages $\gamma$ such that $\f(\gamma) = \omega$ and $|\partial D_{b(\gamma)}(\gamma)| = k$ is bounded above by the number of connected components of bonds of size $k$ times the number of possible points $t_1$. The former is at most $s^k$ for some universal constant $s$, and the latter of which is at most $k^2$ (since we can look at the first intersection of $\omega$ with $x = m$, and $t_1$ is at most distance $k$ from there). Hence, we obtain
		\begin{align}\label{eq:peierls-map-Db}
			Z^n_{V, U; F, \mu}(A, B \mid \partial D_{b(\gamma)}| \geq k) = \sum_{\gamma: |\partial D_{b(\gamma)}| \geq k} q^n_{U; F, \mu}(\gamma) &\leq \sum_{j \geq k} \sum_{\omega \in \mathsf{Image}(\f)}\sum_{\substack{\gamma\in \f^{-1}(\omega):\\|\partial D_{b(\gamma)}| = j}} q^n_{U; F, \mu}(\gamma)\\
			&\leq \sum_{j \geq k} \sum_{\omega \in \mathsf{Image}(\f)}\sum_{\substack{\gamma\in \f^{-1}(\omega):\\|\partial D_{b(\gamma)}| = j}}q^n_{U; F, \mu}(\f(\gamma))e^{-(\beta - C)\tfrac12j}\nonumber\\
			&\leq\sum_{j \geq k} e^{-(\beta - C')\frac12j}Z^n_{V, U; F, \mu}(A, B)\nonumber\\
			&\leq e^{-(\beta - C'')\frac12k}Z^n_{V, U; F, \mu}(A, B)\,.\nonumber\qedhere
		\end{align}
	\end{proof}
	\begin{remark}\label{rem:cut-gamma-in-half-conditional}
		Since the proof of \cref{lem:cut-gamma-in-half} uses a Peierls map type argument, the results can be strengthened if we have more information about the image of the map $\f$. More specifically, suppose we know that for all $\gamma \in \cE_1$, we have $\f(\gamma) \in \cE_2$. Then, in \cref{eq:peierls-map-Db} we can start instead with the sum $\sum_{\gamma: |\partial D_{b(\gamma)}| \geq k, \gamma \in \cE_1} q^n_{U; F, \mu}(\gamma)$ and replace the sum $\sum_{\omega \in \mathsf{Image}(\f)}$ by $\sum_{\omega \in \mathsf{Image}(\f) \cap \cE_2}$. What was previously an upper bound of \[\sum_{\omega \in \mathsf{Image}(\f)} q^n_{U; F, \mu}(\f(\gamma)) \leq Z^n_{V, U; F, \mu}(A, B)\] now becomes \[\sum_{\omega \in \mathsf{Image}(\f) \cap \cE_2}q^n_{U; F, \mu}(\f(\gamma)) \leq Z^n_{V, U; F, \mu}(A, B\mid\cE_2)\,,\] and hence we get 
		\[Z^n_{V, U; F, \mu}(A, B \mid |\partial D_{b(\gamma)}| \geq k, \cE_1) \leq e^{-(\beta - C)k/2}Z^n_{V, U; F, \mu}(A, B \mid \cE_2)\,.\]
	\end{remark}

    Equipped with the previous lemma, we can now take a detour and prove the promised comparison of the surface tensions. Recall the polymer model with weight $\tildeq_{\Z^2}$ given in \cref{eq:def-polymer-weights-tilde} and the resulting surface tension $\tau_\beta$ in \cref{eq:surface-tension-tilde}.
    \begin{proposition}[Comparing $\tau_{\beta, n}$ and $\tau_\beta$]\label{prop:compare-tau}
        Fix $n \geq 0$ and any $\theta$. There exists $\epsilon_\beta \to 0$ as $\beta \to \infty$ such that for sufficiently large $L$, we have
        \begin{equation*}
            |\tau_{\beta, n}(\theta) - \tau_\beta(\theta)| \leq L^{-1+\epsilon_\beta}\,.
        \end{equation*}
    \end{proposition}
    \begin{proof}
        Let $\n$ be the unit vector with angle $\theta$, and let $N$ be such that the point $N\n\in \Z^2$. It suffices to show that
        \begin{equation}\label{eq:comparing-tau}
         \hatZ^n_{\Z^2, \Z^2}(o, N\n)/\tildeZ_{\Z^2, \Z^2}(o, N\n) \geq (1 - L^{-1+\epsilon_\beta})^N\,,
        \end{equation}
        as the above ratio is by definition at most 1. Let $\widetilde\bP_{\Z^2,\Z^2}$ and $\widetilde\bE_{\Z^2,\Z^2}$ denote the measure and expectation for the polymer model with weight $\tildeq_{\Z^2}$ and partition function $\tildeZ_{\Z^2,\Z^2}$. We may rewrite the above ratio as the expectation $\widetilde \bE_{\Z^2, \Z^2}[\prod_{i \geq 2}\hatpi^{h_i}_{D_i^\circ}(\phi_x \geq 0, \forall x \in D_i^\circ)]$, and then condition on the upper and lower envelopes: \[\widetilde \bE_{\Z^2, \Z^2}\Big[\widetilde \bE_{\Z^2, \Z^2}[\prod_{i \geq 2}\hatpi^{h_i}_{D_i^\circ}(\phi_x \geq 0, \forall x \in D_i^\circ)\mid \mathsf{UE}(\gamma) = S_1, \mathsf{LE}(\gamma) = S_2)]\Big]\,.\] Observe that the pockets $\bigcup_{b \in \gamma} D_b$ are determined by $S_1, S_2$. By \cref{lem:cut-gamma-in-half,rem:cut-gamma-in-half-conditional}, we may restrict to envelopes $S_1, S_2$ such that the maximum size of a pocket boundary $|\partial D_b|$ is at most $\log L$ at the cost of $(1-e^{-(\beta - C)\log L})^N$, which can be absorbed into the $\epsilon_\beta$  in the desired bound in \cref{eq:comparing-tau}. Henceforth we will fix such a realization $S_1, S_2$. 
        
        Next we consider the weight of $\gamma$ in the measure
        $\widetilde\bP_{\Z^2,\Z^2}(\cdot \mid \mathsf{UE}(\gamma) = S_1, \mathsf{LE}(\gamma) = S_2)$. Let $\gamma_{\mathsf{com}}$ denote the bonds common to both $S_1$ and $S_2$, and let $\gamma_{\mathsf{poc}}$ be the remaining bonds of $\gamma$. Note that  $\prod_{i \geq 2}\hatpi^{h_i}_{D_i^\circ}(\phi_x \geq 0, \forall x \in D_i^\circ)$ is a random variable (with randomness in $h_i$) that is measurable w.r.t. $\gamma_{\mathsf{poc}}$. Enumerate over the pockets by $D^{(j)}$, and let $\gamma_{\mathsf{poc}}^{(j)}$ denote the subset of bonds $b$ such that $D_b = D^{(j)}$. Let $E(D^{(j)})$ denote the set of bonds contained in $D^{(j)}$. It is straightforward to verify that the event $\{\mathsf{UE}(\gamma) = S_1, \mathsf{LE}(\gamma) = S_2\}$ is equivalent to having the following:
        \begin{enumerate}
            \item Every bond $b$ of $\gamma_{\mathsf{com}}$ has $D_b = \emptyset$,

            \item $\partial D^{(j)} \subset \gamma^{(j)}_\mathsf{poc} \subset \partial D^{(j)} \cup E(D^{(j)})$.\label{it:gamma-contain-boundary}
        \end{enumerate}
        
        By definition, every disagreement polymer $\gamma \in \cP_{\Z^2, \Z^2}(o, N\n)$ with upper/lower envelopes $S_1, S_2$ must necessarily assign a gradient of 1 across bonds of $\gamma_{\mathsf{com}}$. Hence, in the weight $\tildeq(\gamma)$ we can factor out some terms depending only on the bonds of $S_1$ and $S_2$ to obtain that for the appropriate renormalization $Z(S_1, S_2)$ we have
        \[\tildeq(\gamma) = \frac{1}{Z(S_1, S_2)}\exp\Big(-\sE_\beta(\gamma_\mathsf{poc}) + \sum_{\substack{\sfW \cap \Delta_{\gamma_{\mathsf{poc}}} \neq \emptyset\\ \sfW \subset \bigcup_{b \in \gamma}D_b}}\Phi(\sfW)\Big)\,.\]
        
        We can then split the weight as a product over the pockets, obtaining
        \[\tildeq(\gamma) = \frac{1}{Z(S_1, S_2)}\prod_j\exp\Big(-\sE_\beta(\gamma_\mathsf{poc}^{(j)}) + \sum_{\substack{\sfW \cap \Delta_{\gamma_{\mathsf{poc}}^{(j)}} \neq \emptyset\\ \sfW \subset D^{(j)}}}\Phi(\sfW)\Big)\,.\]
        Together with the condition that $\gamma^{(j)}_\mathsf{poc} \subset \partial D^{(j)} \cup E(D^{(j)})$, this is exactly the weight appearing in \cref{prop:CE-law} with $V = D^{(j)}$ and $\xi = \xi_j$ being $H+1-n$ on the arc of $\partial D^{(j)}$ consisting of bonds of $S_1$, and $H-n$ on the arc consisting of bonds of $S_2$. In other words, the marginal of $\widetilde\bP_{\Z^2,\Z^2}(\cdot \mid \mathsf{UE}(\gamma) = S_1, \mathsf{LE}(\gamma) = S_2)$ on $\gamma_\mathsf{poc}$ is a product measure over disagreement polymers $\gamma_j$ in $\hatpi^{\xi_j}_{D^{(j)}}$ with an extra conditioning on $\partial D^{(j)} \subset \gamma_j$ from \cref{it:gamma-contain-boundary}; denote such a measure by $\widetilde\pi^{\xi_j}_{D^{(j)}}$ for simplicity, with expectation $\widetilde\E^{\xi_j}_{D^{(j)}}$. We thus have
        \begin{align*}
            \widetilde \bE_{\Z^2, \Z^2}\big[\prod_{i \geq 2}\hatpi^{h_i}_{D_i^\circ}&(\phi_x \geq 0, \forall x \in D_i^\circ) \mid \mathsf{UE}(\gamma) = S_1, \mathsf{LE}(\gamma) = S_2\big] \\&= \prod_j \widetilde \E^{\xi_j}_{D^{(j)}}\big[\prod_{i \geq 2}\hatpi^{h_i}_{D_i^\circ}(\phi_x \geq 0, \forall x \in D_i^\circ)\big]\,.
        \end{align*}
        The right side is simply the conditional expectation $\prod_j \widetilde \E^{\xi_j}_{D^{(j)}}\big[ \widetilde \E^{\xi_j}_{D^{(j)}}\big[\one_{\phi_x \geq 0, \forall x \in D^{(j)}} \mid \cF_{\gamma_j}\big]\big]$ where $\cF_{\gamma_j}$ is the $\sigma$-algebra generated by $\gamma_j$. By the tower law, this is $\prod_j \widetilde\pi^{\xi_j}_{D^{(j)}}(\phi_x \geq 0,\ \forall x \in D^{(j)})$, which by a union bound is at least $\prod_j (1 - |D^{(j)}|\max_{x \in D^{(j)}}\widetilde\pi^{\xi_j}_{D^{(j)}}(\phi_x < 0))$. Since there are at most $N$ pockets, to show \cref{eq:comparing-tau} it remains to prove the upper bound
        \[|D^{(j)}|\max_{x \in D^{(j)}}\widetilde\pi^{\xi_j}_{D^{(j)}}(\phi_x < 0) \leq L^{-1+\epsilon_\beta}\,.\]
    
    Recalling now the assumption that $|\partial D^{(j)}| \leq \log L$, we then have $|D^{(j)}| \leq (\log L)^2$. Conditioning on $\partial D^{(j)} \subset \gamma_j$ is equivalent to asking that interior boundary vertices are either $\neq H+1-n$ or $\neq H-n$, depending on if the vertex is adjacent to the boundary condition of $H+1-n$ or $H-n$ respectively of $\xi_j$. This in turn occurs with probability $e^{\epsilon_\beta |\partial D^{(j)}|} \leq L^{\epsilon_\beta}$ for some $\epsilon_\beta$, so that $\widetilde\pi^{\xi_j}_{D^{(j)}}(\phi_x < 0) \leq \hatpi^{\xi_j}_{D^{(j)}}(\phi_x < 0)L^{\epsilon_\beta}$. By monotonicity, we can then lower the boundary conditions to $H-n$ and apply \cref{lem:UB-LD-any-point} to obtain the desired upper bound.
    \end{proof}
    
    We next show that we are, with high probability, in a setting where we have cluster expansion. 
    \begin{lemma}\label{lem:cluster-expansion-likely} Let $\gamma$ be the disagreement polymer from $A$ to $B$ in the setting above \cref{thm:growth-gadget}, recalling also $Q, \xi, F$ there. Let $\cG_1$ be the event that $|F| - |D_0 \cap F| - |D_1 \cap F| \leq L^{5/6}$ and $|\gamma| \leq 2N_n^{2/3}(\log L)^a$. Then, $\pi_{Q; F}^\xi(\cG_1^c) \leq e^{-c L^{1/24}}$. Moreover, if $\cG^\sqcup$ is the event that $\fL_n$ stays above a given horizontal line $\cH$ above the bottom of $Q$, then $\pi_{Q; F}^\xi(\cG_1^c \mid \cG^\sqcup) \leq e^{-c L^{1/24}}$.
	\end{lemma}
	\begin{proof}
        By \cref{prop:CE-law-with-floor} and comparing with the definition of $\hatq^n_Q$, we have
        \[\pi^\xi_{Q; F}(\gamma) \propto \hatq^n_Q(\gamma)\prod_{i = 0, 1}\hatpi^{H+1-n-i}_{D_i^\mathsf{o}}(\phi_x \geq 0, \forall x \in D_i^\mathsf{o} \cap F)\,,\]
        with the same statement holding for $\pi^\xi_{Q; F}(\gamma \mid \cG^\sqcup)$ except restricting to $\gamma \in \cG^\sqcup$. To obtain a rough lower bound on the probabilities, a standard computation using FKG and monotonicity in the boundary conditions obtains
        \begin{align*}\hatpi^{H+1-n-i}_{D_i^\mathsf{o}}(\phi_x \geq 0, \forall x \in D_i^\mathsf{o} \cap F) 
        \geq \frac12\exp\bigg(-\sum_{x \in D_i^\mathsf{o} \cap F} \hatpi^0_{D_i^\mathsf{o}}(\phi_x < -(H-n))\bigg)\,.
        \end{align*}
        Using that $\hatpi^0_{D_i^\mathsf{o}}(\phi_x < -(H-n)) \leq L^{-1 + o(1)}$ and the fact that $|F| \leq N_n(\log L)^{a+b}$, we obtain
        \[\prod_{i = 0, 1}\hatpi^{H+1-n-i}_{D_i^\mathsf{o}}(\phi_x \geq 0, \forall x \in D_i^\mathsf{o} \cap F) \geq \frac12e^{-N_n(\log L)^{a+b}L^{-1 + o(1)}} \geq e^{-L^{o(1)}}\,.\]
        Hence, changing the weights to $\hatq^n_Q(\gamma)$ can only tilt the measure by a multiplicative factor of $e^{L^{o(1)}}$, so it suffices to prove that 
        \begin{align}\label{eq:cluster-expansion-likely-reduction}&\hatZ^n_{Q, Q}(A, B \mid \cG_1^c) \leq e^{-cL^{1/24}}\hatZ^n_{Q, Q}(A, B)\,, \mbox{ and }\\ \label{eq:cluster-expansion-likely-reduction-G'}
        &\hatZ^n_{Q, Q}(A, B \mid \cG_1^c, \cG^\sqcup) \leq e^{-cL^{1/24}}\hatZ^n_{Q, Q}(A, B \mid \cG^\sqcup)\,.
        \end{align}
        Now in the polymer model, the bound on the length
		\[\hatZ^n_{Q, Q}(A, B \mid |\gamma| > 2N_n^{2/3}(\log L)^a) \leq e^{-cN_n^{2/3}(\log L)^a}\hatZ^n_{Q, Q}(A, B)\]
		follows by an easy Peierls argument (take the map that sends $\gamma$ to a minimal length path from $A$ to $B$). The same holds when intersecting with $\cG^\sqcup$, as mapping to the minimal length path stays within the set $\cG^\sqcup$.
		
		Bounding the area of $|F| - |D_0 \cap F| - |D_1 \cap F| = \bigcup_{i \geq 2} D_i \cap F$ follows essentially the same proof of \cref{lem:cut-gamma-in-half} but simpler since there is no area term anymore. By the above, we can assume that $|\gamma| \leq 2N_n^{2/3}(\log L)^a$. Thus, if $|F| - |D_0 \cap F| - |D_1 \cap F| \geq L^{5/6}$, then by the pigeonhole principle we must have $|D_b| \geq L^{1/12}$ for some $b \in \gamma$ which is adjacent to a vertex in $F$. So, we fix any bond $b$ adjacent to a vertex in $F$, and let $\gamma$ be such that $b \in \gamma$. As in \cref{lem:cut-gamma-in-half}, let $t_1, t_2$ be the first and last points on $\mathsf{UE}(\gamma)$ that are intersected by $D_b$, and let $\Delta(D_b)$ be the length of the shortest path from $t_1$ to $t_2$ that does not exit $D_b$ so that $|\partial D_b| - \Delta(D_b) \geq \tfrac12|\partial D_b|$.
		Let $\f_b$ be the map which replaces the portion of $\gamma$ between $t_1$ and $t_2$ by a path attaining length $\Delta(D_b)$. By the definition of $\cE^*_\beta(\gamma)$ and the decay property of  $\Phi_Q$, we have \begin{equation}\label{eq:delete-Db-energy-pi}\hatq^n_Q(\gamma) \leq \hatq^n_Q(\f_b(\gamma))e^{-(\beta - C)\tfrac12|\partial D_b|}\,.
		\end{equation}
		Moreover, for a given $\omega$ in the image of $\f_b$, the number of preimages $\gamma$ such that $\f_b(\gamma) = \omega$ and $|\partial D_b(\gamma)| = k$ is bounded above by the number of connected components of bonds of size $k$, which is at most $s^k$ for some universal constant $s$. Finally, observe that by the isoperimetric inequality, $|D_b| \geq L^{1/12}$ implies $|\partial D_b| \geq 4L^{1/24}$. Hence, we obtain
		\begin{align*}
			\sum_{\substack{\gamma:b \in \gamma,\\|D_b| \geq L^{1/12}}} \hatq^n_Q(\gamma) &\leq \sum_{k \geq 4L^{1/24}} \sum_{\omega \in \mathsf{Image}(\f_b)}\sum_{\substack{\gamma\in \f_b^{-1}(\omega):\\|\partial D_b(\gamma)| = k}} \hatq^n_Q(\gamma)\\
			&\leq \sum_{k \geq 4L^{1/24}} \sum_{\omega \in \mathsf{Image}(\f_b)}\sum_{\substack{\gamma\in \f_b^{-1}(\omega):\\|\partial D_b(\gamma)| = k}}\hatq^n_Q(\f_b(\gamma))e^{-(\beta - C)\tfrac12k}\\
			&\leq\sum_{k \geq 4L^{1/24}} e^{-(\beta - C')\frac12k}\hatZ^n_{Q, Q}(A, B)\\
			&\leq e^{-(\beta - C'')2L^{1/24}}\hatZ^n_{Q, Q}(A, B)\,.
		\end{align*}
		We can then take a union bound over the at most $O(N_n(\log L)^{a+b})$ bonds adjacent to a vertex in $F$ to conclude the proof of \cref{eq:cluster-expansion-likely-reduction}. 
        
        By the logic in \cref{rem:cut-gamma-in-half-conditional}, we also obtain \cref{eq:cluster-expansion-likely-reduction-G'} as long as we can show that if $\gamma \in \cG^\sqcup$, then $\f_b(\gamma) \in \cG^\sqcup$, or equivalently that for every $D_b$, there exists a shortest path from $t_1$ to $t_2$ that stays above $\cH$. Without loss of generality, assume $\cH$ is at a half-integer height, and suppose for contradiction that there is no such path. Then, take a shortest path $P$, and let $\sfu, \sfv$ be two consecutive points of $P$ on $\cH$ which mark a drop of $P$ below $\cH$ (i.e., in between $\sfu, \sfv$, $P$ lies strictly below $\cH$). Observe that the region sandwiched between $P$ and the arc of $\mathsf{UE}(\gamma)$ between $t_1, t_2$ is contained in $D_b$. Since $\gamma \in \cG^\sqcup$, $\mathsf{UE}(\gamma)$ must be at or above $\cH$, and in particular at or above the line segment $\overline{uv}$. Moreover, the arc of $P$ from $u$ to $v$ must be strictly below $\overline{uv}$. Hence,  $\overline{uv}$ is in the sandwiched region, and is therefore in $D_b$. This means that the path $P'$ which replaces the arc of $P$ between $u$ and $v$ by $\overline{uv}$ is in $D_b$. Since $|P'| < |P|$, this is a contradiction.
	\end{proof}
	
	The above lemma shows we can restrict our attention to a set of ``good'' $\gamma$ for which we have cluster expansion. We will also want to then consider the partition function with respect to these cluster expansion weights, and it will be convenient to have the partition function sum over all possible $\gamma \subset Q$, and not just the ``good'' $\gamma$ from \cref{lem:cluster-expansion-likely}. For the same reason as above, this difference is negligible. That is, for $Q, F, A, B$ as above \cref{thm:growth-gadget} and for any constant $\mu > 0$, we claim that 
    \begin{align}\label{eq:good-event-CE}&Z^n_{Q, Q; F, \mu}(A, B \mid \cG_1^c) \leq e^{-cL^{1/24}}Z^n_{Q, Q; F, \mu}(A, B)\,, \mbox{ and }\\
    \label{eq:good-event-CE-G'}&Z^n_{Q, Q; F, \mu}(A, B \mid \cG_1^c, \cG^\sqcup) \leq e^{-cL^{1/24}}Z^n_{Q, Q; F, \mu}(A, B \mid \cG^\sqcup)\,.
    \end{align}
    Indeed, the effect of the area term can only tilt the measure by a factor of $O(\exp(-(\log L)^{a+b}))$. Without the area term however, we reduce to showing that 
    \begin{align*}&\hatZ^n_{Q, Q}(A, B \mid \cG_1^c) \leq e^{-cL^{1/24}}\hatZ^n_{Q, Q}(A, B)\,, \mbox{ and }\\
    &\hatZ^n_{Q, Q}(A, B \mid \cG_1^c, \cG^\sqcup) \leq e^{-cL^{1/24}}\hatZ^n_{Q, Q}(A, B \mid \cG^\sqcup)\,,
    \end{align*}
    which were proven above in \cref{eq:cluster-expansion-likely-reduction,eq:cluster-expansion-likely-reduction-G'}.

    \begin{lemma}\label{lem:Db-smaller-log2}
        Let $\gamma$ be the disagreement polymer from $A$ to $B$ in the setting above \cref{thm:growth-gadget}. Let $\cG_2$ be the event that for all $b \in \gamma$, $|\partial D_b| \leq \log L$. Then, $\pi^\xi_{Q; F}(\cG_2^c) \leq e^{-c\beta\log L}$. Moreover, if $\cG^\sqcup$ is the event that $\fL_n$ stays above a given horizontal line $\cH$ above the bottom of $Q$, then $\pi_{Q; F}^\xi(\cG_2^c \mid \cG^\sqcup) \leq e^{-c\beta\log L}$.
    \end{lemma}
    \begin{proof}
        By \cref{lem:cluster-expansion-likely}, it suffices to upper bound $\pi^\xi_{Q; F}(\cG_2^c, \cG_1)$. This in turn is upper bounded by $\pi^\xi_{Q; F}(\cG_2^c \mid \cG_1)$, whence we can use \cref{prop:CE-law-with-floor} for the equality and \cref{eq:good-event-CE} for the inequality to write
        \[\pi^\xi_{Q; F}(\cG_2^c \mid \cG_1) = (1+o(1))\frac{Z^n_{Q, Q; F, 1}(A, B \mid \cG_2^c, \cG_1)}{Z^n_{Q, Q; F, 1}(A, B \mid \cG_1)} \leq (1+o(1))\frac{Z^n_{Q, Q; F, 1}(A, B \mid \cG_2^c)}{Z^n_{Q, Q; F, 1}(A, B)}\,.\]
        For a fixed bond $b \in Q$, we have
        \begin{equation*}
            Z^n_{Q, Q; F, 1}(A, B \mid |\partial D_b| > \log L) \leq e^{-c\beta\log L}Z^n_{Q, Q; F, 1}(A, B)
        \end{equation*}
        by applying yet another map argument which is a minor modification of the ones in \cref{lem:cut-gamma-in-half,lem:cluster-expansion-likely} (in particular, we would apply the map $\f_b$ from \cref{lem:cluster-expansion-likely}, but use the bound \cref{eq:delete-Db-energy} from \cref{lem:cut-gamma-in-half} since there is an area term). The proof now concludes after taking a union bound over all $b \in Q$. The case conditioning on $\cG^\sqcup$ follows similarly as $\f_b$ preserves the event $\cG^\sqcup$, just intersect every partition function above with the event $\cG^\sqcup$ and use \cref{eq:good-event-CE-G'} instead of \cref{eq:good-event-CE}. We record for later use the end result that 
        \begin{equation}\label{eq:G2-likely-given-G'}
            Z^n_{Q, Q; F, 1}(A, B \mid \cG_2^c, \cG^\sqcup) \leq e^{-c\beta\log L}Z^n_{Q, Q; F, 1}(A, B \mid \cG^\sqcup)\,.\qedhere
        \end{equation}
    \end{proof}
    
	Our last item before proving the theorem is to provide a bound on partition functions in terms of the functions $\sG^n_\mu(\ell,\theta)$, defined as
	\begin{equation}\label{eq:define-G-mu}
		\sG^n_\mu(\ell, \theta) = -\tau_\beta(\theta)\ell + \frac{\ell^3\mu^2}{24(\tau_\beta(\theta) + \tau_\beta''(\theta))N_n^2}\,.
	\end{equation}
	Let $\delta = 1/10$. Let $\cR_k$ be the set of pairs of points $(A', B')$ in $Q$ with distance $\ell_{A', B'} \leq 2^{k(1-\delta)}N_n^{2/3}$, angle $\theta_{A', B'}$, and horizontal distance $M_{A', B'}$. Without loss of generality, assume $A'$ is to the left of $B'$.
	Finally, let $\fT_{A', B'}$ be the set of points either to the left of the vertical line $x = A'_1 + 2(\log L)^2$ or to the right of the vertical line $x = B'_1 - 2(\log L)^2$. Fix any $\mu \geq 0$, and define the weights
	\begin{equation*}
		w^n(\gamma) = \exp\big(-\sE^*_\beta(\gamma) + \fI_{\Z^2}(\gamma) + e^{-\beta}|\gamma \cap \fT_{A', B'}| + \frac{\mu}{N_n}\cA_F(\gamma)\big)\,,
	\end{equation*}
	where the energy $\sE^*_\beta(\gamma)$ is defined as in \cref{eq:define-E*_beta} with respect to boundary conditions $H+1-n$ and $H-n$. The new term $e^{-\beta}|\gamma \cap \fT_{A', B'}|$ in the weight $w^n(\gamma)$ should be thought of as a buffer term that allows us to switch between different interaction functions $\fI_U(\gamma)$ (see \cref{cor:growth-bound-partition-functions}). It will also be important for later use to note that the exact dimensions of $Q$ and $F$ (in particular, the choice of $a, b$ above \cref{thm:growth-gadget}) play no role in the following proposition.
    \begin{proposition}\label{prop:bound-partition-functions}
		Let $Z^n_{A', B'} = \sum_{\gamma \in \cP_{\Z^2}(A', B')}w^n(\gamma)$. Then, for all $k \leq \frac{a}{1-\delta}\log_2\log L$, uniformly over all $A', B' \in \cR_k$, we have
		\begin{equation*}
			Z^n_{A', B'} \leq z_ke^{\sG^n_\mu(\ell_{A', B'}, \theta_{A', B'}})
		\end{equation*}
		where $z_1 = e^{c(\log L)^2}$, $z_k = z_1^{2^{k-1}}$.
	\end{proposition}
	\begin{proof}
		Following \cite[Section 5.3]{CLMST16}, this will be a proof by induction.  Let $
		\widehat\bP_{A', B'}$ denote the probability measure on $\gamma$ given by 
		\begin{equation*}
			\widehat\bP_{A', B'}(\gamma) = \frac{\hatq^n_{\Z^2}(\gamma)}{\hatZ^n_{\Z^2, \Z^2}(A', B')}\,,
		\end{equation*}
		and let $\widehat\bE_{A', B'}$ denote the expectation. 
		We begin with the base case taking $k = 1$. Fix $A', B'$ distance $\ell_{A', B'} \leq 2N_n^{2/3}$ apart. 
		Let $h_{\max}(\gamma)$ be the maximum height reached by $\gamma$ with respect to the line segment $\overline{A'B'}$. By a union bound over \cref{it:DKS-large-deviations} of \cref{prop:DKS-propositions}, we obtain
		\begin{equation}\label{eq:hmax-bound}
			\widehat\bP_{A', B'}(h_{\max}(\gamma) \geq j) \leq CM_{A', B'}^{3/2}e^{-c(j \wedge j^2/M_{A', B'})}\,.
		\end{equation}
		By an easy Peierls argument mapping $\gamma$ to a minimal length path from $A'$ to $B'$, we also have
		\begin{align*}
			\widehat\bP_{A', B'}(|\gamma| \geq 2\ell_{A', B'} + j) \leq e^{-(\beta - C)j}\,.
		\end{align*}
		Now, by Cauchy--Schwarz, we write 
		\begin{align*}
			Z^n_{A', B'} &= \hatZ^n_{\Z^2, \Z^2}(A', B')\widehat\bE_{A', B'}\bigg[e^{\tfrac{\mu}{N_n}\cA_F(\gamma) + e^{-\beta}|\gamma \cap \fT_{A', B'}|}\bigg] \\
			&\leq \hatZ^n_{\Z^2, \Z^2}(A', B')\widehat\bE_{A', B'}\bigg[e^{\tfrac{2\mu}{N_n}\cA_F(\gamma)}\bigg]\widehat\bE_{A', B'}\bigg[e^{2e^{-\beta}|\gamma \cap \fT_{A', B'}|}\bigg]\,.
		\end{align*}
		By the fact that $\cA_F(\gamma) \leq |\gamma|\big(h_{\max}(\gamma) \wedge N_n^{1/3}(\log L)^b\big)$, we have \[\widehat\bE_{A', B'}\bigg[e^{\tfrac{2\mu}{N_n}\cA_F(\gamma)}\bigg] \leq \widehat\bE_{A', B'}\bigg[e^{|\gamma|o(1)}\one_{|\gamma| \geq 5N_n^{2/3}}\bigg] + \widehat\bE_{A', B'}\bigg[e^{\tfrac{10\mu}{N_n^{1/3}}h_{\max}(\gamma)}\bigg]\,.\]
        The first term on the right can then be bounded by a constant by the tail bound on $|\gamma|$. The second term can be bounded by $CM^{1/2}_{A', B'}$ by the tail bound on $h_{\max}(\gamma)$.
        
        Similarly, the tail bound on $|\gamma \cap \fT_{A', B'}|$ given by \cref{it:DKS-large-deviations} of \cref{prop:DKS-propositions} implies that $\widehat\bE_{A', B'}\big[e^{2e^{-\beta}|\gamma \cap \fT_{A', B'}|}\big] \leq e^{c(\log L)^2}$. Moreover, by \cref{it:DKS-convergence-rate} of \cref{prop:DKS-propositions} and \cref{prop:compare-tau}, we have 
		\begin{equation*}
			\hatZ^n_{\Z^2, \Z^2}(A', B') \leq Ce^{-\tau_\beta(\theta_{A', B'})\ell_{A', B'}}\,.
		\end{equation*}
		Putting the above together implies the claim for $k = 1$.
		
		Now we show the induction step. Let $(A', B') \in \cR_{k+1}$. Let $b(\gamma)$ be the first horizontal bond where  $\mathsf{UE}(\gamma)$ hits the middle vertical line between $A'$ and $B'$. Let $C = (C_1, C_2)$ be the left endpoint of $b(\gamma)$. Define $\Delta C_2 = C_2 - \frac{A'_2 + B'_2}{2}$, and first consider the case where $|\Delta C_2| \geq N_n^{1/3}(\log L)^{3a}$. We have
		\[Z^n_{A', B'}(|\Delta C_2| \geq N_n^{1/3}(\log L)^{3a}) = \hatZ^n_{\Z^2, \Z^2}(A', B')\widehat\bE_{A', B'}\bigg[\one_{|\Delta C_2| \geq N_n^{1/3}(\log L)^{3a}}e^{\tfrac{\mu}{N_n}\cA_F(\gamma) + e^{-\beta}|\gamma \cap \fT_{A', B'}|}\bigg]\,.\]
		Hence, by another application of Cauchy-Schwarz, it suffices to bound the $\widehat\bP_{A', B'}$ probability that $|\Delta C_2| \geq N_n^{1/3}(\log L)^{3a}$, which is at most $e^{-c(\log L)^{5a}}$ by \cref{eq:hmax-bound}. Hence, we altogether have
		\begin{equation}\label{eq:UB-j-large}Z^n_{A', B'}(|\Delta C_2| \geq N_n^{1/3}(\log L)^{3a}) \leq e^{-\tau_\beta(\theta_{A', B'})\ell_{A', B'}-c(\log L)^{5a}}\,.
		\end{equation}
		
		For the case that $|\Delta C_2| \leq N_n^{1/3}(\log L)^{3a}$, we recall the event $D_{b(\gamma)}$ from \cref{lem:cut-gamma-in-half} where we have (noting \cref{rem:cut-gamma-in-half-conditional})
		\[Z^n_{A', B'}(D_{b(\gamma)} = \emptyset, |\Delta C_2| \leq N_n^{1/3}(\log L)^{3a}) \geq (1 - \epsilon_\beta)Z^n_{A', B'}(|\Delta C_2| \leq N_n^{1/3}(\log L)^{3a})\,.\] 
		On the event $D_{b(\gamma)} = \emptyset$, we can split up $\gamma = \gamma_1 \circ \gamma_2$ as the segments of $\gamma$ before and after $C$. Write
		\[\cA_F(\gamma) = \cA_F^0 + \cA_F(\gamma_1) + \cA_F(\gamma_2)\,,\]
		where $\cA_F(\gamma_1)$ is the signed area above $\gamma_1$ w.r.t. the line segment $\overline{A'C}$, $\cA_F(\gamma_2)$ is the signed area above $\gamma_2$ w.r.t. the line segment $\overline{CB'}$, and $\cA_F^0$ is the signed area of the triangle $A'CB'$ intersected with $F$ (signed negative when $C$ is above $\overline{A'B'}$). Note that $|\cA_F^0| \leq \tfrac{\ell_{A', B'}}{2}|\Delta C_2|\cos(\theta_{A', B'})$. Now, define the correction term
		\begin{equation}\label{eq:correction-term}
		    \Delta\fI_{\Z^2}(\gamma_1, \gamma_2) :=  \fI_{\Z^2}(\gamma_1) + \fI_{\Z^2}(\gamma_2) - \fI_{\Z^2}(\gamma)\,.
		\end{equation}
		By the decay properties of $\Phi$ in \cref{eq:decay-bound-W-gamma}, we have that $|\Delta\fI_{\Z^2}(\gamma_1, \gamma_2)| \leq e^{-(\beta - C)}|\gamma_2 \cap \fT_{C, B'}|$. Hence, we have the bound
		\[Z^n_{A', B'}(D_{b(\gamma)} = \emptyset, |\Delta C_2| \leq N_n^{1/3}(\log L)^{3a}) \leq \sum_{\gamma: |\Delta C_2| \leq N_n^{1/3}(\log L)^{3a}}e^{\frac{1}{N_n}|\cA_F^0|}w^n(\gamma_1)w^n(\gamma_2)\,.\]
		Since both $(A', C)$ and $(C, B')$ are in $\cR_k$, by the induction hypothesis, the above display is bounded by
		\[z_k^2\sum_{|\Delta C_2| \leq N_n^{1/3}(\log L)^{3a}} \exp\bigg(\frac{\mu\ell_{A', B'}}{2N_n}|\Delta C_2|\cos(\theta_{A', B'}) + \sG^n_\mu(\ell_{A', C}, \theta_{A', C}) + \sG^n_\mu(\ell_{C, B'}, \theta_{C, B'})\bigg)\,.\]
		By a Taylor expansion in $\theta$ followed by a Gaussian summation and the convexity of $\tau_\beta$ (see \cite[Prop. 5.11]{CLMST16} for details), this is in turn bounded above by 
		\[C(\beta)\sqrt{\ell_{A', B'}}z_k^2\exp(\sG^n_\mu(\ell_{A', B'}, \theta_{A', B'}))\]
		for some constant $C(\beta)$. In total, we have proved that
		\[Z^n_{A', B'}(|\Delta C_2| \leq N_n^{1/3}(\log L)^{3a}) \leq (1+\epsilon_\beta)C(\beta)\sqrt{\ell_{A', B'}}z_k^2\exp(\sG^n_\mu(\ell_{A', B'}, \theta_{A', B'}))\,,\]
		which together with \cref{eq:UB-j-large}, after increasing the constant $c$ in the definition of $z_1$ to absorb the $C(\beta)\sqrt{\ell_{A', B'}}$ term above, concludes the induction step.    
    \end{proof}
	\begin{corollary}\label{cor:growth-bound-partition-functions}
		For any $A', B'$ in $Q$, we have
		\begin{equation*}
			Z^n_{Q, \Z^2; F, \mu}(A', B') = \sum_{\gamma \in \cP_Q(A', B')} q^n_{\Z^2; F, \mu}(\gamma) \leq \exp(\sG^n_\mu(\ell_{A', B'}, \theta_{A', B'}) + O((\log L)^{2a}))\,.
		\end{equation*}
		Let $W$ be any region containing the infinite vertical strip between the lines $x = A'_1 + (\log L)^2$ and $x = B'_1 - (\log L)^2$. The same bound holds if we replace the weight $q^n_{\Z^2; F, \mu}(\gamma)$ with $q^n_{W; F, \mu}(\gamma)$, or with 
		\[\exp\big(-\sE^*_\beta(\gamma)+ \frac{\mu}{N_n}\cA_F(\gamma) + \fI_W(\gamma) -\Delta\fI_W(\gamma, \gamma_\mathsf{left})\big)\]
		uniformly over all $\gamma_{\mathsf{left}} \in \cP_Q(A, A')$, with $\Delta\fI_W(\gamma, \gamma_\mathsf{left})$ defined as in \cref{eq:correction-term}.
	\end{corollary}
	\begin{proof}
		Note that because of the decay properties of the interaction terms, switching the weights from $\fI_{\Z^2}(\gamma)$ to $\fI_W(\gamma)$ or to $\fI_W(\gamma) - \Delta\fI_W(\gamma, \gamma_{\mathsf{left}})$ comes with a factor of at most $e^{-(\beta - C)}|\gamma \cap \fT_{A', B'}|$. Thus, the corollary follows from \cref{prop:bound-partition-functions} at $k = \frac{a}{1-\delta}\log_2\log L$ and the trivial bound $Q \subset \Z^2$.
	\end{proof}
	
	We are now finally ready to prove \cref{thm:growth-gadget}.
	\begin{proof}[Proof of \cref{thm:growth-gadget}]
        For simplicity, we may assume that in the construction of $Q$ we have $A' = A$ and $B' = B$, as the shift by $(\log L)^5$ is negligible compared to the margin of error provided in the proof of the theorem.
        Define the event $\fU$ as the event that $\fL_n$ intersects the line $x = 0$ at a point higher than $X = (0, Y + \sigma(\log L)^a + \log L)$ (recall the definitions of $Y, \sigma$ from \cref{eq:def-Y-sigma}). Recall the goal of the theorem is to upper bound $\pi^\xi_{Q}(\fU)$. Since $\fU$ is a decreasing event in $\phi$, by FKG we have $\pi^\xi_{Q}(\fU) \leq \pi^\xi_{Q; F}(\fU)$. Then by FKG again we can condition on the decreasing event $\cG^\sqcup$ that $\fL_n$ stays above a horizontal line $\cH$ distanced $2\log L$ above the highest point on the bottom of $Q$,  so that $\pi^\xi_{Q; F}(\fU) \leq \pi^\xi_{Q; F}(\fU \mid \cG^\sqcup)$.
            
        By \cref{lem:cluster-expansion-likely,lem:Db-smaller-log2}, we have $\pi^\xi_{Q; F}(\cG_1^c \mid \cG^\sqcup) \leq e^{-cL^{1/24}}$ and $\pi^\xi_{Q; F}(\cG_2^c \mid \cG^\sqcup) \leq e^{-c\beta\log L}$. We are then left with $\pi^\xi_{Q; F}(\fU, \cG_1, \cG_2 \mid \cG^\sqcup)$, which we upper bound by $\pi^\xi_{Q; F}(\fU \mid \cG_1, \cG_2, \cG^\sqcup)$. Since we are now conditioning on the event $\cG_1$, and also since $\fL$ and hence $\fU, \cG^\sqcup$ can be read from $\gamma$, we can use \cref{prop:CE-law-with-floor} to write the above as the ratio
		\begin{align}\label{eq:growth-initial-fraction}
		    \pi^\xi_{Q; F}(\fU \mid \cG_1, \cG_2, \cG^\sqcup) = (1+o(1))\frac{Z^n_{Q, Q; F, 1}(A, B \mid \fU, \cG_1, \cG_2, \cG^\sqcup)}{Z^n_{Q, Q; F, 1}(A, B\mid \cG_1, \cG_2, \cG^\sqcup)}\,.
		\end{align}

		We first upper bound the numerator of \cref{eq:growth-initial-fraction}. Let $S$ be an infinite vertical strip with sides containing the sides of $Q$. Let $W$ be the extension of $Q$ to $S$ at the top and bottom of $Q$. (As we will see, the exact choice of $S$ and point of extension are irrelevant.) We show we can move from $Q$ interactions to $W$ interactions. Indeed, $\cG_1$ implies $\gamma$ stays a distance $\log L$ from the top of $Q$ (as such an excursion would force $\gamma$ to have length $3N_n^{2/3}(\log L)^a$, coming from $2N_n^{2/3}(\log L)^a$ vertical bonds and $N_n^{1/3}(\log L)^a$ horizontal bonds). Moreover, $\cG_2 \cap \cG^\sqcup$ implies $\gamma$ stays distance $\log L$ away from the bottom of $Q$. Hence, by the decay properties of $\Phi$ and the polynomial bound on the length of $\gamma$, we can replace the weights $q^n_{Q; F, 1}(\gamma)$ with $q^n_{W; F, 1}(\gamma)$ at the cost of $(1+o(1))$ and then disregard the events $\cG_1, \cG^\sqcup$ so that 
        \begin{equation}\label{eq:growth-numerator}
            Z^n_{Q, Q; F, 1}(A, B \mid \fU, \cG_1, \cG_2, \cG^\sqcup) = (1+o(1))Z^n_{Q, W; F, 1}(A, B \mid \fU, \cG_1, \cG_2, \cG^\sqcup) \leq Z^n_{Q, W; F, 1}(A, B \mid \fU, \cG_2)\,.
        \end{equation}
        
        Next, let $b(\gamma)$ be the first horizontal bond where $\mathsf{UE}(\gamma)$ hits the line $x = 0$. Let $\underline\fU$ be the event that $\fL_n$ intersects $x = 0$ above $X - (0, \log L)$. Then, for the map $\f$ in \cref{lem:cut-gamma-in-half}, we have $\gamma \in \fU \cap \cG_2$ implies $\f(\gamma) \in \underline\fU \cap D_{b(\gamma)} = \emptyset$, so that \cref{rem:cut-gamma-in-half-conditional} gives us
		\[Z^n_{Q, W; F, 1}(A, B \mid |\partial D_{b(\gamma)}| > 0, \fU, \cG_2) <\epsilon_\beta Z^n_{Q, W; F, 1}(A, B \mid D_{b(\gamma)} = \emptyset, \underline\fU)\,.\]
		This implies that
        \begin{align}\label{eq:Db-empty}Z^n_{Q, W; F, 1}(A, B \mid \fU, \cG_2) &\leq Z^n_{Q, W; F, 1}(A, B \mid D_{b(\gamma)} = \emptyset, \fU) + \epsilon_\beta Z^n_{Q, W; F, 1}(A, B \mid D_{b(\gamma)} = \emptyset, \underline\fU) \nonumber\\&\leq (1+\epsilon_\beta)Z^n_{Q, W; F, 1}(A, B \mid D_{b(\gamma)} = \emptyset, \fU)\,,
        \end{align}
        where the first inequality uses the previous display, and the second inequality is because $\fU \subset \underline\fU$.
		
		Now let $C = (-\tfrac12, C_2)$ be the left endpoint of $b(\gamma)$. Consider the disjoint union $\{D_{b(\gamma)} = \emptyset\} \cap \underline\fU = \underline\fU_1 \cup \underline\fU_2$, where $\underline\fU_1, \underline\fU_2$ additionally intersect with the events that $C_2$ is smaller or bigger than $\widehat Y = Y + \frac12 \sigma (\log L)^a$, respectively.
		
		We first bound $Z^n_{Q, W; F, 1}(A, B \mid \underline\fU_1)$. Let $\bE_\mu, \bP_\mu$ be the expectation and probability defined for the polymer measure with weights $q^n_{W; F, \mu}$ and partition function $Z^n_{Q, W; F, \mu}(A, B)$. Then, we have
		\begin{align*}
			Z^n_{Q, W; F, 1}(A, B \mid \underline\fU_1) &= \hatZ^n_{Q, W}(A, B)\bE_0[\one_{\underline\fU_1}\exp(\tfrac{1}{N_n}\cA_F(\Gamma))] \\&\leq Ce^{-\tau_\beta(\theta) \ell + O((\log L)^{2a})}\sqrt{Z^n_{Q, W; F, 2}(A, B)/\hatZ^n_{Q, W}(A, B)}\sqrt{\bP_0(\underline\fU_1)}\,,
		\end{align*}
		where the equality is by definition, and then we use Cauchy-Schwarz and \cref{cor:growth-bound-partition-functions} at $\mu = 0$ to get the inequality. We can also use \cref{cor:growth-bound-partition-functions} to upper bound $Z^n_{Q, W; F, 2}(A, B)$ and \cref{lem:strip-to-Z2-interactions} to lower bound $\hatZ^n_{Q, W}(A, B)$, obtaining
		\begin{equation*}
			Z^n_{Q, W; F, 2}(A, B)/\hatZ^n_{Q, W}(A, B) \leq e^{C(\log L)^{3a} + C(\log L)^2}\,.
		\end{equation*}
		Moreover, the event $\underline\fU_1$ implies that $\gamma$ hits the line $x = 0$ once below $\widehat Y$ and once above $Y$, which means there are two points on the vertical line distanced $\frac12\sigma(\log L)^a \geq C(\beta, \theta)N_n^{1/3}(\log L)^{3a/2}$ apart. In the $\hatZ^n_{\Z^2, \Z^2}$ measure, this has probability $e^{-cN_n^{1/3}}$ by \cref{it:DKS-large-deviations} of \cref{prop:DKS-propositions}. By \cref{lem:rare-events-changing-Z} with \cref{lem:cigar-likely-in-Z2,lem:strip-to-Z2-interactions} as the inputs, we have the same for our $\bP_0$ measure:
		\begin{equation*}
			\bP_0(\underline\fU_1) \leq e^{-cN_n^{1/3}}\,.
		\end{equation*}
		Altogether, we obtain the upper bound
		\begin{equation}\label{eq:Z1U1-bound}
			Z^n_{Q, W; F, 1}(A, B \mid \underline\fU_1) \leq Ce^{-\tau_\beta(\theta)\ell}e^{-cN_n^{1/3}}\,.
		\end{equation}
		
		Next we bound $Z^n_{Q, W; F, 1}(A, B \mid \underline\fU_2)$. We will sum over possible $C_2 \geq \widehat Y$. Let $\gamma = \gamma_1 \circ \gamma_2$ be the decomposition of $\gamma$ into components before and after $C$. Note that this decomposition is now well-defined because we are on the event $D_{b(\gamma)} = \emptyset$. Note also that even though $C$ is not necessarily a cut-point, the fact that $D_{b(\gamma)} = \emptyset$ still implies that the energy naturally decomposes as $\sE^*_\beta(\gamma) = \sE^*_\beta(\gamma_1) + \sE^*_\beta(\gamma_2)$. 
		As in the proof of \cref{prop:bound-partition-functions}, we will need the correction term
		\begin{equation*}
			\Delta\fI_W(\gamma_1, \gamma_2) = \fI_W(\gamma_1) + \fI_W(\gamma_2) - \fI_W(\gamma_1 \circ \gamma_2)\,.
		\end{equation*}
		We can also write $\cA_F(\gamma) = \cA_F^0 + \cA_F(\gamma_1) + \cA_F(\gamma_2)$, where $\cA_F(\gamma_1)$ is the signed area above $\gamma_1$ with respect to the line segment $\overline{AC}$, $\cA_F(\gamma_2)$ is defined analogously with respect to $\overline{CB}$, and $\cA_F^0$ is the signed area of the triangle $ACB$ intersected with $F$ (signed as positive when $C$ is below $\overline{AB}$ and negative when $C$ is above it). We obtain
		\begin{align*}
			Z^n_{Q, W; F, 1}(A, B \mid \underline\fU_2) &\leq \sum_{y \geq \widehat Y}\sum_{\gamma\,:\; C_2 = y}e^{\frac{\cA_F^0}{N_n}}e^{-\sE^*_\beta(\gamma_1)+\fI_W(\gamma_1)+\frac{\cA_F(\gamma_1)}{N_n}}e^{-\sE^*_\beta(\gamma_2)+\fI_W(\gamma_2)-\Delta\fI_W(\gamma_1, \gamma_2)+\frac{\cA_F(\gamma_2)}{N_n}}\\
			&=: \sum_{y \leq \widehat Y}e^{\frac{1}{N_n}\cA_F^0}\sum_{\gamma_1: C_2 = y}e^{-\sE^*_\beta(\gamma_1)+\fI_W(\gamma_1)+\frac1{N_n}\cA_F(\gamma_1)}Z_{\gamma_1, y}\,.
		\end{align*}
		But now we can apply \cref{cor:growth-bound-partition-functions} twice to bound the above by 
		\begin{equation*}
			e^{C(\log L)^{2a}}\sum_{y \leq \widehat Y}\exp\Big(\frac1{N_n}\cA_F^0+ \sG^n_1(\ell_{AC}, \theta_{AC}) + \sG^n_1(\ell_{CB}, \theta_{CB})\Big)\,.
		\end{equation*}
		We can further decompose this sum into $\Sigma_1 + \Sigma_2$, where $\Sigma_1$ sums over values of $y \geq N_n^{1/3}(\log L)^{3a}$ and $\Sigma_2$ sums over the remaining values of $N_n^{1/3}(\log L)^{3a} \geq y \geq \widehat Y$.
		The case of $\Sigma_1$ is easy to bound since we have a large negative signed area from $\cA_F^0$. Indeed, by definition we have 
		\begin{equation}\label{eq:bound-G-weak}
			\sG^n_\mu(\ell, \theta) = -\tau_\beta(\theta)\ell + O((\log L)^{3a})\,.
		\end{equation} 
		However, the area term is $\frac{1}{N_n}\cA_F^0 \leq -c(\log L)^{4a}$ for some $c > 0$, using here that $b \geq 3a$ so that the triangle $ACB$ is contained in $F$ when $y = N_n^{1/3}(\log L)^{3a}$. Hence, by the convexity of $\tau$ and \cref{eq:bound-G-weak}, we have
		\begin{equation}\label{eq:Sigma1-bound}
			\Sigma_1 \leq \exp(-\tau_\beta(\theta)\ell - c(\log L)^{4a}) = \exp(\sG^n_1(\ell, \theta) - c(\log L)^{4a})\,.
		\end{equation}
		
		To bound $\Sigma_2$, we wish to do a Taylor expansion on the angles in all the surface tension terms. Through such a computation, combined with the fact that $\cA_F^0 = \frac12N_n^{2/3}(\log L)^ay$ (since in this regime of $y$, the triangle $ACB$ is contained in $F$) and $\ell\cos(\theta) = N_n^{2/3}(\log L)^a$, we eventually obtain
		\begin{align}\label{eq:taylor-expansion-theta}
			\Sigma_2 \leq (1+o(1))e^{\sG^n_1(\ell, \theta)}\sum_{y \in [\widehat Y, N_n^{1/3}(\log L)^{3a}]}e^{-\frac{(y - Y)^2}{2\sigma^2}}\,,
		\end{align}
		where we recall the definitions of $Y$ and $\sigma$ in \cref{eq:def-Y-sigma}. (See the computations leading up to \cite[Eq. 5.11]{CLMST16}, where the same computation was done, for details. A further elaboration of the steps there can be found in, e.g., \cite[Sec.~3.1]{GreenbergIoffe2005}.) We can now interpret the sum as the probability that a Gaussian with mean $Y$ and variance $\sigma^2$ lies in the interval $[ \widehat Y, N_n^{1/3}(\log L)^{3a}]$. Recall we have defined $\widehat Y = Y + \frac12\sigma(\log L)^a$, so that this probability is $\exp(-c(\log L)^{2a})$ for some $c > 0$, resulting in the upper bound
		\begin{equation}\label{eq:Sigma2-bound}
			\Sigma_2 \leq \exp(\sG^n_1(\ell,\theta) - c(\log L)^{2a})\,.
		\end{equation}
		Combining \cref{eq:growth-numerator,eq:Db-empty,eq:Z1U1-bound,eq:Sigma1-bound,eq:Sigma2-bound}, we obtain as our final upper bound for the numerator of  \cref{eq:growth-initial-fraction}
		\begin{equation}\label{eq:final-UB-numerator}
			Z^n_{Q, Q; F, 1}(A, B \mid \fU, \cG_2)  \leq \exp(\sG^n_1(\ell, \theta) - c'(\log L)^{2a})\,.
		\end{equation}
		
		We now turn to getting a lower bound for the denominator, which we recall is \begin{equation*}
			Z^n_{Q, Q; F, 1}(A, B \mid \cG_1, \cG_2, \cG^\sqcup)\,.
		\end{equation*}
        First note that by \cref{eq:good-event-CE-G',eq:G2-likely-given-G'}, we have
        \begin{equation}\label{eq:LB-exclude-G1-G2}
            Z^n_{Q, Q; F, 1}(A, B \mid \cG_1, \cG_2, \cG^\sqcup) \geq (1-o(1))Z^n_{Q, Q; F, 1}(A, B \mid \cG^\sqcup)\,.
        \end{equation}
		The lower bound in \cite{CLMST16} consisted mostly of using the proof of \cite[Lemma A.6]{MartinelliToninelli10}. We summarize these ideas together here and show the crucial modification needed in our setting (though for details that are unchanged, we will refer the reader to \cite{MartinelliToninelli10}). Begin by defining the optimal curve
		\begin{equation}
			\gamma_\mathsf{opt}(x) = \overline{AB}(x) + \frac{(x-A_1)(B_1 - x)}{2 N_n(\tau_\beta(\theta) + \tau_\beta''(\theta))\cos^3(\theta)}\,.
		\end{equation}
		Then, we can take a linear approximation $\sL$ of $\gamma_\mathsf{opt}$ consisting of $\log L$ many line segments. This linear approximation is such that the width of the line segments start at $O(N_n^{2/3}(\log L)^a)$ for the segment containing the middle of $\gamma_\mathsf{opt}$, and decreases by a factor of two with each segment going left/right from there (as was done, e.g., in \cite[Lemma~A.6]{MartinelliToninelli10}\footnote{The proof of \cite[Lemma~A.6]{MartinelliToninelli10} refers to \cite[Thm.~4.16]{DKS92}. The upper bound stated in that theorem was later found to be erroneous, but only the lower bound is required for that proof.}; we reproduce that argument here for completeness). Now let $T$ be any line segment. Denote its endpoints by $A_T, B_T$, its length by $\ell_T$, its angle by $\theta_T$, and its horizontal length by $M_T$. Recall the definition of the cigar shape $\sC(T)$ from \cref{def:cigar-shape}.  
		For brevity, define the union of the cigars as 
		\[\sC = \bigcup_{T \in \sL} \sC(T)\,.\]
		Note that the cigars are constructed such that all $\gamma$ contained in $\sC$ are automatically in $\cG^\sqcup$. We can then lower bound \cref{eq:LB-exclude-G1-G2} by restricting to polymers which lie inside the union of these cigars,
		\begin{equation}\label{eq:LB-restrict-to-cigars}
			\sum_{\gamma \in \cP_Q(A, B) \cap \cG^\sqcup}\exp(-\sE^*_\beta(\gamma) + \fI_Q(\gamma) + \frac{1}{N_n}\cA_F(\gamma)) \geq \sum_{\gamma \in \cP_Q(A, B)\cap \sC}\exp(-\sE^*_\beta(\gamma) + \fI_Q(\gamma) + \frac{1}{N_n}\cA_F(\gamma))\,.
		\end{equation}
		Now forget about the main area term $\cA_F(\gamma)$ for now. The shapes of the cigars imply that we can ignore interactions between the cigars at a cost of $e^{C(\log L)^2}$. Having done so, we can approximate by a product of polymer partition functions:
		\begin{align}\label{eq:cigar-into-products}
			\sum_{\gamma \in \cP_Q(A, B)\cap \sC}\exp(-\sE^*_\beta(\gamma) + \fI_Q(\gamma)) &\geq e^{-C(\log L)^2}\prod_{T \in \sL} \sum_{\gamma_T \in \cP_{\sC(T)}(A_T, B_T)} \!\!\exp(-\sE^*_\beta(\gamma_T) + \fI_Q(\gamma_T))\\
			&= e^{-C(\log L)^2}\prod_{T \in \sL}\hatZ^n_{\sC(T), Q}(A_T, B_T)\,.\nonumber
		\end{align}
		Now with the exception of the first and last few cigars close to the sides of $Q$, all of $\sC(T)$ will be distance $\geq (\log L)^2$ from $\partial Q$ so that
        \[|\log \hatZ^n_{\sC(T), Q}(A_T, B_T) - \log \hatZ^n_{\sC(T), \Z^2}| \leq o(1)\,.\]
        For the cigars within distance $O((\log L)^2)$ from the sides of $Q$, we can still switch $Q$ interactions to $\Z^2$ interactions at a total cost of another $e^{C(\log L)^2}$. By \cref{prop:DKS-propositions} combined with \cref{lem:cigar-likely-in-Z2,prop:compare-tau}, we have
		\begin{equation*}
			\log \hatZ^n_{\sC(T), \Z^2}(A_T, B_T) \geq -\tau_\beta(\theta_T)\ell_T - \tfrac12\log M_T - C\,.
		\end{equation*}
		This implies that \begin{equation}\label{eq:cigar-product-LB}
			\prod_{T \in \sL}\hatZ^n_{\sC(T), Q}(A_T, B_T) \geq \exp(-\sum_{T \in \sL} \tau_\beta(\theta_T)\ell_T-\sum_{T \in \sL}\tfrac12\log M_T - C(\log L))\,.
		\end{equation}
		The sum of the $\log M_T$ terms will be at most $O((\log L)^2)$, as there are $\log L$ many terms contributing at most $O(\log L)$. Finally, as in \cite{CLMST16} we can exchange the sum $\sum_{T \in \sL}\tau_\beta(\theta_T)\ell_T$ by an integral at a cost of $O(\log L)$ (it will be a constant cost for each $T$), resulting in the bound
		\begin{equation}\label{eq:cigar-integral-LB}
			\exp\bigg(-\sum_{T \in \sL} \tau_\beta(\theta_T)\ell_T-\sum_{T \in \sL}C\log M_T\bigg) \geq \exp\bigg(-\int_{\gamma_\mathsf{opt}} \tau_\beta(\theta_s)ds - O((\log L)^2)\bigg)\,.
		\end{equation}
		To summarize, by combining \cref{eq:cigar-into-products,eq:cigar-product-LB,eq:cigar-integral-LB}, we get
		\begin{equation*}
			\sum_{\gamma \in \cP_Q(A, B)\cap \sC}\exp(-\sE^*_\beta(\gamma) + \fI_Q(\gamma)) \geq \exp\Big(-\int_{\gamma_\mathsf{opt}} \tau_\beta(\theta_s)ds - O((\log L)^2)\Big)\,.
		\end{equation*}
		Now to handle the area term $\frac1{N_n}\cA_F(\gamma)$ which we removed from the above analysis, note that for every $\gamma \subset \sC$, we have 
		\begin{equation*}\cA_F(\gamma) = \cA_F(\gamma_\mathsf{opt}) + O(N_n^{2/3}(\log L)^a(N_n^{2/3}(\log L)^a)^{1/2}(\log L)^2) = \cA_F(\gamma_\mathsf{opt}) + O(N_n(\log L)^{3a/2 + 2})\,.\end{equation*}
		Thus, we get
		\begin{equation}\label{eq:LB-cigars-by-integral}
			\sum_{\gamma \in \cP_Q(A, B)\cap \sC}\!\!\!\exp\big(-\sE^*_\beta(\gamma) +  \fI_Q(\gamma)+ \frac{\cA_F(\gamma)}{N_n}\big) \geq e^{(\log L)^{3a/2+2}}e^{\frac{\cA_F(\gamma_\mathsf{opt})}{N_n}}\exp\big(-\int_{\gamma_\mathsf{opt}} \tau_\beta(\theta_s)ds - O((\log L)^2)\big)\,.
		\end{equation}
		We now use the fact that $\gamma_\mathsf{opt}$ deviates at most $O(N_n^{1/3}(\log L)^{2a})$ from $\overline{AB}$, so since $b > 2a$ it never leaves the region $F$. Thus, it is the same as if $\cA_F(\gamma_\mathsf{opt})$ were defined with no area restriction to $F$. Then, as in \cite[Eq~(5.17)]{CLMST16}, a direct computation gives
		\begin{equation*}
			-\int_{\gamma_\mathsf{opt}} \tau_\beta(\theta_s)ds + \frac1{N_n}\cA_F(\gamma_\mathsf{opt}) = \sG^n_1(\ell, \theta) + o(1)\,.
		\end{equation*}
		We can now plug this into \cref{eq:LB-cigars-by-integral} to obtain
		\begin{equation}\label{eq:final-LB-denominator}
			\sum_{\gamma \in \cP_Q(A, B) \cap \sC}\exp(-\sE^*_\beta(\gamma) +  \fI_Q(\gamma)+ \frac{1}{N_n}\cA_F(\gamma)) \geq \exp\big(\sG^n_1(\ell,\theta) + O((\log L)^{3a/2+2} + (\log L)^2)\big)\,,
		\end{equation}
        which together with \cref{eq:LB-exclude-G1-G2,eq:LB-restrict-to-cigars} provides the final lower bound on the denominator.
        
		By combining the lower bound on the denominator in \cref{eq:final-LB-denominator} with the upper bound on the numerator in \cref{eq:final-UB-numerator}, we obtain that for the choice of $a > 4$ (so that the $(\log L)^{2a}$ term in \cref{eq:final-UB-numerator} dominates), the fraction in \cref{eq:growth-initial-fraction} is upper bounded by
		\begin{equation}\label{eq:A-final-bound}
			\frac{Z^n_{Q, Q; F, 1}(A, B \mid \fU, \cG_1, \cG_2, \cG^\sqcup)}{Z^n_{Q, Q; F, 1}(A, B\mid \cG_1, \cG_2, \cG^\sqcup)} \leq \exp\big(-c(\log L)^{2a}\big)
		\end{equation}
		for some $c > 0$. 
	\end{proof}
	
	\subsection{Level lines contain translated Wulff shapes}\label{subsec:level-line-contain-wulff}
    In this subsection, we will use \cref{thm:growth-gadget} to prove \cref{thm:level-line-contains-Wulff}. The strategy will be to begin with a small Wulff shape, and use \cref{thm:growth-gadget} to ``grow'' it to a large Wulff shape (one argues that if a small Wulff shape is present, then w.h.p.\ a larger one is also present). The smallest starting Wulff shape we need has diameter $\leq \tfrac12$ (see \cref{rem:size-of-ell0}). In order to jump-start this growth process, we use the following consequence of the analysis of \cite{LMS16}.
\begin{claim}\label{clm:starting-shape}
For every $0<\epsilon<\frac1{10}$ there exists $\beta_0$ such that, for every $\beta>\beta_0$, if
$V$ is a domain such that $\Lambda_{(1-\epsilon)L} \subset V \subset \Lambda$, then 
w.h.p., $\phi\sim\pi_V^0$ satisfies the event $\cE_1(\Lambda_{(1-2\epsilon)L})$. \end{claim}
Indeed, in \cite[Sec~4.4]{LMS16} the authors argue (as part of the event denoted there $\cR_{H-1}$) that  $\phi\sim\pi_{\Lambda}^0$ contains  a $*$-adjacent circuit $C$ of sites at height at least $H-1$ encapsulating $\Lambda_{L- (H-1)\ell}$ where $\ell \leq L\exp[-c_0  (\beta \log L/\log\log L)^{1/2}]$. This is obtained inductively, establishing such a circuit $C_j$ for heights $j=0,\ldots,H-1$, each time applying \cite[Prop~4.5]{LMS16} (the main growth tool in that work) to every box $\Lambda_\ell$ at distance at least $\log^2\ell$ from the previously found $C_{j-1}$; that tool shows the existence of a circuit at height at least $j$ and distance at most $\epsilon \ell$ from  $\partial\Lambda_\ell$, and ``stitching'' these circuits together gives $C_j$. For the final $\fL_1$, one chooses $\ell'$ that may be of order $L$ (still $\ell'\geq L^{1-o(1)}$ yet it can be as large as $ \frac45 L$, hence the upper bound on $\epsilon$ in our assumptions), and proceeds as before, yielding a  circuit $C_H$ encapsulating $\Lambda_{(1-\epsilon)L}$. As this argument is based on applications of the growth tool on small boxes $\Lambda_\ell$ at distance at least $\log^2\ell$ from $C_{j-1}$ (and in the first iteration~$\partial V$), all that is needed is to have $V\supset \Lambda_{(1-\epsilon)L}$. Thus, in our setting we obtain~$C_j$, for each $j\geq H-1$, encapsulating $\Lambda_{(1-\epsilon)L-j\ell}$, and finally $C_H$ that encapsulates $\Lambda_{(1-2\epsilon)L}$.
    
    We will also need the following deterministic fact about the Wulff shape (see e.g., \cite[Lemma 3.9]{CLMST16} for the proof).
	\begin{lemma}\label{lem:Wulff-shape-dist}
		Fix $\theta \in [0, \pi/4]$ and any $\tau$ satisfying the properties of \cref{prop:surface-tension-properties}. Let $\overline{AB}$ be a line segment of length $d$ and angle $\theta$ such that $A, B$ are on $\partial \cW_1(\tau)$. Recall that $\mathsf{w}_1(\tau)$ is the value of the Wulff functional on $\partial \cW_1(\tau)$. Let $\Delta(d, \theta)$ be the vertical distance between the midpoint $X = \frac{A + B}{2}$ and $\partial \cW_1$. Then we have that for $d \ll 1$, 
		\[\Delta(d, \theta) = \frac{\mathsf{w}_1(\tau) d^2(1+O(d^2))}{16(\tau(\theta)+\tau''(\theta))\cos(\theta)}\,.\]
	\end{lemma}
	
	The next proposition is the key growth result that shows we can go from the Wulff shape to its translates, to be proven over \cref{lem:grow-droplet,lem:Wulff-Dk-convergence,lem:grow-translates}. (See \cref{fig:growth-gadget} for a visualization of the proof.) Fix $1 \leq m \leq \log L$. Recall the definitions of $\ell_n^*$ and $\kappa_{n, b}$ stated above \cref{thm:level-line-contains-Wulff}. For brevity, define $\hat\ell_n := (1+(\log L)^{-5/3})\ell_n^*$.
	
    \begin{figure}
\centering
    \begin{tikzpicture}
    \begin{scope}[scale=.75]
    \draw [color=black,thick,fill=gray!10] (0,0)--(5,0)--(5,5)--(-0,5)--cycle;
    \draw [thick, rounded corners=13,fill=blue!15] (2.5,3.3) rectangle ++(1.35,1.35);
    \node[circle,scale=0.4,fill=gray] at (3.17,3.96) {};
    \end{scope}

    \begin{scope}[scale=.75,shift={(6,0)}]
    \draw [color=black,thick,fill=gray!10] (0,0)--(5,0)--(5,5)--(-0,5)--cycle;
    \draw [thick, rounded corners=18,fill=blue!15] (2.2,3) rectangle ++(1.95,1.95);
    \draw [thick,dashed, rounded corners=13,fill=blue!10] (2.5,3.3) rectangle ++(1.35,1.35);
    \node[circle,scale=0.4,fill=gray] (o) at (3.17,3.96) {};
    \draw[thick,-stealth] ($(o)+(0.5,0.5)$) -- ($(o)+(.75,.75)$);
    \draw[thick,-stealth] ($(o)+(0.5,-0.5)$) -- ($(o)+(.75,-.75)$);
    \draw[thick,-stealth] ($(o)+(-0.5,0.5)$) -- ($(o)+(-.75,.75)$);
    \draw[thick,-stealth] ($(o)+(-0.5,-0.5)$) -- ($(o)+(-.75,-.75)$);
    \end{scope}

    \begin{scope}[scale=.75,shift={(12,0)}]
    \draw [color=black,thick,fill=gray!10] (0,0)--(5,0)--(5,5)--(-0,5)--cycle;
    \draw [thick, rounded corners=20,fill=blue!15] (0,0) rectangle ++(5,5);
    \end{scope}
\end{tikzpicture}
\caption{The growth procedure used to prove \cref{prop:wulff-to-cL}. Left: Fix $x$ such that the Wulff shape $\cW(x, \hat\ell_n)$ (colored blue) is contained in $L\hat\ell_n\cW_1(\tau_\beta)$. We start with the event $\cE_n(\cW(x, \hat\ell_n))$ that $\fL_n$ encapsulates $\cW(x, \hat\ell_n)$. Middle: By \cref{lem:grow-droplet}, we can grow this Wulff shape until it reaches a $o(L)$ distance from the boundary. Right: Repeat with all possible starting $x$ from the left picture to obtain $\cD$. We can then start again from the left, now allowing $x$ such that $\cW(x, \hat\ell_n)$ is contained in $\cD$, and iterate this process.}
\label{fig:growth-gadget}
\end{figure}
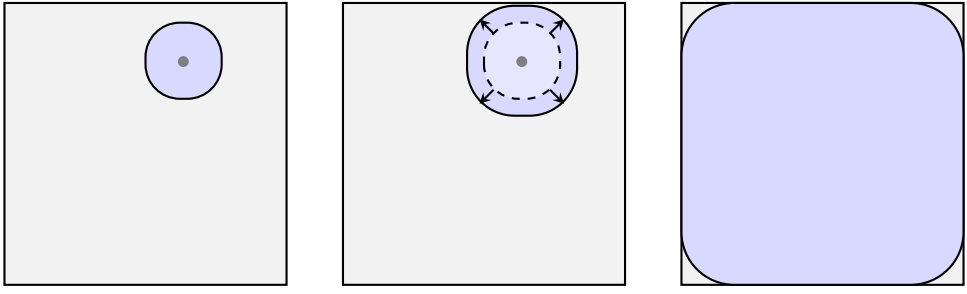

\begin{proposition}\label{prop:wulff-to-cL}
        Fix $n \geq 1$ and $1 \leq m \leq \log L$. Let $\Lambda'$ be any region containing $L\sL(\hat\ell_n, -(m-1)\kappa_{n,15})$. Under $\pi^{H-n}_{\Lambda'}$, if we know that w.h.p.\ $\cE_n(L\hat\ell_n\cW_1(\tau_\beta))$ holds, then w.h.p.\ we also have $\cE_n(L\sL(\hat\ell_n, -m\kappa_{n,15}))$.
	\end{proposition}
	For shorthand, define $\fR := L\sL(\hat\ell_n, -(m-1)\kappa_{n,15})$. Let $\cW(x, \ell)$ denote the scaled Wulff shape $L\ell\cW_1(\tau_\beta)$ centered at $x$. Let $\ell_x$ be the largest value before $\cW(x, \ell)$ gets within distance $N_n^{1/3}(\log L)^{15}$ from $\fR^c$. 
	\begin{lemma}\label{lem:grow-droplet}
		Fix $x \in \fR$ and $\ell$ such that $\hat\ell_n \leq \ell < \ell_x$. Under $\pi^{H-n}_{\Lambda'}$, given that w.h.p.\ $\cE_n(\cW(x, \ell))$ holds, then w.h.p.\ we have $\cE_n(\cW(x, \ell_x))$.
	\end{lemma}
	\begin{proof}
		We show that if $\cE_n(\cW(x, \ell))$ holds, then w.h.p.\ so does $\cE_n(\cW(x, \ell(1+L^{-3/4}))$. Pick an angle $\theta \in [0, \pi/4]$. Let $\overline{AB}$ be the line segment such that its horizontal distance is $N_n^{2/3}(\log L)^5$, $\overline{AB}$ has angle $\theta$, and $A, B$ lie on the bottom right quarter of $\cW(x, \ell)$. Assume that $B$ is to the right of $A$. Let $R$ be the rectangle with width $N_n^{2/3}(\log L)^5$ and height $2N_n^{2/3}(\log L)^5$ such that $A, B$ are on the sides of $R$, and the distance from $A$ to the bottom of $R$ is $N_n^{1/3}(\log L)^{15}$. Let $Z = (Z_1, Z_2)$ be the midpoint of $A$ and $B$. By \cref{lem:boundary-construction} and the fact that \cref{thm:growth-gadget} holds for any $Q$ satisfying the conditions in \cref{lem:boundary-construction}, we have that w.h.p.\ the interior of $\fL_n$ contains the point $W = (Z_1, W_2)$ where 
		\begin{equation*}
			W_2 = Z_2 -\frac{N_n^{1/3}(\log L)^{10}}{8(\tau(\theta) + \tau''(\theta))\cos(\theta)^3} + c(\beta, \theta)N_n^{1/3}(\log L)^{15/2}\,.
		\end{equation*}
		By rescaling the result of \cref{lem:Wulff-shape-dist}, we obtain that $Z$ lies a vertical distance of $\Delta'$ above $\partial \cW(x, \ell)$, where
		\begin{equation*}
			\Delta' = \ell L \Delta\bigg(\frac{N_n^{2/3}(\log L)^5}{\ell L\cos\theta}, \theta\bigg) = \frac{\mathsf{w}_1(\tau_\beta) N_n^{4/3}(\log L)^{10}}{16(\tau(\theta) + \tau''(\theta))\ell L\cos(\theta)^3} + o(1)\,.
		\end{equation*}
		Hence, as long as, say, \begin{equation}\label{eq:min-start-Wulff}1 - \frac{\mathsf{w}_1(\tau_\beta) N_n}{2\ell L} \geq (\log L)^{-5/2}c'(\beta, \theta)\,
		\end{equation}
		for some other constant $c'(\beta, \theta)$, then $W$ is at a height at least $O(N_n^{1/3}(\log L)^{10-5/2})$ below $\partial \cW(x, \ell)$, whereas the enlarged Wulff shape $\cW(x, \ell(1 + L^{-3/4}))$ is only distance $O(L^{1/4})$ below $\partial \cW(x, \ell)$. It is easy to check that \cref{eq:min-start-Wulff} is satisfied by the assumption $\ell \geq  \hat\ell_n$. By repeating this $O(L)$ times to cover all angles $\theta \in [0, \pi/4]$ and (using symmetry to conclude for other angles), we obtain $\cE_n(\cW(x, \ell(1+L^{-3/4}))$. We can then repeat this at most $o(L^{3/4}\log L)$ times to obtain $\cE_n(\cW(x, \ell_x))$.
	\end{proof}
	Now we can define (deterministic) sets $\cD_k$ as part of the enlargement/translation procedure. Fix $\ell$ such that $\hat\ell_n \leq \ell < \ell_x$. We start with $\cD_0$ defined as the largest rescaling of $\cW_1(\tau_\beta)$ that is contained in the unit square $S$. Now given $\cD_k(\ell)$, define $\cD_{k+1}(\ell)$ as follows. For any $\zeta \in \cD_k(\ell)$, define 
	\[t_\zeta = \max\{t \geq 1: t\ell\cW_1(\tau_\beta) + \zeta \subset S\}\,,\]
	with $t_\zeta = 0$ if there is no such $t$. Then, define
	\[\cD_{k+1}(\ell) = \bigcup_{\zeta \in \cD_k}\{t_\zeta\ell\cW_1(\tau_\beta) + \zeta\}\,.\]
	Note that by construction, $\cD_k(\ell) \subset \cD_{k+1}(\ell) \subset S$ for each $k \geq 0$, and there is no dependence of $\cD_0$ on $\ell$. The following lemma showing convergence of the $\cD_k(\ell)$ to $\sL(\ell, 0)$ was proved in \cite{CLMST16}.
	
	\begin{lemma}[{\cite[Lemma 6.8]{CLMST16}}]\label{lem:Wulff-Dk-convergence}
    The sequence $\cD_k(\ell)$ converges to $\cD_\infty:=\sL(\ell, 0)$, and the Hausdorff distance between $\cD_k$ and $\cD_\infty$ is bounded by $c^k$ for some constant $c \in (0, 1)$.
	\end{lemma}

	We can now apply \cref{lem:grow-droplet} in the iterative procedure of the $\cD_k$ sets.
	\begin{lemma}\label{lem:grow-translates}
		Let $x_0$ be the center of $\fR$. Under $\pi^{H-n}_{\Lambda'}$, given that w.h.p.\ $\cE_n(\cW(x_0, \ell_{x_0}))$ holds, then w.h.p.\ we have $\cE_n(L\sL(\hat\ell_n, -m\kappa_{n,15}))$.
	\end{lemma}
	\begin{proof}
		By definition, we have $L(1 - (m-1)\kappa_{n,15})\cD_0 \subset \fR$. Hence, by adding another factor of $\kappa_{n,15}$, we get $L(1 - m\kappa_{n,15})\cD_0 \subset \cW(x_0, \ell_{x_0})$. We show that with high probability, 
		\begin{equation}\label{eq:Dk-to-Dk+1}\cE_n(L(1 - m\kappa_{n,15})\cD_k(\hat\ell_n)) \implies \cE_n(L(1 - m\kappa_{n,15})\cD_{k+1}(\hat\ell_n)))\,.
		\end{equation}
		Indeed, for any $\cW(x, \hat\ell_n) \subset L(1 - m\kappa_{n,15})\cD_k(\hat\ell_n)$, we can apply \cref{lem:grow-droplet} to conclude that
		\[\cE_n(\cW(x, \hat\ell_n)) \implies \cE_n(\cW(x, \ell_x))\,.\]
		Suppose that we could grow $\cW(x, \ell)$ to $\cW(x, \widetilde \ell_x)$, defined as the largest scaling of $\cW(x, \ell)$ contained in $\Lambda_{L(1 - m\kappa_{n,15})}$ as long as $\widetilde \ell_x \geq \ell$. Then we could apply the above display to all $x$ such that $\cW(x, \hat\ell_n) \subset L(1-m\kappa_{n,15})\cD_k(\hat\ell_n)$ to obtain \cref{eq:Dk-to-Dk+1}. However, it is a-priori possible that $\widetilde \ell_x > \ell_x$. That is, the additional restriction that we need to consider is that by definition of $\ell_x$, we can only grow each Wulff shape until we reach distance $N_n^{1/3}(\log L)^{15}$ from $\fR^c$. But, since $\cD_{k+1} \subset \cD_\infty$, the ending shape $L(1-m\kappa_{n,15})\cD_{k+1}(\hat\ell_n)$ is already at a distance $N_n^{1/3}(\log L)^{15}$ away from $\fR^c$, so this additional restriction does not change anything. Note also that we only need to repeat this up to $k = O(\log L)$ times to get $\cD_k$ within distance $1/L$ of $\cD_\infty$ by the bound given in \cref{lem:Wulff-Dk-convergence}, which is all that is needed due to the discretization of the lattice. As noted in \cref{rem:choice-of-a-b}, we can afford this many iterations.
	\end{proof}
	Combining \cref{lem:grow-droplet,lem:grow-translates} concludes the proof of \cref{prop:wulff-to-cL}. Now we prove the main theorem of this section.
	\begin{proof}[Proof of \cref{thm:level-line-contains-Wulff}]
		First, additionally assume that we have 
		\begin{equation*}
			\cE_{n+1}(L\sL(\hat\ell_n, -(H-n)\kappa_{n,15}))\,.
		\end{equation*} 
		Reveal $\fL_{n+1}$. This reveals an interior region that contains $L\sL(\hat\ell_n, -(H-n)\kappa_{n,15})$ with a boundary of sites with height $\geq H-n$. By monotonicity, we can drop this boundary condition down to exactly $H-n$. Then, we can apply \cref{prop:wulff-to-cL} to conclude that with high probability we have $\cE_n(L\sL(\hat\ell_n, -(H+1-n)\kappa_{n,15}))$.
		
		It thus suffices to justify the above assumption. We again use monotonicity, this time dropping the floor to $-(H-n)$. This allows us to control the 1 level line with a floor at 0 by the 1 level line with a floor at height $-(H-n)$, which is then the same as the $(H+1-n)$ level line when the floor is at 0 and boundary conditions are at height $H-n$. By \cref{prop:wulff-to-cL} (whose application is justified by \cref{clm:starting-shape}) with $m = 1$, this level line contains $L\sL(\hat\ell_n, -\kappa_{n,15})$. Thus in $\pi^0_{\Lambda}$ we can reveal the 1 level line, and by monotonicity drop the induced $\geq 1$ boundary conditions to be equal to 1. Then we can repeat, studying the $k$ level line for $2 \leq k \leq H-n$ by dropping the floor to $-(H+1-n-k)$, and concluding that it contains $L\sL(\hat\ell_n, -k\kappa_{n,15})$. 
	\end{proof}
	
	\section{Upper bound}\label{sec:UB}
Our goal in this section is the following bound on the distance of $\fL_n$ from the bottom side of $\Lambda$. 
\begin{theorem}\label{thm:ub}
In the setting of \cref{thm:1}, fix $n\geq 1$ and $K>0$, let $\rho_n(x)$ be the maximum vertical distance of~$\fL_n$ above $x+(\frac{L}2,0)$ for $-N_n^{2/3} \leq x\leq N_n^{2/3}$,
and set $\sigma^2>0$ as per \cref{def:sigma}. 
Then every weak limit point $\mathbf{Y}_n(t)$ of the process $Y_n(t):=N_n^{-1/3}\rho_n(t N_n^{2/3})$ ($t\in[-K,K]$), as $L\to\infty$, satisfies
\[ \mathbf{Y}_n \preceq \mathsf{FS}_{\sigma}\,. \]
Moreover, for every fixed $m$, every weak limit point $(\bY_n(t))_{n\leq m}$ of the processes $(Y_n(t))_{n\leq m}$ satisfies
\[ (\bY_n)_{n\leq m} \preceq \bigotimes_{n\leq m}\mathsf{FS}_{\sigma}\,.\]
\end{theorem}

\begin{figure}
\centering
    \begin{tikzpicture}
    \pgfmathsetseed{271818}
    \begin{scope}[scale=1.25]
    \filldraw [thick, blue, fill=green!20] (0.1,1.5) to[move to] (5.9,1.5) decorate[decoration={random steps,segment length=1.5mm}]{-- (6,0)} decorate[decoration={random steps,segment length=1.5mm}]{-- (0,0)} decorate[decoration={random steps,segment length=1.5mm}]{-- (0.1,1.5)};   
    \filldraw [thick, blue, fill=gray!10] (0.1,1.5) to[move to] (5.9,1.5) decorate[decoration={random steps,segment length=1.5mm}]{-- (5.9,4)} decorate[decoration={random steps,segment length=1.5mm}]{-- (0,4)} decorate[decoration={random steps,segment length=1.5mm}]{-- (0.1,1.5)};  
    \draw[thick, dashed, color=black] (0.1,1.5)--(5.9,1.5);
    
    \draw[color=gray!50!black,|-|](-0.25,0)--(-0.25,4);
    \node[color=gray!50!black,font=\footnotesize] at (-1.1,2) {$N_n^{\frac23}(\log L)^a$};
    
    \draw[color=gray!50!black,|-|](6.25,0)--(6.25,1.5);
    \node[color=gray!50!black,font=\footnotesize] at (7.15,.75) {$N_n^{\frac13}(\log L)^b$};

    \draw [thick, color=orange, decorate, decoration={random steps,segment length=1.5mm}]
    (-1,-.5)--(7,-.5);

    \draw [thick, color=black]
    (-2,-.75)--(8,-.75);

    \draw[color=gray!50!black,|-|](0, 4.2)--(5.9,4.2);
    \node[color=gray!50!black,font=\footnotesize] at (3.1,4.55) {$N_n^{\frac23}(\log L)^a$};

    \node[font=\small] at (3,3.6) {$H+1-n$};
    \node[font=\small] at (3,0.25) {$H-n$};

    \node[circle,scale=0.4,fill=gray] at (0.02,0.8) {};
    \node[circle,scale=0.4,fill=gray] at (5.9,0.8) {};
    \end{scope}
    \end{tikzpicture}
    \caption{An instantiation of the domain $Q$ used in the proof of the upper bound on~$\fL_n$. The orange line is the lower level line $\fL_{n+1}$, and the floor (the constraint that $\phi_x \geq 0$) is only present in the green region. The two gray points on the boundary mark where the boundary conditions change from $H-n$ to $H+1-n$.}
    \label{fig:ub-rect}
\end{figure}
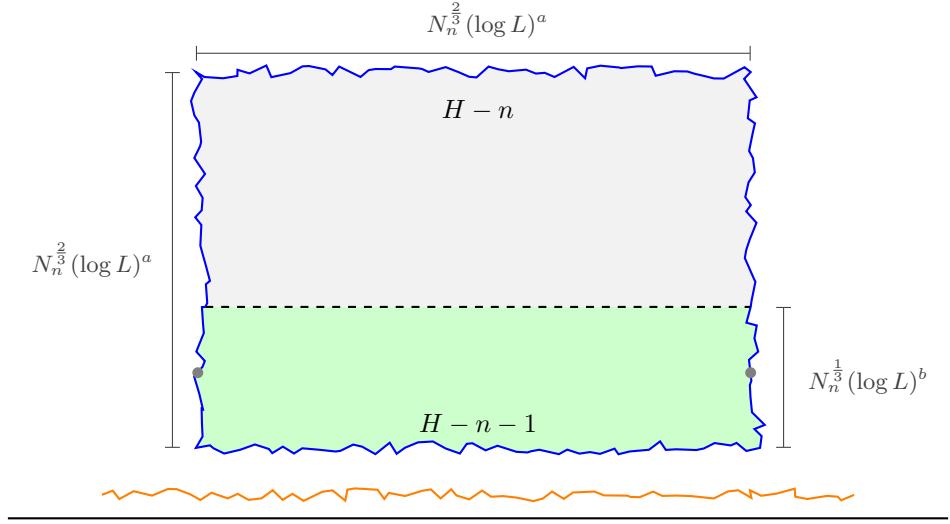

Observe that $\rho_n(x)$ is a decreasing function of $\phi$. Hence, en route to proving \cref{thm:ub} in this section we may apply monotonicity arguments that are decreasing (e.g., decreasing individual floor constraints or heights of boundary vertices; see \cref{clm:grad-p-monotone}).

Let $R$ be the rectangle of size $N_n^{2/3}(\log L)^{25} \times 2N_n^{2/3}(\log L)^{25}$ centered at the $x = 0$ such that the bottom side of $R$ is at distance $2N_n^{1/3} \exp(-c\sqrt{\beta\log L/\log\log L})$ above the bottom side of $\Lambda$, for $c$ as in \cref{thm:level-line-contains-Wulff}. Let $A, B$ be on the sides of $\partial R$ such that the distance from $A, B$ to the bottom of $R$ is $N_n^{1/3}(\log L)^{16}$. The dimensions of $R$ and the location of $A, B$ satisfy the requirements of \cref{lem:boundary-construction}. So, let $Q, \xi$ be any domain and boundary condition satisfying the properties of \cref{lem:boundary-construction} with respect to this $R, A, B$ (see \cref{fig:ub-rect}). By \cref{thm:level-line-contains-Wulff,lem:boundary-construction}, for the stochastic domination of a single $\bY_n$, it suffices to prove \cref{thm:ub} where $\fL_n$ is the $H+1-n$ level line in $\pi^\xi_Q$.

	Next, by FKG, we can first remove the floor in some parts of $Q$, and condition on decreasing events. Let $F$ be the portion of $Q$ below the line $y = 2N_n^{1/3}(\log L)^{16}$. Let $\cG^\sqcup$ be the event that $\fL_n$ stays above the horizontal line $\cH$ distance $2\log L$ above the highest point on the bottom side of $Q$. Then, $\pi^\xi_Q \succeq \pi^\xi_{Q; F}(\cdot \mid \cG^\sqcup)$. Note that the distance between $\cH$ and the bottom of $\Lambda$ is $o(N_n^{1/3})$, and goes to 0 under the rescaling of $Y_n$. For simplicity, in the rest of this section we will refer to the $y$-coordinate of $\cH$ as 0, and use ``height'' to denote the vertical distance of a point above $\cH$ (as opposed to referencing $\phi$ as a height function).

    From here, we can follow the same preliminary steps done in the proof of \cref{thm:growth-gadget} to move to a polymer model. By \cref{lem:cluster-expansion-likely}, we can condition further on the event $\cG_1$ defined there, which says that we have the cluster expansion form of the measure given by \cref{prop:CE-law-with-floor}. That is, (recalling \cref{eq:weight-of-q-E*} and the notation in \cref{eq:def-qUmu-weights}) we have
    \begin{align}\label{eq:UB-repeat-qn-weight}\pi^\xi_{Q; F}(\gamma \mid \cG^\sqcup, \cG_1) \propto \exp\bigg(-\sE^*_\beta(\gamma) + \frac{\cA_F(\gamma)}{N_n}+ \fI_Q(\gamma) \bigg) = q^n_{Q; F, 1}(\gamma)\,.
    \end{align} 
    Hence, by \cref{eq:good-event-CE-G'} we can now move to the polymer model with weights $q^n_{Q; F, 1}(\gamma)$ and partition function \[Z^n_{Q, Q; F, 1}(A, B \mid \cG^\sqcup) = \sum_{\gamma \in \cP_Q(A, B) \cap \cG^\sqcup}q^n_{Q; F, 1}(\gamma)\,.\]
    
     Recalling the animal decomposition from \cref{def:animal} and the notational comment following it, we extend the above definitions to weights and sums over animals $\Gamma$. Now define $W$ as ``$Q$ with no top or bottom'', i.e. the extension of $Q$ to an infinite vertical strip containing it at the top and bottom sides of $Q$. (Both the choice of the strip and exact point of extension are irrelevant.) By the event $\cG^\sqcup$ and the construction of $Q$, we first argue that except for a bad set of $\gamma$ making up a $o(1)$ fraction of $Z^n_{Q, Q; F, 1}(\cdot \mid \cG^\sqcup)$ and of $Z^n_{Q, W; F, 1}(\cdot \mid \cG^\sqcup)$, we have 
     \[q^n_{Q; F, 1}(\gamma) = (1+o(1))q^n_{W; F, 1}(\gamma)\,,\]
    and we can instead study the polymer model with weights $q^n_{W; F, 1}$ and partition function $Z^n_{Q, W; F, 1}(\cdot \mid \cG^\sqcup)$.
    By the decay properties of $\fI$, it suffices to show that $\gamma$ stays distance at least $\log L$ away from the top and bottom of $Q$, besides an exceptional set of $\gamma$. If $\gamma$ approaches the top of $Q$ then it must have length greater than $1.1N_n^{1/3}(\log L)^{16}$, which can be ruled out via a Peierls argument mapping to a minimal length path from $A$ to $B$. In particular, the distance from $\gamma$ to the top of $Q$ is at least $\log L$. Regarding interactions with the bottom of $Q$, in the proof of \cref{lem:Db-smaller-log2} (see \cref{eq:G2-likely-given-G'}) it is shown that 
        \[Z^n_{Q, Q; F, 1}(A, B \mid \max_{b \in \gamma}|\partial D_b| \leq \log L\,, \cG^\sqcup)\geq (1-o(1))Z^n_{Q, Q; F, 1}(A, B \mid \cG^\sqcup)\,,\]
        and the same proof also holds when replacing the interaction domain with $W$. For $\gamma$ in the left side of the inequality, the distance between $\gamma$ and the bottom side of $Q$ is at least $\log L$.

\subsection*{Random walk preliminaries and notation} Let $\bP^0$ denote the polymer measure on $\Gamma$ with weight $q^n_{W; F, 1}$ and partition function $Z^n_{Q, W; F, 1}(A, B)$, so our starting point is $\bP^0(\cdot \mid \cG^\sqcup)$. First observe that in \cref{eq:UB-repeat-qn-weight}, we could have replaced $\cA_F(\gamma)$ with $-|D_1 \cap F|$, the area below $\gamma$ intersected with $F$, by just a renormalization which does not change the measure $\bP^0$ (in fact we only started with $\cA_F(\gamma)$ to keep the same notation as \cref{sec:initial-UB-JEMS}). Now define $\cA_F^{\downarrow}(\Gamma)$ as the area of the region below the linear interpolation of the cone-points of $\Gamma$ intersected with $F$. Similarly for an irreducible component $\Gamma^{(i)}$ between cone-points $\sfu, \sfv$, we define $\cA_F^{\downarrow}(\Gamma^{(i)})$ as the area below $\overline{\sfu\sfv}$ intersected with $F$. In \cref{lem:bdd,rem:area-difference}, we show that except for a bad set of $\Gamma$ occurring with $\bP^0(\cdot \mid \cG^\sqcup)$-probability $o(1)$, we can replace $|D_1 \cap F|$ with $|\cA_F^{\downarrow}(\Gamma)|$ at the multiplicative cost of $(1+o(1))$ to the weight of $\Gamma$. We will make this change and call the new measures $\bP$ and $\bP(\cdot \mid \cG^\sqcup)$. 

    Now observe that every diamond in $\Gamma$ is fully contained in $W$, and hence $\fI_W(\Gamma^{(i)}) = \fI_{\Z^2}(\Gamma^{(i)})$ for irreducible components $\Gamma^{(i)}$. For $\bP$, recall from \cref{prop:OZ-normalization} that, after renormalizing the weights by $e^{\sfh_{(1, 0)} \cdot (B - A)}$,
    every irreducible component $\Gamma^{(i)}$ in $\Gamma$ has weight given by the probability $\P^{\sfh_{(1, 0)}}(\Gamma^{(i)})e^{-\frac{1}{N_n}\cA_F^{\downarrow}(\Gamma^{(i)})}$. By the product structure of the weights (\cref{eq:product-structure}), this implies that the law of $\Gamma$ in between any two cone-points $\sfu, \sfv$ is an area tilted random walk bridge with increments taking values in the set of irreducible components, independent of the portion of $\Gamma$ before $\sfu$ and after $\sfv$. This random walk measure will be denoted by $\bP_{\sfu, \sfv}(\Gamma) := \bP(\Gamma \mid \sfu, \sfv \in \Cpts(\Gamma))$, where it is implicit that the measure is on the restriction of $\Gamma$ in between $\sfu$ and $\sfv$. It will be useful also to define the one ended random walk measure $\bP_\sfu$ that starts at $\sfu$, so that $\bP_{\sfu, \sfv} = \bP_\sfu(\cdot \mid \sfv \in \Cpts(\Gamma))$. We can also define the above notation for the case when there is no area term, replacing $\bP$ by $\hat\bP$. Finally, the notation for the expectation with respect to the above measures will be to replace $\bP$ with $\bE$.

    \subsection*{Proof strategy} The strategy for proving \cref{thm:ub} is as follows:
    \begin{enumerate}[label=\textbf{Step~\arabic*}:, ref=\arabic*, wide=0pt, itemsep=1ex]
    \item \label[step]{st-1-bdd}
    We show in \cref{lem:bdd} that w.h.p.\ we have the event $\Bdd$, informally stating that the first and last cone-points of $\Gamma$ are not far away from $A$ and $B$, and the maximum distance between any adjacent cone-points of $\Gamma$ is suitably small. This moreover allows us to modify the area tilt so that it only depends on the cone-points of $\Gamma$, and not the full path. Calling the first and last cone-points $A', B'$, we move to the measure $\bP_{A', B'}(\cdot \mid \cG^\sqcup, \Bdd)$ for the next step.

\item\label[step]{st-2-drop} We next show in \cref{lem:dropping} that w.h.p.\ in the first and last quadrant of the interval $[A'_1, B'_1]$, we can find points with height at most $1000N_n^{1/3}$. Denoting these points $\bar\sfu, \bar\sfv$, we proceed to the next step with the measure $\bP_{\sfu, \sfv}(\cdot \mid \cG^\sqcup, \Bdd)$.

\item\label[step]{st-3-drop-to-K} Moving inwards, we show in \cref{lem:drop-to-Kepsilon} that for any $\epsilon > 0$, with probability $1 - \epsilon$, the cone-points with $x$-coordinates $-(\log\log L)N_n^{2/3}$ and $(\log\log L)N_n^{2/3}$, call them $\bar\sfu',\bar\sfv'$, are at height at most $K_\epsilon N_n^{1/3}$. Calling these points $\bar\sfu', \bar\sfv'$, we move to the next step with the measure $\bP_{\sfu', \sfv'}(\cdot \mid \cG^\sqcup, \Bdd)$.

\item\label[step]{st-4-repel}
 In \cref{lem:repel-with-area} we show that with high probability, we have the event $\mathsf{Repel}$ that $\Cpts(\Gamma)$ stays above $N^\delta$ in the interval $J$ that is $N^{8\delta}$ away from the endpoints $\bar\sfu'$ and $\bar\sfv'$. Call the first and last cone-points in $J$ by $\bar\sfu'', \bar\sfv''$ and consider the measure $\bP_{\bar\sfu'', \bar\sfv''}(\cdot \mid \cG^\sqcup, \Bdd, \mathsf{Repel})$. Since $\Bdd \cap \mathsf{Repel}$ implies $\cG^\sqcup$, the conditioning on $\cG^\sqcup$ can be removed. The resulting measure $\bP_{\bar\sfu'', \bar\sfv''}(\cdot \mid \Bdd, \mathsf{Repel})$ is such that the marginal on $\Cpts(\Gamma)$ is a 2\Dim area-tilted random walk, as (recalling also \cref{rem:area-difference}) we have removed everything not measurable with respect to $\Cpts(\Gamma)$. The convergence of such a random walk to the Ferrari--Spohn diffusion was proven in \cite{IOSV21}.
    \end{enumerate}

\subsection*{Reducing to a 2D RW with an area-tilt}

We begin with establishing the boundedness of irreducible components.

    \begin{lemma}\label{lem:bdd}
        Let $\Bdd$ be the set of $\Gamma$ such that $|\mathsf{Cpts}(\Gamma)| \geq 2$, and the maximum displacement $\max_i(|X(\Gamma^{(i)})|, |X(\Gamma^{(L)})|, |X(\Gamma^{(R)})|)$  is no more than $(\log L)^{50}$. Then, 
        \[\bP^0(\Bdd^c \mid \cG^\sqcup) \leq e^{-c(\log L)^{50}}\,,\]
        or equivalently,
        \[Z^n_{Q, W; F, 1}(\Bdd^c \mid \cG^\sqcup) \leq e^{-c(\log L)^{50}}Z^n_{Q, W; F, 1}(\cG^\sqcup)\,.\]
    \end{lemma}
    \begin{proof}
        Since the area tilt can change the weight of events by at most a factor of $e^{(\log L)^{41}}$, it suffices to prove the bounds with respect to $\hatZ^n_{Q, W}$ and weights $\hatq^n_W$ (i.e. having removed the area term). By \cref{lem:cpts-length} we can assume we have cone-points. Now by construction of $Q$, all diamonds of $\Gamma$ are entirely contained in $W$, whence $\hatq^n_{W}(\Gamma^{(i)}) = \hatq^n_{\Z^2}(\Gamma^{(i)})$. Ignoring the event $\cG^\sqcup$ by an upper bound, we can thus write
		\begin{align*}
			\hatZ^n_{Q, W}\big(&A, B  \;\big|\; |X(\Gamma^{(L)})| > (\log L)^{50}, |\mathsf{Cpts}(\Gamma)| \geq 2, \cG^\sqcup \big) \\
			&\leq \sum_{\substack{u \in \cY^{\btl}(A)\\v \in \cY^{\btr}(B)\\
            |A - u| > (\log L)^{50}}}\sum_{\Gamma^{(\sfL)} \in \sfA_\sfL \cap \cP_Q(A, u)}\!\!\hatq^n_{W}(\Gamma^{(\sfL)})\sum_{\Gamma^{(\sfR)} \in \sfA_\sfR \cap \cP_Q(v, B)}\!\!\hatq^n_{W}(\Gamma^{(\sfR)})\sum_{\substack{k \geq 1\\\Gamma^{(1)} \circ \ldots \Gamma^{(k)} \in \cP_{\Z^2}(\sfu, \sfv)}}\!\!\prod_{i =1}^k\hatq^n_{\Z^2}(\Gamma^{(i)})\\
			&= \sum_{\substack{u \in \cY^{\btl}(A)\\v \in \cY^{\btr}(B)\\
            |A - u| > (\log L)^{50}}}\sum_{\Gamma^{(\sfL)} \in \sfA_\sfL \cap \cP_Q(A, u)}\!\!\hatq^n_{W}(\Gamma^{(\sfL)})\sum_{\Gamma^{(\sfR)} \in \sfA_\sfR \cap \cP_Q(v, B)}\!\!\hatq^n_{W}(\Gamma^{(\sfR)})\hatZ^n_{\Z^2, \Z^2}(\sfu, \sfv)\,.
		\end{align*}
		Fixing $u$, every $\Gamma^{(L)} \in \cP_Q(A, u)$ has length at least $|A - u|$, so by \cref{prop:OZ-normalization} (and the definition of $\sfh_\sfy$ there) we have
		\begin{equation*}
            \sum_{\Gamma^{(\sfL)} \in \sfA_\sfL \cap \cP_Q(A, u)}\!\!\hatq^n_{W}(\Gamma^{(\sfL)}) \leq Ce^{-\nu_g\beta |A - u|}e^{-\tau_\beta(A - u)}\,,
		\end{equation*}
		with an analogous bound for $\Gamma^{(\sfR)}$. Recall that $\hatZ^n_{\Z^2, \Z^2}(\sfu, \sfv) \leq Ce^{-\tau_\beta(v - u)}$ by \cref{it:DKS-convergence-rate} of \cref{prop:DKS-propositions} combined with \cref{prop:compare-tau}. Together with the convexity of $\tau_\beta$, we obtain an upper bound of 
		\begin{align*}
		      \hatZ^n_{Q, W}(A, B \mid |X(\Gamma^{(L)})| > (\log L)^{50},& |\mathsf{Cpts}(\Gamma)| \geq 2, \cG^\sqcup) \\&\leq Ce^{-\tau_\beta(B-A)}\sum_{\substack{u \in \cY^{\btl}(A)\\
            |A - u| > (\log L)^{50}}} e^{-\nu_g\beta|A - u|}\sum_{v \in \cY^{\btr}(B)} e^{-\nu_g\beta|B - v|}\,.
		\end{align*}
		The number of points $u$ such that $|A - u| = r$ is $O(r)$, so that the sum over $u$ is bounded by $e^{-\tfrac{\nu_g}2\beta (\log L)^{50}}$. By the same logic, the sum over $v$ is bounded by a constant. Hence, we have
	\begin{align*}\hatZ^n_{Q, W}(A, B \mid |X(\Gamma^{(L)})| > (\log L)^{50}, |\mathsf{Cpts}(\Gamma)| \geq 2, \cG^\sqcup)&\leq Ce^{-\tfrac{\nu_g}2\beta (\log L)^{50}}e^{-\tau_\beta(B-A)}\\
    &\leq Ce^{-c\beta (\log L)^{50}}\hatZ^n_{\Z^2, \Z^2}(A, B)\\
    &\leq Ce^{-c\beta (\log L)^{50}}\hatZ^n_{Q, W}(A, B)\,,
    \end{align*}
        where we again used \cref{it:DKS-convergence-rate} of \cref{prop:DKS-propositions} with \cref{prop:compare-tau} in the second line, and \cref{lem:cigar-likely-in-Z2,lem:strip-to-Z2-interactions} in the third.

    The proof that the irreducible components $|X(\Gamma^{(i)}|$ have the same bound follows similarly, so the details are omitted. The lemma now concludes by a union bound.
\end{proof}

\begin{remark}\label{rem:area-difference}
    Given $\Gamma$, let $S$ be the linear interpolation of $\Cpts(\Gamma)$. On the event $\Bdd$, the area in between $\Gamma$ and $S$ between two cone-points of $\Gamma$ is at most $O(\log L)^{50}$. There are (deterministically) at most $N_n^{2/3}(\log L)^{25}$ cone-points, so the total difference in area between $\Gamma$ and $S$ is at most $O(N_n^{2/3}(\log L)^{75})$. After dividing this area by $N_n$, this is $o(1)$. Hence, as commented in the preliminaries of this section, we can change the area tilt term in \cref{eq:UB-repeat-qn-weight} to $-\cA_F^\downarrow(\Gamma)/N_n$ at the cost of a multiplicative $(1+o(1))$. 
\end{remark}

Next, we show that $\Gamma$ will drop to height $O(N_n^{1/3})$. The following lemma may be viewed as a 2\Dim analog of the 1\Dim dropping lemma in \cite[Lem. 4.2]{HKS25}. (In both cases, the proof uses the mechanism of splitting the interval into segments where the random walk is forced to follow a parabolic descent, with parameters derived from balancing the area term vs.\ the large deviations of a random walk. The 2\Dim nature of the walk introduces several technical difficulties; see \cref{fig:drop-points}.) 
For an easier read of the proof, we leave the dimensions in terms of $a, b$; our application is for $a = 25$ and $b = 16$.
\begin{lemma}\label{lem:dropping}
Fix $a,b>0$ such that $a>3b/2$, let $I:=\llb -\frac12 N_n^{2/3} (\log L)^a,-\frac12 N_n^{2/3} (\log L)^a\rrb$, and 
let $\sfu,\sfv$ be such that $\sfu_1<\sfv_1$ are within an additive $(
\log L)^{50}$-term of the two endpoints of $I$, and $\sfu_2,\sfv_2 \in N_n^{1/3}(\log L)^b \pm (\log L)^{50}  $. Let $\mathrm{Drop}$ be the event that there exist cone-points of $\Gamma$ in first and last quarters of $I$ which lie below height $KN_n^{1/3}$. Then for some constant $K > 0$,
\[\bP_{\sfu, \sfv}(\mathrm{Drop} \mid \cG^\sqcup, \Bdd) \geq 1 - e^{-c(\log L)^a}\,.\]
\end{lemma}
\begin{proof}
Roughly, the sought bound is achieved by forcing the 2\Dim random walks to proceed as follows:
\begin{itemize}
\item In the two intervals of length $N_n^{2/3} / (\log L)^{b}$ closest to $\sfu,\sfv$, drop by height $N_n^{1/3}/(\log L)^{b/2}$. 
This is repeated up to $(\log L)^{3b/2}$ times from each side so the height drops to at most $\frac{K}2 N_n^{1/3}$. 

\item Not exceed height $KN_n^{1/3}$ throughout the remaining middle interval. 
\end{itemize}
We will argue that the probability cost associated to these two items in a random walk with no area-tilt is $\exp[-c (\log L)^{3b/2}]$ and $\exp[-c (\log L)^a]$ respectively, whereas the associated area of such $\Gamma$ has $\cA^\downarrow_F(\Gamma)=O((\log L)^{3b/2}+ (\log L)^a)N_n)$, giving a term of the same order as above in the probability.
    
    By a union bound, we may as well assume that $\mathrm{Drop}$ only refers to the existence of such a cone-point in the first quarter of $I$. Isolating the area term, we have
    \[\bP_{\sfu, \sfv}(\mathrm{Drop}^c \mid \cG^\sqcup, \Bdd) = \frac{\hat\bE_{\sfu, \sfv}[e^{-\frac{1}{N_n}\cA^\downarrow_F(\Gamma)}\one_{\mathrm{Drop}^c} \mid \cG^\sqcup, \Bdd]}{\hat\bE_{\sfu, \sfv}[e^{-\frac{1}{N_n}\cA^\downarrow_F(\Gamma)} \mid \cG^\sqcup, \Bdd]}\,.\]
    On the event $\mathrm{Drop}^c \cap \Bdd$, we know that $\Gamma$ lies above height $KN_n^{1/3} - (\log L)^{50} =: K'N_n^{1/3}$ over an interval of length $N_n^{2/3}(\log L)^a/4$. Thus, we have $\cA^\downarrow_F(\Gamma) \geq K'(\log L)^aN_n/4$, whence
    \[\hat\bE_{\sfu, \sfv}[e^{-\frac{1}{N_n}\cA^\downarrow_F(\Gamma)} \mid \cG^\sqcup, \Bdd] \leq e^{-K'(\log L)^a/4}\,.\]
    It suffices to now show the key lower bound, which captures the dropping behavior of $\Gamma$ to the equilibrium height of $O(N_n^{1/3})$:
    \begin{equation}\label{eq:key-lower-bound}
        \hat\bE_{\sfu, \sfv}[e^{-\frac{1}{N_n}\cA^\downarrow_F(\Gamma)} \mid \cG^\sqcup, \Bdd] \geq e^{-c^*(\log L)^a}\,,
    \end{equation}
    whence we can take $K' > 8c^*$ to conclude.

    To prove this lower bound, we break $I$ into three intervals $I = I_1 \cup I_2 \cup I_3$ such that $|I_1| = |I_3| = N_n^{2/3}(\log L)^{b/2}$. Let $\Low$ be the event that the maximum height of the cone-points of $\Gamma$ in the interval $I_2$ is at most $KN_n^{1/3}$.  Then, on $\Bdd \cap \Low$, the maximum height of $\Gamma$ in $I_2$ is at most $KN_n^{1/3} + (\log L)^{50}$, in which case
    \[\frac{\cA^\downarrow_F(\Gamma)}{N_n} \leq \frac{1}{N_n}N_n^{1/3}(\log L)^b(|I_1| + |I_3|) + \frac{C}{N_n}KN_n^{1/3}|I_2| \leq (\log L)^{3b/2} + C'K(\log L)^a\,.\]
    Hence, 
    \begin{align*}\hat\bE_{\sfu, \sfv}[e^{-\frac{1}{N_n}\cA^\downarrow_F(\Gamma)} \mid \cG^\sqcup, \Bdd] &\geq e^{-(\log L)^{3b/2} - C'K(\log L)^a}\hat\bP_{\sfu, \sfv}\big(\Low \mid \Bdd, \cG^\sqcup\big) \\
    &= e^{-(\log L)^{3b/2} - C'K(\log L)^a}\hat\bP_\sfu\big(\Low \mid \Bdd, \cG^\sqcup, \sfv \in \Cpts(\Gamma)\big)\,.
    \end{align*}
    As long as $a > 3b/2$, the $(\log L)^a$ term dominates the prefactor, which is in the form of the desired lower bound in \cref{eq:key-lower-bound}. It now remains to prove a lower bound of the same form for the probability. We trivially have
    \[\hat\bP_\sfu\big(\Low \mid \Bdd, \cG^\sqcup, \sfv \in \Cpts(\Gamma)\big) \geq \hat\bP_\sfu\big(\Low, \cG^\sqcup, \sfv \in \Cpts(\Gamma) \mid \Bdd \big)\,.\]
    
    First, let $\cG^\sqcup_*$ be the event that the minimum height of any cone-point of $\Gamma$ is at least $(\log L)^{50}$. Then, $\Bdd \cap \cG^\sqcup_* \subset \cG^\sqcup$, so that
    \[\hat\bP_\sfu\big(\Low, \cG^\sqcup, \sfv \in \Cpts(\Gamma) \mid \Bdd\big) \geq \hat\bP_\sfu\big(\Low, \cG^\sqcup_*, \sfv \in \Cpts(\Gamma) \mid \Bdd\big)\,.\]
    By this move, we have reduced to a situation where all the events in question are measurable with respect to $\Cpts(\Gamma)$. Hence, we can turn to the marginal on $\Cpts(\Gamma)$, which has the law of a 2\Dim random walk with increment law $\P(x, y) = \P^{\sfh_{(1, 0)}}(X(\Gamma) = (x, y) \mid |(x, y)| \leq (\log L)^{50})$. Let $\sfS$ denote this random walk started at $\sfu$ when the measure referred to is $\hat\bP_\sfu(\cdot \mid \Bdd)$.

\begin{figure}
\vspace{0.1in}
\centering
    \begin{tikzpicture}
    \pgfmathsetseed{314159}

    \pgfmathsetmacro{\xuv}{0.55}
    \pgfmathsetmacro{\yuv}{0.25}
    \pgfmathsetmacro{\xw}{0.7}

    \begin{scope}[scale=1.5]
    \node[circle,scale=0.3,fill=black,label={[label distance=-2pt]above left:{\footnotesize$\mathsf{u}$}}] (u0) at (0,0) {};

    \foreach \i in {1,...,5} { \node[circle,scale=0.3,fill=blue] (u\i) at (\i*\xuv,-\yuv*\i) {};
    \draw[fill=blue!50, color=blue,fill opacity=0.25] (u\i) circle (1mm);
    \draw[color=black,thick,dotted] (u\i)--++(0,\yuv) -- ++(-\xuv,0);
    }

    \node[circle,scale=0.3,fill=black,label={[label distance=-2pt]above right:{\footnotesize$\mathsf{v}$}}] (v0) at (10*\xuv+6*\xw,0) {};

    \foreach \i in {1,...,4} { \node[circle,scale=0.3,fill=blue] (v\i) at (10*\xuv+6*\xw-\i*\xuv,-\yuv*\i) {};
    \draw[fill=blue!50, color=blue,fill opacity=0.25] (v\i) circle (1mm);   
    \draw[color=black,thick,dotted] (v\i)--++(0,\yuv) -- ++(\xuv,0);
    }

    \foreach \i in {1,...,4} {
    \node[circle,scale=0.3,color=red,fill=red] (w\i) at (5*\xuv+\i*\xw,-\yuv*5) {};    
    \draw[color=orange,fill=orange!50, fill opacity=0.5] (w\i) circle (1.8mm);
    \draw[color=black,thick,dotted] (w\i)-- ++(-\xw,0);
    }

    \foreach \i in {5,6} {
    \node[circle,scale=0.3,color=red,fill=red] (w\i) at (5*\xuv+\i*\xw,-\yuv*5+0.1*\i-0.4) {};    
    \draw[color=orange,fill=orange!50, fill opacity=0.5] (w\i) circle (1.8mm); 
    \draw[color=black,thick,dotted] (w\i)--++(-\xw,0) -- ++(0,-0.05);
    }
    
    \node[font=\footnotesize] at ($(u1)+(-0.1,-.25)$) {$\mathsf{u}^{(1)}$};
    \node[font=\footnotesize] at ($(u2)+(-0.1,-.25)$) {$\mathsf{u}^{(2)}$};
    \node[font=\footnotesize,rotate=-27] at ($(u3)+(0.,-.4)$) {$\ldots$};
    \node[font=\footnotesize] at ($(u5)+(-0.3,-.1)$) {$\mathsf{u}^{(m)}$};

    \node[font=\footnotesize] at ($(v1)+(0.2,-.2)$) {$\mathsf{v}^{(1)}$};
    \node[font=\footnotesize] at ($(v2)+(0.2,-.2)$) {$\mathsf{v}^{(2)}$};
    \node[font=\footnotesize,rotate=27] at ($(v3)+(0.34,-.22)$) {$\ldots$};
    \node[font=\footnotesize] at ($(v4)+(0.4,-.2)$) {$\mathsf{v}^{(m')}$};

    \node[font=\footnotesize] at ($(w1)+(0,-.35)$) {$\mathsf{w}^{(1)}$};
    \node[font=\footnotesize] at ($(w2)+(0,-.35)$) {$\mathsf{w}^{(2)}$};
    \node[font=\footnotesize] at ($(w3)+(0.,-.4)$) {$\ldots$};
    \node[font=\footnotesize] at ($(w4)+(0,-.35)$) {$\mathsf{w}^{(k)}$};
    \node[font=\footnotesize,rotate=15] at ($(w5)+(0.,-.4)$) {$\ldots$};
    \node[font=\footnotesize] at ($(w6)+(0,-.35)$) {$\mathsf{w}^{(k')}$};
    
    \draw[color=purple!80!black, thick, decorate, decoration={random steps,segment length=1mm}] 
     (u0)--(u1)--(u2)--(u3)--(u4)--(u5)--(w1)--(w4)--(w6)--(v4)--(v3)--(v2)--(v1)--(v0);

    \draw[|-|,gray!50!black] ($(u2)+(0,0.4)$)-- ++(\xuv,0);
    \node[color=gray!50!black,font=\tiny] at ($(u2)+(0.4,0.6)$) {$N_n^{2/3}(\log L)^{-b}$};

    \draw[|-|,gray!50!black] ($(u3)+(0,\yuv+0.03)$)-- ++(0,\yuv-0.03);
    \node[color=gray!50!black,font=\tiny] at ($(u3)+(0.75,0.4)$) {$N_n^{1/3}(\log L)^{-b/2}$};

    \draw[|-|,gray!50!black] ($(v2)+(0,0.4)$)-- ++(-\xuv,0);
    \node[color=gray!50!black,font=\tiny] at ($(v2)+(-0.15,0.6)$) {$N_n^{2/3}(\log L)^{-b}$};

    \draw[|-|,gray!50!black] ($(v3)+(0,\yuv+0.03)$)-- ++(0,\yuv-0.03);
    \node[color=gray!50!black,font=\tiny] at ($(v3)+(-0.75,0.4)$) {$N_n^{1/3}(\log L)^{-b/2}$};

    \draw[|-|,gray!50!black] ($(w3)+(0,0.4)$)-- ++(\xw,0);
    \node[color=gray!50!black,font=\tiny] at ($(w3)+(0.4,0.55)$) {$N_n^{2/3}$};

    \draw[thick,black] ($(u0)+(-0.25,-2.5)$)--($(v0)+(0.25,-2.5)$);
    
    \draw[|-|,gray!50!black] ($(u0)+(0,-2.6)$)-- ++(\xuv*5,0);
    \node[color=gray!50!black,font=\tiny] at ($(u0)+(1.5,-2.75)$) {$N_n^{2/3}(\log L)^{b/2}$};

    \draw[|-|,gray!50!black] ($(u5)+(0,-0.12)$)-- ++(0,-1.1);
    \node[color=gray!50!black,font=\tiny] at ($(u5)+(-0.35,-0.55)$) {$\frac{K}2 N_n^{1/3}$};

    \draw[|-|,gray!50!black] ($(u0)+(0,-0.07)$)-- ++(0,-2.4);
    \node[color=gray!50!black,font=\tiny] at ($(u0)+(0.6,-1.4)$) {$N_n^{1/3}(\log L)^b$};

    \draw[|-|,gray!50!black] ($(v0)+(0,-0.07)$)-- ++(0,-2.4);
    \node[color=gray!50!black,font=\tiny] at ($(v0)+(-0.6,-1.4)$) {$N_n^{1/3}(\log L)^b$};

    \draw[|-|,gray!50!black] ($(v0)+(0,-2.6)$)-- ++(-\xuv*4,0);
    \node[color=gray!50!black,font=\tiny] at ($(v0)+(-1.2,-2.75)$) {$N_n^{2/3}(\log L)^{b/2}$};

    \draw[|-|,gray!50!black] ($(v4)+(0,-0.12)$)-- ++(0,-1.35);
    \node[color=gray!50!black,font=\tiny] at ($(v4)+(.5,-0.8)$) {$N_n^{1/3}\log L$};

    \draw[|-|,gray!50!black] ($(w1)+(-.65,-1.35)$)-- ++(\xw*5+1.15,0);
    \node[color=gray!50!black,font=\tiny] at ($(w1)+(1.65,-1.5)$) {$N_n^{2/3}(\log L)^{a}$};
    \end{scope}
\end{tikzpicture}
\caption{The dropping points $\sfz^{(j)}$ and target balls $\cB_j$ around them in proving \cref{lem:dropping}. We lower bound the probability that $\Gamma: \sfu \rightarrow \sfv$ hits each target ball along the way. Between $\sfw^{(k')}$ and $\sfv^{(m')}$, the size of the balls decreases in order. It is now rare for a random walk started from the orange ball around $\sfw^{(k')}$ to hit the much smaller blue ball around $\sfv^{(m')}$. Controlling this probability is easier if this transition occurs away from the bottom, hence we first climb back up to height $N_n^{1/3}(\log L)$.}
\label{fig:drop-points}
\end{figure}
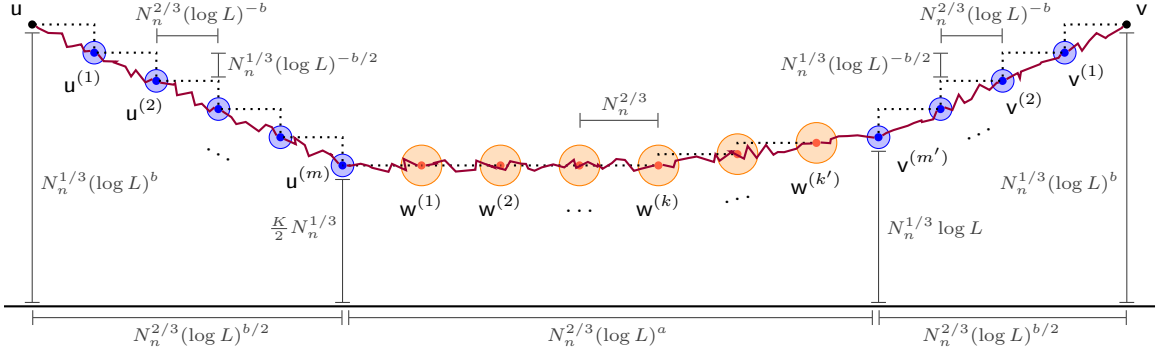

    Let
    \[ m:=\lfloor (\log L)^{3b/2} - \tfrac{K}2(\log L)^{b/2}\rfloor \quad ,\quad m' := \lfloor (\log L)^{3b/2} - (\log L)^{b/2+1}\rfloor \,,\]
    and define the points
\begin{align*}
    \sfu^{(j)} &= \left(\sfu_1+ j  N_n^{2/3}(\log L)^{-b}, \sfu_2 - jN_n^{1/3}(\log L)^{-b/2} \right)\qquad(j=0,\ldots,m)\,,\\
    \sfv^{(j)} &=  \left( \sfv_1 - j N_n^{2/3}(\log L)^{-b}, \sfv_2 -j N_n^{1/3}(\log L)^{-b/2}\right)\qquad(j=0,\ldots,m')\,,
\end{align*}
so that, recalling $|\sfu_2 - N_n^{1/3}(\log L)^b |< (\log L)^{50}$, we have $|\sfu^{(m)}_2 - \frac{K}{2}N_n^{1/3}| = O( N_n^{1/3}(\log L)^{-b/2}))$ (accounting for integer rounding in $m$), and similarly $|\sfv_2^{(m')}- N_n^{1/3}\log L| = O(N_n^{1/3}(\log L)^{-b/2}))$.
Further set \[ k :=  \lfloor (\log L)^a - (m+m') (\log L)^{-b}\rfloor - \lceil\log L - \tfrac{K}{2}\rceil \quad,\quad 
 k' = \lfloor (\log L)^a - (m+m') (\log L)^{-b}\rfloor - 1
\,,
\]
and define the points \begin{align*}
\sfw^{(j)} = \begin{cases}
    \left(\sfu^{(m)}_1 + j\lfloor N_n^{2/3}\rfloor, \sfu^{(m)}_2\right) & j=1,\ldots,k\,,\\
    \noalign{\medskip}
    \left(\sfu^{(m)}_1 + j\lfloor N_n^{2/3}\rfloor, \sfu^{(m)}_2 + (j-k)N_n^{1/3} \right) & j=k+1,\ldots,k' \,.\\
\end{cases}
\end{align*}
so that the jump from $\sfw^{(k')}$ to $\sfv^{(m')}$ is such that
\begin{align*} \lfloor N_n^{2/3}\rfloor \leq \sfv^{(m')}_1-\sfw^{(k')}_1 \leq 2\lfloor N_n^{2/3}\rfloor\,.
\end{align*} 

We will control the probability that in time $m+m' + k' + 1$, the random walk $\sfS$ passes through appropriately sized balls around the points $\sfu^{(j)}, \sfw^{(j)}, \sfv^{(j)}$, as illustrated in \cref{fig:drop-points}. That is, $\sfS$ will first drop down along the points $\sfu^{(j)}$ during the interval $I_1$, then remain low at the points $\sfw^{(j)}$ along the interval $I_2$, and finally climb back up to $\sfv$ along the points $\sfv^{(j)}$ along the interval $I_3$. We defined separately the above points as $\sfu, \sfw, \sfv$ for ease of definitions and an emphasis on the above geometric description, but it will be convenient now to just define one sequence of points $\sfz^{(j)}$ such that
\begin{equation*}
    \sfz^{(j)} = \begin{cases}
        \sfu^{(j)} & \text{for } 0 \leq j \leq m\\
        \sfw^{(j - m)} & \text{for } m < j \leq m + k'\\
        \sfv^{(m+m'+k'+1-j)} & \text{for } m + k' < j \leq m+m' + k' + 1\,.
    \end{cases}
\end{equation*}
Now along this sequence, for $1 \leq j  \leq m+m'+k'+1$, define $\alpha = \E^{\sfh_{(1, 0)}}[X(\Gamma)_1 \mid |X(\Gamma)| \leq (\log L)^{50}]$, the expectation in the $x$-coordinate of an increment. (The expectation in the $y$-coordinate is 0 by the symmetry of the model.) Then define $\ell_j := \frac{1}{\alpha}(\sfz^{(j)}_1 - \sfz^{(j-1)}_1)$, noting that with the exception of $j = m+k'$, we have $\ell_j \leq \ell_{j+1}$. 
Define $\cB_j$ as the ball of radius $\sqrt{\ell_j}$ centered about $\sfz^{(j)}$ in $\Z^2$, except when $j=m+k'$ (the index where $\ell_j > \ell_{j+1}$), where we set $\cB_j$ to have radius $\sqrt{\ell_{j+1}}$.
For convenience, define $\cB_0 = \sfz^{(0)}$. Observe now that for all~$j$ we have 
\begin{equation}\label{eq:dist-z-tildez}
|\widetilde \sfz^{(j)} - z^{(j)}| \leq \sqrt{\ell_{j+1}}\,.
\end{equation}
Indeed for $j \neq m+k'$, we have by construction that $|\widetilde \sfz^{(j)} - z^{(j)}| \leq \sqrt{\ell_j} \leq \sqrt{\ell_{j+1}}$. For $j = m+k'$, since we adjusted the radius of $\cB_j$ to be $\sqrt{\ell_{j+1}}$, we directly have $|\widetilde \sfz^{(j)} - z^{(j)}| \leq \sqrt{\ell_{j+1}}$.

We now show that there exists $C > 0$ such that for any $j$ and any $\widetilde \sfz^{(j)} \in \cB_j$, we have
\begin{equation}\label{eq:control-mean}
    \left|\sfz^{(j+1)} -\hat\E_{\widetilde \sfz^{(j)}}[\sfS(\ell_{j+1})]\right| \leq C\sqrt{\ell_{j+1}}\,.
\end{equation}
Indeed, we have
\begin{align*}
\hat\E_{\widetilde \sfz^{(j)}}[\sfS(\ell_{j+1})] &= (\widetilde \sfz^{(j)}_1 + \ell_{j+1}\alpha_1, \widetilde \sfz^{(j)}_2)\\
&=(\widetilde \sfz^{(j)}_1 + z^{(j+1)}_1 - z^{(j)}_1, \widetilde \sfz^{(j)}_2)\,.
\end{align*}
The bound on the $x$-coordinates in \cref{eq:control-mean} now follows from \cref{eq:dist-z-tildez}. For the $y$-coordinates, \[|\widetilde \sfz^{(j)}_2 - \sfz^{(j+1)}_2| \leq |\widetilde\sfz^{(j)}_2 - \sfz^{(j)}_2| + |\sfz^{(j+1)}_2 - \sfz^{(j)}_2| \leq 2\sqrt{\ell_{j+1}}\,,\]
using \cref{eq:dist-z-tildez} and the fact that $|\sfz^{(j+1)}_2 - \sfz^{(j)}_2| \leq \sqrt{\ell_{j+1}}$ by the construction of the sequence $\sfz^{(j)}$.

With \cref{eq:control-mean} in hand, we aim to prove the following claim.
\begin{claim}\label{clm:prob-hitting-targets}
If we denote, for $j=0,\ldots,m+m'+k'$,
\[ p_{j}(\sfz,\cB) := \hat\bP_{\sfz}\Big(\sfS(\ell_{j+1}) \in \cB,\, \sfS(t)_2 \in [\tfrac12\sfz_2,\tfrac32\sfz_2] \;\forall\, t \leq \ell_{j+1} \;\big|\; \Bdd\Big)\,,\] 
then there exists a constant $c = c(\beta) > 0$ such that for every $0 \leq j \leq m+m'+k'$ except $j = m+k'$, we have
    \begin{align}\label{eq:hit-target-box} \min_{\sfz\in \cB_j} p_{j}(\sfz,\cB_{j+1}) &\geq c&\mbox{for } j\in\{0,\ldots,m+m'+k'\}\setminus\{m+k'-1\}\,,\\
    \label{eq:hit-target-box-m+k'} \min_{\sfz\in\cB_j} p_j(\sfz,\cB_{j+1}) & \geq c(\log L)^{-b}&\mbox{for }j=m+k'-1\,.
    \end{align}
   Furthermore, we have
    \begin{align}\label{eq:hit-last-point}
 \min_{\sfz\in\cB_{j}} p_j(\sfz,\{\sfv\}) &\geq cN_n^{-2/3} (\log L)^b &\qquad\qquad\qquad\mbox{for }j=m+m'+k'\,.
    \end{align}
\end{claim}
\begin{proof}[Proof of \cref{clm:prob-hitting-targets}]
    By \cref{eq:control-mean} and the local limit theorem for the $\Z^2$ random walk in~\cite{Richter58}, there exists some constant $c_1 =c_1(\beta)> 0$ such that for every $j$ and every $\sfz\in\cB_j$,
    \begin{align*}
    &\hat\bP_{\sfz}\left(\sfS(\ell_{j+1})   \in \cB_{j+1} \mid \Bdd\right) \geq c_1 \frac{|\cB_{j+1}|}{\ell_{j+1}}\,.
        \end{align*}
    Recall that for all $j \neq m+k'-1$ we defined $\cB_{j+1}$ to be the ball of radius $\sqrt{\ell_{j+1}}$ about $\sfz^{(j+1)}$. For that exceptional $j=m+k'-1$, the radius is $\sqrt{\ell_{j+2}} = N_n^{1/3} (\log L)^{-b/2} = \sqrt{\ell_{j+1}} (\log L)^{-b/2}$. Thus, we can infer from the above display that there exists some constant $c_2=c_2(\beta)>0$ such that 
\begin{align*} \min_{\sfz\in\cB_j}&\hat\bP_{\sfz}\left(\sfS(\ell_{j+1}) \in \cB_{j+1} \mid \Bdd\right) \geq c_2 \qquad\mbox{for all $j\neq m+k'-1$}\,,
\end{align*}
and 
\begin{align*} \min_{\sfz\in\cB_{m+k'-1}}&\hat\bP_{\sfz}\left(\sfS(\ell_{m+k'}) \in \cB_{m+k'} \mid \Bdd\right) \geq c_2 (\log L)^{-b}\,.\end{align*}
    Similarly, at $j=m+m'+k'$, we may target the singleton $\{\sfv\}$ and obtain the lower bound
        \begin{align*}
    &\min_{\sfz\in\cB_{m+m'+k'}} \hat\bP_{\sfz}(\sfS(\ell_{m+m'+k'+1}) = \sfv \mid \Bdd) \geq \frac{c_2}{\ell_{m+m'+k'+1}} = \frac{c_2}{N_n^{2/3} (\log L)^{-b}}\,.
    \end{align*}

Turning to establish \cref{eq:hit-target-box} for $j\neq m+k'-1$, it suffices via a union bound to show that, for every such $j$ and all $\sfz\in \cB_j$, we have 
    \[\hat\bP_{\sfz}(\max_{t \leq \ell_{j+1}} |\sfS(t)_2 - \sfz_2| > \tfrac12 \sfz_2 \mid \Bdd) \leq c_2/2\,.\]
Indeed, by construction, for all $j$ and $\sfz\in\cB_j$ we have $\sfz_2 \geq (K/4) N_n^{1/3} > (K/4)\sqrt{\ell_{j+1}}$ for all $j$, and said bound readily holds for our  random walk $\sfS(\cdot)$ by Hoeffding's inequality (for large enough $K$).
To establish \cref{eq:hit-target-box-m+k'}, we use that the location of the transition point $j=m+k'-1$ was chosen precisely so that $\sfz_2$ would then be suitably large. Namely, as $\cB_{m+k'-1}$ is centered at $\sfz^{(m+k'-1)}$, we have $\sfz_2 = (1+o(1))(k'-k) N_n^{1/3} = (1+o(1))N_n^{1/3} \log L$ for all $\sfz\in\cB_{m+k'-1}$. In particular, for the exceptional $j=m+k'-1$ we have by Hoeffding's inequality (recall $\ell_{j+1} = N_n^{2/3}$ for this $j$) that
\[\hat\bP_{\sfz}(\max_{t \leq \ell_{j+1}} |\sfS(t)_2 - \sfz_2| > \tfrac12 \sfz_2 \mid \Bdd) < L^{-100}\,.\]
Finally, the last inequality holds true also for the final $j=m+m'+k'$ (in which case for all $\sfz\in\cB_j$ we have $\sfz_2 \geq (1-o(1))N_n^{1/3}(\log L)^b $ vs.\ an interval length $\ell_{j+1}=N_n^{2/3}(\log L)^{-b}$), thereby establishing \cref{eq:hit-last-point} and concluding the proof of \cref{clm:prob-hitting-targets}.
\end{proof}

Equipped with the above claim, we can now apply the Domain Markov property after every $\ell_j$ steps of the random walk, and use the fact that the intersection over $j$ of the events of the form $\sfS(t)_2 \in [\tfrac12\sfz_2,\tfrac32\sfz_2]$ imply $\Low \cap \cG^\sqcup_*$, so that \cref{clm:prob-hitting-targets} obtains that 
\[\hat\bP_\sfu\big(\Low, \cG^\sqcup_*, \sfv \in \Cpts(\Gamma) \mid \Bdd\big)) \geq e^{-c(m+m'+k')}N_n^{-2/3} \geq e^{-c'(\log L)^a}\,,\]
for some constant $c' = c'(\beta) > 0$ when $a > 3b/2$.
This completes the proof of \cref{lem:dropping}.
\end{proof}

\begin{remark}
    \label{rem:fs-strategy}
    At this point, we have established that (a) w.h.p., when looking at the rectangle delimiting an interval of length $ N_n^{2/3}(\log L)^a$ centered on the bottom boundary, the intersection of~$\Gamma$ with its left and right sides, denoted $\bar\sfu,\bar\sfv$ respectively, is at distance at most $K N_n^{1/3}$ from the bottom boundary; 
    (b) the law of $\Gamma$ in said rectangle is that of a 2\Dim random walk $\sfS(\cdot)$ (cone-points), with i.i.d.\ animal decorations between every two cone-points, started at $\bar\sfu$ and conditioned (i) to hit~$\bar\sfv$, and (ii) that $\fL_n$, a subset of $\Gamma$ stays at nonnegative heights, thereafter tilted by $\exp(-\cA^\downarrow_F(\Gamma)/N_n)$.
    The subtle point is that the marginal on cone-points is not simply a 2\Dim random walk bridge conditioned to be nonnegative (and tilted by the area): the event that $\fL_n$ is nonnegative is not measurable w.r.t.\ the marginal on cone-points, and could potentially have delicate pinning effects. Further disruptive is the area term: without random walk estimates, we cannot control the area and rule out the case where it pushes $\Gamma$ towards the boundary.    
    We wish to rule out these scenarios: indeed, if this had been a standard random walk, entropic repulsion would repel it to height $N_n^{\delta}$, where conditioning on $\fL_n$ or $\Cpts(\Gamma)$ being nonnegative is basically the same. The strategy is as follows:
\begin{enumerate}[1.]
    \item Show (in \cref{lem:drop-to-Kepsilon}) that at $\pm(\log\log L)N_n^{2/3}$ the height of $\Gamma$ is at most $K_\epsilon N_n^{2/3}$ with probability at least $1-\epsilon$. (This would have followed from \cite[\S6.6 and Prop.~6.2]{IOSV21} if we had a random walk on cone-points, but the conditioning on nonnegative $\fL_n$ vs.\ cone-points precludes that. Also note that the effect of $\fL_n$ vs.\ cone-point conditioning could have alternatively been addressed by \cite[Proposition 13]{IST15}, but the area tilt precludes that approach as well.)
    \item Establish (\cref{lem:repel}) that, without an area term, in the above interval, $\Gamma$ stays above height $N_n^{\delta}$ away from its two endpoints, except with probability $O(N_n^{-\delta})$.
    \item Deduce the same (\cref{lem:repel-with-area}) for the area-tilted $\Gamma$ by showing that the area term  $\tfrac{\cA^\downarrow_F(\Gamma)} {N_n}$ in the above interval is $O((\log \log L)^{3/2})$.
\end{enumerate}
At the end of these steps, we arrive at a 2\Dim random walk on the interval $J$ (we may restrict attention to the cone-points), with endpoints at height at most $K_\epsilon N_n^{1/3}$, conditioned to be above $N_n^\delta$ and tilted by the area delimited by the random walk (rescaled by $1/N_n$).
\end{remark}
   
Thanks to \cref{lem:dropping}, we move to $\bP_{\bar\sfu, \bar\sfv}(\cdot \mid \cG^\sqcup, \Bdd)$ with $\bar\sfu_1 \in [-\tfrac12N_n^{2/3}(\log L)^a, -\tfrac14N_n^{2/3}(\log L)^a]$, $\bar\sfv_1 \in [\tfrac14N_n^{2/3}(\log L)^a, \tfrac12N_n^{2/3}(\log L)^a]$, and $\bar\sfu_2, \bar\sfv_2 \leq KN_n^{1/3}$. We next show that we can find two cone-points $\bar\sfu', \bar\sfv'$ at the $N_n^{2/3}$ scale (moving inwards from $N_n^{2/3}(\log L)^a$) at height at most $O(N_n^{1/3})$. Define first the event
\[\cG^\sqcup_{\Cpts} := \{\text{All cone-points of $\Gamma$ have height $\geq 0$}\}\,.\] We start with a tightness result from \cite{IOSV21}. 

\begin{proposition}[{\cite[Section 6.6]{IOSV21}}]\label{prop:tightness}
    For every $\epsilon > 0$, there exists a constant $K_\epsilon > 0$ such that the following holds. For any $\sfu, \sfv$ such that $\sfu_2, \sfv_2 \leq CN_n^{1/3}$, and for every fixed vertical line $x = m$ in between $\sfu, \sfv$, the height of (the linear interpolation between points of) $\sfS$ at $x = m$ is at most $K_\epsilon N_n^{1/3}$ except with probability $\epsilon$ under $\bP_{\sfu, \sfv}(\cdot \mid \Bdd, \cG^\sqcup_{\Cpts})$. 
\end{proposition}
\begin{proof}
    The measure $\bP_{\sfu, \sfv}(\cdot \mid \Bdd, \cG^\sqcup_{\Cpts})$ is exactly an area tilted 2\Dim random walk conditioned to stay above height 0. This tightness result was shown in \cite[Section 6.6]{IOSV21}, noting Proposition 6.2 there in order to extend the result to scales larger than $N_n^{2/3}$ (e.g., applying to the case when $\sfv_1 - \sfu_1 = O(N_n^{2/3}(\log L)^{25})$).
\end{proof}

\begin{lemma}\label{lem:drop-to-Kepsilon}
    Let $\cE(K)$ be the event that at the first cone point after $-(\log\log L)N_n^{2/3}$ and the last cone point before $(\log\log L)N_n^{2/3}$, the height of $\gamma$ is at most $KN_n^{1/3}$. Then, for any $\epsilon > 0$, there exists a $K_\epsilon$ such that
    \[\bP_{\bar\sfu, \bar\sfv}(\cE(K_\epsilon) \mid \Bdd, \cG^\sqcup) \geq 1 - \epsilon\,.\]
\end{lemma}
\begin{proof}
    It suffices to prove that under $\bP_{\bar\sfu, \bar\sfv}$, the first cone-point after $-(\log\log L)N_n^{2/3}$ has height $\leq K_\epsilon N_n^{2/3}$ except with probability $\epsilon/2$. Fix any $\delta < 1/3$, and let $\sfu^{(L)}$ be the last cone-point of $\Gamma$ such that $\sfu^{(L)}_1 \leq -(\log\log L)N_n^{2/3}$ (``to the left'' of $-(\log\log L)N_n^{2/3}$) and $\sfu^{(L)}_2 \leq N^\delta$. If there is no such point, define $\sfu^{(L)} = \bar\sfu$. Similarly, let $\sfv^{(R)}$ be the first cone-point of $\Gamma$ with $\sfv^{(R)}_1\geq -(\log\log L)N_n^{2/3}$ (``to the right'' of $-(\log\log L)N_n^{2/3}$) such that $\sfv^{(R)}_2 \leq N_n^\delta$. If there is no such point, define $\sfv^{(R)} = \bar\sfv$. Then, the law of the cone-points of $\Gamma$ in between $\sfu^{(L)}$ and $\sfv^{(R)}$ under the measure $\bP$ is exactly the area tilted 2\Dim random walk from $\sfu^{(L)}$ to $\sfv^{(R)}$  conditioned to stay above $N_n^\delta$ (since the event $\cG^\sqcup$ is implied by the cone-points staying above $N_n^\delta$ and $\Bdd$, and hence can be ignored in the conditioning). We can now conclude by \cref{prop:tightness} after translating vertically by $N_n^\delta$.
\end{proof}

By \cref{lem:drop-to-Kepsilon}, we can now move to $\bP_{\bar\sfu', \bar\sfv'}(\cdot \mid \cG^\sqcup, \Bdd)$, where $\bar\sfu'_1 = -(\log\log L)N_n^{2/3} \pm (\log L)^{50}$, $\bar\sfv'_1 = (\log\log L)N_n^{2/3} \pm (\log L)^{50}$, and $\bar\sfu'_2, \bar\sfv'_2 \leq K_{\epsilon}N_n^{1/3}$. Before we can appeal to the convergence to Ferrari--Spohn shown in \cite{IOSV21}, we still need to handle the issue of having a floor with respect to $\fL_n$ vs. a floor with respect to $\Cpts(\Gamma)$. Furthermore, we want to show that even though $\bar\sfu'_2, \bar\sfv'_2$ are potentially lower than $O(N^{1/3})$, the entropic repulsion will bring the height back to $O(N^{1/3})$ on an interval $[-TN^{2/3}, TN^{2/3}]$. 

Up until \cref{lem:repel-with-area}, we now forget the area tilt. The next proposition from \cite[Proposition 13]{IST15} deals with the fact that our floor condition is with respect to $\fL_n$ and not $\Cpts(\Gamma)$ when there is no area tilt, capturing the fact that $\Gamma$ quickly repels away from the floor after which point the two events are essentially the same. The statement in \cite{IST15} assumed further that $\sfu$ and $\sfv$ are on the floor at height 0, but as commented after \cite[Theorem 3]{IST15}, such repulsion results are only easier when the end-points are away from the floor. See also the comment before \cref{prop:IST} regarding Ising polymers vs. disagreement polymers.
\begin{proposition}[{\cite[Proposition 13]{IST15}}]\label{prop:cone-points-vs-animal-floor}
    There exists a constant $c(\beta) > 0$ such that for all $\sfu, \sfv$ with $\sfu_2, \sfv_2 \geq 0$,
    \[\hat\bP_{\sfu, \sfv}(\gamma \subset \mathbb H_+ \mid \cG^\sqcup_{\Cpts}) \geq c(\beta)\,.\]
    In particular, since $\gamma \subset \mathbb H_+$ implies $\fL_n \subset \mathbb H_+$, we have
    \[\hat\bP_{\sfu, \sfv}(\cG^\sqcup \mid \cG^\sqcup_{\Cpts}) \geq c(\beta)\,.\]
\end{proposition}

In order to show entropic repulsion in the measure $\hat\bP_{\bar\sfu', \bar\sfv'}(\cdot \mid \cG^\sqcup, \Bdd)$, we will use the following random walk estimates from \cite{IOVW20}, recorded here for convenience.
\begin{theorem}[{\cite[Thm~5.1]{IOVW20}}]\label{thm:iovw-rw-hitting}
Consider the 2\emph{D} effective random walk $\sfS(\cdot)$. Let $\tau_G$ be the hitting time of a set $G\subset \Z^2$. There exists a constant $C>0$ such that, for every sequence $\delta_\ell$ that goes to~$0$ arbitrarily slowly as $\ell\to\infty$, and for all $u,v\in \Z$ such that $u/\sqrt{\ell},v/\sqrt{\ell}\in (0,\delta_\ell)$,    
\begin{equation*}
\hat\bP_{(0,u)}(\tau_{(\ell,v)}<\tau_{\mathbb H_-}<\infty\mid \Bdd) \sim C \frac{h^+(u)h^-(v)}{\ell^{3/2}}\,, 
\end{equation*}
where $h^\pm$ are positive harmonic (in particular, asymptotically linear) functions. 
\end{theorem}

Applying the above for $\ell = N_n^{2/3}(\log\log L)$, we show entropic repulsion to $N_n^{\delta}$. 
\begin{lemma}\label{lem:repel}
    Fix a constant $0 < \delta < 1/12$, and consider the event $\mathsf{Repel}$ that in the interval $J := [-(\log\log L)N_n^{2/3} + N_n^{8\delta}, (\log\log L)N_n^{2/3} - N_n^{8\delta}]$, all cone-points of $\Gamma$ lie above height $N_n^\delta$. Then, there exists $C > 0$ such that
    \[\hat\bP_{\bar\sfu', \bar\sfv'}(\mathsf{Repel} \mid \cG^\sqcup, \Bdd) \geq 1 - CN_n^\delta\,.\]
\end{lemma}
\begin{proof}[Proof of \cref{lem:repel}]
    
    By \cref{prop:cone-points-vs-animal-floor} (and \cref{lem:bdd}) it suffices to prove that 
    \[\hat\bP_{\bar\sfu', \bar\sfv'}(\mathsf{Repel}^c \mid \cG^\sqcup_{\Cpts}, \Bdd) \leq C'N_n^\delta\,.\]
    Set $\ell = N_n^{2/3}(\log\log L)$. Since $\bar\sfu_2', \bar\sfv_2' \leq K_\epsilon N_n^{1/3} = o(\sqrt{\bar\sfv_1' - \bar\sfu_1'})$, we are in the regime of \cref{thm:iovw-rw-hitting}.
    We follow the computation of \cite[Lemma 3.6]{IOVW20} (except simpler in our case, as we have already moved to the setting of the effective random walk). We can sum over the first cone-point $\sfz$ below height $\ell^\delta$ in the interval $J$ and apply random walk estimates, obtaining
    \begin{align*}
        \hat\bP_{\bar\sfu', \bar\sfv'}(\mathsf{Repel}^c \mid \cG^\sqcup_{\Cpts}, \Bdd) &\leq \sum_{\sfz_1 = -\ell + \ell^{8\delta}}^{\ell - \ell^{8\delta}}\,\sum_{\sfz_2 = 0}^{\ell^\delta}\hat\bP_{\bar\sfu', \bar\sfv'}(\tau_{\sfz}<\infty \mid \cG^\sqcup_{\Cpts}, \Bdd)\\
        &=\sum_{\sfz_1 = -\ell + \ell^{8\delta}}^{\ell - \ell^{8\delta}}\,\sum_{\sfz_2 = 0}^{\ell^\delta}\frac{\hat\bP_{\bar\sfu'}(\tau_\sfz<\tau_{\bar\sfv'} < \tau_{\H_-} < \infty \mid \Bdd)}{\hat\bP_{\bar\sfu'}(\tau_{\bar\sfv'} < \tau_{\H_-} < \infty \mid \Bdd)}\\
        &= \sum_{\sfz_1 = -\ell + \ell^{8\delta}}^{\ell - \ell^{8\delta}}\,\sum_{\sfz_2 = 0}^{\ell^\delta} \frac{\hat\bP_{\bar\sfu'}(\tau_{\bar\sfz} < \tau_{\H_-} < \infty \mid \Bdd)\bP_{\sfz}(\tau_{\bar\sfv'} < \tau_{\H_-} < \infty \mid \Bdd)}{\bP_{\bar\sfu'}(\tau_{\bar\sfv'} < \tau_{\H_-} < \infty \mid \Bdd)}\\
        &\leq C\ell^{3/2 + \delta}\sum_{\sfz_1 = -\ell + \ell^{8\delta}}^{\ell - \ell^{8\delta}}\frac{\ell^\delta}{(\sfz_1 +\ell)^{3/2}}\cdot \frac{\ell^\delta}{(\ell - \sfz_1)^{3/2}}\\
    &\leq 2C\ell^{3\delta}\sum_{\sfz_1 = -\ell + \ell^{8\delta}}^{0} \frac{1}{(\sfz_1 +\ell)^{3/2}}\cdot \frac{1}{(1 - \tfrac{\sfz_1}{\ell})^{3/2}}\\
    &\leq 2C\ell^{3\delta}\sum_{\sfz_1 = -\ell + \ell^{8\delta}}^{0} \frac{1}{(\sfz_1 +\ell)^{3/2}}\\
    &\leq4C\ell^{3\delta}\ell^{-4\delta} = 4C\ell^{-\delta}\,.
    \end{align*}
    The first two lines follow by a union bound and conditional probability, noting that the event $\tau_{\H_-} < \infty$ occurs with probability 1 as the random walk has no drift in the $y$-coordinate. The third line follows by the markov property, and the fourth line uses \cref{thm:iovw-rw-hitting}. The rest of the computations are purely algebraic, requiring only that $\delta < 1/8$.
\end{proof}

We can now add back the area tilt to obtain entropic repulsion in the measure $\bP_{\bar\sfu', \bar\sfv'}(\cdot \mid \cG^\sqcup, \Bdd)$.
\begin{lemma}\label{lem:repel-with-area}
    For the event $\mathsf{Repel}$ defined in \cref{lem:repel},
    \[\bP_{\bar\sfu', \bar\sfv'}(\mathsf{Repel} \mid \cG^\sqcup, \Bdd) \geq 1 - o(1)\,.\]
\end{lemma}
\begin{proof}
    We have \begin{equation}
        \label{eq:repel-CS}
        \bP_{\bar\sfu', \bar\sfv'}(\mathsf{Repel}^c \mid \cG^\sqcup, \Bdd) = \frac{\hat\bE_{\bar\sfu', \bar\sfv'}[e^{-\tfrac{\cA^\downarrow_F(\Gamma)}{N_n}}\one_{\mathsf{Repel}^c} \mid \cG^\sqcup, \Bdd]}{\hat\bE_{\bar\sfu', \bar\sfv'}[e^{-\tfrac{\cA^\downarrow_F(\Gamma)}{N_n}} \mid \cG^\sqcup, \Bdd]} \leq \frac{\hat\bE_{\bar\sfu', \bar\sfv'}[\one_{\mathsf{Repel}^c} \mid \cG^\sqcup, \Bdd]}{\hat\bE_{\bar\sfu', \bar\sfv'}[e^{-\tfrac{\cA^\downarrow_F(\Gamma)}{N_n}} \mid \cG^\sqcup, \Bdd]}
    \end{equation}
    The numerator is $\leq CN_n^\delta$ by \cref{lem:repel}. To lower bound the denominator, we wish to show that with at least constant probability in $\hat\bP_{\bar\sfu', \bar\sfv'}(\cdot \mid \cG^\sqcup, \Bdd)$, the maximum of $\Gamma$ is at most $O(N_n^{1/3}\sqrt{\log\log L})$, for then we will have a lower bound on the denominator by $e^{-c(\log\log L)^{3/2}}$, which is negligible as $e^{c(\log\log L)^{3/2}} \ll N_n^\delta$. By \cref{prop:cone-points-vs-animal-floor} again, we can lower bound this probability under $\hat\bP_{\bar\sfu', \bar\sfv'}(\cdot \mid \cG^\sqcup_{\Cpts}, \Bdd)$ instead. Yet this is now a 2\Dim random walk whose start and end-points are of a lower order than the square root of the horizontal distance between them. By \cite[Theorem 5.3]{IOVW20}, this converges to a Brownian excursion, where the bound on the maximum is a well known fact.
\end{proof}

\begin{proof}[Proof of \cref{thm:ub}]
    As per \cref{st-1-bdd,st-2-drop,st-3-drop-to-K,st-4-repel} in
    the strategy outlined at the beginning of this section, by \cref{lem:bdd,lem:dropping,lem:drop-to-Kepsilon,lem:repel-with-area}, we can move to the measure $\bP_{\bar\sfu', \bar\sfv'}(\cdot \mid \mathsf{Repel}, \cG^\sqcup, \Bdd)$. As noted in the proof of \cref{lem:repel-with-area}, we can reveal the first and last cone-points $\bar\sfu'', \bar\sfv''$ in $J$, and by the linearity of the cones we have $|\bar\sfu''_2 - \bar\sfu'_2|, |\bar\sfv''_2 - \bar\sfv'_2| \leq N^{8\delta}$. On the interval $J$ in between $\bar\sfu''$ and $\bar\sfv''$, the event $\mathsf{Repel} \cap \Bdd$ already implies $\cG^\sqcup$, so we have reduced to the measure $\bP_{\bar\sfu'', \bar\sfv''}(\cdot \mid \mathsf{Repel}, \Bdd)$. At this point, we have an area-tilted 2\Dim random walk bridge conditioned to stay above height $N_n^\delta$ with increment law $\P(x, y) = \P^{\sfh_{(1, 0)}}(X(\Gamma) = (x, y) \mid |(x, y)| \leq (\log L)^{50})$, as we have removed every constraint on $\Gamma$ that is not measurable with respect to $\Cpts(\Gamma)$. By the exponential tails on the increments in \cref{prop:OZ-normalization}, changing the increment law to $\P^{\sfh_{(1, 0)}}(X(\Gamma) = (x, y))$ only changes the law of the random walk in total variation by $e^{-c(\log L)^{50}}$. Moreover, the endpoints $\bar\sfu'', \bar\sfv''$ satisfy $\bar\sfv''_1 - \bar\sfu''_1 \geq TN^{2/3}$ for a  constant $T$ sufficiently large depending on $K_\epsilon$, and $\bar\sfu''_2, \bar\sfv''_2 \leq K_\epsilon N_n^{1/3}$. In particular, we have reduced to the starting point of \cite[Sec.~6]{IOSV21}, where it was shown that the area-tilted 2\Dim random walk bridge above a floor with increment law $\P^{\sfh_{(1, 0)}}(X(\Gamma) = (x, y))$ converges to $\mathsf{FS}_{\sigma}$ (see \cite[Eq.~(6.55)]{IOSV21}) in the limit $L\to\infty$ followed by $T\to\infty$.

Thus far we have  established that, w.h.p., for every fixed $n\geq 1$, the process $Y_n(t)$ is stochastically dominated by a process $Z_n(t)$ (the 2\Dim random walk with an area-tilt) that weakly converges to 
$\mathsf{FS}_{\sigma}$ as~$L\to\infty$. Moreover, as the lower level lines are w.h.p.\ below the randomly constructed $Q$, the proof in fact showed that, w.h.p., $Y_n(t) \preceq Z_n(t)$   conditionally on $\fL_k$ for $k > n$ (and in particular, conditionally on the rescaled portions of these level lines, $Y_k(t)$ for $k>n$). Hence, the stronger statement concerning the joint law follows from the next elementary fact.
\begin{fact}\label{fact:stoch-dom-indep}
    Suppose $Y_1, \ldots Y_m, Z_1, \ldots Z_m$ are $\R^d$-valued random variables such that $Y_n \preceq Z_n$ for each $n$. Suppose furthermore that for each $n$, conditional on $\{Y_k, k > n\}$, we have $Y_n \preceq Z_n$. Then, the joint law satisfies $(Y_1, \ldots Y_m) \preceq (Z'_1,\ldots,Z'_m)$ where $\{Z'_k\}_{n=1}^m$ are independent with $Z'_k \stackrel{\mathrm{d}}{=} Z_k$.
\end{fact}
Note that here we conditioned on the lower level lines $\fL_k$ for $k>n$; we cannot condition on the upper level lines $\fL_k$ for $k<n$ (unlike the proof of \cref{thm:lb}, the matching lower bound for \cref{thm:ub}, where  we condition on the upper level lines  but cannot condition on the lower ones). The reason for this is that, in our proof, we construct a rectangle $Q$ of height $N_n^{2/3}$ surrounding the typical location of $\fL_n$, and while it is positioned so as not to intersect the lower level line $\fL_{n+1}$, it will impede on the support of $\fL_{n-1}$. (The necessity for such a rectangle, as opposed to one of height (say)  $N^{1/3}(\log L)^c$, is that the latter  entails pinning issues with the top side of $Q$, and the present tools cannot resolve those when said boundary is random and wiggly.)
\end{proof}

\begin{remark}
In the proof of the convergence of the area-tilted 2\Dim random walk to $\mathsf{FS}_{\sigma}$ in \cite[Sec.~6]{IOSV21}, the initial hypothesis that the height of the endpoints (our $\bar\sfu''_2,\bar\sfv''_2$) is replaced by 
\cite[Eq.~(6.55)]{IOSV21}, requiring that they would belong to $ [c_\epsilon N_n^{1/3},K_\epsilon N_n^{1/3}] $. Said assumption qualifies for an application of a coupling tool (\cite[Prop.~6.2]{IOSV21}). At this stage, both here and in the setting of \cite{IOSV21}, the heights of the endpoints could potentially be $o(N_n^{1/3})$. This is not an issue for their argument, as the proof of \cite[Prop.~6.2]{IOSV21} remains valid even without the assumption of the $c_\epsilon N_n^{1/3}$ lower bound.

Alternatively, it is not difficult to show that one can move from the assumption  $\bar\sfu''_2,\bar\sfv''_2 \leq K_\epsilon N_n^{1/3}$ to  $\bar\sfu''_2,\bar\sfv''_2 \in [c_\epsilon N_n^{1/3}, K_\epsilon N_n^{1/3}]$ with probability at least $1-\epsilon$. Indeed, it suffices to consider $\sfu''$ (then applying the same argument to $\sfv''$ by a union bound).
Let $g_L^\epsilon$ be the maximum integer such that $\sfu''_2 \geq g_L^\epsilon$ with probability at least $1-\epsilon$. If $\liminf_{L\to\infty} g_L^\epsilon N_n^{-1/3} > 0$ then $\sfu''$ satisfies the sought condition (with $c_\epsilon$ as this $\liminf$ value). Otherwise, we may apply the same logic to the ``last'' cone-point $\sfu'''$ in the interval of length $N_n^{2/3}$ starting at $\sfu''$ (i.e., the cone-point $\sfx$ with maximizing $\sfx_1-\sfu''_1$ out of those with $\sfx_1-\sfu''_1 < N_n^{2/3}$), looking at the maximum integer $h_L^\epsilon$ such that $\sfu'''_2\geq h_L^\epsilon$ with probability at least $1-\epsilon$. If $\liminf_{L\to\infty} h_L^\epsilon N_n^{-1/3}>0$, we may use $\sfu'''$ as our desired endpoint. It remains to handle the case where $\sfu''_2 ,\sfu'''_2 < \delta_L N_n^{1/3}$ for $\delta_L=o(1)$ (for $\delta_L = (g_L^\epsilon \vee h_L^\epsilon) N_n^{-1/3}$). For such $\sfu'',\sfu'''$, by \cite[Thm.~5.3]{IOVW20}, the 2\Dim random walk with \emph{no area-tilt} converges to a Brownian excursion. Consequently, we claim that there exists some $c_\epsilon$ such that the cone-points at the middle of this interval---concretely, take the first cone-point $\sfw$ such that $\sfw_1-\sfu''_1 > N_n^{2/3}/2$---are at height at least $c_\epsilon N_n^{1/3}$ with probability at least $1-\epsilon$. Indeed, we may bound the complement as in \cref{eq:repel-CS}: the denominator is uniformly bounded away from zero, e.g., we may bound it from below by $\frac12 \exp(-A)$ where $A$ is the median of the area of a standard Brownian excursion $B_t$ ($0<t<1$), whereas the numerator is bounded by $\P(B_{1/2}<\epsilon)$.
\end{remark}

    \section{Lower bound}\label{sec:LB}
	
	Our goal in this section is the following bound on the distance of $\fL_n$ from the bottom side of $\Lambda$. 

\begin{theorem}\label{thm:lb}
In the setting of \cref{thm:1}, fix $n\geq 1$ and $K>0$, let $\rho_n(x)$ be the maximum vertical distance of~$\fL_n$ above $x+(\frac{L}2,0)$ for $-N_n^{2/3} \leq x\leq N_n^{2/3}$,
and set $\sigma^2>0$ as per \cref{def:sigma}. 
Then every weak limit point $\mathbf{Y}_n(t)$ of the process $Y_n(t):=N_n^{-1/3}\rho_n(t N_n^{2/3})$ ($t\in[-K,K]$), as $L\to\infty$, satisfies
\[ \mathbf{Y}_n \succeq \mathsf{FS}_{\sigma}\,. \]
Moreover, for every fixed $m$, every weak limit point $(\bY_n(t))_{n\leq m}$ of the processes $(Y_n(t))_{n\leq m}$ satisfies
\[ (\bY_n)_{n\leq m} \succeq \bigotimes_{n\leq m}\mathsf{FS}_{\sigma}\,.\]
\end{theorem}

Since $\rho_n(x)$ is a decreasing function of $\phi$, en route to proving \cref{thm:lb} in this section we may apply monotonicity arguments that are increasing (e.g., raising the heights of boundary vertices).

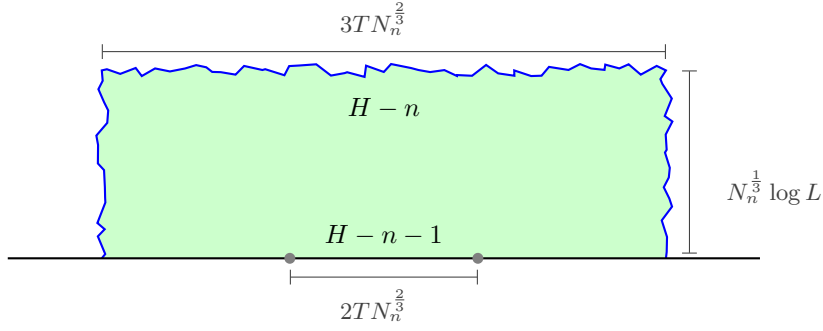
\begin{figure}
\centering
    \begin{tikzpicture}
    \begin{scope}[scale=1.25]
    \pgfmathsetseed{314159}
    \filldraw [thick, blue, fill=green!20] (0,0) to[move to] (6,0) decorate[decoration={random steps,segment length=1.5mm}]{-- (6,2)} decorate[decoration={random steps,segment length=1.5mm}]{-- (0,2)} decorate[decoration={random steps,segment length=1.5mm}]{-- (0,0)};
     
    \draw[thick, color=black] (-1,0)--(7,0);
    
    \draw[color=gray!50!black,|-|](6.25,0.05)--(6.25,2);
    \node[color=gray!50!black,font=\footnotesize] at (7.15,.75) {$N_n^{\frac13}\log L$};

    \draw[color=gray!50!black,|-|](0, 2.2)--(6,2.2);
    \node[color=gray!50!black,font=\footnotesize] at (2.9,2.55) {$3T N_n^{\frac23}$};

    \draw[color=gray!50!black,|-|](2, -0.2)--(4,-0.2);
    \node[color=gray!50!black,font=\footnotesize] at (2.9,-0.5) {$2TN_n^{\frac23}$};

    \node[font=\small] at (3,1.6) {$H+1-n$};
    \node[font=\small] at (3,0.25) {$H-n$};

    \node[circle,scale=0.4,fill=gray] at (2,0.) {};
    \node[circle,scale=0.4,fill=gray] at (4,0.) {};
    \end{scope}
    
    \end{tikzpicture}
    \caption{An instantiation of the domain $Q$ used in the proof of the lower bound on $\fL_n$. The bottom boundary of $Q$ coincides with the bottom boundary of $\Lambda$. The two gray points on the boundary mark the change in boundary conditions from $H-n$ to $H+1-n$. In contrast with \cref{fig:ub-rect}, the conditioning on $\phi_x \geq 0$ is present in all of $Q$.}
    \label{fig:lb-rect}
\end{figure}

Fix $T > K$ (eventually we will take $T \to \infty$). Let $R$ be the $3TN_n^{2/3} \times N_n^{1/3}(\log L)$ rectangle centered at $x = 0$ such that the bottom side  of $\partial R$ coincides with the bottom of $\partial \Lambda$. Let $A = (-TN_n^{2/3}, 0)$ and $B = (TN_n^{2/3}, 0)$. We start with a simpler analogue of \cref{lem:boundary-construction}, the proof of which we also postpone to \cref{sec:random-boundary}.

\begin{lemma}\label{lem:LB-boundary-construction}
        Fix $n \geq 1$. For $R, A, B$, as defined above, assume we know that w.h.p.\ under $\pi^0_{\Lambda}$, all of $R$ lies in the exterior of $\fL_{n-1}$ (for $n = 1$, there is no assumption, since in the setting of \cref{thm:1}, w.h.p.\ $\fL_0$ does not exist). Then, there exists a $\pi^0_{\Lambda}$-measurable distribution on connected regions $Q\subset R$ with marked boundary conditions $\xi$ satisfying  \cref{it:LB-Q-simp-conn,it:LB-dist-Q-R-bdy,it:LB-Q-bc-arc}
        below, such that the following holds. If $\cA_1$ is the area of the interior of $\fL_n$ intersected with~$Q$, and $\cA_2$ is the area above the $(H+1-n)$ level line in $Q$ under $\pi^\xi_Q$, then $\cA_1 \subset \cA_2$ w.h.p.\ under $\pi^0_{\Lambda}$.        
        \begin{enumerate}
            \item \label{it:LB-Q-simp-conn}$Q$ is simply connected,

            \item\label{it:LB-dist-Q-R-bdy} $\dist(\partial Q, \partial R) \leq \log L$, and the bottom side of $\partial Q$ coincides with the bottom side of $\partial R$,

            \item \label{it:LB-Q-bc-arc} The boundary conditions $\xi$ assigns height $H-n$ on the straight line between $A$ and $B$, and $H+1-n$ on the remainder of $\partial Q$.
        \end{enumerate}
    \end{lemma}

Let $Q, \xi$ be any domain and boundary condition satisfying \cref{lem:LB-boundary-construction} (see \cref{fig:lb-rect}). It now suffices to prove \cref{thm:lb} where $\fL_n$ is the $H+1-n$ level line in $\pi^\xi_Q$. We start by showing the lower bound for $\fL_1$ so that w.h.p.\ the $y$-coordinate of $\fL_1$ is $O(N_1^{1/3})$ above the interval $[-\tfrac32TN_2^{2/3}, \tfrac32TN_2^{2/3}]$. This justifies the assumption in \cref{lem:LB-boundary-construction} for $n = 2$, since $N_2^{1/3} \leq N_1^{1/3}e^{-c\beta\sqrt{\log L/\log\log L}}$ by \cref{eq:LD-ratio-inf}. We can then prove the lower bound for $\fL_2$, and proceed inductively. 

In the rest of this section, we will use ``height'' to refer to the vertical distance above the bottom of $\Lambda$. Let $\cG^\sqcap$ be the event that $\fL_n$ stays under the horizontal line $\cH$ at distance $2\log L$ below the lowest point on the top side of $Q$. Then, $\pi^\xi_Q \preceq \pi^\xi_Q(\cdot \mid \cG^\sqcap)$. We can now  follow the same preliminary steps as in \cref{sec:UB} to move to a polymer model. By \cref{lem:cluster-expansion-likely}, we can condition further on the event $\cG_1$ defined there, so that by \cref{prop:CE-law-with-floor} (and \cref{eq:weight-of-q-E*}) we have
\begin{align}\label{eq:LB-repeat-qn}\pi^\xi_Q(\gamma \mid \cG^\sqcap, \cG_1) &\propto \exp\bigg(-\sE^*_\beta(\gamma) + \frac{\cA_Q(\gamma)}{N_n}+ \fI_Q(\gamma) \bigg) = q^n_{Q; Q, 1}(\gamma)\,.
\end{align}
By \cref{eq:good-event-CE-G'} we can move to the polymer model with weights $q^n_{Q; Q, 1}$ and partition function
\[Z^n_{Q, Q; Q, 1}(A, B \mid \cG^\sqcap) = \sum_{\gamma \in \cP_Q(A, B) \cap \cG^\sqcap} q^n_{Q; Q, 1}(\gamma)\,.\]

Forgetting the area term for now, the first step is to control the effect of the interactions with the bottom of $Q$, effectively eliminating any potential pinning effects. This was shown in the half-space $\H_+$ in \cite{IST15} for Ising polymers. However, as mentioned in the end of the introduction there, the proof is more robust and allows for more complicated geometries in $\gamma$ (e.g. our connected components of bonds), and more complicated energies (e.g. our $\cE^*_\beta(\gamma)$), as long as the polymer model in question features the Ornstein--Zernike results found in \cref{lem:cpts-length,prop:OZ-normalization}. The particular structure of the Ising polymer model is only used to show that the increment measure has $1-\epsilon_\beta$ mass on three basic irreducible components (see \cite[Definition 18]{IST15}), which simplifies the proof of certain random walk estimates. The same is true for the class of polymer models considered in this paper, for the same reason that allowed us to transfer the results of \cite{DKS92} in \cref{sec:geom-disagreement-polymers}: the assumption in \cref{eq:energy-property} that $e^{-\cE^*_\beta(\gamma)}$ decays exponentially in $\beta\sN(\gamma)$ implies that the lowest energy disagreement polymers coincide with those of the Ising polymer setting, and the increase in entropy of other polymers is negligible compared to their energy.
\begin{proposition}[{\cite[Theorem 2]{IST15}}]\label{prop:IST}
    There exists a constant $C(\beta)>0$ such that
    \[\tfrac{1}{C(\beta)}\hatZ^n_{\H_+,\Z^2}(A, B) \leq \hatZ^n_{\H_+,\H_+}(A, B) \leq C(\beta)\hatZ^n_{\H_+, \Z^2}(A, B)\,.\]
\end{proposition}
In our setting, we will need the following corollary.

\begin{corollary}\label{cor:move-to-H-H}
For the above defined $Q$ and endpoints $A,B$, one has
\[    \hatZ^n_{Q, Q}(A, B \mid \cG^\sqcap) = (1+o(1))\hatZ^n_{\H_+, \H_+}(A, B)\,.
\]
\end{corollary}
\begin{proof}
    First note that by a simple Peierls argument (mapping the polymer to the straight line between $A,B$), we can rule out the event that $|\gamma| > 1.1|B-A|$ regardless of the domain or interaction with the boundary. In particular, this implies that
\begin{equation*}
        \hatZ^n_{\H_+, \H_+}(A, B \mid \cG^\sqcap) = (1+o(1))\hatZ^n_{Q, \H_+}(A, B \mid \cG^\sqcap)\,,
    \end{equation*}
    and by the decay of the $\Phi$ functions (since by the above Peierls argument we can assume on $\cG^\sqcap$ that $\gamma$ stays $\log L$ away from the top and sides of $\partial Q$, and the bottom side of $\partial Q$ coincides with $\partial H$), we also have
    \begin{equation*}
        \hatZ^n_{Q, \H_+}(A, B \mid \cG^\sqcap) = (1+o(1))\hatZ^n_{Q, Q}(A, B \mid \cG^\sqcap)\,.
    \end{equation*}
    Combining the last two displays results in
    \begin{equation}
    \label{eq:change-Q-H}
        \hatZ^n_{Q, Q}(A, B \mid \cG^\sqcap) = (1+o(1))\hatZ^n_{\H_+, \H_+}(A, B \mid \cG^\sqcap)\,.
    \end{equation}
    
Next, to eliminate the conditioning on $\cG^\sqcap$, we turn to $\hatZ^n_{\H_+, \Z^2}$ and argue that
    \begin{equation}\label{eq:ignore-soft-ceiling}
        \hatZ^n_{\H_+, \Z^2}(A, B \mid \cG^\sqcap) \geq (1-o(1))\hatZ^n_{\H_+, \Z^2}(A, B)\,.
    \end{equation}
    Indeed, suppose first that the domain restriction induces only that the cone-points have nonnegative heights (instead of also forcing all of $\gamma$ to be nonnegative). Then, \cite[Theorem 5.3]{IOVW20} shows that $\Gamma$ with weight $\hatq^n_{\Z^2}$ and partition function $\hatZ^n_{\Z^2, \Z^2}(A, B \mid \text{height}(\Cpts(\Gamma)) \geq 0)$ converges weakly to a Brownian excursion upon rescaling, whence the probability of reaching height $N_n^{1/3}(\log L)$ is $o(1)$ by standard estimates on the Brownian excursion. We can then use \cref{prop:cone-points-vs-animal-floor} to translate the same result to our setting with a partition function of $\hatZ^n_{\H_+, \Z^2}$, proving \cref{eq:ignore-soft-ceiling}.  
    We now deduce the analogue of \cref{eq:ignore-soft-ceiling} for interactions in $\H_+$ as opposed to $\Z^2$:
    \begin{equation}\label{eq:ignore-soft-ceiling-H}
        \hatZ^n_{\H_+, \H_+}(A, B \mid \cG^\sqcap) \geq (1-o(1))\hatZ^n_{\H_+, \H_+}(A, B)\,.
    \end{equation}
    Indeed, recall that \cref{lem:rare-events-changing-Z} implies that events that have a probability of $o(1)$ in the measure associated to  $\hatZ^n_{\H_+, \H_+}(A, B)$ also have probability $o(1)$ in the analogue for $\hatZ^n_{\H_+, \Z^2}(A, B)$ as long as the two partition functions are up to a multiplicative constant apart; \cref{prop:IST} provides exactly that hypothesis. 
    Since $\hatZ_{\H_+,\H_+}^n(A,B\mid\cG^\sqcap)$ is a subset of the sum in 
$\hatZ_{\H_+,\H_+}^n(A,B\mid\cG^\sqcap)$,  \cref{eq:ignore-soft-ceiling-H} yields
\[
        \hatZ^n_{\H_+, \H_+}(A, B \mid \cG^\sqcap) = (1-o(1))\hatZ^n_{\H_+, \H_+}(A, B)\,,
\]
which, after combining with \cref{eq:change-Q-H}, completes the proof.
\end{proof}

\begin{proof}[Proof of \cref{thm:lb}]
First observe that we can replace $\cA_Q(\gamma)$ with $-|D_1|$, the area below $\gamma$, in \cref{eq:LB-repeat-qn} by just a renormalization. Now we are in the same setting as in \cite{CKL24}, and we recall the proof of \cite[Theorems 7.1]{CKL24} to show the necessary adjustments. For convenience, let $\bP^n_{V, U}, \hat\bP^n_{V, U}$ denote the measures on $\Gamma$ associated to $Z^n_{V, U}, \hatZ^n_{V, U}$ respectively. For the case of the Ising polymer model with no area tilt $\hat\bP^n_{\H_+, \H_+}$, the authors prove that w.h.p.,
\begin{enumerate}
    \item\label{it:LB-length-cpts} $\Gamma$ has a linear length and number of cone-points (\cite[Lemma 5.7]{CKL24}, also proved here in \cref{lem:cpts-length}),
    
    \item\label{it:LB-bounded-components} The left, right, and irreducible components of $\Gamma$ have size no more than $(\log L)^2$ (\cite[Lemmas 5.5, 5.8]{CKL24}), and

    \item\label{it:LB-entropic-repulsion} The cone-points of $\Gamma$ stay above height $N_n^\delta$ in the interval $[A_1 +N_n^{4\delta}, B_2 - N_n^{4\delta}]$ (\cite[Lemma 5.11]{CKL24}). 
\end{enumerate}
The proof of the above inputs relied only on results from \cite{IST15} and the product structure from the Ornstein--Zernike theory, both of which we have in our disagreement polymer model, whence their results translate immediately to our setting in $\hat\bP^n_{\H_+, \H_+}$ for general disagreement polymers. We can then use \cref{cor:move-to-H-H} (with \cref{lem:rare-events-changing-Z}) to obtain that the same results hold w.h.p.\ in $\hat\bP^n_{Q, Q}(\cdot \mid \cG^\sqcap)$.

To reintroduce the area tilt, we first replace $\cA_Q(\gamma)$ in \cref{eq:LB-repeat-qn} by $-|D_1|$ by just a renormalization, which does not change the above measures. Then, the same argument leading to \cref{eq:ignore-soft-ceiling,eq:ignore-soft-ceiling-H} (i.e., using convergence to Brownian excursion and then \cref{prop:IST}) proves that for some constant $c = c(T) > 0$,
\[\hat\bE^n_{\H_+, \H_+}[e^{-|D_1|/N_n}] \geq c\,.\]
(See also \cite[Claim 7.4]{CKL24} for more details.\footnote{More work had to be done in \cite{CKL24} to show convergence to a Brownian excursion because the 2\Dim random walk there was not symmetric in the $y$-coordinate, and hence the authors could not directly apply \cite[Theorem 5.3]{IOVW20}. Our setting does feature this symmetry, so we cite \cite{IOVW20} for the convergence.}) Again \cref{cor:move-to-H-H} and \cref{lem:rare-events-changing-Z} implies the same lower bound with respect to $\hat\bE^n_{Q, Q}(\cdot \mid \cG^\sqcap)$. This in turn implies that the events in \cref{it:LB-length-cpts,it:LB-bounded-components,it:LB-entropic-repulsion} hold w.h.p.\ in $\bP^n_{Q, Q}(\cdot \mid \cG^\sqcap)$, as per a computation analogous to \cref{eq:repel-CS}. Finally, as in \cref{rem:area-difference}, this allows us to replace $|D_1|$ with the area below the linear interpolation of the cone-points of $\Gamma$ at the cost of a multiplicative $1+o(1)$ to the weight. 

The consequence of all of the above is that in $\bP^n_{Q, Q}(\cdot \mid \cG^\sqcap)$, upon looking at the first and last cone-points $\sfu, \sfv$ in the interval $[A_1 + N_n^{4\delta}, B_1 - N_n^{4\delta}]$, the law of $\Cpts(\Gamma)$ in between $\sfu, \sfv$ can be coupled to an area-tilted 2\Dim random walk bridge conditioned to stay above height $N_n^\delta$, with increment law given by $\P(x, y) = \P^{\sfh_{(1, 0)}}(X(\Gamma) = (x, y))$. The convergence to the Ferrari--Spohn diffusion $\mathsf{FS}_{\sigma}$ in the limit $L \to \infty$ followed by $T \to \infty$ now follows by \cite[Sec. 6]{IOSV21}.

Thus, we have established that w.h.p., for every fixed $n\geq 1$, the process $Y_n(t)$ stochastically dominates a process $Z_n(t)$ (the 2\Dim random walk with an area-tilt) that weakly converges to 
$\mathsf{FS}_{\sigma}$. As noted in the induction below \cref{lem:LB-boundary-construction}, the upper level lines are w.h.p.\ above the randomly constructed $Q$, so the proof in fact shows that w.h.p., $Y_n(t) \succeq Z_n(t)$ conditionally on $\fL_k$ for $k < n$ (and in particular, conditionally on the rescaled portions of these level lines, $Y_k(t)$ for $k<n$). Hence, the stronger statement concerning the joint law follows from \cref{fact:stoch-dom-indep} (for the other direction of stochastic domination, which is easily deduced by multiplying the random variables by $-1$).

Note that this conditioning is in the reverse direction as in the proof of \cref{thm:ub}, and we cannot instead condition on the lower lines $\fL_k$ for $k > n$. Indeed, revealing $\fL_{n+1}$ reveals a wiggly boundary at $o(N_n^{1/3})$ with only size $\log L$ perturbations, which appears to be no different from the flat boundary we have in the construction of $Q$ at the macroscopic level. However, our proof relies strongly on the depinning proved in \cite{IST15}, and it is unclear how to extend those results to the case of a wiggly boundary.
\end{proof}
	
\section{Extension to all \texorpdfstring{$|\nabla\phi|^p$}{grad-phi} models}\label{sec:gradphi}
	In this section we prove \cref{thm:grad-phi-p}, extending \cref{thm:1} to $|\nabla\phi|^p$ models for fixed $p>1$. 
	
	\subsection{Large deviations}\label{subsec:ld-grad-phi}
	The following analogues of \cref{eq:LD-ratio-inf,eq:LD-inf,eq:LD-conditional-inf} for $1<p<2$ (\cite[Thm.~5.1]{LMS16}) and $2<p<\infty$ (\cite[Thm.~5.5]{LMS16}) are known:
	\begin{align}
		\frac{\hatpi^{(p)}_\infty(\phi_o = h)}{\hatpi^{(p)}_\infty(\phi_o = h-1)} &\leq \exp\Big[-c \beta h^{(p-1)\,\wedge\, 1}\Big]\,,\label{eq:gradphi-LD-ratio-inf}\\
		\exp\Big[-c_1 \beta h^{p \,\wedge\, 2}\Big] \leq \hatpi^{(p)}_\infty(\phi_o = h) &\leq \exp\Big[-c_2\beta h^{p \,\wedge\, 2}\Big]\,,\label{eq:gradphi-LD-inf}\\        
		\hatpi^{(p)}_\infty(\phi_z = h \mid \phi_o = h) &\leq \exp\Big[-c \beta h^{(p-1)\,\wedge\, 1}\Big]\,.\label{eq:gradphi-LD-conditional-inf}
	\end{align}
	
	As was the case for $p=2$, our proof will require a refined version of \cref{eq:gradphi-LD-conditional-inf}, namely that one appearing below in \cref{eq:gradphi-LD-conditional} below. Let $\cB_R(x)$ be the ball of radius $R$ centered at the site $x$. 
\begin{theorem}\label{thm:gradphi-large-deviations}
    There exists a constant $C=C(p)>0$ such that the following holds for any domain $V$ that contains $\cB_R(o)$, where $R = Ch^{p-1}$ for $1 < p < 2$ and $R = Ch$ for $2 < p < \infty$. Let $z$ be such that $\cB_{R+1}(z) \subset V$. Then there exist absolute constants $c_0,c_1,c_2,c_3,c_4>0$ such that
    \begin{align}
		\exp\Big[-c_0 \beta h^{(p-1)\,\wedge\, 1}\Big] \leq \frac{\hatpi^{(p)}_V(\phi_o = h)}{\hatpi^{(p)}_V(\phi_o = h-1)} &\leq \exp\Big[-c_1 \beta h^{(p-1)\,\wedge\, 1}\Big]\,,\label{eq:gradphi-LD-ratio}\\
		\exp\Big[-c_2 \beta h^{p \,\wedge\, 2}\Big] \leq \hatpi^{(p)}_V(\phi_o = h) &\leq \exp\Big[-c_3\beta h^{p \,\wedge\, 2}\Big]\,,\label{eq:gradphi-LD}\\        
		\hatpi^{(p)}_V(\phi_z = h \mid \phi_o = h) &\leq \exp\Big[-c_4 \beta h^{p\,\wedge\,\frac{p}{p-1}}\Big]\,.\label{eq:gradphi-LD-conditional}
	\end{align}
\end{theorem}

\begin{remark}\label{rem:bad-set-calc}
In \cref{rem:gradphi-bad-set} we defined the set $\sB$ of exceptional values of $L$ (extending \cref{rem:Lh-B-concrete} for the \ZGFF) as $\bigcup_{h\geq 1}\llb \frac34 L_h,L_h\rrb$ where $L_h = \lceil 5\beta/\hatpi_\infty^{(p)}(\phi_o=h)\rceil$. For this choice, for every $h$ we have
$\sum_{k\in\sB\cap\llb 1,L_h\rrb]}\frac1k = O(h) = o(\log 
L_h)$---i.e., $\sB$ has zero logarithmic density---as \cref{eq:gradphi-LD} shows $\lim_{h\to\infty}-\frac1h \hatpi_\infty^{(p)}(\phi_o)=h) = \infty$. (Precisely, $\frac1{\log n}\sum_{k\in\sB\cap\llb 1,n\rrb}\frac1k=O((\log n)^{-(\frac{p-1}p\,\wedge\, \frac12)})$ for all $p>1$.)
\end{remark}

The extensions of the large deviation estimates in \cref{eq:gradphi-LD-ratio-inf,eq:gradphi-LD-inf,eq:gradphi-LD-conditional-inf} to a general domain $V$ containing $\cB_R(o)$ follow immediately from their proofs in \cite{LMS16}, which we summarize below. Looking at the $1 < p < 2$ case, the proof begins by showing that the outermost 1 level-line loop containing the origin (denoted $\Gamma_1$, not to be confused with the notation for animals used previously) has size at most $Ch^{p-1}$ except with probability $e^{-\beta h^{p-1}}$, via a Peierls map argument which holds in any domain $V$. Hence we obtain
\begin{claim}[extending {\cite[Lem.~5.3]{LMS16}}]\label{clm:extend-lem53}
    Suppose $\phi_o \geq 1$. Let $\Gamma_1$ be the outermost 1 level-line loop surrounding the origin $o$. There exists a constant $C = C(p) > 0$ such that for any domain $V$ containing $o$, 
    \[\hatpi_V^{(p)}(|\Gamma_1| > Ch^{p-1} \mid \phi_o \geq h) \leq e^{-\beta h^{p-1}}\,.\]
\end{claim}
This is used in \cite[Lem.~5.4]{LMS16} to show the large deviation rate by comparing to the real-valued energy minimizer $\phi^*$ in a similar (but simpler) strategy as in the $p = 2$ case. The lower bound of \cref{eq:gradphi-LD-inf} is established by taking the probability of a single $\phi$ which is an integer valued approximation to $h\phi^*$, and is supported on $\cB_{R - 1}(o)$, thus holding also on $\hatpi_V^{(p)}$. The upper bound of \cref{eq:gradphi-LD-inf} was proved by revealing $\Gamma_1$ and using monotonicity to set the boundary conditions outside $\Gamma_1$ to 0, and hence follows from \cref{clm:extend-lem53}. The proof of \cref{eq:gradphi-LD-ratio-inf} also uses the same revealing strategy, combined with a map argument which holds in any domain. Thus, we will have \cref{thm:gradphi-large-deviations} for $1 < p < 2$ once we prove the lower bound in \cref{eq:gradphi-LD-ratio} (stated as \cref{lem:ratio-lb-p12}) and the conditional probability bound in \cref{eq:gradphi-LD-conditional}\footnote{Here we mention how the proof of \cref{eq:gradphi-LD-ratio-inf,eq:gradphi-LD-inf} extend to $\hatpi^{(p)}_V$ for $1<p<2$; the same is valid for \eqref{eq:gradphi-LD-conditional-inf}, but we will need a stronger version of that inequality (namely, \cref{eq:gradphi-LD-conditional}) anyway.}.

For the $2 < p < \infty$ case, the lower bound of \cref{eq:gradphi-LD-inf} was again proven by taking a single $\phi^*$ with $\phi^*_o = h$ that has probability at least $e^{-c\beta h}$, which is already supported on $\cB_R(o)$. Both the upper bound of \cref{eq:gradphi-LD-inf} and the bound on the ratio of \cref{eq:gradphi-LD-ratio-inf} then apply a map argument, which holds in any domain, together with \cite[Claim~5.7]{LMS16}, which refers only to the energy of a set of level-line loops. Thus, \cref{thm:gradphi-large-deviations} for $2 < p < \infty$ will be established upon proving the lower bound in~\cref{eq:gradphi-LD-ratio} (stated as \cref{cor:ratio-lb-p2inf}) and \cref{eq:gradphi-LD-conditional}. One extra ingredient which will be needed for this is to bound the effective support of solutions to the large deviation problem (see \cref{lem:bd-outermost-loop}) in $\cB_{R}(o)$ where $R=C h$ for a sufficiently large absolute constant $C>0$. 

\begin{remark}\label{rem:gradphi-radius-req}
We emphasize that for extending the results of \cite{LMS16} for $2 < p < \infty$, only the lower bound of $\hatpi^{(p)}_V(\phi_o = h)$ requires a restriction on $V$, which is that it must contain the support of a particular $\phi^*$ defined as $\phi^*(x) = (h - \norm{x}_1) \vee 0$. In particular, we can take $R = \sqrt{2}h$, with no dependence on $\beta$. Then to prove the additional lower bound on \cref{eq:gradphi-LD} and the bound in \cref{eq:gradphi-LD-conditional}, we increase $R$ to some other constant $Ch$, as needed to capture the outermost 1 level-line loop. If instead $R$ had the form $C\beta h$, then our proof of the upper bound for points near the boundary of~$V$ in \cref{lem:gradphi-UB-LD-any-point} would fail.
\end{remark}
\begin{lemma}\label{lem:ratio-lb-p12}
    Fix $1 < p < 2$, and $V$ as in \cref{thm:gradphi-large-deviations}. There exists a constant $c > 0$ such that for all $h\geq 1$,
    \begin{equation*}
 \frac{\hatpi^{(p)}_V(\phi_o = h)}{\hatpi^{(p)}_V(\phi_o = h-1)}\geq \exp\big[-c\beta h^{p-1}\big]
 \,. 
    \end{equation*}
\end{lemma}
\begin{remark}
The proof of \cref{lem:ratio-lb-p12} is valid for all $p\geq1$, though its lower bound of $\exp[-c\beta h^{p-1}]$ is weaker than the matching upper bounds of $\exp(-c\beta h/\log h)$ for $p=2$ and $\exp(-c\beta h)$ for $p> 2$.
\end{remark}
\begin{proof}[Proof of \cref{lem:ratio-lb-p12}]
    It is easy to see that we have
    \[\hatpi^{(p)}_V(\min_{x \sim o}\phi_x < 0 \mid \phi_o = h) \leq \epsilon_\beta\,.\] Indeed, by the upper bound of \cref{eq:gradphi-LD-ratio}, we have that $\hatpi^{(p)}_V(\phi_o = h) = (1+o(1))\hatpi^{(p)}_V(\phi_o \geq h)$, so it suffices to prove the bound when the conditioning is instead on $\phi_o \geq h$. But by FKG, we can then forget about the conditioning entirely, whence the bound is immediate by a standard Peierls argument. 

    Now we can write
    \[\hatpi^{(p)}_V(\phi_o = h-1) \leq (1+\epsilon_\beta)\hatpi^{(p)}_V(\phi_o = h-1,\, \min_{x \sim o}\phi_x \geq 0)\,.\]
    On this latter event, we can raise the height of the origin by 1, mapping to the set $\{\phi: \phi_o = h\}$, and this changes the weight of $\phi$ by a factor of at most $e^{4\beta(h^p-(h-1)^p)} \leq e^{4\beta ph^{p-1}}$, so that
    \[\hatpi^{(p)}_V(\phi_o = h-1,\, \min_{x \sim o}\phi_x \geq 0) \leq e^{4\beta ph^{p-1}}\hatpi^{(p)}_V(\phi_o = h)\,.\qedhere\]
\end{proof}
We now look to prove the lower bound of \cref{eq:gradphi-LD-ratio} for $2 < p < \infty$. For each configuration $\phi$ with $\phi_o = h$, let $\Gamma_i(\phi)$ for $1 \leq i \leq h$ be the outermost $i$ level-line loop surrounding the origin $o$, dropping $\phi$ from the notation if it is clear from context. For a collection of such loops $\{\Gamma_i\}$, define the energy of the collection by \[\sE(\{\Gamma_i\}) = \sum_e \Delta_e^p\,,\]
where $\Delta_e$ denotes the number of $\Gamma_i$ in the collection that contain the dual bond $e$. Let the weight of $\Gamma_i$ with respect to $\{\Gamma_i\}$ be defined as
    \[\mathscr{W}(\Gamma_i) = \sum_{e \in \Gamma_i} \Delta_e^p/\Delta_e = \sum_{e \in \Gamma_i} \Delta_e^{p-1}\,.\]
    (The weight depends on the full collection $\{\Gamma_i\}$ and not just $\Gamma_i$, but we will omit this from the notation for simplicity as it will be obvious from the context.) Observe that $\sE(\{\Gamma_i\}) = \sum_i \mathscr{W}(\Gamma_i)$, and that $|\Gamma_i| \leq \mathscr{W}(\Gamma_i)$.
\begin{lemma}\label{lem:bd-outermost-loop}
    Fix $p\geq1$ and $V$ as in \cref{thm:gradphi-large-deviations}. Then, there exist constants $C, c > 0$ such that for all $h \geq 1$, $r \geq Ch$,
    \[\hatpi^{(p)}_V(\mathscr{W}(\Gamma_1) \geq r\mid \phi_o = h) < e^{-(\beta - c)(r - Ch)}\,.\]
\end{lemma}
\begin{proof}
    By a standard Peierls argument (enumerate over the $h$ level-line loops $\{\Gamma_i\}$, where a length-$k$ loop intersects (say) a horizontal ray from the origin at distance at most $k$),
    \[
    \hatpi^{(p)}_V(\sE(\{\Gamma_i\}) \geq C_1h^2 \,,\,\phi_o = h) = e^{-(C_1/2)\beta h^2}\,,
    \]
    and hence the lower bound on $\hatpi^{(p)}_V(\phi_o = h)$ in \cref{eq:gradphi-LD} implies that    \begin{equation}\label{eq:gradphi-E-bigger-h2}\hatpi^{(p)}_V(\sE(\{\Gamma_i\}) \geq C_1h^2 \mid \phi_o = h) = e^{-\beta h^2}\,
    \end{equation}
    provided that $C_1/2>c_2+1$ for $c_2$ from \cref{eq:gradphi-LD}.
    Hence, we can assume that $\sE(\{\Gamma_i\}) \leq C_1h^2$, whence by pigeonhole principle there exists some $\Gamma_{i^*}$ such that $\mathscr{W}(\Gamma_{i^*}) \leq C_1h$. Now consider the map $T$ on the set $\{\phi: \phi_o = h,\, \sE(\{\Gamma_i\}) \leq C_1h^2\}$ that takes the collection $\{\Gamma_i\}$ in $\phi$, deletes the outermost contour $\Gamma_1$, and adds a copy of $\Gamma_{i^*}$. This map preserves the fact that the origin is at height $h$. Calling this resulting collection $\{\Gamma_i(T(\phi))\}$, we obtain that
    \begin{align*}
    -\frac{1}{\beta}\log\frac{\hatpi^{(p)}_V(T(\phi))}{\hatpi^{(p)}_V(\phi)} = \sE(\{\Gamma_i(T(\phi))\}) - \sE(\{\Gamma_i(\phi)\}) &\leq -\mathscr{W}(\Gamma_1(\phi)) + \sum_{e \in \Gamma_{i^*}(\phi)} (\Delta_e+1)^p - \Delta_e^p\\
    &\leq -\mathscr{W}(\Gamma_1(\phi)) + \sum_{e\in \Gamma_{i^*}(\phi)} p(\Delta_e + 1)^{p-1}\\
    &\leq -\mathscr{W}(\Gamma_1(\phi)) + p2^{p-1}C_1h \\
    &=: -\mathscr{W}(\Gamma_1(\phi)) + C_2h\,.
    \end{align*}
    
    Moreover, the multiplicity of this map is such that for any given $\psi$ in the image of $T$,
    \[|\{\phi: T(\phi) = \psi, \mathscr{W}(\Gamma_1(\phi)) = k\}| \leq hks^k\,,\]
    for some absolute constant $s > 0$,
    since given the image $\psi$ we can read off $T(\{\Gamma_i(\phi)\})$, and reconstruct $\{\Gamma_i(\phi)\}$ if we know what $i^*$ and $\Gamma_1$ are. There are $h$ choices of $i^*$, $k$ choices for the length $|\Gamma_1| = j \leq k$, and $s^j \leq s^k$ choices of loops of length $j$ which surround the origin. Hence, if $\mathscr{W}(\Gamma_1) \geq r \geq Ch$, we can apply a Peierls argument with this map to obtain that
    \begin{align}\label{eq:gradphi-contour-Bh}
        \sum_{\substack{\phi: \phi_o = h,\\ \sE(\{\Gamma_i\}) \leq C_1h^2,\\
        \mathscr{W}(\Gamma_1) \geq r}} \hatpi^{(p)}_V(\phi) &\leq \sum_{k \geq r}\sum_{\psi \in \mathsf{Im}(T)}\sum_{\substack{\phi: \phi_o = h,\\ \sE(\{\Gamma_i\}) \leq C_1h^2,\\
        \mathscr{W}(\Gamma_1) = k,\\
        T(\phi) = \psi}} \hatpi^{(p)}_V(\psi) e^{-\beta(k - C_2h)}\nonumber\\
        &\leq \sum_{k \geq r}\sum_{\psi \in \mathsf{Im}(T)}\hatpi^{(p)}_V(\psi)e^{-\beta(k - C_2h)}hks^k\nonumber\\
        &\leq e^{-(\beta - C_3)(r - C_2h)}\hatpi^{(p)}_V(\phi_o = h)\,.
    \end{align}
    Combining \cref{eq:gradphi-E-bigger-h2,eq:gradphi-contour-Bh} concludes the proof.
\end{proof}

The constant $C$ used for $R = Ch$ in \cref{thm:gradphi-large-deviations} should now be taken to be the maximum between $2C$ from \cref{lem:bd-outermost-loop} and the constant needed previously from \cite{LMS16} (which was just to ensure that the single $\phi$ used to lower bound \cref{eq:gradphi-LD-inf} is supported in $\cB_R(o)$).

\begin{corollary}\label{cor:ratio-lb-p2inf}
    Fix $2 < p < \infty$ and let $V$ be as in \cref{thm:gradphi-large-deviations}. There exists a constant $c > 0$ such that for all $h \geq 1$,
    \begin{equation*}
 \frac{\hatpi^{(p)}_V(\phi_o = h)}{\hatpi^{(p)}_V(\phi_o = h-1)}\geq \exp\big[-c\beta h\big] \,. 
    \end{equation*}
\end{corollary}
\begin{proof}
    By \cref{lem:bd-outermost-loop}, there exists $C$ such that 
    \[\hatpi^{(p)}_V(\phi_o = h-1) = (1+o(1))\hatpi^{(p)}_V(\phi_o = h-1,\,  \mathscr{W}(\Gamma_1) \leq Ch)\,.\]
    On the latter event, consider the map $T(\phi) = \psi$ which raises the interior of $\Gamma_1(\phi)$ by 1. Then, we have 
    \begin{align*}
    -\frac{1}{\beta}\log\frac{\hatpi^{(p)}_V(T(\phi))}{\hatpi^{(p)}_V(\phi)} &= \sum_{e \in \Gamma_1(\phi)} (\Delta_e + 1)^p - \Delta_e^p\\
    &\leq p2^{p-1}\mathscr{W}(\Gamma_1(\phi))\,.
    \end{align*}
    As this map is injective and maps into the set $\{\phi: \phi_o = h\}$, we obtain
    \[\hatpi^{(p)}_V(\phi_o = h-1,\,  \mathscr{W}(\Gamma_1) \leq Ch) \leq \hatpi^{(p)}_V(\phi_o = h)e^{-\beta C'h}\]
    for $C' = p2^{p-1}C$.
\end{proof}

Finally, to show the bound on the conditional probability in \cref{eq:gradphi-LD-conditional}, we begin with a simple lemma comparing $\hatpi_V^{(p)}$ to $\hatpi_{\cB_R(o)}^{(p)}$.
\begin{lemma}\label{lem:gradphi-hatpi-V-vs-Br}
    Let $V \supset \cB_R(o)$ be as in \cref{thm:gradphi-large-deviations}. Then, there exists $\epsilon_\beta \to 0$ as $\beta \to \infty$ such that
    \[ e^{-\epsilon_\beta R} \leq \frac{\hatpi_V^{(p)}(\phi_o=h)}{\hatpi_{\cB_R(o)}^{(p)}(\phi_o=h)} \leq e^{\epsilon_\beta R}\,.\] 
\end{lemma}
\begin{proof}
    For any $V_1 \subset V_2$, it is a straightforward FKG argument (see \cite[Eq.~3.14]{LMS16}, or the proof of \cref{lem:UB-LD-any-point}) to show that
    \begin{equation}\label{eq:gradphi-force-bdy-0}
        \hatpi^{(p)}_{V_2}(\phi_o \geq h) \geq e^{-\epsilon_\beta|\partial V_1 \setminus \partial V_2|} \hatpi^{(p)}_{V_1}(\phi_o \geq h)\,.
    \end{equation}
    Moreover, by a standard Peierls argument, (see, e.g., the proof of \cite[Clm.~3.6]{LMS16}) one obtains that for any $V$, $h \geq 0$,
    \begin{equation}\label{eq:gradphi-eq-vs-ineq}
        \hatpi^{(p)}_V(\phi_o > h) \leq \epsilon_\beta\hatpi^{(p)}_V(\phi_o = h)\,.
    \end{equation}
    Combining the above two displays immediately implies the lower bound. 
    
    For the upper bound, by \cref{lem:bd-outermost-loop} it suffices to upper bound $\hatpi^{(p)}_V(\mathscr{W}(\Gamma_1) \leq R,\, \phi_o = h)$ (noting the comment about the constants after \cref{lem:bd-outermost-loop}). The event $\mathscr{W}(\Gamma_1) \leq R$ implies $|\Gamma_1| \leq R$, which in turn implies that there is a chain of sites with height $\phi_x \leq 0$ contained inside $\cB_R(o)$ which surrounds the origin, whose length is at most $R$. By monotonicity, the heights along this chain can be raised to 0. That is, calling the interior of this chain $\Lambda$, we have argued that
    \[\hatpi^{(p)}_V(\mathscr{W}(\Gamma_1) \leq R,\, \phi_o = h) \leq \hatpi^{(p)}_V(\mathscr{W}(\Gamma_1) \leq R,\, \phi_o \geq h) \leq \max_{\substack{o \in \Lambda \subset \cB_R(o)\\
    |\partial \Lambda|\leq R}}\hatpi^{(p)}_\Lambda(\phi_o \geq h)\,.\]
    The upper bound now follows by applying \cref{eq:gradphi-force-bdy-0} for $V_1 = \Lambda$,  $V_2 = \cB_R$, along with~\cref{eq:gradphi-eq-vs-ineq}.
\end{proof}

\begin{proof}[Proof of \cref{thm:gradphi-large-deviations}, \cref{eq:gradphi-LD-conditional}.]
With the above lemma in hand (replacing \cref{clm:hatpi-V-vs-Br} from the case $p=2$), the proof follows from the same argument used to establish \cref{eq:LD-conditional} in \cref{thm:LD-DG}, with the only difference that one defines the two events $E_1$ and $E_2$ for
\[ \Delta := \Big\lfloor \frac{c_1}{6^{p+1}} h^{1\,\wedge\,\frac{1}{p-1}} \Big\rfloor\,,\]
with $c_1>0$ from the upper bound on $\hatpi_V^{(p)}(\phi_o=h)/\hatpi_V^{(p)}(\phi_o=h-1)$, now given by \eqref{eq:gradphi-LD-ratio}. Note that for this choice of $\Delta$ we have
\[ \Delta h^{(p-1)\,\wedge\, 1} \asymp \Delta^p \asymp h^{p\,\wedge\,\frac{p}{p-1}}\,.\]
As was argued in \cref{eq:E1c-upper}, one can iterate the upper bound of \cref{eq:gradphi-LD-ratio} and thereafter apply \cref{lem:gradphi-hatpi-V-vs-Br} to arrive at the analogue here:
\[ \hatpi_V^{(p)}(E_1^c\mid \phi_o = h) \leq \exp\bigg(-(c_1-o(1))\beta \Delta h^{(p-1)\,\wedge\, 1}\bigg)\,.\]
(NB.\ We apply \cref{lem:gradphi-hatpi-V-vs-Br} for $R=Ch^{(p-1)\wedge 1}$ whence the term $e^{\epsilon_\beta R}$ from said lemma is absorbed in the $o(1)$-term in the above exponent.) 
Proceeding as we did below \cref{eq:E1c-upper}, one then has
\[ \hatpi_V^{(p)}(\phi_z = h \mid \phi_o = h,\, E_1,\, E_2^c) \leq (1+\epsilon_\beta)e^{-\beta(3 \Delta^p + (5\Delta)^p)}=(1+\epsilon_\beta)e^{-(5^p+3)\beta \Delta^p}\,,\]
where the constants here, apart from the $p$-power of the interaction, were unchanged as we did not alter the definition of $E_1$ and $E_2$ (but only the definition of $\Delta$). For the same reason we further get
\[ \hatpi_V^{(p)}(\phi_z \geq h+\Delta \mid \phi_o = h,\, E_1,\, E_2) \geq (1-\epsilon_\beta)e^{-4\beta  (6\Delta)^p}
\]
and
\[
	\hatpi_V^{(p)}(\phi_z \geq h+\Delta \mid \phi_o =h,\, E_1\,, E_2) \leq \frac{
				\exp(-(c_1-o(1))\beta \Delta h^{(p-1)\,\wedge\,1})}
			{\hatpi_V^{(p)}(E_2\mid \phi_o=h,\,E_1)}\,,
\]
combining to yield
\begin{align*}
\hatpi_V^{(p)}(E_2 \mid\phi_o=h\,,\,E_1) &\leq (1+\epsilon_\beta)\exp\bigg(\beta\Big(4\cdot 6^p\Delta-(c_1-o(1))h^{(p-1)\,\wedge\,1} \Big) \Delta\bigg) \\ 
&\leq \exp\Big(-(c_1/3)\beta \Delta h^{(p-1)\,\wedge\,1}\Big)\,,    
\end{align*}
as required.
\end{proof}

Analogously to \cref{lem:UB-LD-any-point} from the $p = 2$ case, as a consequence of extending the large deviation results to an arbitrary domain containing $\cB_R(o)$, we obtain an upper bound on the probability that $\phi_x \geq h$ even for points $x$ close to the boundary.
\begin{lemma}\label{lem:gradphi-UB-LD-any-point}
		Fix $1 < p < 2$, or $2 < p < \infty$. There exists $\epsilon_\beta \to 0$ as $\beta \to \infty$ such that, for every $V \subset \Z^2$, $x \in V$, and $h$ sufficiently large compared to $\beta$,
		\begin{equation}
			\hatpi_V^{(p)}(\phi_x \geq h) \leq \hatpi_\infty^{(p)}(\phi_o \geq h)^{1 + \epsilon_\beta}\,.
		\end{equation}
	\end{lemma}
	\begin{proof}
		Consider first $1 < p < 2$. Here we take $R = Ch^{p-1}$. Let $V' = V \cup \cB_R(x)$. By \cref{eq:gradphi-force-bdy-0}, we have
        \[\hatpi_{V'}^{(p)}(\phi_x \geq h) \geq e^{-\epsilon_\beta R^2}\hatpi_V^{(p)}(\phi_x \geq h)\,,\]
        using the fact that $\partial V \setminus \partial V' \subset \cB_R(x)$ so that $|\partial V \setminus \partial V'| \leq \pi R^2$. Combining with \cref{lem:gradphi-hatpi-V-vs-Br} to compare both $\hatpi_V^{(p)}$ and $\hatpi_\infty^{(p)}$ with $\hatpi_{\cB_R(o)}^{(p)}$, we obtain that
        \[\hatpi_V^{(p)}(\phi_x \geq h) \leq \hatpi_\infty^{(p)}(\phi_x \geq h)e^{\epsilon_\beta C h^{2p-2}}\,.\]
        By the lower bound of \cref{eq:gradphi-LD} and the fact as $h \to \infty$, $  h^{2p-2} \ll h^p$ in the regime $1 < p < 2$, we have that
        \[e^{\epsilon_\beta C h^{2p-2}} \leq \hatpi^{(p)}_\infty(\phi_o \geq h)^{o(1)}\,,\]
        where $o(1)$ is as $h \to \infty$. We conclude by combining the above two displays with the translation invariance of $\hatpi_\infty^{(p)}$.

        For $2 < p < \infty$, we can apply the same logic as above with $R = Ch$ to obtain that
        \[\hatpi_V^{(p)}(\phi_x \geq h) \leq \hatpi_\infty^{(p)}(\phi_x \geq h)e^{\epsilon_\beta C h^2}\,.\]
        In this case, the error term is of the same order as the large deviation with respect to $h$ (both are $\exp[O(h^2)]$). We instead use the $\epsilon_\beta$ term to ensure that the error is sufficiently small. That is, by the lower bound of \cref{eq:gradphi-LD} for $2 < p < \infty$, we have
        \[e^{\epsilon_\beta Ch^2} \leq \hatpi_\infty^{(p)}(\phi_o \geq h)^{\epsilon_\beta}\,,\]
        and the proof concludes by the translation invariance of $\hatpi_\infty^{(p)}$.
	\end{proof}

    We now have all the tools needed to prove the crucial bound on the probability of having a floor (analogous to \cref{lem:area-estimate}), which results in the area term in the cluster expansion. Recall the definition of $H^{(p)}$ in \cref{eq:gradphi-def-H-Nn}.
    \begin{lemma}\label{lem:gradphi-area-estimate}
        Fix $1 < p < 2$ or $2 < p < \infty$. Fix $0<\delta<\frac13$ and $n\geq 1$, let $V \subset \Z^2$ be a connected region, let $F\subset V$ be a subset satisfying $|F| \leq Le^{\kappa\sqrt{\log L}}$ and $|\partial F| = O(L^{1-\delta})$, and set $h = H^{(p)}+1-n$. Then, for $\beta$ sufficiently large,
		\[
		\hatpi_V^{(p)}(\phi_x \geq -h,\, \forall x \in F) = (1+o(1))\exp\Big(-\hatpi_\infty^{(p)}(\phi_o < -h)|F| \Big)\,.
		\]
	\end{lemma}
    \begin{proof}
        The proof follows exactly as in \cref{lem:area-estimate}, only plugging in the estimates above when needed. We summarize the crucial relationships between the exponents in \cref{thm:gradphi-large-deviations} for the reader's convenience. 
        
        To lower bound the left hand side, we firstly need that the exponent in (the lower bound of) \cref{eq:gradphi-LD-ratio} is of a smaller order than in \cref{eq:gradphi-LD}. This ensures that for each finite $h = H^{(p)}$ to $h = H^{(p)}+1-n$, $\hatpi^{(p)}_\infty(\phi_o \geq h) = L^{-1 + o(1)}$, a fact used throughout the proof. To control points near $\partial F$, by \cref{lem:gradphi-UB-LD-any-point}, we have that for all $x$,
        \[\hatpi^{(p)}_V(\phi_x < -h) \leq L^{-1 + \epsilon_\beta}\,,\]
        which replaces \cref{eq:x-close-to-boundary}. The $\epsilon_\beta$ in place of $o(1)$ is not an issue as we only needed something smaller than $\delta$ for the application in \cref{eq:sum-dist-log2-from-F} (though it does mean that for every choice of $\delta$, we need to take a sufficiently large $\beta$). The rest of the proof of the lower bound just involved FKG and coupling to infinite volume, holding exactly in the same manner. 

        The upper bound was a more involved ``grill'' estimate. The key equation to prove (analogous to \cref{eq:grill-good-points}) is 
        \begin{equation}\label{eq:gradphi-grill-good-points}\hatpi_V^{(p)}\bigg(\fD_\fs\cap \bigcap_{x \in Q'_{i_j}}
			\{\phi_x \geq -h\} \;\Big|\; \cF_{j-1}\bigg) \leq \exp\Big(-\Big(1 -e^{-c\beta h^{p\,\wedge\,\frac{p}{p-1}}}\Big)\hatpi_\infty^{(p)}(\phi_o < -h)|Q'_{i_j}| \Big)\,,
        \end{equation}
        where the definitions of $\fD_\fs, Q'_{i_j}, \cF_j$ all remain the same. The exponent is motivated by the bound on the conditional probability in \cref{eq:gradphi-LD-conditional}. Indeed, the justification for the above equation is to use Bonferroni's inequalities, with the main term to estimate being
        \[\sum_{\substack{x,y\in Q'_{i_j}\\ 0<\dist(x,y)\leq \log L}} \!\!\!\!\!\hatpi_\infty^{(p)}(\phi_x<-h,\, \phi_y<-h) \leq (\log L)^2\exp(-c\beta h^{p\,\wedge\,\frac{p}{p-1}})\hatpi_\infty^{(p)}(\phi_o < -h)|Q'_{i_j}|\,,\]
        and this follows directly from \cref{eq:gradphi-LD-conditional}, \cref{eq:gradphi-LD}, and the lower bound on \cref{eq:gradphi-LD-ratio}.
        We highlight that 
        when justifying this point, as was done in the proof of \cref{lem:area-estimate}, one needs that 
        \begin{enumerate}[(a)]
            \item the probability $\exp[-c \beta h^{(p-1)\,\wedge\, 1}]$ given
        in (the lower bound of) \cref{eq:gradphi-LD-ratio} is of a larger order than the probability $\exp[-c\beta h^{p\,\wedge\,\frac{p}{p-1}}]$ from \cref{eq:gradphi-LD-conditional}, and is thus absorbed (see \cref{eq:grill-absorbing-cond}); 
        \item the bound on $|F|$ ensures  $|F|\hatpi_\infty^{(p)}(\phi_o < -H^{(p)}) = o(\exp[-c\beta h^{p\,\wedge\,\frac{p}{p-1}}])$ (again, cf.~\cref{eq:grill-absorbing-cond});
        \item for $h=H^{(p)}+1-n$, which satisfies $h \asymp(\beta \log L)^{1/(p\,\wedge\,2)}$ by \cref{eq:gradphi-LD}, the quantity
        $\exp[-c \beta h^{(p-1)\,\wedge\, 1}]$ from
        the (lower bound of) \cref{eq:gradphi-LD-ratio} 
        outweighs the $(\log L)^2$ pre-factor. 
        \end{enumerate}
        The remaining computations are unchanged.
    \end{proof}
	
    We can now establish the cluster expansion law. We use the same definitions of disagreement polymers as per \cref{def:gamma,def:disagree-polymer-len-energy-decor}, with the exception that the energy of $\gamma$ is now defined to be
	\[ \sE_\beta(\gamma) = \beta \sum_{e\in\gamma}|(\nabla\phi)_e|^p\,.\]
	With this definition, we have the following analog of \cref{prop:CE-law-with-floor}. 
	
	\begin{proposition}[Cluster expansion with a floor for $|\nabla\phi|^p$]\label{prop:gradphi-CE-with-floor}
		In the setting of the $|\nabla\phi|^p$ model for fixed $p>1$, with $\pi_V^{(p),\xi}$ replacing $\pi_V^\xi$ and $\sE_\beta$ as above, the statement of
		\Cref{prop:CE-law-with-floor} holds unchanged.
	\end{proposition}
	\begin{proof}
		The proof is identical as in \cref{prop:CE-law-with-floor}. The cluster expansion without a floor only requires that $\cE_\beta(\gamma) \geq \sN(\gamma)$, where $\sN(\gamma)$ is the number of bonds in $\gamma$, and we have this. Obtaining from this the cluster expansion with a floor requires only the area estimate of \cref{lem:gradphi-area-estimate}.
	\end{proof}

    \begin{proof}[Proof of {\cref{thm:grad-phi-p}}]
    
    With \cref{prop:gradphi-CE-with-floor} replacing \cref{prop:CE-law-with-floor}, we can proceed through \cref{sec:geom-disagreement-polymers,sec:initial-UB-JEMS,sec:UB,sec:LB} just as in the case $p = 2$, since the only property needed from the polymer model was that the energy $\cE^*_\beta(\gamma)$ satisfies \cref{eq:energy-property}. Otherwise, the only properties we used of the actual measure $\pi_\Lambda$ were the FKG inequality and basic Peierls arguments. We simply highlight that in part (b) of \cref{thm:level-line-contains-Wulff}, the initial upper bound on the displacement of the $(H-n)$ line is obtained by the estimate of the ratio $N_{n+1}/N_n$ coming directly from the upper bound in \cref{eq:LD-ratio}, now replaced with the upper bound in \cref{eq:gradphi-LD-ratio} at $h = H^{(p)}+1-n$. What is needed is that this ratio is smaller than (and easily absorbs) $(\log L)^{16}$, which holds for all $1 < p < \infty$.
    \end{proof}
	\appendix
	
	\section{Monotonicity of \texorpdfstring{$|\nabla\phi|^p$}{grad-phi} models with non-uniform floors and ceilings}

It is well-known that the $|\nabla\phi|^p$ measures enjoy the FKG inequality for any $p\geq 1$; we will use this fact in the presence of non-uniform floors and ceilings, and we include the short proof for completeness.
    \begin{claim}\label{clm:grad-p-monotone}
        Let $\pi_V^{(p),\xi,\mathbf{a},\mathbf{b}}$ be the $|\nabla \phi|^p$ model ($p\in[1,\infty)$ and $\beta>0$) on a domain $V$ with boundary conditions $\xi$ and arbitrary floors and ceilings $\mathbf{a},\mathbf{b}\in (\Z\cup\{\pm\infty\})^V$ (i.e., $\mathbf{a}_x\leq \phi_x\leq \mathbf{b}_x$ for all~$x\in V$). This measure satisfies the FKG lattice condition. In particular, if $\xi \leq \xi'$, $\mathbf{a}\leq \mathbf{a}'$ and $\mathbf{b}\leq \mathbf{b}'$ pointwise, then $\pi_V^{(p),\xi,\mathbf{a},\mathbf{b}}\preceq \pi_V^{(p),\xi',\mathbf{a}',\mathbf{b}'}$.
    \end{claim}
    \begin{proof}
    Write $\pi=\pi_V^{(p),\xi,\mathbf{a},\mathbf{b}}$ for brevity.
    It suffices (see, e.g.,~\cite[Thm.~2.22]{Grimmett_RC}) 
    to verify the condition 
\begin{equation}\label{eq:fkg-lattice}\pi(\phi_1\vee\phi_2)\pi(\phi_1\wedge\phi_2)\geq \pi(\phi_1)\pi(\phi_2)\end{equation} for two configurations $\phi_1,\phi_2$ that differ on exactly $2$ sites $x,y\in V$. We may further assume that $x, y$ are nearest-neighbors (or there is equality in \cref{eq:fkg-lattice}, as all  interactions are nearest-neighbor). Writing $a_i:=\phi_i(x)$ and $b_i := \phi_i(y)$, we may further assume without loss of generality that $a_1 >a_2$ and $b_1<b_2$ (or else again there is trivially an equality). The above then translates into showing 
\[ |a_1-b_2|^p + |a_2-b_1|^p \leq |a_1-b_1|^p + |a_2-b_2|^p\,.\]
Writing $k:= a_1-b_1$, $\Delta_a := a_1-a_2>0$ and $\Delta_b := b_2-b_1>0$, this is equivalent to 
    \[ |k-\Delta_b|^p + |k-\Delta_a|^p \leq  |k - \Delta_a- \Delta_b|^p + |k|^p \,.\]
If $\Delta_a>\Delta_b$ then we have 
 $[|k-\Delta_a|^p - |k-\Delta_a-\Delta_b|^p]/\Delta_b <
[|k|^p - |k-\Delta_b|^p]/\Delta_b
$ by the monotonicity of slopes due to
the convexity of the function  $x\mapsto |x|^p$ for $p\geq 1$. Similarly, when $\Delta_b>\Delta_a $ the monotonicity of slopes implies 
$[|k-\Delta_b|^p - |k-\Delta_a-\Delta_b|^p]/\Delta_a <
[|k|^p - |k-\Delta_a|^p]/\Delta_a
$, thus \cref{eq:fkg-lattice} holds in both cases.
The final conclusion of the claim is a consequence of FKG in $\pi_V^{(p),\xi,\mathbf{a},\mathbf{b}'}$ if $\xi=\xi'$, whereas the statement for $\xi\neq\xi'$ follows as we may implement $\xi$ via $\mathbf{a}_x=\mathbf{b}_x=\xi_x$ for all $x\in\partial V$ (in the measure on $V\cup\partial V$ with, say, $0$ boundary conditions).
\end{proof}    
     
    \section{Random boundary construction}\label{sec:random-boundary}

In this section, we prove \cref{lem:boundary-construction,lem:LB-boundary-construction}. We begin with the following fact, which follows easily by a Peierls argument. 
\begin{observation}[e.g.,~{\cite[Eq.~(4.2)]{LMS16}}]\label{obs:dist-bdy-R}
     Let $\cC_x$ be the connected component of~$x$ of sites $y$ with $\phi_y \leq h-1$. W.h.p.\ under $\pi^k_V$ ($k$ can be any uniform b.c.), the interior of an $h$ line does not contain $x$ with $|\cC_x|>\log L$.
\end{observation}
\begin{proof}[Proof of \cref{lem:boundary-construction}]
        Start by revealing all of $\phi\restriction_{\Lambda \setminus R}$. Now on $R$, we will continue to reveal $\phi$ at certain sites and lower some of the revealed heights in order to obtain $Q, \xi$ in the manner described below, and this defines the $\pi^k_V$-measurable distribution on $Q$. 
        
        As the procedure will reveal all of $\phi\restriction_{Q^c}$, the boundary conditions $\xi$ leave two cases to consider: first, that $\fL_n$ does not intersect $R$ at all. Then, either $R$ is entirely in the interior of $\fL_n$ or entirely in the exterior, and the requirement that $A, B \in \partial R$ are in the interior of $\fL_n$ rules out the latter possibility. Hence, $\cA_1$ is all of $Q$, so $\cA_2 \subset \cA_1$ deterministically. In the second case, $\fL_n$ intersects $R$, in which case $\xi$ implies that the restriction $\fL_n\restriction_Q$ coincides with the $H+1-n$ line in $\pi^\xi_Q$. In this case, the law of $\fL_n$ has also been changed due to the lowering of heights in the revealing procedure. However, since the area inside a line is an increasing function of $\phi$, we obtain that $\cA_2 \subset \cA_1$ by monotonicity.
                 
        The procedure to construct $Q, \xi$ is as follows: 
	\begin{enumerate}
        \item Call $\partial^1 R, \partial^2 R$ the top and bottom arcs of $\partial R$ from $A$ to $B$, respectively. For every $x$ on $\partial^1R$ with $\phi_x \leq H-n$, reveal the connected component $\cC^1_x$ of sites $y$ with $\phi_y \leq H-n$ containing $x$. Similarly, for every $x$ on $\partial^2R$ such that $\phi_x \leq H-n-1$, reveal the connected component $\cC^2_x$ of sites $y$ with $\phi_y \leq H-n-1$ containing $x$. 
		
		\item By maximality, the sites in the outer boundary of  $\cC^1_x$ have height $\geq H+1-n$, which we can drop by monotonicity to be exactly $H+1-n$. Similarly, the sites in the outer boundary of the $\cC^2_x$ have height $\geq H-n$, which we can drop to be exactly $H-n$.
		
		\item The remaining sites in $\partial^1R$ all have height $\geq H+1-n$, which can be lowered to $H+1-n$ by monotonicity. Similarly, the remaining sites in $\partial^2R$ have height $\geq H-n$, which can be lowered to $H-n$.
		
		\item By construction, every $x \in \partial^1R$ is now either at height $H+1-n$, or enclosed in a boundary of sites with height $H+1-n$. Similarly, every $x \in \partial^2R$ is at height $H-n$ or enclosed in a boundary of sites with height $H-n$. Hence, the union of these sites at height $H+1-n$ and $H-n$ enclose a simply connected domain $Q$ satisfying the conditions of \cref{lem:boundary-construction} except for \cref{it:good-boundary-points}. (Here, we use \cref{obs:dist-bdy-R} to ensure \cref{it:dist-Q-R-bdy}.)
		
		\item To obtain \cref{it:good-boundary-points}, let $A^{(0)}=(A^{(0)}_1,A^{(0)}_2)$ be the rightmost point on the left side of $\partial Q$ such that $A^{(0)}_2 = A_2 + (\log L)^5 + (\log L)^2$. We start the following iterative procedure starting with $A^{(0)}$. Look at the $(\log L)^2 \times 2(\log L)^3$ rectangle with $A^{(0)} + (1, 0)$ as the midpoint of the left side of the rectangle. If there is no point in $\partial Q$ which intersects the rectangle, then set $A' = A^{(0)}$. Otherwise, let $A^{(1)}$ be the rightmost intersection of $\partial Q$ with the rectangle. Then, repeat starting with $A^{(1)}$, continuing until we set $A' = A^{(k)}$ for some $k$. This procedure must stop after at most $(\log L)^2$ steps because each $A^{(i+1)}$ is to the right of $A^{(i)}$, while $\partial Q$ only deviates at most $(\log L)^2$ distance from the left side of $\partial R$. Moreover, the change in the $y$-coordinate from $A^{(i)}$ to $A^{(i+1)}$ is at most $(\log L)^3$ at each step, so that $|A'_2 - A^{(0)}_y| \leq (\log L)^5$. In particular, $A'_2 \geq A_2 + (\log L)^2$, so that by \cref{obs:dist-bdy-R}, $A_2$ lies on the arc of $\partial Q$ with boundary condition $H+1-n$. 
		
		\item Repeat the above procedure analogously on the right side of $Q$ to obtain the point $B$. Finally, lower the boundary conditions from $H+1-n$ to $H-n$ on the arcs of $Q$ between $A$ and $A'$, and $B$ and $B'$. 
	\end{enumerate}	

        The fact that $\sC(\overline{A'B'}) \subset Q$ follows because of the assumption that $A, B$ are at least distance $\sqrt{\ell_1}(\log L)^2$ away from the top or bottom of $R$, and the shifts by up to $(\log L)^5$ to $A', B'$ are irrelevant as $\ell_1 \gg (\log L)^6$. It is straightforward to check that the remaining conditions in \cref{it:good-boundary-points} are now satisfied by $A', B'$. 
\end{proof}

We turn now to prove \cref{lem:LB-boundary-construction}, which is essentially the same as \cref{lem:boundary-construction} but with reversed monotonicity. The minor complication is that the $C_x$ components we wish to reveal now will now be of the form $\{y: \phi_y \geq h\}$. The Peierls argument that deletes the contour $\partial C_x$ would now want to shift the sites in $C_x$ down, but this may run into problems with the conditioning that $\phi \geq 0$. This was handled in \cite{LMS16}, which we recall here.
\begin{observation}[{\cite[Theorem 2 and Rem. 1.5]{LMS16}}]\label{obs:LB-dist-bdy-R}
    With high probability under $\pi^0_{\Lambda}$, for each $h = 0, \ldots, H$, there is at most one \texttt{large} $h$ line, and there is no \texttt{large} $H+1$ line. In particular, if $C_x$ is the connected component of $x$ of sites $y$ with $\phi_y \geq H+2-n$, then for $n \geq 2$, w.h.p.\ the exterior of $\fL_{n-1}$ does not contain $x$ with $|C_x| > \log L$. For $n = 1$, w.h.p.\ there are no points $x$ with $|C_x| > \log L$.
\end{observation}

\begin{proof}[Proof of \cref{lem:LB-boundary-construction}]
As described in the beginning of the proof of \cref{lem:boundary-construction}, by monotonicity it suffices to construct $Q, \xi$ via a procedure which reveals $\phi\restriction_{Q^c}$ under $\pi^0_{\Lambda}$ and raises some of the revealed heights. The procedure to construct $Q, \xi$ is as follows:
\begin{enumerate}
    \item Call $\partial^1 R$ the top and sides of $\partial R$. For every $x$ on $\partial^1 R$ with $\phi_x \geq H+2-n$, reveal the connected component $C_x$ of sites $y$ with $\phi_y \geq H+2-n$ containing $x$.

    \item By maximality, the sites in the outer boundary of $C_x$ have height $\leq H+1-n$, which we can raise by monotonicity to be exactly $H+1-n$. 

    \item Call $\partial^2 R$ the bottom side of $\partial R$ in between $A, B$. Call $\partial^3 R$ the remainder of the bottom side of $\partial R$. Since all the boundary heights on $\partial^2 R \cup \partial^3 R$ are 0, by monotonicity we can raise the boundary condition on $\partial^2 R$ to $H-n$ and on $\partial^3 R$ to $H+1-n$.

    \item By construction, every $x \in \partial^1 R$ is now either at height $H+1-n$, or enclosed in a boundary of sites with height $H+1-n$. Hence, the union of these sites at height $H+1-n$ with $\partial^2 R$ and $\partial^3 R$ enclose a simply connected domain $Q$ satisfying the conditions of \cref{lem:LB-boundary-construction}, using \cref{obs:LB-dist-bdy-R} to ensure \cref{it:LB-dist-Q-R-bdy}.\qedhere
\end{enumerate}
\end{proof}

\subsection*{Acknowledgments}
This research was supported by the NSF grant \textsc{dms}-2451083.

	\bibliographystyle{abbrv}
	\bibliography{dg_limit_ref}

\end{document}